\documentclass[a4paper, oneside,11pt]{amsart}

\usepackage[english]{babel}
\usepackage{mathrsfs,amssymb}
\usepackage{mathtools}
\usepackage[colorlinks, citecolor = blue]{hyperref}
\usepackage[shortlabels]{enumitem}
\usepackage{esint}
\usepackage[dvipsnames]{xcolor}

\usepackage{todonotes}
\usepackage{fullpage}

\usepackage{tikz-cd}
\usetikzlibrary{calc,arrows.meta}

\newtheorem{theorem}{Theorem}[section]
\newtheorem{lemma}[theorem]{Lemma}
\newtheorem{prop}[theorem]{Proposition}
\newtheorem{cor}[theorem]{Corollary}

\theoremstyle{definition}
\newtheorem{definition}[theorem]{Definition}

\theoremstyle{remark}
\newtheorem{remark}[theorem]{Remark}

\numberwithin{equation}{section}

\newcommand{\mE}{{\mathcal E}}
\newcommand{\ta}{{\scriptscriptstyle \parallel}}
\newcommand{\no}{{\scriptscriptstyle\perp}}
\newcommand{\scl}[2]{\langle #1,#2 \rangle}
\newcommand{\dom}{\textsf{D}}
\newcommand{\sgn}{\mathrm{sgn}}

\DeclareMathOperator*{\esssup}{ess\,sup}

\DeclareMathOperator{\ch}{ch}
\DeclareMathOperator{\diam}{diam}
\DeclareMathOperator{\supp}{supp}

\DeclareMathOperator{\mat}{Mat}
\DeclareMathOperator{\Div}{div}
\newcommand{\Cc}{\ensuremath{C^\infty_c}}
\DeclareMathOperator{\tr}{Tr}
\DeclareMathOperator{\id}{Id}
\usepackage{forloop}
\newcommand{\defcal}[1]{\expandafter\newcommand\csname 
	c#1\endcsname{{\mathcal{#1}}}}
\newcommand{\defbb}[1]{\expandafter\newcommand\csname 
	b#1\endcsname{{\mathds{#1}}}}
\newcommand{\defbf}[1]{\expandafter\newcommand\csname 
	bf#1\endcsname{{\mathbf{#1}}}}
\newcommand{\defrm}[1]{\expandafter\newcommand\csname 
    rm#1\endcsname{{\mathrm{#1}}}}
\newcommand{\deffr}[1]{\expandafter\newcommand\csname 
	frak#1\endcsname{{\mathfrak{#1}}}}
\newcommand{\defov}[1]{\expandafter\newcommand\csname 
	ov#1\endcsname{{\overline{#1}}}}
\newcommand{\deftil}[1]{\expandafter\newcommand\csname 
	til#1\endcsname{{\widetilde{#1}}}}
\newcommand{\defhat}[1]{\expandafter\newcommand\csname 
	hat#1\endcsname{{\widehat{#1}}}}
\newcommand{\defscr}[1]{\expandafter\newcommand\csname 
	scr#1\endcsname{{\mathscr{#1}}}}
\newcommand{\defdot}[1]{\expandafter\newcommand\csname 
	Dot#1\endcsname{{\Dot{#1}}}}
\newcommand{\defbold}[1]{\expandafter\newcommand\csname 
	bold#1\endcsname{{\boldsymbol{#1}}}}
\newcommand{\deful}[1]{\expandafter\newcommand\csname 
	ul#1\endcsname{{\underline{#1}}}}

\newcounter{calBbCounter}
\forLoop{1}{26}{calBbCounter}{
	\edef\letter{\Alph{calBbCounter}}
	\expandafter\defcal\letter
	\expandafter\defbb\letter
	\expandafter\defbf\letter
    \expandafter\defrm\letter
	\expandafter\deffr\letter
	\expandafter\defov\letter
	\expandafter\deftil\letter
	\expandafter\defhat\letter
	\expandafter\defscr\letter
    \expandafter\defdot\letter
    \expandafter\defbold\letter
    \expandafter\deful\letter
}
\forLoop{1}{26}{calBbCounter}{
	\edef\letter{\alph{calBbCounter}}
	\expandafter\defbf\letter
    \expandafter\defrm\letter
	\expandafter\deffr\letter
	\expandafter\defov\letter
	\expandafter\deftil\letter
	\expandafter\defhat\letter
	\expandafter\defscr\letter
    \expandafter\defdot\letter
    \expandafter\defbold\letter
    \expandafter\deful\letter
}

\newcommand{\boldtilF}{\ensuremath{\boldsymbol{\widetilde{F}}}}
\newcommand{\boldtilC}{\ensuremath{\boldsymbol{\widetilde{C}}}}

\newcommand{\Ltx}[2]{%
  L^{\vphantom{\smash[d]{#1}\smash[d]{#2}}\smash[d]{#1}}_{\vphantom{tx}t}%
  L^{\vphantom{\smash[d]{#1}\smash[d]{#2}}\smash[d]{#2}}_{\vphantom{tx}x}%
}
\newcommand{\Lx}[1]{L^{#1_{\vphantom{x}}}_{\vphantom{t}x}}

\newcommand{\N}{\ensuremath{\mathbb{N}}}
\newcommand{\Z}{\ensuremath{\mathbb{Z}}}
\newcommand{\Q}{\ensuremath{\mathbb{Q}}}
\newcommand{\R}{\ensuremath{\mathbb{R}}}
\newcommand{\C}{\ensuremath{\mathbb{C}}}

\newcommand{\mc}{\mathcal}
\newcommand{\ms}{\mathscr}

\newcommand{\mbs}{\boldsymbol}

\renewcommand\vec{\mbs}

\DeclarePairedDelimiter\abs{\lvert}{\rvert}
\DeclarePairedDelimiter\brac[]
\DeclarePairedDelimiter\cbrace\{\}
\DeclarePairedDelimiter\br()
\DeclarePairedDelimiter{\ip}\langle\rangle
\DeclarePairedDelimiter{\nrm}\lVert\rVert

\newcommand{\nrmb}[1]{\bigl\|#1\bigr\|}
\newcommand{\absb}[1]{\bigl|#1\bigr|}
\newcommand{\brb}[1]{\bigl(#1\bigr)}
\newcommand{\cbraceb}[1]{\bigl\{#1\bigr\}}
\newcommand{\ipb}[1]{\bigl\langle#1\bigr\rangle}
\newcommand{\bracb}[1]{\bigl[#1\bigr]}

\newcommand{\nrmB}[1]{\Bigl\|#1\Bigr\|}
\newcommand{\absB}[1]{\Bigl|#1\Bigr|}
\newcommand{\brB}[1]{\Bigl(#1\Bigr)}

\newcommand{\dd}{\hspace{2pt}\mathrm{d}}
\newcommand{\ddn}{\mathrm{d}}

\newcommand{\loc}{\mathrm{loc}}
\newcommand{\whit}{\mathrm{w}}
\newcommand{\ca}{\mathrm{ca}}

\newcommand{\Qwtilde}[1]{\widetilde{\mbs{#1}}\vphantom{\mbs{#1}}^{\whit}}
\newcommand{\Qcatilde}[1]{\widetilde{\mbs{#1}}\vphantom{\mbs{#1}}^{\ca}}
\newcommand{\Qw}[1]{\mbs{#1}^{\whit}}
\newcommand{\Qca}[1]{\mbs{#1}^{\ca}}

\DeclareMathOperator{\ind}{\mathbf{1}}

\allowdisplaybreaks

\author{Pascal Auscher}
\address{Pascal Auscher\\
Universit\'e Paris-Saclay\\ CNRS\\ Laboratoire de Math\'{e}matiques d'Orsay\\ 91405 Orsay\\ France}
\email{pascal.auscher@universite-paris-saclay.fr}

\author{Hedong Hou}
\address{Hedong Hou\\ Westlake Institute for Advanced Study\\ Westlake University\\ 310030 Hangzhou\\ China}
\email{houhedong@westlake.edu.cn}

\author[Lorist]{Emiel Lorist}
\address{Emiel Lorist \\
Delft Institute of Applied Mathematics \\
Delft University of Technology \\
P.O. Box 5031 \\
2600 GA Delft, The Netherlands}
\email{e.lorist@tudelft.nl}

\author{Andreas Ros\'en}
\address{Andreas Ros\'en\\ Mathematical Sciences\\ Chalmers University of Technology and University of Gothenburg\\
SE-412 96 G{\"o}teborg\\ Sweden}
\email{andreas.rosen@chalmers.se}

\begin{document}
\title{Form domination and tent space estimates for operators on the half-space}

\begin{abstract} 
We study operators acting on functions defined on the half-space, with methods inspired by 
sparse domination. Using the specific link between dyadic grids on the boundary and Whitney regions on the half-space, we obtain a  very precise form domination with model operators of Hardy type. 
We introduce mixed-norm off-diagonal estimates that measures  both tangential and transversal decay without requiring pointwise control in the transversal variable, substantially weakening the assumptions used in earlier tent space theories.
This allows us to prove optimal weighted tent space extrapolation from local boundedness plus off-diagonal decay. The resulting framework yields new bounds for operators arising from non-autonomous elliptic and parabolic PDEs with non-smooth coefficients.    
\end{abstract}




\thanks{This work was supported by CNRS through the International Emerging Action program. We thank our institutions for their hospitality and support, which enabled us to meet at various stages of the project. During the final stages of this work, Pascal Auscher was based at the France Australia Mathematical Sciences and Interactions International Research Lab, Australian National University--CNRS, Canberra, ACT 2601, Australia.
Andreas Rosén was supported by the Swedish Research Council (Grant 2022-03996).
Emiel Lorist was partially supported by the Dutch Research Council (NWO) through the project ``The sparse revolution for stochastic partial differential equations'' under project number \href{https://doi.org/10.61686/ZGRMR99948}{VI.Veni.242.057}.
Hedong Hou would like to thank Alexey Cheskidov for kindly providing the funding for his trip to Europe.}
\thanks{A CC-BY 4.0 \url{https://creativecommons.org/licenses/by/4.0/} public copyright license has been applied by the authors to the present document and will be applied to all subsequent versions up to the Author Accepted Manuscript arising from this submission.}

\maketitle

\section{Introduction}
In this work, we study the boundedness of operators on (weighted) tent spaces, motivated by applications to elliptic and parabolic PDEs. Our main objective is to derive a range of tent space estimates by combining a local boundedness estimate with suitable off-diagonal decay. Several earlier works have addressed related aspects of this problem. In \cite{AMP12}, maximal regularity operators were studied; subsequent articles \cite{AKMP12, AH25a} developed a singular integral operator (SIO) theory on weighted tent spaces for operators formally given by 
$$
Tf(t,x)= {\int_0^\infty} (K(t,s)f(s,\cdot))(x)\dd s,\qquad {(t,x)\in\R^{1+n}_+,}
$$
and applied it, in particular, to autonomous problems (including maximal regularity operators), see also \cite{AH25b}. Early applications of tent space estimates to stochastic PDEs can be found in \cite{AvNP14, PV19}. This SIO theory
does not cover operators whose operator-valued distributional kernel $K(t,s)$ is supported on the diagonal $t=s$, which we call \emph{slice operators}. An important class of examples consists of tensor extensions to $\R^{1+n}_+$ of operators on $\R^n$, for which tent space estimates were obtained in \cite{APr17,MP24} using Rubio de Francia extrapolation. The operators arising in the non-autonomous problems studied in \cite{AMP19,Z20,AP25,Hou25} likewise fall outside this SIO theory; the corresponding tent space estimates are obtained instead using PDE techniques. Further results on tent space boundedness are in \cite{Hua17,Hua18}. Hence, the panorama is diverse and lacks a systematic, unified approach.

Of particular relevance to our work is the Carleson measure estimate for anti-causal operators proved in \cite{HR23}, using sparse domination methods from Calder\'on--Zygmund theory. The result was motivated by Beurling-type operators and elliptic PDEs. Anti-causality (downward causality in the terminology of \cite{HR23}) means that if the input vanishes after a given time, then so does the output.
Causality (upward causality) is obtained by reversing time:
if the input vanishes before a given time, then so does the output. The proof in \cite{HR23} assumes pointwise kernel bounds, and the authors explicitly ask whether these can be replaced by off-diagonal estimates in view of potential applications.
  
In seeking to connect these works, our intention was to see whether there is a sparse domination method to prove tent space boundedness, improve on existing results proved by different techniques and obtain this unified framework. Sparse domination relies on the existence of a dyadic system. The half-space can be equipped with such a system, but it also has specific geometric features that can be exploited. As we shall see, this  ``trivializes'' the sparse domination in a sense. We will prove a form domination in the spirit of \cite{BFP16} for operators acting on functions defined on the half-space.  This yields the desired unified framework with many improvements on boundedness results, except for the slice operators mentioned above.

Concerning the geometric features of the half-space,    the first observation is that one can take
advantage of a chosen dyadic grid $\mc{D}$ on its boundary to create the Whitney regions and the Carleson boxes. We shall see that these are the only geometric objects  we need to deal with. For example, as exploited in \cite{HR23}, in elliptic scaling, the set of Carleson boxes $$\mbs{Q}^{\ca}:=(0,\ell(Q))\times Q,\qquad Q\in \mc{D}$$ is a sparse family, as they contain the mutually disjoint Whitney regions  $$\mbs{Q}^{\whit}:=(\ell(Q)/2,\ell(Q))\times Q,\qquad Q\in \mc{D}$$ and $|\mbs{Q}^{\whit}|= \tfrac 1 2 |\mbs{Q}^\ca|$.  Families of  Whitney regions themselves are disjoint and thus  trivially sparse. 
Next, sparse domination relies on the existence of simple model operators whose bounds are easily established. The definition of these model operators in our setting follows from the second observation that there are specific directions. In the tangential directions, we benefit from the  harmonic analysis tools available on $\R^n$ and, more generally, on any  space of homogeneous type, notably averaging operators. The transverse variable, instead, belongs to the semi-infinite interval $(0,\infty)$ and this makes a huge difference as the origin plays a special role. Therefore, the relevant averaging operators  will be extensions of  the ``vertical'' Hardy operator 
$$
Hf(t,x)=\frac 1 t \int_{0}^t f(s,x)\, \dd s
$$
 and its variant, not seeing the origin but with the same scaling, 
 $$
 \tilde Hf(t,x)= \frac 2 t \int_{t/2}^t f(s,x)\, \dd s.
 $$
 We combine these operators with an average in the $x$-variable at a scale related to $t$, which we denote by
  $H^{\ca}_{\vec{u},\delta}$ and $H^{\whit}_{\vec{u},\delta}$ respectively, as we integrate over a Carleson box for the first one and over a Whitney region for the second one. Here $\delta>0$ is a thickness parameter expressing how horizontally stretched the regions of integration are at the scale given by $t$ and  $\vec{u}=(u_{t},u_{x})\in [1,\infty]^2$ denotes a  pair of Lebesgue exponents  with which we consider mixed-norm  $L^{u_t}_tL^{u_x}_{\vphantom{t}x}$-averages instead of $L^1_{t,x}$-averages. 
     
 We are able to handle general subadditive operators $T$ and prove form  domination estimates of $\int_{{\mathbb R}^{1+n}_+}|Tf| |g|$ by a localized term and an error term, see Theorem~\ref{thm:maindominationform}. A special instance for  locally $L^r$-bounded linear operators of order $\kappa=0$ in the elliptic scaling is  
 \begin{align}\label{eq:domintro}
\absB{\int_{\R^{1+n}_+} Tf\cdot g - \sum_{Q \in \mc{D}} \int_{\mbs{Q}^{\whit}}T(f\ind_{\widetilde{\mbs{Q}}^{\whit}})\cdot g}
\lesssim \sum_{j\ge 0} 2^{-j\gamma_T} \mathcal{E}_j(f,g),\qquad\qquad
\intertext{where $\widetilde{\mbs{Q}}^{\whit}$ is an enlarged Whitney region, $\gamma_T>0$ depends on the assumed off-diagonal decay of $T$, measured in mixed Lebesgue norms, and}
\mathcal{E}_j(f,g):=
\int_{{\mathbb R}^{1+n}_+} \absb{H^{\whit}_{(r,u_x),2^j}f}|H^{\whit}_{(r',v_x'), 2^j}g|+  \absb{H^{\ca}_{\vec{u},2^j}f}|H^{\whit}_{\vec{v}', 2^j}g|+\absb{H^{\whit}_{\vec{u},2^j}f}|H^{\ca}_{\vec{v}', 2^j}g|. \notag
\end{align}
We call the sum in $j$ the decay term of the domination.  Operators that preserve support within the Whitney regions have no decay term. We shall see that for specific classes of operators,  the integral in the decay term simplifies. Namely,  for causal operators, the third integrand is absent; for anti-causal, the second is absent.  
It is  important in applications that the estimate regularizes in $t$, i.e., we need $u_t<v_t$, and this is why our domination applies to all situations but the slice operators,
where $u_t=v_t$. 

 The decomposition \eqref{eq:domintro} resembles an almost diagonalization and is reminiscent of the $T(1)$ decomposition for singular integrals, where we need a weak (local) boundedness property and control of the error (decay) terms. 
Indeed, heuristically, if we,  through a Calder\'on formula,  set up a continuous wavelet transform between functions $h:\R^n\to\R$ and
 $f:\R^{1+n}_+\to\R$, then the action of our $T$ on $f(t,x)$ would
 correspond to the action of a usual $\R^n$ singular integral $S$ 
 on $h(x)$. Thus, neglecting the local variations of $f$ on Whitney regions, $T$ describes the action of $S$ on the wavelet coefficients of $h$.
 Moreover, under such a correspondence, causality of $T$ would correspond to $S1=0$, i.e., the downward mapping paraproduct part of $S$ vanishes. And similarly $T$ being anti-causal would correspond to $S^*1=0$. However, for general $T$ we do not necessarily have such a relation to an operator $S$ on $\R^n$.

\medskip
 
Using \eqref{eq:domintro} we can perform extrapolation. If we take $f$ in a weighted tent space $T^{p,q}_{\beta}$ (and even its generalization with Whitney averages, denoted by $T^{p,q,r}_{\beta}$) and $g$ in its dual,  the localized term only requires local, uniform $L^r$-boundedness of $T$ on Whitney regions and control of the norm of Whitney sums $\sum_{Q \in \mc{D}} a_Q \ind_{{\mbs{Q}^{\whit}}}$ in weighted tent spaces, which is rather elementary.  For the decay term, we need boundedness of the operators $H^{\ca}$ and $H^{\whit}$  and growth control of their operator norms in order to find  exponents $\beta\in \R$, $p,q\in (0,\infty]$ for which the last sum converges. That is, for an explicit range of the parameters $\beta, p,q$, we  extrapolate from local $L^r$-bounds to global $T^{p,q,r}_{\beta}$ bounds. The detailed statement can be found in Theorem~\ref{thm:mainestimate}, which is the main result of this paper. The conclusion is optimal for the class of operators considered in the statement.
In particular when $p=q=r$, the space $T^{p,q,r}_{\beta}$ is a weighted $L^r$-space and   we obtain global weighted $L^r$-boundedness for a range of $\beta$ as a consequence. This method not only applies to   operators  of order zero but also to   operators of non-zero order.  Examples are the regularizing Duhamel operators arising from evolution equations.  As we are able to treat subadditive operators, we can also handle the quasi-Banach range $\min\{p,q\}<1$, avoiding the use of atomic theory, relying instead on power rules in tent spaces. 
For the main tent space operator class and parameter ranges of Theorem~\ref{thm:mainestimate}, we are not aware of prior results when \(q\neq2\) or \(r\neq2\). At \(q=r=2\) our results overlap with and improve specific ranges in \cite{AP25, Hou25}, and the earlier SIO literature when, in addition, $u_t=v_t'=1$.

When $q=r=2$, a key feature of the SIO method was to replace $R$-bounds by off-diagonal estimates for the operator-valued kernel $K(t,s)$. This assumed that these kernels  had, for each $s\ne t$, decay estimates depending on $t-s$, measured in the $L^{u_x}(E_x)$ to $L^{v_x}(F_x)$ operator norm with $u_x\le 2 \le v_x$, and $E_x,F_x$ being subsets of $\R^n$. These ``pointwise-in-time" estimates cannot be expected to hold in situations coming from non-autonomous problems.  Here, we go two steps further. First, we allow for $q,r$ different from 2, which forces us sometimes to reinject $R$-bounds for maximal regularity operators. Second, we define and  implement a new, weaker version of decay  estimates that do not require an operator-valued kernel representation and allow
 $L^{u_t}(E_t)$ to $L^{v_t}(F_t)$, $u_t,v_t \in [1,\infty]$, control of this decay rather than of pointwise control (corresponding to $u_t=v_t'=1$).  This  explains the use of the mixed-norm spaces above.
Moreover, the SIO theory was formulated for a chosen homogeneity parameter $m
$ where $t$ scales like $|x|^m$.  As we shall see, having chosen the homogeneity, this new form of measuring off-diagonal decay applies to the situation of  \cite{Hou25}  in parabolic homogeneity ($m=2$) and  the  situation in elliptic homogeneity  ($m=1$) proposed in \cite{HR23}.   

Under pointwise in time off-diagonal estimates, we improve the results of \cite[Section 3.2]{AH25a}  in several ways.  
\begin{enumerate}[(i)]
    \item We  obtain weighted $T^{p,q}_{\beta}$ estimates with $q\ne 2$ (and even $T^{p,q,r}_{\beta}$ estimates).
\item We do not assume global $L^r_{\beta}$ to $L^r_{\beta+\kappa}$ a priori boundedness for the given $\beta$: only uniform local, $\beta$-independent $L^r$-estimates on Whitney regions suffice. The above global bound becomes a particular consequence.
\item Our assumptions do not require any kernel representation.
\item We relax the order of  polynomial decay in the required off-diagonal estimates. 
\end{enumerate}   
The last point is a major improvement and greatly streamlines the ranges of exponents in applications. In parabolic applications,  the decay is often exponential. However, in elliptic applications, we face polynomial decay that can be very small.  Before (when $q=2$), a lower bound for the order of decay, roughly $\frac{n}{2m}$, 
was needed with the previous techniques for any positive result. By our new method, any positive order of decay leads to a non-empty range of applicability. See Section~\ref{sec:comparison}. 

\medskip

We  showcase our results with applications to problems that motivated this study. Prototypical examples are the non-autonomous elliptic and parabolic equations with bounded, measurable, possibly complex, uniformly elliptic coefficients $A(t,x)$, where no regularity or symmetry assumptions are further imposed.

Section~\ref{sec:ellipticPDES} contains tent space estimates for operators relevant for elliptic problems, derived from Theorem~\ref{thm:mainestimate}.
Theorem~\ref{thm:cz} establishes tent space estimates 
for classical Calder\'on--Zygmund operators, 
extending \cite{HR23}.
Theorem~\ref{thm:beurling} shows tent space estimates for maximal regularity operators for the
Poisson semigroup generated by $-\sqrt{L}$, for divergence form elliptic operators $L$, which extends
\cite{AKMP12}.
Moreover, we prove tent space estimates for maximal regularity operators associated with more general perturbed Dirac 
operators $DB$. In doing so, we prove new off-diagonal estimates for these non-injective $DB$ operators; the estimates extend the results in 
\cite{AS16}.
Finally, Theorem~\ref{thm:riesz} shows new tent space estimates for Riesz transforms associated with
divergence form elliptic operators in the half-space, with Dirichlet or Neumann boundary conditions.
These operators are non-causal and only have off-diagonal estimates with $u_t=u_x$ and $v_t=v_x$ and not pointwise-in-$t$, thus fully making use of our weaker off-diagonal hypothesis.
Again, we prove corresponding results for 
generalized Riesz transforms associated with $DB$.
We note that all these estimates concern operators $T$ acting in tent spaces on
$\R^{1+n}_+$, 
rather than operators $S$ acting in the corresponding Hardy type spaces on $\R^n$, as in for example \cite{FMcP18, AA18}, which is the more common setup in the elliptic PDE literature.

In Section~\ref{sec:parabolic}, we apply our new singular operator theory to study non-autonomous, complex, parabolic Cauchy problems, deriving the tent space estimates for solutions and their gradients from tent space bounds on the source terms and Triebel-Lizorkin bounds on the initial data. This requires estimates for the Duhamel and Lions operators, together with suitable trace theorems. Given an initial $L^2$ data, existence for the Cauchy problem of solutions  with $\nabla u \in L^2(\R^{1+n}_+; \C^n)$ was established in \cite{Lions1957}, while uniqueness in this class was proved in \cite{AMP19}. As a consequence, the Lions operator
\[ \nabla_x(\partial_t-\Div_x(A(t,x)\nabla_x))^{-1}\Div_x \]
is well-defined and bounded on $L^2(\R^{1+n}_+; \C^n)$. Well-posedness in non-Hilbertian settings require other bounds for this operator, which remained unknown for a long time for non-smooth (even real) coefficients $A$. The $L^2$ bound was recently extrapolated to $T^{p,2}_\beta$ bounds in \cite{Hou25}, allowing to prove well-posedness for initial data in Hardy-Sobolev spaces,  extending the well-posedness results in \cite{AP25} when $\beta=0$ with $L^p$-data building on the bounds in \cite{AMP19} and \cite{Z20}  for $p$ finite and $p$ infinite respectively. To extrapolate to $T^{p,q,r}_\beta$ bounds when $q\ne 2$ or $r\ne 2$, one needs mixed Lebesgue bounds for the Lions operator. For applications to related boundary value problems, mixed Lebesgue estimates in the full space $\R^{1+n}$ for in particular $\nabla_x(1+ \partial_t-\Div_x(A(t,x)\nabla_x))^{-1}\Div_x$ were obtained in \cite{AEN20} through a variational argument and Sneiberg extrapolation \cite{Sneiberg1974}. We adapt this approach to the Lions operator and combine it with local estimates for weak solutions to obtain off-diagonal decay and then  $T^{p,q,r}_\beta$ estimates for $q \ne2$ or $r\ne 2$ from Theorem~\ref{thm:mainestimate}, which we subsequently incorporate into the analysis of the Cauchy problem.

\medskip

We conclude in Section~\ref{sec:extensions} with adaptations of our results to weighted $Z$-spaces, a more recent scale of spaces  adapted to boundary value problems with Besov data introduced in \cite{BaM16}
and further developed in \cite{Ame18}. An up-to-date treatment and references can be found in \cite{ABH26}. We note that the requirements we obtain for boundedness on $Z$-spaces are a little simpler compared to those on tent spaces. This is due to the way the norms separate time scales, in contrast with the tent space norms,  analogously to Besov versus Triebel-Lizorkin norms. Related results not overlapping with ours are in \cite{AB26}. 

We also mention that with dyadic versions of our Hardy-type model operators, one can certainly develop our theory on $\R_+\times X$, where $X$ is a space of homogeneous type, as our arguments for the upper half-space rely only on the existence of  dyadic lattices on $\R^n$. Technical difficulties on tent spaces coming from such a setting are  addressed in \cite{Ame14}. Operator bounds based on the earlier SIO theory are in \cite{Hua18a}.

\medskip

We shall use throughout the following notation.
\begin{itemize}
\item   $\R_+:=(0,\infty)$ and $\R^n$, $n \in \N$, unless otherwise specified,  are equipped with Lebesgue measure and Euclidean distance.  We set $\R^{1+n}_+=(0,\infty)\times \R^n$.
We tacitly assume all subsets of $\R^n$, $\R_+$ or $\R^{1+n}_+$  to be Lebesgue measurable. For $E \subseteq \R^n$,  $\R_+$ or $\R^{1+n}_+$ with $0<\abs{E}<\infty$ and $f \in L^1(E)$ or $f$ non-negative, we write
\[
\fint_E f  := \frac{1}{\abs{E}} \int_E f 
\]
and for $p>0$ we define
\[
\ip{f}_{p,E}:= \brB{\fint_E\abs{f}^p}^{1/p}
\]
with the usual adaptations for $p=\infty$. 
\item A set \({\mbs{E}}\subseteq \R^{1+n}_+\) is called a \emph{rectangle} (or a ``cubic'' cylinder)  if it can be written as the Cartesian product of an interval in \(\R_+\) and a cube with sides parallel to the axes in \(\R^n\).
We reserve the letter \(t\) for the \(\R_+\)-variable and the letter \(x\) for the \(\R^n\)-variable, and use the same letters in subscripts to denote the corresponding components of product sets and integrability indices. Thus, if \({\mbs{E}}\subseteq \R^{1+n}_+\) is a rectangle, we write
\[
{\mbs{E}}=E_t\times E_x,
\]
where \(E_t\subseteq \R_+\) and \(E_x\subseteq \R^n\). We hope that this suggestive notation makes the subsequent discussion easier to follow. 
\item We shall encounter mixed-norm spaces $L^{p}(\R_+; L^q(\R^n))$  whose norms are written $\nrm{f}_{L^p_tL^q_{\vphantom{t}x}}$. When $p=q$, we write the space as $L^p(\R^{1+n}_+)$ and its norm as $\nrm{f}_{L^p_{t,x}}$. 
\item Throughout  we fix  a dyadic system $\mc{D}$ in $\R^n$ {such that for any bounded set $E\subseteq \R^n$ there is a $Q \in \mc{D}$ with $E\subseteq Q$. For example, construct the one-dimensional dyadic system from $[-\frac13,\frac23)$ and take Cartesian products to pass to higher dimensions.}
\item For $Q \in \mc{D}$, denote its side length by  $\ell(Q)$. Any ${Q} \in \mc{D}$ can be decomposed into $2^{n}$ dyadic subcubes of side length $\frac{1}{2}\ell(Q)$, which are called the \emph{dyadic children} of ${Q}$ and are denoted by $\ch({Q})$. Conversely, $\widehat{Q} \in \mc{D}$ denotes the unique \emph{dyadic parent} of $Q$. The collection of all  $P\in \mc{D}$ such that $P \subseteq Q$ is denoted by ${\mc{D}}({Q})$.
\item For $Q \in \mc{D}$, we denote by $\lambda Q\subseteq \R^n$ the cube with the same centre as $Q$ and side length $\lambda \ell(Q)$ for $\lambda>0$. Similar notations also apply to balls.
\item For $x \in \R^n$ and $r>0$, denote the ball around $x$ with radius $r$ by $B(x,r)$. For a ball $B$ in $\R^n$, write $r(B)$ for the radius of $B$. Similar notations apply to the cube $Q$ (ball with respect to $d_\infty$ distance).
\item For $p,q \in (0,\infty]$ we will write $[p,q] := \frac{1}{p}-\frac{1}{q}$. This notation should not be confused with that of closed intervals, and the context will make it clear whether we mean a closed interval or this number.
Note that if $\min\{p,q\}\ge 1$, we have $[p,q]=[q',p']$ where $p',q'$ are the H\"older conjugate exponents to $p,q$.
\item We use the standard Vinogradov convention $\eqsim$, $\lesssim$ for comparisons and inequalities with inessential implicit constants.
\end{itemize}

\section{Singular operators}\label{sec:singularoperators}
In this section we introduce our class of singular operators, prove some of its functional-analytic properties, and then show that it  encompasses earlier definitions.  Throughout, we fix a sufficiently small geometric constant \(c\in (0,1)\), depending only on the dimension \(n\). For our purposes it will be enough to take
\[
c \leq \frac{1}{16\sqrt n}.
\]

\subsection{Motivating examples}

As a model, and to build intuition,
consider a singular integral operator
\begin{equation*}
    Tf(t,x)= \int_0^\infty \int_{\R^n} k(t,x; s,y)\,f(s,y)\,\dd y\,\dd s,
\qquad (t,x)\in \R^{1+n}_+.
\end{equation*}
Assume that $T$ is bounded from $L^2(\R^{1+n}_+)$ to $L^2(\R_+^{1+n},t^{-2\kappa}\ddn t\dd x)$, with a
kernel having pointwise estimates
\begin{equation}\label{eq:pointwise}
      |k(t,x; s,y)|\lesssim 
\brb{|t-s|+|x-y|^m}^{-1-\frac nm+\kappa}\cdot
\brB{1+\frac{|x-y|^m}{|t-s|}}^{-M}.
\end{equation}
Here $\kappa$ denotes the order of the operator in  the scaling determined by the parameter $m>0$ and the parameter $M\geq 0$ quantifies the additional off-diagonal decay of the kernel in the $x$-direction. Operators arising from elliptic PDEs, corresponding to \(m=1\), typically exhibit only limited off-diagonal decay; for instance, the Poisson kernel satisfies \eqref{eq:pointwise} with \(\kappa=1\) and \(M=1\). By contrast, operators arising from parabolic PDEs, corresponding to \(m=2\), typically enjoy much stronger off-diagonal decay; for example, the heat kernel satisfies \eqref{eq:pointwise} with \(\kappa=1\) and arbitrarily large \(M\). Examples with \(\kappa=0\) include first-order derivatives of the Poisson kernel and second-order derivatives of the heat kernel. These arise, respectively, in connection with the single layer potential for the Laplace equation and with maximal regularity for the heat equation.

Pointwise kernel estimates can be relaxed to so-called off-diagonal estimates. Indeed, if two rectangles \({\mbs{E}},{\mbs{F}}\subseteq \R^{1+n}_+\) are separated in at least one of the variables, then for every \((t,x)\in {\mbs{F}}\) and \((s,y)\in {\mbs{E}}\) the quantities \(|t-s|\) and \(|x-y|\) are bounded from below by the corresponding separations $d(E_t,F_t)$ and $d(E_x,F_x)$, respectively. Hence for every $(t,x)\in {\mbs{F}}$ and  $(s,y)\in {\mbs{E}}$, we see that \eqref{eq:pointwise} implies
\[
|k(t,x;s,y)|
\lesssim
\brb{d(E_t,F_t)+d(E_x,F_x)^m}^{-1-\frac{n}{m}+\kappa}\cdot 
\brB{1+\frac{d(E_x,F_x)^m}{\diam(E_t\cup F_t)}}^{-M},
\]
for $\kappa\le 1+n/m$.
Integrating over \({\mbs{E}}\) immediately yields so-called \(L^1\)-\(L^\infty\) off-diagonal estimates. Namely, whenever \({\mbs{E}}\) and \({\mbs{F}}\) are separated, one has
\begin{align*}
   \|\ind_{\mbs{F}} T(f\ind_{\mbs{E}})\|_{L^\infty_{t,x}}
&\lesssim 
\brb{d(E_t,F_t)+d(E_x,F_x)^m}^{-1-\frac{n}{m}+\kappa}
\cdot\brB{1+\frac{d(E_x,F_x)^m}{\diam(E_t\cup F_t)}}^{-M}
\cdot \|f\ind_{\mbs{E}}\|_{L^1_{t,x}}.
\end{align*}
Conversely, under suitable kernel representation assumptions, by {shrinking the rectangles to points} and invoking the Lebesgue differentiation theorem, such off-diagonal estimates imply the kernel estimates in \eqref{eq:pointwise}.

Motivated by the preceding discussion, heuristically speaking, one can weaken the assumptions   by requiring  $T$ to be   (locally) bounded from $L^r(\R^{1+n}_+)$ to $L^r(\R_+^{1+n},t^{-\kappa r}\ddn t\dd x)$ for some $r \in [1,\infty]$, with the convention that one uses the  norm $\nrm{t^{-\kappa} f}_{\infty}$ when $r=\infty$,
and by assuming that, for separated rectangles  \({\mbs{E}},{\mbs{F}}\subseteq \R^{1+n}_+\), we have an estimate of the form
\begin{equation} \begin{aligned} \label{eq:odeprototypeinit}
\nrmb{\ind_{\mbs{F}}T(f\ind_{\mbs{E}})}_{L^{v_t}_tL^{v_{\vphantom{t}x}}_{\vphantom{t}x}}
\leq C\cdot \brb{d(E_t,F_t)^{\frac1m}&+d(E_x,F_x)}^{-m[u_t,v_t]-n[u_x,v_x]+m\kappa}\\
&\quad \cdot \brB{1+\frac{d(E_x,F_x)^m}{\diam(E_t\cup F_t)}}^{-M}
\cdot \nrmb{f\ind_{\mbs{E}}}_{L^{u_t}_tL^{u_{\vphantom{t}x}}_{\vphantom{t}x}},\end{aligned}
\end{equation}
for 
$
    1\leq u_t,u_x\leq   r\leq v_t,v_x\leq \infty
$ and $M\ge 0$. This can be viewed as the mixed-norm analogue of the \(L^1\)-\(L^\infty\) off-diagonal estimates discussed above. Here the first factor on the right-hand side measures the singular behaviour of the operator in terms of the separation of ${\mbs{E}}$ and ${\mbs{F}}$, while the second factor captures the additional decay when the \(x\)-separation dominates the \(t\)-scale. Our actual assumption, Definition \ref{def:ODE} below, is a refined version of the above prototypical inequality. We note the following about Definition \ref{def:ODE}.
\begin{enumerate}[(i)]
\item It takes advantage of working in the half-space $\R^{1+n}_+$.
   \item It only requires local $L^r$-boundedness on Whitney-type regions, where $t\eqsim \inf(E_t)$.
    \item It allows for negative off-diagonal decay parameters $M<0$, with potentially important applications to elliptic PDEs. However, in the examples considered in Section~\ref{sec:ellipticPDES}, $M\geq 0$ suffices.
    \item It only requires estimates for sets ${\mbs{E}}$ and ${\mbs{F}}$ with $\sup(E_t\cup F_t)^{\frac1m}$ bounded by $\diam(E_x\cup F_x)$, which will be important in Section \ref{sec:parabolic}.
    \item It minimizes the collection of sets ${\mbs{E}}$ and ${\mbs{F}}$ for which the off-diagonal estimates need to be checked in applications. 
\end{enumerate}

Finally, let us stress that this point of view is more flexible than the typical formulations of off-diagonal estimates that appear in the literature; see for instance \cite{AH25a} and the references therein. Indeed, in previous works  the \(t\)-variable is treated pointwise which, as discussed, corresponds to \(L^1\)-\(L^\infty\) off-diagonal estimates. By contrast, our assumptions allow \(L^{u_t}\)-\(L^{v_t}\) off-diagonal estimates in the $t$-variable with \(u_t>1\) and/or \(v_t<\infty\). This additional flexibility will be essential for the operators arising in our applications.

\subsection{Definition}

Our definition will only require bounds and decay on special configurations for sets ${\mbs{E}},{\mbs{F}} \subseteq \R^{1+n}_+$, depicted in Figure \ref{fig:EF-picturesdef}. They  all come with finite diameters, allowing us to use 
 $$\ell_x:= \diam(E_x\cup F_x)
  \  \text{and} \ \ell_t:= \sup (E_t\cup F_t)$$
    as scaling variables. Note that $\ell_t$ measures the maximum distance to $t=0$ and is well-adapted to taking Whitney decompositions. The inequality $$\diam(E_t\cup F_t)\le \sup(E_t\cup F_t)=\ell_t$$  for all $E_t,F_t \subseteq \R_+$ reflects the special role of the half-line $\R_+$.

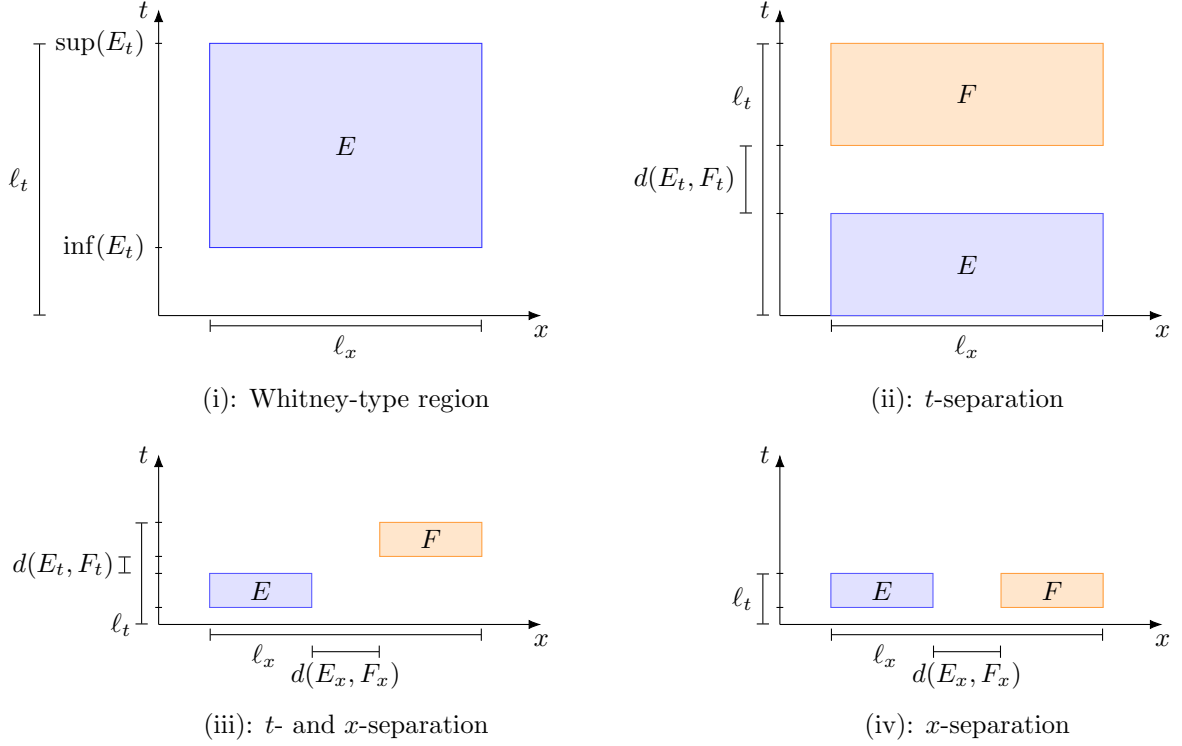
\begin{figure}[ht]
\centering

\tikzset{ref/.style={black!100,dashed}, refLight/.style={black!35,dashed}}

\begin{minipage}[t]{0.45\textwidth}
\begin{tikzpicture}[x=0.45cm,y=0.45cm,>=Latex, font=\small]
\useasboundingbox (-10,-3) rectangle (6,10);
  \draw[->] (-5.5,0) -- (5.75,0) node[below] {$x$};
  \draw[->] (-5.5,0) -- (-5.5,9) node[left] {$t$};

  \path[fill=blue!12]   (-4,2) rectangle (4,8);      
  \draw[blue!80]   (-4,2) rectangle (4,8);


  \node at (0,{0.5*(2+8)}) {$E$};

  \foreach \y/\lab in {2/{\inf(E_t)},8/{\sup(E_t)}}{
    \draw (-5.4,\y) -- (-5.6,\y);
    \node[left] at (-5.6,\y) {$\lab$};
  }
  \draw[|-|] (-4,-0.30) -- (4,-0.30) node[midway,below=0pt] {$\ell_x$};
  \draw[|-|] (-9,0) -- (-9,8) node[midway,left=0pt] {$\ell_t$};
  \node at (0,-2.5) {(i): Whitney-type region};
\end{tikzpicture}
\end{minipage}
\hspace{0.05\textwidth}
\begin{minipage}[t]{0.45\textwidth}
\begin{tikzpicture}[x=0.45cm,y=0.45cm,>=Latex, font=\small]
\useasboundingbox (-10,-3) rectangle (6,10);
  \draw[->] (-5.5,0) -- (5.75,0) node[below] {$x$};
  \draw[->] (-5.5,0) -- (-5.5,9) node[left] {$t$};

  \path[fill=blue!12]        (-4,0) rectangle (4,3);
  \path[fill=orange!20] (-4,5) rectangle (4,8);       
  \draw[blue!65]   (-4,0) rectangle (4,3);
  \draw[orange!75] (-4,5) rectangle (4,8);

  \node at (0,{0.5*(5+8)}) {$F$};
  \node at (0,{0.5*(3)}) {$E$};

  \foreach \y/\lab in {3/{},5/{},8/{}}{
    \draw (-5.4,\y) -- (-5.6,\y);
    \node[left] at (-5.6,\y) {$\lab$};
  }
  \draw[|-|] (-4,-0.30) -- (4,-0.30) node[midway,below=0pt] {$\ell_x$};
  \draw[|-|] (-6,0) -- (-6,8) node[pos=0.8,left=0pt] {$\ell_t$};
  \draw[|-|] (-6.5,3) -- (-6.5,5) node[midway,left=0pt] {$d(E_t,F_t)$};

  \node at (0,-2.5) {(ii): $t$-separation};
\end{tikzpicture}
\end{minipage}
\begin{minipage}[t]{0.45\textwidth}

\begin{tikzpicture}[x=0.45cm,y=0.45cm,>=Latex, font=\small]
\useasboundingbox (-10,-3) rectangle (6,6);
  \draw[->] (-5.5,0) -- (5.75,0) node[below] {$x$};
  \draw[->] (-5.5,0) -- (-5.5,5) node[left] {$t$};

  \path[fill=blue!12]         (-4,0.5) rectangle (-1,1.5); 
  \path[fill=orange!20] (1,2) rectangle (4,3);     
  \draw[blue!65]   (-4,0.5) rectangle (-1,1.5);
  \draw[orange!75] (1,2) rectangle (4,3); 

  \node at (2.5,2.5) {$F$};
  \node at (-2.5,1) {$E$};

  \foreach \y/\lab in {0.5/{},1.5/{},2/{},3/{}}{
    \draw (-5.4,\y) -- (-5.6,\y);
    \node[left] at (-5.6,\y) {$\lab$};
  }
  \draw[|-|] (-4,-0.30) -- (4,-0.30) node[pos=0.2,below=0pt] {$\ell_x$};
  \draw[|-|] (-1,-0.80) -- (1,-0.80) node[midway,below=0pt] {$d(E_x,F_x)$};
  \draw[|-|] (-6,0) -- (-6,3) node[pos=0,left=0pt] {$\ell_t$};
  \draw[|-|] (-6.5,1.5) -- (-6.5,2) node[midway,left=0pt] {$d(E_t,F_t)$};

   \node at (0,-3) {(iii): $t$- and $x$-separation};
\end{tikzpicture}
\end{minipage}
\hspace{0.05\textwidth}
\begin{minipage}[t]{0.45\textwidth}
\begin{tikzpicture}[x=0.45cm,y=0.45cm,>=Latex, font=\small]
\useasboundingbox (-10,-3) rectangle (6,6);
  \draw[->] (-5.5,0) -- (5.75,0) node[below] {$x$};
  \draw[->] (-5.5,0) -- (-5.5,5) node[left] {$t$};

  \path[fill=blue!12]         (-4,0.5) rectangle (-1,1.5); 
  \path[fill=orange!20] (1,0.5) rectangle (4,1.5);     
  \draw[blue!65]   (-4,0.5) rectangle (-1,1.5);
  \draw[orange!75] (1,0.5) rectangle (4,1.5);

  \node at (2.5,1) {$F$};
  \node at (-2.5,1) {$E$};

  \foreach \y/\lab in {0.5/{},1.5/{}}{
    \draw (-5.4,\y) -- (-5.6,\y);
    \node[left] at (-5.6,\y) {$\lab$};
  }
  \draw[|-|] (-4,-0.30) -- (4,-0.30) node[pos=0.2,below=0pt] {$\ell_x$};
  \draw[|-|] (-1,-0.80) -- (1,-0.80) node[midway,below=0pt] {$d(E_x,F_x)$};
  \draw[|-|] (-6,0) -- (-6,1.5) node[midway,left=0pt] {$\ell_t$};

   \node at (0,-3) {(iv): $x$-separation};
\end{tikzpicture}
\end{minipage}
\caption{Examples of the sets ${\mbs{E}},{\mbs{F}} \subseteq \R^{1+n}_+$, with $n=1$, in configurations \ref{it:ODEDef1}-\ref{it:ODEDef4} of Definition \ref{def:ODE} for $m=1$.}
\label{fig:EF-picturesdef}
\end{figure}

Let us describe the various configurations for sets ${\mbs{E}},{\mbs{F}} \subseteq \R^{1+n}_+$.
\begin{enumerate}[(i)]
    \item $\mbs{E}=\mbs{F}$ is a Whitney-type region: the $t$- and $x$-scales are comparable, i.e., \(\ell_t^{1/m}\eqsim \ell_x\), and the set is separated from $t=0$, i.e., $\inf(E_t) \eqsim \ell_t$.
    \item Separation occurs in the $t$-variable with $d(E_{t},F_{t})\eqsim \ell_{t}$ and comparable $t$- and $x$-scales.  One of the sets $\mbs{E},\mbs{F}$ is contained in a Whitney-type region  and the other is contained in a  Carleson-type region. 
    \item Separation occurs in both variables  with  \(d(E_x,F_x)\eqsim \ell_x\) and $d(E_{t},F_{t})\eqsim \ell_{t}$, but   the $t$-scale may be  much smaller than the $x$-scale (\(\ell_t^{1/m}\lesssim \ell_x\)).
    \item Separation occurs in the $x$-variable  with  \(d(E_x,F_x)\eqsim \ell_x\) but  $d(E_{t},F_{t})\lesssim \ell_{t}$ and  the $t$-scale may be much smaller than the $x$-scale.
\end{enumerate}

Turning to our definition, we call an operator $T\colon L^\infty_c(\R^{1+n}_+) \to L^0(\R^{1+n}_+)$ \emph{subadditive} if for all $f,g \in L^\infty_c(\R^{1+n}_+)$ we have
$$
\abs{T(f+g)} \leq \abs{Tf}+\abs{Tg}
$$
almost everywhere. Note that the constant in this estimate is required to be $1$.

\begin{definition}[Singular operators] \label{def:ODE}
  Let $r\in[1,\infty]$, $\vec{u}= (u_t,u_x) \in [1,\infty]^2$, $\vec{v} = (v_t,v_x) \in [1,\infty]^2$, $\kappa, M \in \R$ and $m\in (0,\infty)$.
  We say that a {subadditive} operator $T\colon L^\infty_c(\R^{1+n}_+) \to L^0(\R^{1+n}_+)$ is  a
  \emph{{singular operator of type} $(\kappa,m,r,\vec u,\vec v, M)$} if there exists a constant {$C<\infty$} 
   such that the following assertions hold for all $f \in L^\infty_c(\R^{1+n}_+)$ and all rectangles ${\mbs{E}},{\mbs{F}}\subseteq \R^{1+n}_+$ in the following configurations. 
  \begin{enumerate}[(i)]
    \item\label{it:ODEDef1} {(Local boundedness on Whitney-type regions)}
      {If ${\mbs{F}}={\mbs{E}}$, $c\cdot  \ell_x \leq \ell_t^{\frac1m} \leq \ell_x$ and       
      $\inf(E_t) \geq c^m \cdot \ell_t$,} then
\begin{align*}
    \nrmb{ \ind_{{\mbs{E}}} T(f\ind_{{\mbs{E}}})}_{L^r_{t,x}} \leq C\cdot \inf(E_t)^{\kappa} \cdot \nrmb{ f\ind_{{\mbs{E}}}}_{L^r_{t,x}}.
\end{align*} 
      \item\label{it:ODEDef2}  ($t$-separation) 
      {If $c\cdot  \ell_x \leq \ell_t^{\frac1m} \leq \ell_x$ and $d(E_t,F_t)\ge c^m \cdot \ell_t$,} then
      \begin{align*}
  \nrmb{\ind_{\mbs{F}} T(f\ind_{\mbs{E}})}_{L^{v_t}_tL^{v_{\vphantom{t}x}}_{\vphantom{t}x}} 
  &
  \leq C\cdot {d(E_t,F_t)^{-[u_t,v_t]-\frac nm[u_x,v_x]+\kappa}}
\cdot \nrmb{f\ind_{\mbs{E}}}_{L^{u_t}_tL^{u_{\vphantom{t}x}}_{\vphantom{t}x}}.
\end{align*}
  \item\label{it:ODEDef3}  ($t$- and $x$-separation)
    {If $\ell_t^{\frac1m} \leq \ell_x$, $d(E_t,F_t)\ge c^m\cdot \ell_t$ and $d(E_x,F_x)\ge c\cdot  \ell_x$,}
    then
\begin{align*}
  \nrmb{\ind_{\mbs{F}} T(f\ind_{\mbs{E}})}_{L^{v_t}_tL^{v_{\vphantom{t}x}}_{\vphantom{t}x}} \leq C&\cdot   {d(E_x,F_x)^{-m[u_t,v_t]-{n}[u_x,v_x]+m\kappa}}
  \cdot  \brB{\frac{\diam(E_t\cup F_t)}{d(E_x,F_x)^{m}}}^M 
\cdot \nrmb{f\ind_{\mbs{E}}}_{L^{u_t}_tL^{u_{\vphantom{t}x}}_{\vphantom{t}x}}.
\end{align*}
    \item\label{it:ODEDef4}  ($x$-separation)
    {If $\ell_t^{\frac1m}\leq \ell_x$ and $d(E_x,F_x)\ge c \cdot \ell_x$,} then
\begin{align*}
  \nrmb{\ind_{\mbs{F}} T(f\ind_{\mbs{E}})}_{L^{r_{\vphantom{t}}}_{\vphantom{t}t}L^{v_{\vphantom{t}x}}_{\vphantom{t}x}} \leq C
  \cdot   {d(E_x,F_x)^{-{n}[u_x,v_x]+m\kappa}}
  \cdot  \brB{\frac{\diam(E_t\cup F_t)}{d(E_x,F_x)^{m}}}^{M+[u_t,v_t]} 
\cdot \nrmb{f\ind_{\mbs{E}}}_{L^{r_{\vphantom{t}}}_{\vphantom{t}t}L^{u_{\vphantom{t}x}}_{\vphantom{t}x}}.
\end{align*}
\end{enumerate}
\end{definition}

In Definition \ref {def:ODE} the parameter $\kappa$ is the order of the operator, usually non-negative in applications. The parameter $m$ is the scaling parameter, usually an integer in applications. The parameter $M$ measures decay and only appears in  \ref{it:ODEDef3} and \ref{it:ODEDef4}. In these configurations we have
 $$\frac{\diam(E_t\cup F_t)}{d(E_x,F_x)^m }\lesssim \frac{\ell_t}{\ell_x^m} \leq 1,$$
so larger $M$, with $M>0$ or  $M>-[u_t,v_t]$, is respectively, implies stronger decay of $T$.

Definition \ref {def:ODE} looks significantly more involved than the prototypical estimate given in \eqref{eq:odeprototypeinit}. In the next proposition, we show that \eqref{eq:odeprototypeinit} actually implies Definition~\ref{def:ODE}. We will see the benefits of the assumptions in 
Definition~\ref{def:ODE} in our applications in Sections \ref{sec:ellipticPDES} and~\ref{sec:parabolic}.

\begin{prop}
    \label{prop:SIO-simple}
 Let $r\in[1,\infty]$, $\vec{u},\vec{v} \in [1,\infty]^2$ with $u_t \le r \le v_t$, $\kappa, M \in \R$ and $m\in (0,\infty)$. Let $T$ be a bounded subadditive operator from $L^r(\R^{1+n}_+; t^{\gamma r}\dd t\dd x)$ to $L^r(\R_+^{1+n}; t^{(\gamma-\kappa) r}\dd t\dd x)$ for some $\gamma\in \R$.  If for any $f \in L^\infty_c(\R^{1+n}_+)$, and any  rectangles $\boldE,\boldF \subseteq \R^{1+n}_+$ satisfying one of the conditions
 \begin{itemize}
     \item  ($t$-separation) 
      {$c\cdot  \ell_x \leq \ell_t^{\frac1m} \leq \ell_x$ and $d(E_t,F_t)\ge c^m \cdot \ell_t$,}
      \item  ($x$-separation)
    $\ell_t^{\frac1m}\leq \ell_x$ and $d(E_x,F_x)\ge c \cdot \ell_x,$
 \end{itemize} we have the inequality
\begin{equation} \begin{aligned} \label{eq:odeprototype}
\nrmb{\ind_{\mbs{F}}T(f\ind_{\mbs{E}})}_{L^{v_t}_tL^{v_{\vphantom{t}x}}_{\vphantom{t}x}}
\leq C\cdot \brb{d(E_t,F_t)^{\frac1m}&+d(E_x,F_x)}^{-m[u_t,v_t]-n[u_x,v_x]+m\kappa}\\
&\quad \cdot \brB{1+\frac{d(E_x,F_x)^m}{\diam(E_t\cup F_t)}}^{-M}
\cdot \nrmb{f\ind_{\mbs{E}}}_{L^{u_t}_tL^{u_{\vphantom{t}x}}_{\vphantom{t}x}},\end{aligned}
\end{equation}    
    then $T$ is a singular operator of type $(\kappa,m,r,\vec{u},\vec{v},M)$.    
\end{prop}

\begin{proof}   
The assumed boundedness of $T$ implies Definition \ref{def:ODE}\ref{it:ODEDef1} since $\sup(E_t)\eqsim t \eqsim \inf(E_t)$ for all $t \in E_t$. Next, 
   in the configuration in \ref{it:ODEDef2} we have\begin{align*}
       d(E_x,F_x) &\le \ell_{x}\eqsim \ell_{t}^{1/m}\eqsim d(E_t,F_t)^{1/m}\\
       \diam(E_t\cup F_t)&\ge  d(E_t,F_t) \eqsim \ell_{t}\gtrsim d(E_x,F_x)^m
   \end{align*} and thus \eqref{eq:odeprototype} implies Definition \ref{def:ODE}\ref{it:ODEDef2}. Definition \ref{def:ODE}\ref{it:ODEDef3} follows similarly. Finally in the configuration in \ref{it:ODEDef4}, together with $u_t\leq r\leq v_t$ and H\"older's inequality, we have 
    \begin{align*}
        \nrmb{\ind_{\mbs{F}} T(f\ind_{\mbs{E}})}_{L^{r_{\vphantom{x}}}_tL^{v_x}_{\vphantom{t}x}} &\leq |F_t|^{[r,v_t]} \nrmb{\ind_{\mbs{F}} T(f\ind_{\mbs{E}})}_{L^{v_t}_t L^{v_x}_{\vphantom{t}x}} ,\\
        \nrmb{f\ind_{\mbs{E}}}_{L^{u_t}_tL^{u_{\vphantom{t}x}}_{\vphantom{t}x}} &\leq |E_t|^{[u_t,r]} \nrmb{f\ind_{\mbs{E}}}_{L^{r_{\vphantom{x}}}_tL^{u_{x}}_{\vphantom{t}x}}.
    \end{align*}
    Since $|F_t|,|E_{t}|\le \diam(E_t\cup F_t)$, $[u_t,r],[r,v_t] \ge 0$ and $[u_t,r]+[r,v_t]=[u_t,v_t]$, combining these two estimates with \eqref{eq:odeprototype}   yields  Definition \ref{def:ODE}\ref{it:ODEDef4}.
    \end{proof}

\subsection{Functional-analytic properties}

We now record a few elementary properties of singular operators as in
Definition~\ref{def:ODE}. The first one says that the class is stable under
contracting the admissible range of mixed-norm exponents, trading $t$-integrability for additional decay.

\begin{prop}[Contraction of parameters]
    \label{prop:ODE-contraction}  Let $r\in[1,\infty]$, $\vec{u},\vec{v} \in [1,\infty]^2$, $\kappa, M \in \R$ and $m\in (0,\infty)$.
    Let $T$ be a singular operator of type $(\kappa,m,r,\vec u,\vec v,M)$. If $\widetilde{\vec u},\widetilde{\vec v} \in [1,\infty]^2$ satisfy
    \[
        u_t\leq \widetilde u_t\leq \widetilde v_t\leq v_t,
        \qquad
        u_x\leq \widetilde u_x\leq \widetilde v_x\leq v_x,
    \]
    then $T$ is a singular operator of type
    $(\kappa,m,r,\widetilde{\vec u},\widetilde{\vec v},M+[u_t,v_t]-[\widetilde u_t,\widetilde v_t])$.
\end{prop}

\begin{proof}
    Let ${\mbs E}, {\mbs F}\subseteq \R_+^{1+n}$ be rectangles. By
    H\"older's inequality, we have for $f \in L^\infty_c\br{\R^{1+n}_+}$
    \begin{align*}
      \nrmb{\ind_{\mbs F} f}_{L^{\widetilde v_t}_tL^{\widetilde v_x}_x}
      &\leq
      |F_t|^{[\widetilde v_t,v_t]}|F_x|^{[\widetilde v_x,v_x]}
      \nrmb{\ind_{\mbs F} f}_{L^{v_t}_tL^{v_x}_x},\\
      \nrmb{f\ind_{\mbs E}}_{L^{u_t}_tL^{u_x}_x}
      &\leq
      |E_t|^{[u_t,\widetilde u_t]}|E_x|^{[u_x,\widetilde u_x]}
      \nrmb{f\ind_{\mbs E}}_{L^{\widetilde u_t}_tL^{\widetilde u_x}_x}.
    \end{align*}
We check each of the  configurations in Definition~\ref{def:ODE} separately.
\begin{enumerate}[(i)]
    \item The local Whitney estimate is
    unchanged, since it only involves the exponent $r$.
    \item We have
    $|E_t|,|F_t|\lesssim d(E_t,F_t)$ and
    $|E_x|,|F_x|\lesssim d(E_t,F_t)^{n/m}$. Combining these bounds with the
    estimate for the original exponents changes the decay powers exactly from
    $[u_t,v_t]$ and $[u_x,v_x]$ to
    $[\widetilde u_t,\widetilde v_t]$ and
    $[\widetilde u_x,\widetilde v_x]$.
    \item We similarly use
    $|E_x|,|F_x|\lesssim d(E_x,F_x)^n$ and
    $|E_t|,|F_t|\leq \diam(E_t\cup F_t)$. Since 
    \[
       \diam(E_t\cup F_t)^{[u_t,v_t]-[\widetilde u_t,\widetilde v_t]}  =  d(E_x,F_x)^{m( [u_t,v_t]-[\widetilde u_t,\widetilde v_t])}\Bigl(\frac{\diam(E_t\cup F_t)}{d(E_x,F_x)^m}\Bigr)^{[u_t,v_t]-[\widetilde u_t,\widetilde v_t]},
    \]
    we obtain the claimed increase in $M$.
    \item H\"older's inequality is only needed in the
    spatial variable, because the time exponent is fixed at $r$ on both sides of
    Definition~\ref{def:ODE}\ref{it:ODEDef4}. The spatial powers contract as above.
    Since
    \[
        M+[u_t,v_t]=\brb{M+[u_t,v_t]-[\widetilde u_t,\widetilde v_t]}+[\widetilde u_t,\widetilde v_t],
    \]
    we also obtain exactly the right decay parameter.
\end{enumerate}
 This proves the claim.
\end{proof}

Our next proposition shows that mixed-norm boundedness automatically yields a
singular operator.  Proposition \ref{prop:ODE-contraction} lets us first prove the endpoint
estimate and then pass
to all intermediate exponents.

\begin{prop}
\label{prop:boundedimplySOwithM=0}
Let $\vec p,\vec q\in [1,\infty]^2$ with $p_t\leq q_t$ and $p_x\leq q_x$, and let $m \in (0,\infty)$. Let
$T$ be a bounded subadditive operator from
$L^{p_t}(\R_+;L^{p_x}(\R^n))$ to $L^{q_t}(\R_+;L^{q_x}(\R^n))$. Set
\[
    \kappa:= [p_t,q_t]+\tfrac nm[p_x,q_x].
\]
Then $T$ is a singular operator of type
$(\kappa,m,r,\vec u,\vec v,[p_t,q_t]-[u_t,v_t])$ for all $r\in[1,\infty]$ and
$\vec u,\vec v\in[1,\infty]^2$ satisfying
\(
    p_t\leq u_t\leq r\leq v_t\leq q_t\), and \(
    p_x\leq u_x\leq r\leq v_x\leq q_x.
\)
\end{prop}

\begin{proof}
By Proposition~\ref{prop:ODE-contraction} it suffices to prove that $T$ is a
singular operator of type $(\kappa,m,r,\vec p,\vec q,0)$. Fix rectangles
${\mbs E},{\mbs F}\subseteq\R^{1+n}_+$. Let us first consider the local Whitney configuration in
Definition~\ref{def:ODE}\ref{it:ODEDef1}. H\"older's
inequality and the boundedness of $T$ give
\begin{align*}
  \nrmb{\ind_{\mbs E}T(f\ind_{\mbs E})}_{L^r_{t,x}}
  &\lesssim
  |E_t|^{[r,q_t]}|E_x|^{[r,q_x]}
  \nrmb{T(f\ind_{\mbs E})}_{L^{q_t}_tL^{q_x}_x} \\
  &\lesssim
  |E_t|^{[r,q_t]}|E_x|^{[r,q_x]}
  |E_t|^{[p_t,r]}|E_x|^{[p_x,r]}
  \nrmb{f\ind_{\mbs E}}_{L^r_{t,x}} \\
  &=
  |E_t|^{[p_t,q_t]}|E_x|^{[p_x,q_x]}
  \nrmb{f\ind_{\mbs E}}_{L^r_{t,x}}.
\end{align*}
In this configuration we have $|E_t|\lesssim \inf(E_t)$ and
$|E_x|\lesssim \inf(E_t)^{n/m}$, so the final factor is bounded by
$\inf(E_t)^\kappa$, as required. For the configurations in
Definition~\ref{def:ODE}\ref{it:ODEDef2}-\ref{it:ODEDef3}, the boundedness
assumption immediately yields
\[
    \nrmb{\ind_{\mbs F}T(f\ind_{\mbs E})}_{L^{q_t}_tL^{q_x}_x}
    \lesssim
    \nrmb{f\ind_{\mbs E}}_{L^{p_t}_tL^{p_x}_x}.
\]
When $\vec u=\vec p$, $\vec v=\vec q$, and
$\kappa=[p_t,q_t]+\frac nm[p_x,q_x]$, the distance powers in
Definition~\ref{def:ODE}\ref{it:ODEDef2} and
Definition~\ref{def:ODE}\ref{it:ODEDef3} cancel exactly, while $M=0$. Hence these two
estimates follow.

It remains to check Definition~\ref{def:ODE}\ref{it:ODEDef4}. Write
$D_t:=\diam(E_t\cup F_t)$ and $d_x:=d(E_x,F_x)$. By H\"older's inequality in the
time variable and the boundedness assumption,
\begin{align*}
  \nrmb{\ind_{\mbs F}T(f\ind_{\mbs E})}_{L^r_tL^{q_x}_x}
  &\lesssim
  |F_t|^{[r,q_t]}
  \nrmb{T(f\ind_{\mbs E})}_{L^{q_t}_tL^{q_x}_x} \\
  &\lesssim
  |F_t|^{[r,q_t]}|E_t|^{[p_t,r]}
  \nrmb{f\ind_{\mbs E}}_{L^r_tL^{p_x}_x} \\
  &\lesssim
  D_t^{[p_t,q_t]}
  \nrmb{f\ind_{\mbs E}}_{L^r_tL^{p_x}_x}.
\end{align*}
Since
\[
  D_t^{[p_t,q_t]}
  =
  d_x^{-n[p_x,q_x]+m\kappa}
  \Bigl(\frac{D_t}{d_x^m}\Bigr)^{[p_t,q_t]},
\]
this is precisely Definition~\ref{def:ODE}\ref{it:ODEDef4} with
$\vec u=\vec p$, $\vec v=\vec q$ and $M=0$. 
\end{proof}

{\begin{remark}
   The converse to Proposition \ref{prop:boundedimplySOwithM=0} fails in general. For example, one can find  bounded singular operators on $L^r(\R^{1+n}_+)$ of type $(0,1, r, (u,u), (v,v), 0)$ which are not bounded from $L^u(\R^{1+n}_+)$ to $L^v(\R^{1+n}_+)$, see Section~\ref{sec:GRT}. 
\end{remark}}

Singular operator properties for linear operators are also stable under interpolation and transposition in a natural way. For later use, we only state the  duality result,  whose proof is immediate.

\begin{prop}[Duality]
    \label{prop:duality}
    Let  $T$ and $S$ be linear operators $L^\infty_c(\R^{1+n}_+)\to L^1_{\loc}(\R^{1+n}_+)$  such that for all $f,g\in L^\infty_c(\R^{1+n}_+)$,
    \[
    \int_{\R^{1+n}_+} Tf\cdot \overline g \ \dd x \dd t =\int_{\R^{1+n}_+} f\cdot \overline{Sg} \ \dd x \dd t.
    \]  
    Then $T$ is a singular operator of type $(\kappa,m,r,\vec u,\vec v, M)$ if and only if $S$ is a singular operator of type 
    $(\kappa,m,r',\vec v',\vec u', M)$, where the parameters are assumed to be in the ranges as in Definition~\ref{def:ODE}. 
    \end{prop}

\subsection{Comparison with earlier definitions}
Next, we compare Definition \ref{def:ODE} to the off-diagonal estimates assumed in \cite[Section 2]{AH25a} and its precursor  \cite{AKMP12}.
For $f \colon \R^{1+n}_+\to \C$, we write
$$
\pi(f):= \pi_1(\supp(f)) \times \pi_2(\supp(f)),
$$
where $\pi_1$ and $\pi_2$ denote the projections onto $\R_+$ and $\R^n$, respectively. We use the notation $f(s)$  to denote the slice restriction at level $s$ defined by $f(s)(x):=f(s,x)$. For sufficient conditions for the well-definedness of the kernel representation in the  following definition, we refer to \cite[Section 2]{AH25a}.

\begin{definition}[Singular integral operators] 
    \label{def:kernelODE}
 Let $r, u_x,v_x\in[1,\infty]$, $\kappa , M\in \R$ and $m\in (0,\infty)$. We say that
  a {linear} operator $T\colon L^\infty_c(\R^{1+n}_+) \to L^0(\R^{1+n}_+)$ is  a
  \emph{{singular integral operator of type} $(\kappa,m,r,u_x,v_x, M)$} if
  there is a strongly measurable $K \colon \{(t,s) \in \R_+^2:t \ne s\} \to \mc{L}(L^r(\R^n))$ such that for all $f \in L^\infty_c(\R^{1+n}_+)$ we have the kernel representation
  \begin{equation}
      \label{eq:representation} Tf(t,x) = \int_0^\infty \bracb{K(t,s)f(s)}(x)\dd s, \qquad (t,x) \notin \pi(f),
  \end{equation}
   and   there exists a constant $C<\infty$  such that the following assertions hold for all $f \in L^\infty_c(\R^{1+n}_+)$ and $E_x,F_x \subseteq \R^n$.
  \begin{enumerate}[(i)]
        \item\label{it:Lr} ($L^r$-boundedness) 
        There is a $\gamma\in \R$ such that $T$ is  bounded from $L^r(\R^{1+n}_+;t^{\gamma r}\ddn t\dd x)$ to {$L^r(\R_+^{1+n};t^{(\gamma -\kappa) r}\ddn t\dd x)$} 
        with norm $\le C$.
        
        \item \label{item:kernelODE-t-sep} {($t$-separation)} 
        If $t\ne s$, we have
        \[ \nrmb{ K(t,s)f(s)}_{L^{v_x}_{x}} \leq C \,{\abs{t-s}}^{-1-\frac{n}m[u_x,v_x]+\kappa} \cdot\nrmb{f(s)}_{L^{u_x}_{x}}. \]
    
        \item \label{item:kernelODE-x-sep} {($x$-separation)} 
        If {$d(E_x,F_x) \geq c\, |t-s|^{\frac1m}>0$}, we have
        \[ \nrmb{\ind_{F_x} K(t,s) ( f(s)\ind_{E_x} )}_{L^{v_x}_{x}} \leq C\, d(E_x,F_x)^{-m-n[u_x,v_x]+m\kappa} \left( \frac{\abs{t-s}}{d(E_x,F_x)^m} \right)^{M} \nrmb{f(s) \ind_{E_x}}_{L^{u_x}_{x}}. \]
   \end{enumerate}
\end{definition}

\begin{remark}Let us make a few comments on Definition \ref{def:kernelODE} and its relation to its precursor in \cite{AH25a, AKMP12}.\label{rem:KernelODE}
\begin{enumerate}[(i)]
    \item\label{it:kernelODEAH} Denote by 
    \[ M_{\text{AH}} := M+ 1-\kappa+\tfrac{n}{m}[u_x,v_x]. \]
    When $M_{\text{AH}}\ge 0$, the combination of the $t$- and $x$-separation conditions in Definition \ref{def:kernelODE} is equivalent to 
    \begin{align*}
        \nrmb{\ind_{F_x} & K(t,s)(f\ind_{E_x})}_{L^{v_x}_{x}}\le C \abs{t-s}^{-1-\frac{n}{m}[u_x,v_x]+\kappa} \left( 1+\frac{d(E_x,F_x)^m}{\abs{t-s}} \right)^{-M_{\text{AH}}} \nrmb{f \ind_{E_x}}_{L^{u_x}_{x}}.
    \end{align*}  
When $M_{\text{AH}}> 0$, this is exactly the $L^{u_x}(\R^n)$-$L^{v_x}(\R^n)$ off-diagonal decay of type $(\kappa, m, M_{\text{AH}})$ for the family of operators $K(t,s)$  appearing in \cite[Definition 1]{AH25a}. Hence, our definition of a singular integral operator in Definition \ref{def:kernelODE}  includes the classes SIO$^{\kappa \pm}_{m,q,M_{\text{AH}}}$
of \cite[Definitions 3 and 4]{AH25a} for which {either $q=u_x\le 2$ and $v_x=2$ or $u_x=2$ and $q=v_x\ge 2$} and the integral in the representation \eqref{eq:representation} is limited to either $(0,t)$ or $(t,\infty)$.
\item {When $M_{\text{AH}}>\frac n m $, one can prove the $L^r$-boundedness of $K(t,s)$ for any $r\in [u_x,v_x]$  from the assumed off-diagonal estimates 
{(see, e.g., \cite[Lemma 4.7]{Auscher-Egert2023-book})}. However, when $M_{\text{AH}}\le \frac n m$, this is unclear.}
\item The formulation in Definition \ref{def:kernelODE} involves an operator-valued kernel representation of $T$, hence the addition of ``integral'' in the terminology. In stark contrast, Definition \ref{def:ODE} is free of any representation and also accommodates subadditive operators.
\end{enumerate}
\end{remark}

We now show that every singular integral operator $T$ as in Definition \ref{def:kernelODE}, including the operators considered in \cite{AH25a},  is a singular operator as in Definition \ref{def:ODE} with $u_t=1$ and $v_t=\infty$. The proof crucially uses $M>-1= -[u_t,v_t]$.

\begin{prop}[Singular integral operators are singular operators] \label{prop:compareAH}
  Let $r, u_x,v_x\in[1,\infty]$, $\kappa \in \R$, $m\in (0,\infty)$ and $M> -1$. Let $T$ be a
  {{singular integral operator of type} $(\kappa,m,r,u_x,v_x, M)$}. Then $T$ is a singular operator of type $(\kappa,m,r,\vec{u},\vec{v}, M)$ with $\vec{u}:=(1,u_x)$ and $\vec{v}:=(\infty,v_x)$.
\end{prop}

\begin{proof}
Since $\sup(E_t) \eqsim t \eqsim \inf(E_t)$ on Whitney-type regions, it is clear that the weighted $L^r$-boundedness in Definition \ref{def:kernelODE}\ref{it:Lr} implies Definition \ref{def:ODE}\ref{it:ODEDef1}.
Fix $f \in L^\infty_c(\R^{1+n}_+)$ and let  ${\mbs{E}},{\mbs{F}} \subseteq \R^{1+n}_+$ be rectangles. We will consider  Definition \ref{def:ODE}\ref{it:ODEDef2}-\ref{it:ODEDef4} separately.
\begin{enumerate}[(i)]\setcounter{enumi}{1}
  \item {($t$-separation)} Assuming $d_{t}:=d(E_t,F_t)>0$, we have
\begin{align*}
   \nrmb{\ind_{\mbs{F}} T(f\ind_{\mbs{E}})}_{L^{\infty_{\vphantom{t}}}_tL^{v_{\vphantom{t}x}}_{\vphantom{t}x}} &= \esssup_{t \in F_t} \, \nrmB{x \mapsto \ind_{F_x}(x) \int_{E_t} \bracb{K(t,s)(f(s)\ind_{E_x})}(x)\dd s }_{L^{v_x}_{\vphantom{t}x}}\\
   &\leq \esssup_{t \in F_t} \,\int_{E_t} \nrmB{x \mapsto \ind_{F_x}(x)  \bracb{K(t,s)(f(s)\ind_{E_x})}(x) }_{L^{v_x}_{\vphantom{t}x}}\dd s\\
   &\lesssim  d_{t}^{-1-\frac{n}m[u_x,v_x]+\kappa} \int_{E_t} \nrm{f(s)\ind_{E_x}}_{L^{u_x}_{\vphantom{t}x}}\dd s\\
   &=  d_{t}^{-[1,\infty]-\frac{n}m[u_x,v_x]+\kappa} 
\cdot \nrmb{f\ind_{\mbs{E}}}_{L^{1_{\vphantom{t}}}_tL^{u_{\vphantom{t}x}}_{\vphantom{t}x}},
\end{align*}
as desired.

\item {($t$- and $x$-separation)} Assuming  $d_{t}:=d(E_t,F_t) \geq c^m \cdot D_{t} $ and $d_{x}:=d(E_x,F_x) \geq c\cdot D_{t}^{1/m}$, with $D_{t}:=\diam(E_t\cup F_t)$, we have
\begin{align*}
   \nrmb{\ind_{\mbs{F}} T(f\ind_{\mbs{E}})}_{L^{\infty_{\vphantom{t}}}_tL^{v_{\vphantom{t}x}}_{\vphantom{t}x}} &= \esssup_{t \in F_t} \, \nrmB{x \mapsto \ind_{F_x}(x) \int_{E_t} \bracb{K(t,s)(f(s)\ind_{E_x})}(x)\dd s }_{L^{v_x}_{\vphantom{t}x}}\\
   &\leq \esssup_{t \in F_t} \,\int_{E_t} \nrmB{x \mapsto \ind_{F_x}(x)  \bracb{K(t,s)(f(s)\ind_{E_x})}(x) }_{L^{v_x}_{\vphantom{t}x}}\dd s
   \\&
   \lesssim  d_{x}^{-m-{n}[u_x,v_x]+m\kappa-mM}  
   \cdot \esssup_{t \in F_t} \, \int_{E_t} {\abs{t-s}^M} \nrm{f(s)\ind_{E_x}}_{L^{u_x}_{\vphantom{t}x}}\dd s
   \\&
   \lesssim  d_{x}^{-m[1,\infty]-{n}[u_x,v_x]+m\kappa} 
   \cdot \brB{\frac{D_{t}}{d_{x}^{m}}}^M \cdot  \nrmb{f\ind_{\mbs{E}}}_{L^{1_{\vphantom{t}}}_tL^{u_{\vphantom{t}x}}_{\vphantom{t}x}}.
\end{align*}
The last inequality is always valid when $M\ge 0$, whereas  $d_{t}\geq c^m\cdot D_{t} $ is used when $M<0$.
\item {($x$-separation)} Assume $d_{x}:=d(E_x,F_x) \geq c \cdot D_{t}^{1/m}$, with $D_{t}:=\diam(E_t\cup F_t)$, but  $E_t$ and $F_t$ are not necessarily separated. Since $M>-1$, we have the estimate
$$
  \esssup_{t\in F_t}\int_{E_t} |t-s|^M \dd s\lesssim 
  D_{t}^{M+1},
$$
which yields 
\begin{align*}
\nrmb{\ind_{\mbs{F}} T(f\ind_{\mbs{E}})}_{L^{\infty_{\vphantom{t}}}_tL^{v_{\vphantom{t}x}}_{\vphantom{t}x}}
& \lesssim d_{x}^{-m[1,\infty]-{n}[u_x,v_x]+m\kappa-mM} 
\cdot D_{t}^{M+[1,\infty]}\cdot  
\nrmb{f\ind_{\mbs{E}}}_{L^{\infty_{\vphantom{t}}}_tL^{u_{\vphantom{t}x}}_{\vphantom{t}x}}.
\end{align*}
The same estimate holds with $L^{\infty}(\R_+)$ replaced by
$L^1(\R_+)$, using instead that 
$$\esssup_{s \in E_t}\int_{F_t} |t-s|^M \dd t\lesssim 
  D_{t}^{M+1}.
$$
The result for $L^r(\R_+)$ follows from Schur's lemma or interpolation.
\end{enumerate}
This verifies all assumptions of Definition \ref{def:ODE} and completes the proof.
\end{proof}

\begin{remark}
\label{rem:pointwiseimmliesnonpointwise} Let us comment on Proposition \ref{prop:compareAH}.
\begin{enumerate}
    
    \item In the proof  we obtained the required estimates for a larger class of rectangles \({\mbs{E}}, {\mbs{F}}\subseteq \R^{1+n}_+\) than those appearing in  Definition \ref{def:ODE}. Furthermore, the \(L^r\)-boundedness of \(K(t,s)\) was not used in the proof.

\item Even when restricting Definition \ref{def:kernelODE} to the same class of rectangles as in Definition \ref{def:ODE}, there is no reason to expect the converse implication in the statement to hold for linear operators.  
\begin{itemize}
    \item   Even if one could recover an operator-valued kernel representation in an appropriate sense, when the parameter \(M\) in the \(x\)-separation condition of Definition \ref{def:ODE} is small, it is not clear how to prove the \(L^r\)-boundedness of \(K(t,s)\) away from the diagonal \(s=t\) as mentioned in Remark~\ref{rem:KernelODE}. Such bounds enter quantitatively in the singular integral theory developed in \cite{AH25a} for $r=2$. 
\item  The local $L^r$-boundedness \ref{it:ODEDef1} in Definition \ref{def:ODE} is strictly weaker than the weighted $L^r$-boundedness assumption \ref{it:Lr} in Definition \ref{def:kernelODE}, and can be implied by other conditions.
\end{itemize}
\end{enumerate}
\end{remark}

\section{Form domination of singular operators }  
\label{sec:domination}
In this section we prove a domination principle for  singular operators introduced in Section~\ref{sec:singularoperators} in terms of suitable model operators. We will focus solely on singular operators as in Definition \ref{def:ODE}, noting that our results will also be applicable to singular \emph{integral} operators as in Definition \ref{def:kernelODE} by using Proposition \ref{prop:compareAH}.
The key idea is to decompose the action of \(T\) in bilinear form, i.e., by testing \(Tf\) against a function \(g\) for $f,g \in L^\infty_c(\R^{1+n}_+)$. Once written in this way, the first step is geometric: one has to organize the interaction between the supports of \(f\), \(Tf\), and \(g\). 
This will be done dyadically, using Carleson boxes and Whitney regions in the scaling imposed by $m>0$.
For a dyadic cube
$Q \in {\mc{D}}$ we  define its Carleson box and Whitney region as
\begin{align*}
  \mbs{{Q}}^{\ca} &:= (0,\ell(Q)^m)\times Q,\\
  \mbs{{Q}}^{\whit} &:= (2^{-m} \ell(Q)^m,\ell(Q)^m)\times Q.
\end{align*}
Furthermore, we define their enlargements by
\begin{align*}
  \Qcatilde{Q} &:= (0,2^m\ell(Q)^m)\times 3Q,\\
  \Qwtilde{Q} &:= (4^{-m}\ell(Q)^m,2^m\ell(Q)^m)\times 3Q,
\end{align*}
where we recall that $3 Q\subseteq \R^n$ denotes the cube with the same centre as $Q$ and side length $3 \ell(Q)$.

With this geometry in place, the bilinear form splits into a local contribution on Whitney regions and three families of  off-diagonal terms that are either $t$- or $x$-separated.

\begin{prop}\label{prop:decomposition} Let $T\colon L^\infty_c(\R^{1+n}_+) \to L^0(\R^{1+n}_+)$ be a subadditive operator.
Then, for any  $f,g \in L^\infty_c(\R^{1+n}_+)$, we have
\begin{align*}
\int_{{\mathbb R}^{1+n}_+}\abs{Tf}\abs{g}  \leq
\sum_{Q \in \mc{D}} \int_{\mbs{Q}^{\whit}} \absb{T(f\ind_{\Qwtilde{Q}})}\abs{g} &+ 
\sum_{k=1}^3 \sum_{Q \in \mc{D}} \int_{{\mbs{F}}_Q^k} \absb{T(f\ind_{{\mbs{E}}_Q^k})}\abs{g}
\end{align*}
where for $Q \in \mc{D}$ we denote its dyadic parent by $\widehat{Q}$ and set $($see Figure \ref{fig:EF-pictures}$\, )$
\begin{align*}
   {\mbs{E}}_Q^1 &:= (0,4^{-m}\ell(Q)^m)\times 3Q,  & {\mbs{F}}_Q^1&:= \mbs{Q}^{\whit},\\
    {\mbs{E}}_Q^2 &:= (\ell(Q)^m,2^m\ell(Q)^m)\times 3Q, & {\mbs{F}}_Q^2&:= (0,2^{-m} \ell(Q)^m)\times Q,\\
    {\mbs{E}}_Q^3 &:= (0,2^m\ell(Q)^m)\times ( 3\widehat{Q}\setminus 3{Q} ), & {\mbs{F}}_Q^3&:= \Qca{Q}.
\end{align*}
\end{prop}

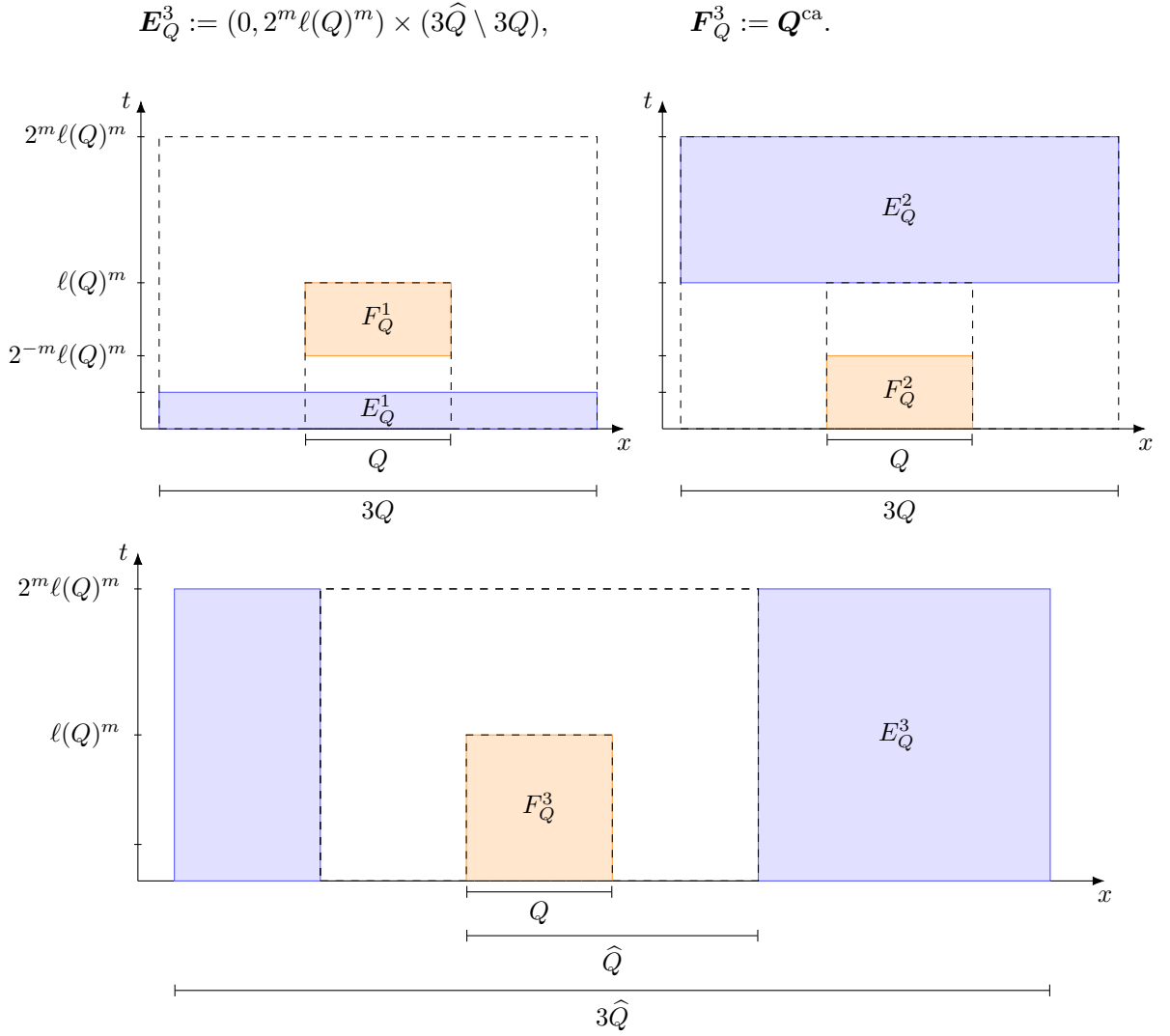
\begin{figure}[ht]
\centering

\def\q{2}    
\def\Q{6}    
\def\h{4}    
\def\yEone{1}
\def\yWh{2}  
\def\yTop{8} 

\tikzset{ref/.style={black!100,dashed}, refLight/.style={black!35,dashed}}

\begin{minipage}[t]{0.45\textwidth}
\begin{tikzpicture}[x=0.5cm,y=0.5cm,>=Latex, font=\small]
  \draw[->] (-6.5,0) -- (6.75,0) node[below] {$x$};
  \draw[->] (-6.5,0) -- (-6.5,\yTop+1) node[left] {$t$};

  \path[fill=blue!12]   (-\Q,0) rectangle (\Q,\yEone);      
  \path[fill=orange!20] (-\q,\yWh) rectangle (\q,\h);       
  \draw[blue!65]   (-\Q,0) rectangle (\Q,\yEone);
  \draw[orange!75] (-\q,\yWh) rectangle (\q,\h);

  \draw[ref] (-\q,0) rectangle (\q,\h);
  \draw[ref] (-\Q,0) rectangle (\Q,\yTop);

  \node at (0,{0.5*\yEone}) {$E_Q^1$};
  \node at (0,{0.5*(\yWh+\h)}) {$F_Q^1$};

  \foreach \y/\lab in {\yEone/{},\yWh/{2^{-m}\ell(Q)^m},\h/{\ell(Q)^m},\yTop/{2^{m}\ell(Q)^m}}{
    \draw (-6.4,\y) -- (-6.6,\y);
    \node[left] at (-6.6,\y) {$\lab$};
  }
  \draw[|-|] (-\q,-0.30) -- (\q,-0.30) node[midway,below=0pt] {$Q$};
  \draw[|-|] (-\Q,-1.7) -- (\Q,-1.7) node[midway,below=0pt] {$3Q$};
\end{tikzpicture}
\end{minipage}
\hspace{0.08\textwidth}
\begin{minipage}[t]{0.45\textwidth}
\begin{tikzpicture}[x=0.5cm,y=0.5cm,>=Latex, font=\small]
  \draw[->] (-6.5,0) -- (6.75,0) node[below] {$x$};
  \draw[->] (-6.5,0) -- (-6.5,\yTop+1) node[left] {$t$};

  \path[fill=blue!12]   (-\Q,\h) rectangle (\Q,\yTop);      
  \path[fill=orange!20] (-\q,0) rectangle (\q,\yWh);        
  \draw[blue!65]   (-\Q,\h) rectangle (\Q,\yTop);
  \draw[orange!75] (-\q,0) rectangle (\q,\yWh);

  \draw[ref] (-\q,0) rectangle (\q,\h);
  \draw[ref] (-\Q,0) rectangle (\Q,\yTop);

  \node at (0,{0.5*(\h+\yTop)}) {$E_Q^2$};
  \node at (0,{0.5*\yWh}) {$F_Q^2$};

  \foreach \y/\lab in {\yEone/{4^{-m}\ell(Q)^m},\yWh/{2^{-m}\ell(Q)^m},\h/{\ell(Q)^m},\yTop/{2^{m}\ell(Q)^m}}{
    \draw (-6.4,\y) -- (-6.6,\y);
  }
 \draw[|-|] (-\q,-0.30) -- (\q,-0.30) node[midway,below=0pt] {$Q$};
  \draw[|-|] (-\Q,-1.7) -- (\Q,-1.7) node[midway,below=0pt] {$3Q$};
\end{tikzpicture}
\end{minipage}

\begin{minipage}[t]{0.98\textwidth}
\begin{tikzpicture}[x=0.5cm,y=0.5cm,>=Latex, font=\small]
  \draw[->] (-13,0) -- (13.5,0) node[below] {$x$};
  \draw[->] (-13,0) -- (-13,\yTop+1) node[left] {$t$};

  \def\pL{-2}
  \def\pR{0}
  \def\tPL{-4}
  \def\tPR{2}

  \path[fill=blue!12] (2*-\Q,0) rectangle (2*\tPL,2*\h);
  \path[fill=blue!12] (2*\tPR,0) rectangle (2*\Q,2*\h);
  \draw[blue!65] (-2*\Q,0) rectangle (-2*\Q+4,2*\h);
  \draw[blue!65] (2*\Q-8,0) rectangle (2*\Q,2*\h);
  \draw[ref] (2*\tPL,0) rectangle (2*\tPR,2*\h);

  \path[fill=orange!20] (2*\pL,0) rectangle (2*\pR,2*\yWh);
  \draw[orange!75] (2*\pL,0) rectangle (2*\pR,2*\yWh);

  \draw[ref] (2*\pL,0) rectangle (2*\pR,2*\yWh);
  \draw[ref] (2*\tPL,0) rectangle (2*\tPR,2*\h);

  \node at (2*3.9,{\h}) {$E_Q^3$};
  \node at (-2,{\yWh}) {$F_Q^3$};

  \foreach \y/\lab in {\yEone/{},2*\yWh/{\ell(Q)^m},2*\h/{2^m\ell(Q)^m}}{
    \draw (-12.9,\y) -- (-13.1,\y);
    \node[left] at (-13.2,\y) {$\lab$};
  }
 \draw[|-|] (-2*\q,-0.30) -- (0,-0.30) node[midway,below=0pt] {$Q$};
   \draw[|-|] (-2*\q,-1.5) -- (2*\q,-1.5) node[midway,below=0pt] {$\widehat{Q}$};
  \draw[|-|] (-2*\Q,-3) -- (2*\Q,-3) node[midway,below=0pt] {$3\widehat{Q}$};
\end{tikzpicture}
\end{minipage}
\caption{The sets in the decomposition in Proposition \ref{prop:decomposition}}
\label{fig:EF-pictures}
\end{figure}

\begin{proof}
  Take ${Q} \in {\mc{D}}$. Since $\Qca{Q} = \Qw{Q}\cup\bigcup_{P\in \ch(Q)} \Qca{P}$ modulo null sets, we can decompose
\begin{align*}
  \int_{\Qca{Q}} \absb{T(f\ind_{\Qcatilde{Q}})}|g|&\leq \int_{\Qw{Q}} \absb{T(f\ind_{\Qcatilde{Q}})}|g| + \sum_{{P} \in \ch\br{{Q}}}\int_{\Qca{P}} \absb{T(f\ind_{\Qcatilde{Q}\setminus \Qcatilde{P}})}|g| \\&\hspace{1cm}+ \sum_{{P} \in \ch\br{{Q}}} \int_{\Qca{P}} \absb{T(f\ind_{\Qcatilde{P}})}|g|. 
\end{align*}
The terms in the final sum on the right-hand side are of the same form as the left-hand side with $Q$ replaced by its dyadic children, so we can iterate. {Since $g \in L^\infty_c(\R^{1+n}_+)$, the distance between the support of $g$ and the boundary $\partial \R^{1+n}_+$ is strictly positive, so the iteration in the final sum on the right-hand side will be identically zero after a finite number of iterations.}
Starting from a dyadic cube ${Q_0} \in {\mc{D}}$  such that $\Qca{Q}_0$ contains the supports of $f$ and $g$, we have 
$$ \int_{{\mathbb R}^{1+n}_+}|Tf||g| =  \int_{\Qca{Q}_0} \absb{T(f\ind_{\Qcatilde{Q}_0})}|g|$$
and thus the iteration procedure yields, with $\mc{D}(Q_0)$ the set of dyadic cubes contained in $Q_0$,  
\begin{align*}
  \int_{{\mathbb R}^{1+n}_+}|Tf||g|&\leq  \sum_{Q \in \mc{D}(Q_0)}\int_{\Qw{Q}} \absb{T(f\ind_{\Qcatilde{Q}})}|g| + \sum_{Q \in \mc{D}(Q_0)}\sum_{{P} \in \ch\br{{Q}}}\int_{{\Qca{P}}} \absb{T(f\ind_{\Qcatilde{Q}\setminus\Qcatilde{P}})}|g|\\
  &\leq \sum_{Q \in \mc{D}(Q_0)} \int_{\Qw{Q}} \absb{T(f\ind_{\Qwtilde{Q}})}|g|+\sum_{Q \in \mc{D}(Q_0)} \int_{\Qw{Q}} \absb{T(f\ind_{\Qcatilde{Q}\setminus \Qwtilde{Q}})}|g| \\&\hspace{1cm}+\sum_{Q \in \mc{D}(Q_0)} \int_{\bigcup_{P\in \ch(Q)} {\Qca{P}}} \absb{T(f\ind_{(\ell(Q)^m,2^m\ell(Q)^m)\times 3Q})}|g|
  \\&\hspace{1cm}+\sum_{Q \in \mc{D}(Q_0)} \sum_{{P} \in \ch\br{{Q}}}\int_{{\Qca{P}}} \absb{T(f\ind_{(0,\ell(Q)^m)\times (3Q \setminus 3P)})}|g|.
\end{align*}
Noting that the second, third and fourth terms are bounded precisely by the $k=1,2,3$ terms {(relabelling $P$ as $Q$ in the fourth term)} and dropping the restriction that cubes are subcubes of $Q_0$, this finishes the proof.
\end{proof}

\begin{remark}
Suppose  $T$ is linear and consider the associated bilinear form \(\int_{\R^{1+n}_+} Tf\cdot g\). In this case, {all inequalities become equalities} once the absolute values are omitted. Rearranging the terms, one can estimate
\[
\absB{\int_{\R^{1+n}_+} Tf\cdot g-{\sum_{Q \in \mc{D}(Q_0)}} \int_{\mbs{Q}^{\whit}} T(f\ind_{\Qwtilde{Q}})\cdot g}
\]
by the same last three error terms appearing in Proposition \ref{prop:decomposition}.
\end{remark}

Our next step is to estimate the three error terms in Proposition \ref{prop:decomposition}. Our strategy will be to treat each of these pieces by means of suitable positive model operators of Carleson or Whitney type.
These model operators are designed to capture the mixed-norm local averages that arise when the off-diagonal estimates  from Definition \ref{def:ODE}, in the scaling determined by $m$, are applied. 
  For measurable \(f\colon \R^{1+n}_+\to \C\), \(\vec{u}\in (0,\infty]^2\) and \(\delta>0\), we define
\begin{align*}
  H^{\ca}_{\vec{u},\delta} f(t,x)&:=  \brB{\fint_0^t \brB{\fint_{B(x,(\delta t)^{1/m})} \abs{f(s,y)}^{u_x} \dd y}^{\frac{u_t}{u_x}} \dd s}^{\frac{1}{u_t}}, \qquad &&(t,x) \in \R^{1+n}_+,\\
  H^{\whit}_{\vec{u}, \delta} f(t,x)&:=  \brB{\fint_{c^mt}^t\brB{\fint_{B(x,(\delta t)^{1/m})} \abs{f(s,y)}^{u_x} \dd y}^{\frac{u_t}{u_x}} \dd s}^{\frac{1}{u_t}}, \qquad &&(t,x) \in \R^{1+n}_+,
\end{align*}
where $c\in (0,1)$ is the constant fixed at the beginning of  Section~\ref{sec:singularoperators} and we do the usual adaptations when $u_t =\infty$ and/or $u_x = \infty$.
The integration regions of \(H^{\ca}_{\vec u,\delta}\) and \(H^{\whit}_{\vec u,\delta}\) are depicted in Figure \ref{fig:HcaHwhit}. Note that \(H^{\whit}_{\vec u,\delta}\) is naturally adapted to the conical geometry of tent spaces.  By contrast, as we shall see in Section \ref{sec:estontentspace}, the operator \(H^{\ca}_{\vec u,\delta}\) is more delicate to estimate, since it involves integration over a horizontal enlargement of the Carleson region.

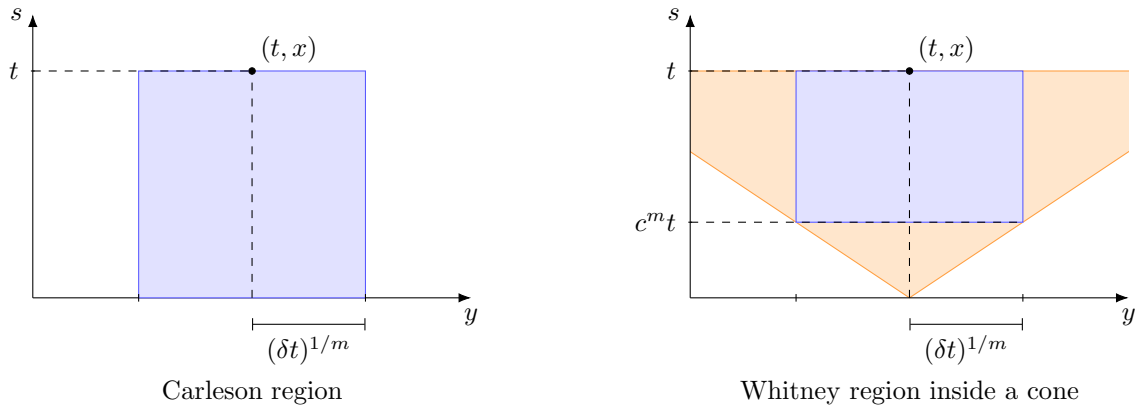
\begin{figure}[ht]
\centering

\def\xc{0}      
\def\tH{6}      
\def\r{3}     
\def\tLow{2}  

\tikzset{ref/.style={black!100,dashed}, refLight/.style={black!35,dashed}}

\begin{minipage}[t]{0.45\textwidth}
\centering
\begin{tikzpicture}[x=0.5cm,y=0.5cm,>=Latex, font=\small]
  \useasboundingbox (-6.5,-3) rectangle (6.5,8);

  \draw[->] (-5.8,0) -- (5.8,0) node[below] {$y$};
  \draw[->] (-5.8,0) -- (-5.8,7.5) node[left] {$s$};

  \path[fill=blue!12] (\xc-\r,0) rectangle (\xc+\r,\tH);
  \draw[blue!65]      (\xc-\r,0) rectangle (\xc+\r,\tH);

  \draw[ref] (\xc,0) -- (\xc,\tH);
  \draw[ref] (-5.8,\tH) -- (\xc,\tH);

  \filldraw[black] (\xc,\tH) circle (1.2pt);
  \node[above right=-1pt] at (\xc,\tH) {$(t,x)$};


  \draw (-5.75,\tH) -- (-5.85,\tH);
  \node[left] at (-5.9,\tH) {$t$};

  \draw (\xc-\r,-0.1) -- (\xc-\r,0.1);
  \draw (\xc+\r,-0.1) -- (\xc+\r,0.1);
  \draw[|-|] (\xc,-0.7) -- (\xc+\r,-0.7)
    node[midway,below=0pt] {$(\delta t)^{1/m}$};
  \node at (0,-2.5) {Carleson region};
\end{tikzpicture}
\end{minipage}
\hspace{0.08\textwidth}
\begin{minipage}[t]{0.45\textwidth}
\centering
\begin{tikzpicture}[x=0.5cm,y=0.5cm,>=Latex, font=\small]
  \useasboundingbox (-6.5,-3) rectangle (6.5,8);

  \def\rTop{\r}         
  \def\rConeTop{9}     

  \begin{scope}
    \clip (-5.8,-2.5) rectangle (5.8,8);

    \path[fill=orange!20] (\xc,0) -- (\xc+\rConeTop,\tH) -- (\xc-\rConeTop,\tH) -- cycle;
    \draw[orange!75] (\xc,0) -- (\xc+\rConeTop,\tH);
    \draw[orange!75] (\xc,0) -- (\xc-\rConeTop,\tH);
    \draw[orange!75] (\xc-\rConeTop,\tH) -- (\xc+\rConeTop,\tH);

    \path[fill=blue!12] (\xc-\rTop,\tLow) rectangle (\xc+\rTop,\tH);
    \draw[blue!65]      (\xc-\rTop,\tLow) rectangle (\xc+\rTop,\tH);
  \end{scope}

  \draw[->] (-5.8,0) -- (5.8,0) node[below] {$y$};
  \draw[->] (-5.8,0) -- (-5.8,7.5) node[left] {$s$};

  \draw[ref] (\xc,0) -- (\xc,\tH);
  \draw[ref] (-5.8,\tH) -- (\xc,\tH);
  \draw[ref] (-5.8,\tLow) -- (\xc+\rTop,\tLow);

  \filldraw[black] (\xc,\tH) circle (1.2pt);
  \node[above right=-1pt] at (\xc,\tH) {$(t,x)$};


  \draw (-5.75,\tLow) -- (-5.85,\tLow);
  \draw (-5.75,\tH) -- (-5.85,\tH);
  \node[left] at (-5.9,\tLow) {$c^mt$};
  \node[left] at (-5.9,\tH) {$t$};

  \draw (\xc-\rTop,-0.1) -- (\xc-\rTop,0.1);
  \draw (\xc+\rTop,-0.1) -- (\xc+\rTop,0.1);
  \draw[|-|] (\xc,-0.7) -- (\xc+\rTop,-0.7)
    node[midway,below=0pt] {$(\delta t)^{1/m}$};

  \node at (0,-2.5) {Whitney region inside a cone};
\end{tikzpicture}
\end{minipage}

\caption{The integration regions defining the model operators \(H^{\ca}_{\vec u,\delta}\) and \(H^{\whit}_{\vec u,\delta}\) at a point \((t,x)\in \R^{1+n}_+\), shown in blue. In the Whitney case, we also display in orange a cone of aperture \({c^{-1} \delta^{1/m}}\) containing the integration region.}
\label{fig:HcaHwhit}
\end{figure}

We start with the \(t\)-separated contributions in Proposition \ref{prop:decomposition}, that is, the terms involving the pairs \({\mbs{E}}_Q^k\) and \({\mbs{F}}_Q^k\) for \(k=1,2\). These are the simplest off-diagonal terms in the decomposition, since the relevant separation occurs only in the \(t\)-variable and the spatial scales remain comparable. Accordingly, they can be treated directly using the \(t\)-separation part of Definition \ref{def:ODE}. 

\begin{lemma}\label{lem:esttseparated}
Let $r \in [1,\infty]$, $\vec{u},\vec{v} \in [1,\infty]^2$, $\kappa, M \in \R$, and $m\in [1,\infty)$. Let $T$ be a
singular  operator  of type $(\kappa,m,r,\vec u,\vec v, M)$. For $Q \in \mc{D}$, we recall from Proposition~\ref{prop:decomposition} that
\begin{align*}
    {\mbs{E}}_Q^1 &:= (0,4^{-m}\ell(Q)^m)\times 3Q,  & {\mbs{F}}_Q^1&:= \mbs{Q}^{\whit},\\
    {\mbs{E}}_Q^2 &:= (\ell(Q)^m,2^m\ell(Q)^m)\times 3Q, & {\mbs{F}}_Q^2&:= (0,2^{-m} \ell(Q)^m)\times Q.
\end{align*}
Then we have for any $f,g \in L^\infty_c(\R^{1+n}_+)$ 
\begin{align*}
\sum_{{Q} \in {\mc{D}}} \int_{{\mbs{F}}_Q^1} \absb{T(f\ind_{{\mbs{E}}_Q^1})}\abs{g} &\lesssim \int_{\R^{1+n}_+} t^\kappa\cdot H^{\ca}_{\vec{u},\br{2\sqrt{n}}^m}f(t,x)\cdot H^{\whit}_{\vec{v}',\br{\sqrt{n}}^m}g(t,x)\dd t \dd x,\\
\sum_{{Q} \in {\mc{D}}} \int_{{\mbs{F}}_Q^2} \absb{T(f\ind_{{\mbs{E}}_Q^2})}\abs{g} &\lesssim \int_{\R^{1+n}_+} t^\kappa\cdot H^{\whit}_{\vec{u},\br{2\sqrt{n}}^m}f(t,x)\cdot H^{\ca}_{\vec{v}',\br{\sqrt{n}}^m}g(t,x)\dd t \dd x.
\end{align*}
\end{lemma}

\begin{proof} 
For the first estimate, fix $Q \in \mc{D}$ and note that for $E_t :=  (0,4^{-m}\ell(Q)^m)$ and $F_t:= (2^{-m}\ell(Q)^m, \ell(Q)^m)$ we have $\ell_t = \ell(Q)^m$, $\ell_x = \diam(3Q) = 3\sqrt{n} \ell(Q)$ and
\begin{align*}
    d(E_t,F_t) &= (2^m-1) 4^{-m}\ell(Q)^m\\ 
    \sup(E_t\cup F_t)+ \diam(3Q)^m &\leq \ell(Q)^m+(3\sqrt{n}\cdot \ell(Q))^m \leq (4\sqrt{n}\cdot \ell(Q))^m
\end{align*}
Therefore, we obtain by H\"older's inequality and the $t$-separated case of Definition \ref{def:ODE} that
\begin{align} \begin{aligned}  \label{eq:toffest}
 \int_{\mbs{Q}^{\whit}} &\absb{T(f\ind_{(0,4^{-m}\ell(Q)^m)\times 3Q})}\abs{g}\\
 &\lesssim  \ell(Q)^{-m[u_t,v_t]-{n}[u_x,v_x]+m\kappa}
 \cdot\nrmb{f\ind_{(0,4^{-m}\ell(Q)^m)\times 3Q}}_{L^{\vphantom{{v'_{\vphantom{t}x}}}u_t}_tL^{\vphantom{{v'_{\vphantom{t}x}}}u_{\vphantom{t}x}}_{\vphantom{t}x}}  \cdot \nrmb{g\ind_{\mbs{Q}^{\whit}}}_{L^{{v'_t}}_tL^{{v'_{\vphantom{t}x}}}_{\vphantom{t}x}} \\
  &\lesssim \ell(Q)^{m+n+m\kappa} \inf_{(t,x) \in \mbs{Q}^{\whit}} H^{\ca}_{\vec{u},\br{2\sqrt{n}}^m}f(2^mt,x)\cdot H^{\whit}_{\vec{v}',\br{\sqrt{n}}^m}g(2^mt,x),
\end{aligned}
\end{align}
where in the last step we used that for any $(t,x) \in \Qw{Q}$ we have $2^mt \geq \ell(Q)^m$ and thus
\begin{align*}
    (0,4^{-m}\ell(Q)^m)\times 3Q &\subseteq (0,2^m t)\times B(x,{(\br{2\sqrt{n}}^m\cdot2^mt)^{1/m}}),\\
    \Qw{Q}&\subseteq (2^{-m}t,2^m t)\times B(x,{(\br{\sqrt{n}}^m\cdot2^mt)^{1/m}}).
\end{align*}
Hence, using  the disjointness of the  Whitney regions, $\abs{\mbs{Q}^{\whit}}\eqsim \ell(Q)^{m+n}$ and $t \eqsim \ell(Q)^m$ for all $(t,x) \in \Qw{Q}$, we deduce
\begin{align*}
  \sum_{{Q} \in {\mc{D}}} \int_{\mbs{Q}^{\whit}} &\absb{T(f\ind_{(0,4^{-m}\ell(Q)^m)\times 3Q})}\abs{g}\\ &\lesssim \sum_{{Q} \in {\mc{D}}}\int_{\mbs{Q}^{\whit}} t^\kappa\cdot H^{\ca}_{\vec{u},\br{2\sqrt{n}}^m}f(2^mt,x)\cdot H^{\whit}_{\vec{v}',\br{\sqrt{n}}^m}g(2^mt,x)\dd t \dd x\\
  &\lesssim \int_{\R^{1+n}_+} t^\kappa\cdot H^{\ca}_{\vec{u},\br{2\sqrt{n}}^m}f(t,x)\cdot H^{\whit}_{\vec{v}',\br{\sqrt{n}}^m}g(t,x)\dd t \dd x,
\end{align*}
using a change of variables in the last step. 
The second estimate is proven analogously.
\end{proof}

\begin{remark}
  We state Lemma~\ref{lem:esttseparated} and all results that follow only for \(m\geq 1\) to avoid technicalities. Indeed, as \(m\downarrow 0\), the factor \(2^m-1\) tends to \(0\), so the fixed geometric constant \(c>0\) would have to be chosen depending on \(m\). For each fixed \(m>0\), one could choose \(c>0\) sufficiently small instead. Since our applications only involve \(m=1,2\), we choose not to complicate the notation.
\end{remark}

For the $x$-separated term in our decomposition in Proposition \ref{prop:decomposition}, i.e., the term involving ${\mbs{E}}_Q^3$ and ${\mbs{F}}_Q^3$,  our domination principle can be strengthened if the operator $T$ is (anti-)causal, which we introduce next.

\begin{definition} Let $T\colon L^\infty_c(\R^{1+n}_+) \to L^0(\R^{1+n}_+)$ be an operator.
\begin{itemize}
    \item We call  $T$ \emph{causal} if for all $f \in L^\infty_c(\R^{1+n}_+)$ and ${\mbs{E}}=E_t\times E_x$  with  $E_t \subseteq \R_+$  and $E_x\subseteq \R^n$ we have
\begin{align*}
    T(f \ind_{\mbs{E}})(t,x) &= 0, & & \makebox[0pt][l]{\ensuremath{(t,x)\in \R^{1+n}_+:\ t<\inf E_t.}}\hphantom{(t,x) \in \R^{1+n}_+}
\end{align*}
\item We call  $T$ \emph{anti-causal} if for all $f \in L^\infty_c(\R^{1+n}_+)$ and ${\mbs{E}}=E_t\times E_x$  with  $E_t \subseteq \R_+$  and $E_x\subseteq \R^n$ we have
\begin{align*}
    T(f \ind_{\mbs{E}})(t,x) &= 0, & & \makebox[0pt][l]{\ensuremath{(t,x)\in \R^{1+n}_+:\ t>\sup E_t.}}\hphantom{(t,x) \in \R^{1+n}_+}
\end{align*}
\end{itemize}
\end{definition}

We are now in a position to state our main domination result. It combines the geometric decomposition in Proposition~\ref{prop:decomposition} with the boundedness of \(T\) on Whitney-type regions, the estimate for the \(t\)-separated terms in Lemma~\ref{lem:esttseparated}, and further estimates for the (\(t\)- and) \(x\)-separated term. In this way, the boundedness of \(T\) is reduced to the boundedness of an averaging operator on Whitney regions and the boundedness of the model operators \(H^{\ca}_{\vec u,\delta}\) and \(H^{\whit}_{\vec u,\delta}\).

\begin{theorem}\label{thm:maindominationform}
Let $r\in[1,\infty]$, $\vec{u},\vec{v} \in [1,\infty]^2$, $\kappa,M \in \R$ and $m\in [1,\infty)$. Let $T$ be a
singular  operator  of type $(\kappa,m,r,\vec u,\vec v, M)$.
Then we have for any $f,g \in L^\infty_c(\R^{1+n}_+)$ that
\begin{align*}
\int_{{\mathbb R}^{1+n}_+}&\abs{Tf}\abs{g} 
\lesssim 
\sum_{Q \in \mc{D}} \ell(Q)^{m\kappa} \cdot \ip{f}_{r,\Qwtilde{Q}}\cdot \ip{g}_{\vphantom{\Qwtilde{Q}}r',\Qw{Q}} \cdot\abs{\Qw{Q}}
&& \\
&+\sum_{j=0}^{\infty} 2^{-jm\br{M+[u_t,v_t]-\kappa}} \cdot \int_{\R^{1+n}_+}t^\kappa\cdot   H^{\whit}_{(r,u_x),2^{jm}}f(t,x)\cdot H^{\whit}_{(r',v_x'), 2^{jm}}g(t,x)\dd x\dd t  &&
\text{$(\mathrm{w}$-$\mathrm{w})$}\\
&+\sum_{j=0}^{\infty} 2^{-jm\br{M+[u_t,v_t]-\kappa}} \cdot \int_{\R^{1+n}_+}t^\kappa\cdot   H^{\ca}_{\vec{u},2^{jm}}f(t,x)\cdot H^{\whit}_{\vec{v}', 2^{jm}}g(t,x)\dd x\dd t &&\text{$(\mathrm{ca}$-$\mathrm{w})$}\\
&+\sum_{j=0}^{\infty} 2^{-jm\br{M+[u_t,v_t]-\kappa}} \cdot
\int_{\R^{1+n}_+}t^\kappa\cdot   H^{\whit}_{\vec{u},2^{jm}}f(t,x)\cdot H^{\ca}_{\vec{v}', 2^{jm}}g(t,x)\dd x\dd t. &&\text{$(\mathrm{w}$-$\mathrm{ca})$}.
\end{align*}
If $T$ is causal, the term {$(\mathrm{w}$-$\mathrm{ca})$} is omitted, and if $T$ is anti-causal, the term {$(\mathrm{ca}$-$\mathrm{w})$} is  omitted.
\end{theorem}

\begin{proof}
Using Proposition \ref{prop:decomposition}, we decompose
\begin{align}\label{eq:decompinproof}
\int_{{\mathbb R}^{1+n}_+}\abs{Tf}\abs{g}  \leq
\sum_{Q \in \mc{D}} \int_{\mbs{Q}^{\whit}} \absb{T(f\ind_{\Qwtilde{Q}})}\abs{g} &+ 
\sum_{k=1,2,3} \sum_{Q \in \mc{D}} \int_{{\mbs{F}}_Q^k} \absb{T(f\ind_{{\mbs{E}}_Q^k})}\abs{g}
\end{align}
where $E_Q^k$ and $F_Q^k$ for $k=1,2,3$ are as in Proposition \ref{prop:decomposition}. For the first term on the right-hand side of \eqref{eq:decompinproof} we have by Definition \ref{def:ODE}\ref{it:ODEDef1}
\begin{align*}
\sum_{{Q} \in {\mc{D}}} \int_{\mbs{Q}^{\whit}} \absb{T(f\ind_{\Qwtilde{Q}})}\abs{g} &\leq \sum_{{Q} \in {\mc{D}}}
\nrm {T(f\ind_{\Qwtilde{Q}})}_{\vphantom{\Qwtilde{Q}}L^{r}(\mbs{Q}^{\whit})} \cdot \nrm {g}_{\vphantom{\Qwtilde{Q}}L^{r'}(\mbs{Q}^{\whit})} \\
&\lesssim \sum_{Q \in \mc{D}} \ell(Q)^{m\kappa} \cdot \ip{f}_{r,\Qwtilde{Q}} \cdot \ip{g}_{\vphantom{\Qwtilde{Q}} r',\Qw{Q}}\cdot \abs{\Qw{Q}},
\end{align*}
which matches the first term on the right-hand side of our claim. The $t$-separated terms in \eqref{eq:decompinproof},
i.e., those involving  ${\mbs{E}}_Q^k$ and ${\mbs{F}}_Q^k$ for $k=1,2$, are handled by Lemma \ref{lem:esttseparated},
and are dominated by terms {$(\mathrm{w}$-$\mathrm{ca})$} and {$(\mathrm{ca}$-$\mathrm{w})$}.
It remains to analyse 
$$
    \sum_{Q \in \mc{D}} \int_{{\mbs{F}}_Q^3} \absb{T(f\ind_{{\mbs{E}}_Q^3})}\abs{g},
$$
using the $x$-separated case and the $t$- and $x$-separated case in Definition~\ref{def:ODE}.

\medskip

Fix $Q \in \mc{D}$ and recall that
\begin{align*}
    {\mbs{E}}_Q^3 &:= (0,2^m\ell(Q)^m)\times (3\widehat{Q}\setminus 3{Q}), & {\mbs{F}}_Q^3&:= \Qca{Q}.
\end{align*}
Note that these sets are $x$-separated. Indeed, one has $\ell_t^{1/m} = 2\ell(Q) \le 6\sqrt{n} \ell(Q) = \ell_x$, and
\[ c\ell_x \le \ell(Q) = d(3\widehat{Q}\setminus 3{Q},Q). \]
To exploit the decay when the $x$-separation is much bigger than the $t$-separation, we further decompose the $t$-interval dyadically.
For $j \geq 0$, define the following {intervals and strips}
\begin{align*}
  R_j&:= (2^{-jm} \ell(Q)^m,2^{-(j-1)m}\ell(Q)^m),&& \mbs{R}_j:= R_j\times \R^n,\\
   \widetilde{R}_j&:= (2^{-(j+1)m} \ell(Q)^m,2^{\min\{-(j-2),1\}m}\ell(Q)^m),&& \widetilde{\mbs{R}}_j:= \widetilde{R}_j\times \R^n,\\
 {R}^{\ca}_j&:= (0,2^{-(j+1)m}\ell(Q)^m), && \mbs{R}^{\ca}_j:={R}_j^{\ca}\times \R^n,
\end{align*}
and set $\mbs{Q}_j:= R_j\times Q$.
Note that, up to  null sets, ${R}^{\ca}_j = \bigcup_{k\ge j+2} {R}_k$ and $\widetilde{R}_j = \bigcup_{|k- j|\le 1, k\ge 0} {R}_k$, and therefore
\begin{align*}
    (0,2^m\ell(Q)^m)\times (0,2^m\ell(Q)^m)= \bigcup_{j=0}^\infty \bigcup_{k=0}^\infty R_j\times R_k = \bigcup_{j=0}^\infty R_j \times \widetilde{R}_j \cup \bigcup_{j=0}^\infty R_j \times {R}_j^{\ca} \cup \bigcup_{j=0}^\infty  {R}_j^{\ca} \times R_j.
\end{align*} 
Hence, we can decompose further as follows
\begin{align*}
     \int_{{\mbs{F}}_Q^3} \absb{T(f\ind_{{\mbs{E}}_Q^3})}\abs{g} &\leq \sum_{j=0}^\infty \int_{{\mbs{F}}_Q^3\cap \mbs{R}_j} \absb{T(f\ind_{{\mbs{E}}_Q^3\cap \widetilde{\mbs{R}}_j})}\abs{g} + \sum_{j=0}^\infty\int_{{\mbs{F}}_Q^3\cap \mbs{R}_j} \absb{T(f\ind_{{\mbs{E}}_Q^3\cap \mbs{R}_j^{\ca}})}\abs{g}\\&\hspace{2cm}+\sum_{j=0}^\infty\int_{{\mbs{F}}_Q^3\cap \mbs{R}_j^{\ca}} \absb{T(f\ind_{{\mbs{E}}_Q^3\cap \mbs{R}_j})}\abs{g}\\&=: \text{\framebox[15pt]{A}}+\text{\framebox[15pt]{B}}+\text{\framebox[15pt]{C}},
\end{align*}
{where the sets in \text{\framebox[15pt]{B}} and \text{\framebox[15pt]{C}} are separated in both $t$ and $x$. Indeed, observe that $$c^m \ell_t = c^m 2^m 2^{-jm} \ell(Q)^m \le (1-2^{-m})2^{-jm}\ell(Q)^m = d(R_j,R_j^{\ca}).$$}

Let us first consider \framebox[15pt]{A}. Note that $\mbs{E}_Q^3$ is a finite union of rectangles in $\R_+^{1+n}$, so the estimates in Definition \ref{def:ODE} are applicable. Fix $j\geq 0$. By 
 H\"older's inequality and the $x$-separated case in Definition~\ref{def:ODE},  we can estimate 
\begin{align}\label{eq:Aestimatebegin}\begin{aligned}
    \int_{{\mbs{F}}_Q^3\cap \mbs{R}_j} \absb{T(f\ind_{{\mbs{E}}_Q^3\cap \widetilde{\mbs{R}}_j})}\abs{g} &\lesssim  \ell(Q)^{-n[u_x,v_x]+m\kappa}\cdot 2^{-jm(M+[u_t,v_t])}   \\&\hspace{1cm} 
    \cdot \nrmb{f\ind_{{\mbs{E}}_Q^3\cap \widetilde{\mbs{R}}_j}}_{L^{\vphantom{r'_{\vphantom{x}}}r}_{t}L^{\vphantom{r'_{\vphantom{x}}}u_x}_{\vphantom{t}x}}\cdot \nrmb{g\ind_{{\mbs{F}}_Q^3\cap \mbs{R}_j}}_{L^{r'_{\vphantom{x}}}_{t}L^{v_x'}_{\vphantom{t}x}}.
\end{aligned} 
\end{align} 
Now note that for $(t,x)  \in \mbs{Q}_j$ we have
\begin{align*}
 \widetilde{R}_j\times (3\widehat{Q}\setminus 3{Q} )&\subseteq (4^{-m} t,4^m t)\times B(x,{(2^{jm}\cdot \br{\sqrt{n}}^m\cdot 4^mt)^{1/m}}),
\end{align*}
and thus 
\begin{align}\label{eq:whitestf}\begin{aligned}
 \nrmb{f\ind_{{\mbs{E}}_Q^3\cap \widetilde{\mbs{R}}_j}}_{L^{r_{\vphantom{x}}}_{t}L^{u_x}_{\vphantom{t}x}}
   &\leq \nrmb{f\ind_{(4^{-m} t,4^m t)\times B(x,{(2^{jm}\cdot\br{\sqrt{n}}^m\cdot4^mt)^{1/m}})}}_{L^{r_{\vphantom{x}}}_{t}L^{u_x}_{\vphantom{t}x}} \\
   &\lesssim (2^{-j}\ell(Q))^{\frac{m}{r}} \ell(Q)^{\frac n{u_x}}\cdot H^{\whit}_{(r,u_x), 2^{jm}\br{\sqrt{n}}^m }f(4^mt,x).\end{aligned}
\end{align} 
Similarly, we have for $(t,x)  \in \mbs{Q}_j$
\begin{align}\label{eq:whitestg}
  \nrmb{g\ind_{{\mbs{F}}_Q^3\cap \mbs{R}_j}}_{L^{r'_{\vphantom{x}}}_{t}L^{v'_x}_{\vphantom{t}x}} \lesssim  
  (2^{-j}\ell(Q))^{\frac{m}{r'}} \ell(Q)^{\frac n{v_x'}}
  \cdot H^{\whit}_{(r',v_x'), 2^{jm}\br{\sqrt{n}}^m }g(4^mt,x).
\end{align} 
Combining \eqref{eq:Aestimatebegin}, \eqref{eq:whitestf}, and \eqref{eq:whitestg}, we conclude 
\begin{align*}\begin{aligned}
  \int_{{\mbs{F}}_Q^3\cap \mbs{R}_j} \absb{T(f\ind_{{\mbs{E}}_Q^3\cap \widetilde{\mbs{R}}_j})}\abs{g}&\lesssim  2^{-jm[u_t,v_t]-jmM}\cdot  2^{-jm}\ell(Q)^{m+n+m\kappa} \\&\cdot \inf_{(t,x) \in \mbs{Q}_j} H^{\whit}_{(r,u_x),2^{jm}\br{\sqrt{n}}^m}f(4^mt,x)\cdot H^{\whit}_{(r',v_x'), 2^{jm} \br{\sqrt{n}}^m }g(4^mt,x).
\end{aligned}\end{align*}
Now note that $2^{-jm} \ell(Q)^{m+n} \eqsim \abs{\mbs{Q}_j}$
{and that $t\eqsim (2^{-j}\ell(Q))^m$ for $(t,x)\in \mbs{Q}_j$}. Furthermore, note that for almost all fixed $(t,x)\in \R^{1+n}_+$ and all $j \geq 0$ there is a unique $Q \in \mc{D}$ such that $(t,x) \in \mbs{Q}_j$. Therefore, we obtain {
\begin{align*}\begin{aligned}
      {\sum_{{Q} \in \mc{D}}}\text{\framebox[15pt]{A}}&\lesssim \sum_{{Q} \in \mc{D}} \sum_{j=0}^{\infty}  2^{-jm(M+[u_t,v_t]-\kappa)}\\&\hspace{2cm}\cdot \int_{\mbs{Q}_j} t^\kappa\cdot   H^{\whit}_{(r,u_x),2^{jm}\br{\sqrt{n}}^m}f(4^mt,x)\cdot H^{\whit}_{(r',v_x'), 2^{jm} \br{\sqrt{n}}^m }g(4^mt,x) \dd t \dd x\\
  &\lesssim \sum_{j=0}^{\infty} 2^{-jm(M+[u_t,v_t]-\kappa)} \cdot \int_{\R^{1+n}_+}t^\kappa\cdot   H^{\whit}_{(r,u_x),2^{jm}}f(t,x)\cdot H^{\whit}_{(r',v_x'), 2^{jm} }g(t,x) \dd t \dd x
\end{aligned}
\end{align*} }
after a change of variables and a shift in the $j$ indices, giving an estimate by term {$(\mathrm{w}$-$\mathrm{w})$}.

Consider next the term \framebox[15pt]{B}.
Applying
the off-diagonal estimate from the $t$- and $x$-separated case in Definition~\ref{def:ODE}
gives that
\begin{align}\begin{aligned} \label{eq:Btermstart}
    \int_{{\mbs{F}}_Q^3\cap\mbs{R}_j} 
    \absb{T(f\ind_{{\mbs{E}}_Q^3\cap \mbs{R}^{\ca}_j})}\abs{g} &\lesssim  \ell(Q)^{-m[u_t,v_t]-n[u_x,v_x]+m\kappa}\cdot 2^{-jmM} \\
    &\hspace{1cm} \cdot 
    \nrmb{f\ind_{{\mbs{E}}_Q^3\cap {\mbs{R}}^{\ca}_j}}_{L^{\vphantom{v'_{t}}u_{t}}_{t}L^{\vphantom{v'_{t}}u_{\vphantom{t}x}}_{\vphantom{t}x}}\cdot 
    \nrmb{g\ind_{{\mbs{F}}_Q^3\cap \mbs{R}_j}}_{L^{v'_{t}}_{t}L^{v'_{\vphantom{t}x}}_{\vphantom{t}x}},
\end{aligned} 
\end{align} 
where
$$
 \nrmb{f\ind_{{\mbs{E}}_Q^3\cap {\mbs{R}}_j^{\ca}}}_{L^{u_{t}}_{t}L^{u_{\vphantom{t}x}}_{\vphantom{t}x}} 
  \lesssim (2^{-j}\ell(Q))^{\frac{m}{u_t}} \ell(Q)^{\frac n{u_x}}\cdot H^{\ca}_{\vec u, 2^{jm}\br{2\sqrt{n}}^m }f(2^mt,x)
$$
and 
$$
  \nrmb{g\ind_{{\mbs{F}}_Q^3\cap \mbs{R}_j}}_{L^{v'_{t}}_{t}L^{v'_{\vphantom{t}x}}_{\vphantom{t}x}} \lesssim  
  (2^{-j}\ell(Q))^{\frac{m}{v_t'}} \ell(Q)^{\frac n{v_x'}}
  \cdot H^{\whit}_{\vec v', 2^{jm}\br{\sqrt{n}}^m }g(2^mt,x).
$$
Collecting these estimates as for \framebox[15pt]{A}
gives an estimate by term {$(\mathrm{ca}$-$\mathrm{w})$}.
Finally, the estimate     
$$
{\sum_{{Q} \in \mc{D}}}\text{\framebox[15pt]{C}}\lesssim    \sum_{j=0}^{\infty} 2^{-jm(M+[u_t,v_t]-\kappa)}\cdot  \int_{\R^{1+n}_+}t^\kappa\cdot   H^{\whit}_{\vec{u},2^{jm}}f(t,x)\cdot H^{\ca}_{\vec{v}', 2^{jm} }g(t,x) \dd t \dd x
$$
is similar, which gives the stated domination.

If $T$ is causal, we note that \framebox[15pt]{C} and the term involving ${\mbs{E}}_Q^2$ and ${\mbs{F}}_Q^2$ in \eqref{eq:decompinproof} are zero. Hence, we do not get a Carleson type operator on $g$ in any of the estimates. Similarly, if $T$ is anti-causal, then \framebox[15pt]{B} and the term involving ${\mbs{E}}_Q^1$ and ${\mbs{F}}_Q^1$ in \eqref{eq:decompinproof} are  zero.  Hence, in this case we do not get a Carleson type operator on $f$ in any of the estimates.
\end{proof}

\subsection{Form domination for quasi-Banach spaces}
In this subsection, we adapt the form domination result in Theorem~\ref{thm:maindominationform} to a setting that is suitable for estimating singular operators in quasi-Banach spaces, such as \(T^{p,q,r}_\beta\) with \(\min\{p,q\}\in(0,1)\). To this end, we study the form

\[
  \int_{\R^{1+n}_+} |Tf|^\nu |g|
\]
for an exponent \(\nu\in(0,1)\) and a singular operator \(T\) of type \((\kappa,m,r,\vec u,\vec v,M)\).
A natural first idea would be to apply Theorem~\ref{thm:maindominationform} to the subadditive operator
\[
\widetilde T f := |Tf|^\nu.
\]
However, this approach does not work since \(\widetilde T\) does not satisfy the same off-diagonal estimates as \(T\). On the other hand, the decomposition in Proposition~\ref{prop:decomposition} only uses subadditivity, and therefore remains valid in this setting. More precisely, for  \(\nu\in(0,1)\) we have
\begin{align}  \label{eq:pdecomp}
\int_{\R^{1+n}_+}\abs{Tf}^\nu\abs{g}
 \leq
\sum_{Q \in \mc{D}} \int_{\mbs{Q}^{\whit}} \absb{T(f\ind_{\Qwtilde{Q}})}^\nu\abs{g}
 + 
\sum_{k=1,2,3} \sum_{Q \in \mc{D}} \int_{{\mbs{F}}_Q^k} \absb{T(f\ind_{{\mbs{E}}_Q^k})}^\nu\abs{g}.
\end{align}
This leads to the following form domination principle in the quasi-Banach range.

\begin{theorem}\label{thm:quasimaindominationform}
Let \(\nu\in(0,1)\), $r\in[1,\infty]$, $\vec{u},\vec{v} \in [1,\infty]^2$,  $\kappa, M \in \R$ and $m \in [1,\infty)$. 
Let $T$ be a
singular  operator  of type $(\kappa,m,r,\vec u,\vec v, M)$.
Then, for any $f,g \in L^\infty_c(\R^{1+n}_+)$ 
\begin{align*}
&\int_{{\mathbb R}^{1+n}_+}\abs{Tf}^\nu\abs{g} 
\lesssim 
\sum_{Q \in \mc{D}} \ell(Q)^{ \nu m\kappa} \cdot \ip{f}_{r,\Qwtilde{Q}}^\nu \cdot \ip{g}_{\vphantom{\Qwtilde{Q}}({r}/{\nu})',\Qw{Q}}^{}\cdot \abs{\Qw{Q}}
&& \\
&+\sum_{j=0}^{\infty} 2^{-jm\nu\br{M+[u_t,v_t]-\kappa}} \cdot \int_{\R^{1+n}_+}t^{\nu\kappa}\cdot   (H^{\whit}_{(r,u_x),2^{jm}}f(t,x))^\nu\cdot H^{\whit}_{\vphantom{\Qwtilde{Q}}(( r/\nu)',({v_x}/\nu)'), 2^{jm}}g(t,x)\dd x\dd t  &&\text{$(\mathrm{w}$-$\mathrm{w})$}\\
&+\sum_{j=0}^{\infty} 2^{-jm\nu\br{M+[u_t,v_t]-\kappa}} \cdot \int_{\R^{1+n}_+}t^{\nu\kappa}\cdot   (H^{\ca}_{\vec{u},2^{jm}}f(t,x))^\nu\cdot H^{\whit}_{({\vec{v}}/\nu)', 2^{jm}}g(t,x)\dd x\dd t &&\text{$(\mathrm{ca}$-$\mathrm{w})$}\\
&+\sum_{j=0}^{\infty} 2^{-jm\nu\br{M+[u_t,v_t]-\kappa}} \cdot
\int_{\R^{1+n}_+}t^{\nu\kappa}\cdot   (H^{\whit}_{\vec{u},2^{jm}}f(t,x))^\nu\cdot H^{\ca}_{({\vec{v}}/\nu)', 2^{jm}}g(t,x)\dd x\dd t. &&\text{$(\mathrm{w}$-$\mathrm{ca})$}
\end{align*}
If $T$ is causal, the term {$(\mathrm{w}$-$\mathrm{ca})$} is omitted, and if $T$ is anti-causal, the term {$(\mathrm{ca}$-$\mathrm{w})$} is  omitted.
\end{theorem}

\begin{proof}
  The proof follows that of Theorem~\ref{thm:maindominationform}, using the 
  decomposition~\eqref{eq:pdecomp} instead. Therefore, we only point out the differences.
The argument for the first term on the right-hand side of \eqref{eq:pdecomp} becomes
\begin{align*}
\sum_{{Q} \in {\mc{D}}} \int_{\mbs{Q}^{\whit}} \absb{T(f\ind_{\Qwtilde{Q}})}^\nu\abs{g} &\leq \sum_{{Q} \in {\mc{D}}}
\nrm {T(f\ind_{\Qwtilde{Q}})}_{\vphantom{\Qwtilde{Q}}L^{r}(\mbs{Q}^{\whit})}^\nu \cdot\nrm {g}^{\vphantom{\nu}}_{\vphantom{\Qwtilde{Q}}L^{( r/\nu)'}(\mbs{Q}^{\whit})} \\
&\lesssim \sum_{Q \in \mc{D}} \ell(Q)^{ \nu m\kappa} \cdot \ip{f}_{r,\Qwtilde{Q}}^\nu \cdot \ip{g}^{\vphantom{\nu}}_{\vphantom{\Qwtilde{Q}}({r}/{\nu})',\Qw{Q}}\cdot \abs{\Qw{Q}},
\end{align*}
which again matches the first term on the right-hand side of our claim.  
The H\"older estimate replacing \eqref{eq:toffest} for the term $k=1$ is
\begin{align*} 
 \int_{\mbs{Q}^{\whit}} &\absb{T(f\ind_{(0,4^{-m}\ell(Q){^m})\times 3Q})}^\nu\abs{g}\\
 &\lesssim  \Big(\ell(Q)^{-m[u_t,v_t]-{n}[u_x,v_x]+{m}\kappa} \nrmb{f\ind_{(0,4^{-m}\ell(Q)^m)\times 3Q}}_{L^{u_{t}}_{t}L^{u_{\vphantom{t}x}}_{\vphantom{t}x}}\Big)^\nu
 \cdot\nrmb{g\ind_{\mbs{Q}^{\whit}}}_{L^{(v_t/\nu)'}_{t}L^{(v_x/\nu)'}_{\vphantom{t}x}} \\
  &\lesssim \ell(Q)^{m+n+m\kappa \nu} \inf_{(t,x) \in \mbs{Q}^{\whit}} (H^{\ca}_{\vec{u},\br{2\sqrt{n}}^m}f(2^mt,x))^\nu\cdot H^{\whit}_{(\vec{v}/\nu)',\br{\sqrt{n}}^m}g(2^mt,x),
\end{align*}
using that 
\begin{equation}  \label{eq:pduality}
 -\nu[u_t,v_t]+\nu/u_t+ 1/(v_t/\nu)'=1= 
-\nu[u_x,v_x]+\nu/u_x+ 1/(v_x/\nu)'.   
\end{equation}
The estimate replacing \eqref{eq:Aestimatebegin} for term \framebox[15pt]{A}, using \eqref{eq:pduality}, is
\begin{align*}
    \int_{{\mbs{F}}_Q^3\cap \mbs{R}_j} &\absb{T(f\ind_{{\mbs{E}}_Q^3\cap \widetilde{\mbs{R}}_j})}^\nu\abs{g} \\
    &\lesssim  \Big(\ell(Q)^{-n[u_x,v_x]+m\kappa}\cdot 2^{-jm(M+[u_t,v_t])}
     \nrmb{f\ind_{{\mbs{E}}_Q^3\cap \widetilde{\mbs{R}}_j}}_{L^{r_{\vphantom{t}}}_{t}L^{u_{\vphantom{t}x}}_{\vphantom{t}x}}\Big)^\nu
     \cdot \nrmb{g\ind_{\mbs{F}_Q^3\cap \mbs{R}_j}}_{L^{(r_{\vphantom{t}}/\nu)'}_{t}L^{(v_x/\nu)'}_{\vphantom{t}x}} \\
     &\lesssim  2^{-jm\nu([u_t,v_t]+M)}\cdot  2^{-jm}\ell(Q)^{m+n+m\kappa \nu} \\
     &\hspace{1cm}\cdot \inf_{(t,x) \in \mbs{Q}_j} (H^{\whit}_{(r,u_x),2^{jm}\br{\sqrt{n}}^m}f(4^mt,x))^\nu\cdot H^{\whit}_{((r/\nu)',(v_x/\nu)'), 2^{jm} \br{\sqrt{n}}^m }g(4^mt,x), 
\end{align*}
which leads to an estimate by term {$(\mathrm{w}$-$\mathrm{w})$}.
The estimate replacing \eqref{eq:Btermstart} for term \framebox[15pt]{B}, using \eqref{eq:pduality}, is
\begin{align*}
    \int_{{\mbs{F}}_Q^3\cap\mbs{R}_j} &
    \absb{T(f\ind_{{\mbs{E}}_Q^3\cap  \mbs{R}^{\ca}_j})}^\nu\abs{g} \\
    &\lesssim  \Big(\ell(Q)^{-m[u_t,v_t]-n[u_x,v_x]+m\kappa}\cdot 2^{-jmM} 
    \nrmb{f\ind_{{\mbs{E}}_Q^3\cap {\mbs{R}}^{\ca}_j}}_{L^{u_{t}}_{t}L^{u_{\vphantom{t}x}}_{\vphantom{t}x}}\Big)^\nu
    \cdot\nrmb{g\ind_{\mbs{F}_Q^3\cap \mbs{R}_j}}_{L^{(v_t/\nu)'}_{t}L^{(v_x/\nu)'}_{\vphantom{t}x}}\\
     &\lesssim  2^{-jm\nu([u_t,v_t]+M)}\cdot  2^{-jm}\ell(Q)^{m+n+m\kappa \nu} \\
     &\hspace{1cm}\cdot \inf_{(t,x) \in \mbs{Q}_j} (H^{\ca}_{\vec u,2^{jm}\br{2\sqrt{n}}^m}f(2^mt,x))^\nu\cdot H^{\whit}_{(\vec v/\nu)', 2^{jm} \br{\sqrt{n}}^m }g(2^mt,x), 
\end{align*} 
which leads to an estimate by term {$(\mathrm{ca}$-$\mathrm{w})$}. The term \framebox[15pt]{C} is handled similarly, and the causal and anti-causal discussion is identical to that in the proof of Theorem \ref{thm:maindominationform}.
\end{proof}

\section{Boundedness of singular operators on tent spaces}\label{sec:estontentspace}

In this section we will first introduce a scale of $m$-parabolic tent spaces. Afterwards, we will show boundedness of singular operators on these spaces, using our domination results in Section~\ref{sec:domination}.

\subsection{Weighted, Whitney-averaged tent spaces}

We start by introducing  the scale of $m$-parabolic tent spaces $T^{p,q,r}_\beta$ with a
Whitney averaging parameter $r$ and a power weight $t^{-\beta}$, following \cite{ABH26,Haa26}.  These spaces 
appeared in a slightly different form,  {based on the classical tent spaces introduced
in \cite{CMS85} and with equivalent norms}, in \cite{HR13} at the endpoints $q=1$, $q=\infty$, and in  the general case  in \cite{Hua16}. Tent spaces, and even these weighted, Whitney-averaged tent spaces, arise naturally in the analysis of elliptic boundary value or parabolic initial value problems.
The basic observation is that a point source at $(0,x)$ mainly contributes to the
solution $f(t,y)$ of the PDE in a region $|y-x|\lesssim t^{1/m}$, where $m=1$ for elliptic boundary value problems and $m=2$ for parabolic initial value problems. This region can be decomposed into Whitney regions, over which the $T^{p,q,r}_\beta$-norm takes $L^r$-averages, followed by a weighted mixed-norm in $t$ and $x$.
The endpoint non-tangential maximal function spaces $T^{p,\infty,2}_0$ at $q=\infty$ were introduced in \cite{KP93} for $p=2$, where the $L^2$ Whitney averaging compensates for the lack of interior pointwise bounds for gradients of solutions to elliptic PDEs. The versions for general $p$ are introduced in  \cite{HR13}.
{The additional $L^r$-averaging on Whitney regions in the tent space norms fits well  with  our method based on local averaging. This recovers tent spaces by taking $q=r$ in our statements and, in particular, weighted Lebesgue spaces by taking $p=q=r$.}
New applications to PDEs, {already in the context of tent spaces in which $q=r$}, are given in Section~\ref{sec:ellipticPDES} for elliptic problems and in Section~\ref{sec:parabolic} for parabolic problems. 

Let $p,q,r \in (0,\infty]$ and \(\beta\in\R\). For a measurable \(f\colon\R^{1+n}_+\to\C\), define the weighted mixed-norm
\(L^p_{\vphantom{\beta}}(\R^n;L^q_\beta(\R_+))\) by
\begin{align*}
  \nrm{f}_{L^p_{\vphantom{\beta}}(\R^n;L^q_\beta(\R_+))}  &:= \nrm{(t,x) \mapsto t^{-\beta}f(t,x)}_{L^p_{\vphantom{\beta}}(\R^n;L^q_{\vphantom{\beta}}(\R_+))}\\&\phantom{:}=
  \brB{   \int_{\R^n} \brB{\int_0^\infty |t^{-\beta}f(t,x)|^q \dd t}^{p/q}\dd x
  }^{1/p},
\end{align*}
with the usual modifications for infinite exponents.
For $r \in (0,\infty)$ define the Whitney average
\begin{align*}
  W_{r}f(t,x)
  :=  \brB{    \fint_{2^{-m}t}^{t}
    \fint_{B(x, t^{1/m})}
    |f(s,y)|^r \dd y \dd s
  }^{1/r},
  \qquad (t,x)\in\R^{1+n}_+.
\end{align*}
 For $r=\infty$,  $W_{\infty}f(t,x)$ is defined by replacing the integral by the  essential supremum over the same region.

Throughout this section, we fix $m\in[1,\infty)$ and suppress it from the notation for the Whitney averages. In principle, one could allow $m>0$, at the expense of some modifications. The parameter $m$ is a homogeneity parameter determined by the intended application. Typical choices are $m=1$ for elliptic problems and $m=2$ for second-order parabolic problems.

For $p<\infty$, we define $T^{p,q,r}_{\beta}$ as the space of all measurable $f \colon \R_+^{1+n} \to \C$ such that
\begin{align*}
\nrm{f}_{T^{p,q,r}_{\beta}}
  &:=
  \nrm{W_r f}_{L^p_{\vphantom{\beta}}(\R^n;L^q_\beta(\R_+))}<\infty.
\end{align*}
Note that for $p=q=r$, this is the weighted space {$L^r_\beta(\R_+^{1+n})= L^r_{\vphantom{\beta}}(\R_+^{1+n}; t_{\vphantom{\beta}}^{-\beta r}\dd x \dd t)$ by Fubini's theorem.}
For $p=\infty$ and $q<\infty$ we define $T^{\infty,q,r}_\beta$ as the space of all measurable $f \colon \R_+^{1+n} \to \C$ such that
$$
  \nrm{f}_{T^{\infty,q,r}_\beta}
  := \sup_{(\tau,y)\in\R^{1+n}_+} \brB{ 
  \int_0^\tau \fint_{B(y,\tau^{1/m})} |t^{-\beta}W_r f(t,x)|^q \dd x\dd t }^{1/q}<\infty.
$$
{Lower semi-continuity of the averages on Carleson boxes shows that one can use $\esssup$ and $\sup$ interchangeably.} For $p=q=\infty$, set {
$$\nrm{f}_{T^{\infty,\infty,r}_\beta}:= \nrm{t^{-\beta} W_r f}_{L^\infty_{\vphantom{\beta}}(\R^{1+n}_+)}=\sup_{(\tau,y)\in\R^{1+n}_+} \nrm{t^{-\beta}W_r f(t,x)}_{L^\infty((0,\tau)\times B(y,\tau^{1/m}))},$$ where the last term is also the  limit as $q\to \infty$ of $\nrm{f}_{T^{\infty,q,r}_\beta}$.
   }
Furthermore, we set
\begin{align*}
  T^{p,q}_{\beta}
  :=
  T^{p,q,q}_{\beta},
\end{align*}
which coincides with the more classical tent spaces with an equivalence of norms. In particular, for $p<\infty$ we have
$$
\nrm{f}_{T^{p,q}_\beta} \eqsim \brB{\int_{\R^n} \brB{\int_0^\infty \fint_{B(x,t^{1/m})} \abs{t^{-\beta} f(t,y)}^q \dd y \dd t}^{p/q}\dd x}^{1/p}.
$$

Let us record some basic properties of these weighted tent spaces with Whitney averages. Note that we use the Lebesgue measure $\ddn t$ on $\R_+$ in the definition of $T^{p,q,r}_{\beta}$, while earlier references use the Haar measure $\ddn t/t$. This only induces a trivial shift in the parameter $\beta$.
First, \(T^{p,q,r}_\beta\) is a quasi-Banach function space (see \cite{Hua16,LN24}).
For \(\nu \in (0,\infty)\) we have
\begin{equation}  \label{eq:rescalenorm}
  \|f\|_{T^{p,q,r}_\beta}^\nu = \||f|^\nu\|_{T^{p/\nu,q/\nu,r/\nu}_{\beta \nu}},
\end{equation}
from which it follows immediately that \(T^{p,q,r}_\beta\) is \(\nu\)-convex for
\(\nu := \min\cbrace{p,q,r}\), i.e., for all \(f_1,\ldots,f_k \in T^{p,q,r}_\beta\) we have
\begin{align}\label{eq:convex}
    \nrmB{\brB{\sum_{j=1}^k \abs{f_j}^\nu}^{1/\nu}}_{T^{p,q,r}_\beta}
    \leq
    \brB{\sum_{j=1}^k \nrm{f_j}_{T^{p,q,r}_\beta}^{\nu}}^{1/\nu}.
\end{align}
Indeed, for this choice of \(\nu\) one has \(p/\nu, q/\nu, r/\nu \geq 1\), so that
\(T^{p/\nu,q/\nu,r/\nu}_{\beta \nu}\) is a Banach function space.
Furthermore, for \(\lambda > 0\) the norm has the following dilation behaviour
\begin{equation*}
    \| (t,x)\mapsto f(\lambda^m t, \lambda x) \|_{T^{p,q,r}_\beta}
    =
    \lambda^{m(\beta - \frac{1}{q}) - \frac{n}{p}} \|f\|_{T^{p,q,r}_\beta},
\end{equation*}
which follows directly from a change-of-variables argument.

Finally, the precise choice of Whitney averages is flexible, provided the averaging region remains uniformly separated from the boundary \(t=0\). This follows from \cite[Observation 2.4]{Hua16} after using the change of variables $t = s^m$ to reduce to the case $m= 1$.
\begin{prop}  \label{prop:equivnorms}
  Let $p,q,r \in (0,\infty]$, fix $0<a<b<\infty$ and $c>0$.
    Replacing $W_r f$ by 
    \begin{align*}
  W^{a,b,c}_{r}f(t,x)
  :=  \brB{    \fint_{at}^{bt}
    \fint_{B(x, (ct)^{1/m})}
    |f(s,y)|^r \dd y \dd s
  }^{1/r},
  \qquad (t,x)\in\R^{1+n}_+,
\end{align*}
in the definition of the tent space norms
yields (quasi-)norms equivalent to $\nrm{f}_{T^{p,q,r}_\beta}$.
\end{prop}

Next, we note that we can discretize the tent space norms, using our fixed dyadic system $\mc{D}$.

\begin{prop}\label{prop:dyadic}
Let $p,q,r \in (0,\infty]$ and $\beta \in \R$. For $Q \in \mc{D}$, define
\[
  W_{r,Q}(f) := \brB{ \fint_{\mbs{Q}^{\whit}} \abs{f(t,y)}^r \dd y \dd t }^{1/r}.
\]
If $p<\infty$, then
\[
  \nrm{f}_{T^{p,q,r}_\beta} \eqsim \nrmb{\cbrace{W_{r,Q}(f)}_{Q \in \mc{D}}}_{t^{p,q}_\beta}:=\nrmB{ \brB{ \sum_{Q \in \mc{D}} \ell(Q)^{m-m\beta q}\cdot  W_{r,Q}(f)^q \ind_Q }^{1/q} }_{L^p_{\vphantom{\beta}}(\R^n)}.
\]
If $p=\infty$, then
$$
 \nrm{f}_{T^{\infty,q,r}_\beta} \eqsim \nrmb{\cbrace{W_{r,Q}(f)}_{Q \in \mc{D}}}_{t^{\infty,q}_\beta}:=
 \sup_{Q_0\in\mc{D}}
 \brB{ \fint_{Q_0}\sum_{Q\in \mc{D}(Q_0)}
 \ell(Q)^{m-m\beta q}\cdot  W_{r,Q}(f)^q \ind_Q(x) \dd x
 }^{1/q},
$$
with the usual modifications  when $q=\infty$.
\end{prop}

\begin{proof}
  The proof is a straightforward generalization of \cite[Propositions 3.12 and 3.13]{Haa26} to $m\ge 1$. In the proofs there, it suffices to use 
  Proposition~\ref{prop:equivnorms} instead of a precise change of angle for the norms. 
\end{proof}

We end our introduction of weighted, Whitney-averaged tent spaces with a duality,  density 
and norming statement for $T^{p,q,r}_\beta$. To have a unified duality statement including $p,q,r = \infty$, we use the \emph{K\"othe dual} or \emph{associate space} $X'$ of a (quasi)-Banach function space $X \subseteq L^0(\R^{1+n}_+)$. It is defined as
\[
X':=\{g\in L^0(\R^{1+n}_+):fg\in L^1(\R^{1+n}_+)\text{ for all $f\in X$}\},
\]
with norm 
\[
\nrm{g}_{X'}:=\sup_{\nrm{f}_X=1}\|fg\|_{L^1(\R^{1+n}_+)}, \qquad  g\in X',
\]
see \cite[Section 3.1]{LN24} and the references therein for an introduction.

\begin{prop}\label{prop:tentspacedensityduality} 
Let $p,q,r\in (0,\infty]$ and $\beta\in \R$.
\begin{enumerate}[(i)]
    \item\label{it:density} $L^\infty_c(\R^{1+n}_+)$ is dense in $T^{p,q,r}_\beta$ if $\max\{p,q,r\}<\infty$ and weak-star dense if  $\min\{p,q,r\}>1$.
    \item\label{it:duality} If $\min\{p,q,r\} \ge  1$, the K\"othe dual of $T^{p,q,r}_\beta$ is
$T^{p',q',r'}_{-\beta}$.
In particular 
$$
  \nrm{f}_{T^{p,q,r}_\beta}\eqsim
  \sup_g\int_{\R^{1+n}_+} |f(t,x)|\cdot |g(t,x)| \dd t \dd x,
$$
where the supremum is taken over all $g\in L^\infty_c(\R^{1+n}_+)$ with 
$\nrm{g}_{T^{p',q',r'}_{-\beta}}=1$.
\item\label{it:quasiduality} If $0<\nu \le \min(p,q,r)$, 
\begin{equation*} 
  \|f\|_{T^{p,q,r}_{\beta}}^\nu \eqsim \sup_{g} \int_{\R^{1+n}_+} |f(t,x)|^\nu \cdot|g(t,x)| \dd t \dd x,
\end{equation*}  
where the supremum is taken over all $g\in L^\infty_c(\R^{1+n}_+)$ with $\|g\|_{T^{(p/\nu)',(q/\nu)', (r/\nu)'}_{-\beta \nu}} = 1$.
\end{enumerate}
\end{prop}

\begin{proof}
{Strong density in \ref{it:density}  follows from a standard approximation argument and weak-star density follows as a consequence of \ref{it:duality}.}
For \ref{it:duality} we note that, after a change of variables $s=t^m$ to reduce to the case $m=1$ and in our notation, \cite[Theorem 5.2]{Hua16} states  that 
\begin{align}\label{eq:multiplierkothe}
    T^{p',q',r'}_{-\beta} = \mc{M}(T^{p,q,r}_\beta, T^{1,1,1}_0) = \mc{M}(T^{p,q,r}_\beta, L^1(\R^{1+n}_+)).
\end{align}
Here, $\mc{M}(T^{p,q,r}_\beta, L^1(\R^{1+n}_+))$ denotes the space of all multipliers from $T^{p,q,r}_\beta$ to $L^1(\R^{1+n}_+)$, i.e., all $g \in L^0(\R_{+}^{1+n})$ such that $f \mapsto gf$ is a bounded operator from $T^{p,q,r}_\beta$ to $L^1(\R^{1+n}_+)$. Hence, \eqref{eq:multiplierkothe}  is exactly the same as saying that $T^{p',q',r'}_{-\beta}$ is the K\"othe dual of $T^{p,q,r}_\beta$. 
Now by the Lorentz-Luxemburg theorem (see \cite[Theorem~71.1]{Za67}), we know that
$$
  \nrm{f}_{T^{p,q,r}_\beta}\eqsim
  \sup_{g}\int_{\R^{1+n}_+} |f(t,x)|\cdot |g(t,x)| \dd t \dd x,
$$
where the supremum is taken over all $g \in T^{p',q',r'}_{-\beta}$ with $\nrm{g}_{T^{p',q',r'}_{-\beta}} = 1$. Finally, we note that it suffices to take the supremum over $g \in L^\infty_c(\R^{1+n}_+)$ by truncations and monotone convergence. 
Finally, \ref{it:quasiduality} follows from the power rule formula \eqref{eq:rescalenorm} and \ref{it:duality} applied with $p/\nu, q/\nu, r/\nu$.  
\end{proof}

\begin{remark}\label{rem:p<1}
Even though the dual of $T^{p,q,r}_\beta$ in the quasi-Banach range $\min\{p,q,r\}<1$ has been identified in \cite[Theorems 3.23 and 3.25]{Haa26}, this  is not of use in our analysis. Indeed, since the 
Lorentz-Luxemburg theorem does not hold for quasi-Banach function spaces, we cannot use the dual of $T^{p,q,r}_\beta$ to calculate the norm of $f \in T^{p,q,r}_\beta$ in this range. Instead, we will use Proposition \ref{prop:tentspacedensityduality}\ref{it:quasiduality}, which is based on the power rule \eqref{eq:rescalenorm}. However, the dual of $T^{p,q,r}_\beta$ can be used to deduce boundedness results for the adjoints {$T^*$ of linear singular operators $T$ to which our statements apply}.
\end{remark}

\subsection{Model operators} \label{sec:modelops}
In Section~\ref{sec:domination} we showed that singular operators can be dominated by a local Whitney averaging operator and the model operators \(H^{\ca}_{\vec{u},\delta}\) and \(H^{\whit}_{\vec{u},\delta}\). Therefore, in order to prove boundedness of singular operators on the tent spaces $T^{p,q,r}_\beta$, we need to study the boundedness of these operators on $T^{p,q,r}_\beta$, which is the aim of this subsection.
The basic idea is to extrapolate from $L^r$- to $T^{p,q,r}_\beta$-estimates.
Indeed, an early version of this work used Rubio de Francia extrapolation from weighted $L^r(w)$ to 
$T^{p,q,r}_\beta$, similarly to \cite{MP24}, to obtain the estimates in this subsection in the case $q=r$.  
However, the proofs below use direct estimates, valid also for $q\ne r$.

As we shall see, even if one is interested only in the two-parameter tent spaces \(T^{p,q}_\beta\), it is useful to view them in the proofs as Whitney-averaged tent spaces \(T^{p,q,r}_\beta\) with \(r=q\). Moreover, we shall prove boundedness in a somewhat larger range of exponents than is used later, for example allowing \(r=\infty\), as this requires  no additional effort. We will also track the dependence on the aperture parameter \(\delta\) carefully. This is crucial because our domination result for singular operators in Section~\ref{sec:domination} contains infinite sums of model operators whose convergence depends on the rate of growth in~\(\delta\).

We start with the estimate for the local Whitney averaging operator.

\begin{lemma}\label{lemma:Whitneymap}
Let $p,q,r \in (0,\infty]$, $\beta,\kappa \in \R$ . Then
$$
  A_{r,\kappa}^{\whit} f:=  \sum_{Q \in \mc{D}} \ell(Q)^{m\kappa} \cdot \ip{f}_{r,\Qwtilde{Q}} \ind_{\Qw{Q}}, \qquad f\in L^0(\R^{1+n}_+),
$$
defines a bounded operator  from $T^{p,q,r}_\beta$ to $T^{p,q,r}_{\beta+\kappa}$.
\end{lemma}

\begin{proof} 
Consider first $p<\infty$. Applying first Proposition~\ref{prop:dyadic} and then 
Proposition~\ref{prop:equivnorms}, we estimate
\begin{align*}
    \nrm{A_{r,\kappa}^{\whit}f }_{T^{p,q,r}_{\beta+\kappa}} &\eqsim \nrmB{ \brB{ \sum_{Q \in \mc{D}} \ell(Q)^{m-m(\beta+\kappa) q}\cdot  W_{r,Q}\brb{\ell(Q)^{m\kappa}\cdot  \ip{f}_{r,\Qwtilde{Q}} \ind_{\Qw{Q}}}^q \ind_Q }^{1/q} }_{L^p_{\vphantom{\beta}}(\R^n)}\\
    &= 
    \nrmB{ \brB{ \sum_{Q \in \mc{D}} \ell(Q)^{m-m\beta q}\cdot  \brB{\fint_{\Qwtilde{Q}}|f|^r}^{q/r} \ind_Q }^{1/q} }_{L^p_{\vphantom{\beta}}(\R^n)} \\
    &\lesssim
    \nrmb{(t,x)\mapsto t^{-\beta} W^{a,b,c}_r f(t,x)
    }_{L^p_{\vphantom{\beta}}(\R^n; L^q(\R_+))} \\
    &\eqsim\nrm{f}_{T^{p,q,r}_\beta}, 
\end{align*}
using that for suitable $a,b,c$ we have
$\Qwtilde{Q}\subseteq (at,bt)\times B(x,(ct)^{1/m})$
for $x\in Q$ and $2^{-m}\ell(Q)^m\le t\le \ell(Q)^m$.

For $p=\infty$, we similarly estimate 
\begin{align*}
   \nrm{A_{r,\kappa}^{\whit}f}_{T^{\infty,q,r}_{\beta+\kappa}} 
    &\eqsim
 \sup_{Q_0\in\mc{D}}
 \brB{ \fint_{Q_0}\sum_{Q\in \mc{D}(Q_0)}
 \ell(Q)^{m-m\beta q}\cdot  \brB{\fint_{\Qwtilde{Q}}|f|^r}^{q/r}  \ind_Q(x) \dd x
 }^{1/q}\\
    &\lesssim
    \sup_{\tau>0}\sup_{y\in\R^n} \brB{ 
  \int_0^\tau \fint_{B(y,\tau^{1/m})} |t^{-\beta}W^{a,b,c}_r f(t,x)|^q \dd x\dd t }^{1/q} \\
    &\eqsim\nrm{f}_{T^{p,q,r}_\beta}, 
\end{align*}
again by embedding 
$\Qwtilde{Q}\subseteq (at,bt)\times B(x,(ct)^{1/m})$.
\end{proof}

To estimate \(H^{\ca}_{\vec{u},\delta}\) and \(H^{\whit}_{\vec{u},\delta}\), we will repeatedly use the following elementary martingale-type fact:
for a measurable $f \colon \R^n \to [0,\infty)$, $r \in (0,\infty)$, $s,t \in \R_+$ with $s\leq t$ and $x \in \R^n$ we have 
\begin{align}\label{eq:martingale}
  \brB{\fint_{B(x,s)}\brB{\fint_{B(y,t)} f(z) \dd z}^r \dd y}^{1/r} \leq \esssup_{y \in B(x,s)}\fint_{B(y,t)} f(z) \dd z\lesssim  \fint_{B(x,s+t)} f(z) \dd z.
\end{align}
Conversely, we note that
\begin{align}\label{eq:martingaleback}
  \brB{\fint_{B(x,t)} f(z)^r \dd z}^{1/r} \lesssim  \fint_{B(x,s)}\brB{\fint_{B(y,s+t)} f(z)^r \dd z}^{1/r} \dd y.
\end{align}

Our estimates for \(H^{\ca}_{\vec{u},\delta}\) and \(H^{\whit}_{\vec{u},\delta}\) will be based on estimates for the following auxiliary averaging operator: for $u \in (0,\infty)$, $\delta\geq 1$ and a measurable $f \colon \R^{1+n}_+ \to \C$  define
\begin{equation}  \label{eq:avop}
  A_{u,\delta}f(t,x)
  :=
  \brB{\fint_{B(x,(\delta t)^{1/m})}
    |f(t,z)|^u\dd z
  }^{1/u},
  \qquad (t,x)\in\R^{1+n}_+ .    
\end{equation}

We will start by considering the case $p<\infty$ and  treat the endpoint $p=\infty$ afterwards.

\begin{lemma}\label{lem:slice}
    Let $p,u \in (0,\infty)$, $q,r \in (0,\infty]$ with $u\leq r$, $\beta \in \R$  and $\delta \ge 1$ and $\varepsilon>0$. Set
\[
  \gamma
  :=\tfrac{n}{m}
  \bracb{\min\cbrace{p,q,u},u}.
\]
Then we have for all $f \in T^{p,q,r}_\beta$
\[
  \nrm{A_{u,\delta}f}_{T^{p,q,r}_\beta}
  \lesssim
  \delta^{\gamma+\varepsilon}\,
  \|f\|_{T^{p,q,r}_\beta},
\]
The implicit constant is independent of $\delta$.
\end{lemma}

\begin{proof}
By Proposition \ref{prop:dyadic}, it suffices to bound the discrete sequence norm $\nrm{ \{ W_{r,Q}(A_{u,\delta} f) \}_{Q \in \mc{D}} }_{t^{p,q}_\beta}$. 
For $\delta\ge 1$, there exists an integer $J$, only depending on $n,m$ and $\delta$, such that the index set 
\begin{equation}\label{eq:J}
 \Lambda_\delta := \cbraceb{ k \in \Z^n : \abs{k}_\infty \le 2^J },
\end{equation}
  (where $\abs{k}_\infty$ denotes the max-norm) has  cardinality \(N:=\abs{\Lambda_\delta}\eqsim \delta^{n/m}\) and we have the uniform covering
\[
  B(x, (\delta t)^{1/m}) \subseteq \bigcup_{k \in \Lambda_\delta} (Q + k \ell(Q))
\]
for all $Q \in \mc{D}$ and $(t,x) \in \mbs{Q}^{\whit}$. Indeed, this is obtained by translating and rescaling a corresponding covering for the unit dyadic cube.

 Fix a dyadic cube $Q \in \mc{D}$. For $k \in \Lambda_\delta$ let $Q_k := Q + k\ell(Q)$ denote the shifted dyadic cubes. For all $(t,x) \in \Qw{Q}$ we have by the covering and Jensen's inequality
\begin{align*}
  A_{u,\delta} f(t,x) 
  &= \brB{ \fint_{B(x, (\delta t)^{1/m})} \abs{f(t,z)}^u \dd z }^{\frac{1}{u}} \\
  &\lesssim \brB{ \frac{1}{N} \sum_{k \in \Lambda_\delta} \fint_{Q_k} \abs{f(t,z)}^u \dd z }^{\frac{1}{u}} \\
  &\le \brB{ \frac{1}{N} \sum_{k \in \Lambda_\delta} \brB{ \fint_{Q_k} \abs{f(t,z)}^r \dd z }^{\frac{u}{r}} }^{\frac{1}{u}}.
\end{align*}
Fix $\nu < \min\{p, q, u\}$ such that $\gamma_\nu := \frac{n}{m} [\nu,u] =\gamma+\varepsilon$.
Since the right-hand side is independent of $x$, averaging this over $\Qw{Q}$ and using Minkowski's inequality since $u\le r$, we obtain 
\begin{align}\label{eq:whitofav}
\begin{aligned}
  W_{r,Q}(A_{u,\delta} f) &\lesssim
  \brB{\fint_{2^{-m}\ell(Q)^m}^{\ell(Q)^m}\brB{ \frac{1}{N} \sum_{k \in \Lambda_\delta} \brB{ \fint_{Q_k} \abs{f(t,z)}^r \dd z }^{\frac{u}{r}} }^{\frac{r}{u}} \dd t}^{1/r}
  \\ &\leq\brB{ \frac{1}{N} \sum_{k \in \Lambda_\delta} W_{r,Q_k}(f)^u }^{\frac{1}{u}} \\
  &\leq N^{\frac{1}{\nu} - \frac{1}{u}} \brB{ \frac{1}{N} \sum_{k \in \Lambda_\delta} W_{r,Q_k}(f)^\nu }^{\frac{1}{\nu}} \\
  &\eqsim \delta^{\gamma_\nu} \brB{ \frac{1}{N} \sum_{k \in \Lambda_\delta} W_{r,Q_k}(f)^\nu }^{\frac{1}{\nu}}. 
\end{aligned}
\end{align}
This estimate also holds for $r=\infty$ with the usual modifications. 

Define for $j\in \Z$, 
\[
  F_j(y) := \sum_{P \in \mc{D}: \ell(P) = 2^{-j}} W_{r,P}(f)^\nu \ind_P(y), \qquad y \in \R^n.
\]
Note that when $\ell(Q)=2^{-j}$, 
\[ \fint_{\cup_{k \in \Lambda_\delta } Q_k} F_j(y)\dd y= \frac{1}{N} \sum_{k \in \Lambda_\delta} W_{r,Q_k}(f)^\nu.
\] 
As $\cup_{k \in \Lambda_\delta } Q_k$ is a region contained within a ball centered at $x$ of radius $\eqsim \delta^{\frac{1}{m}}\ell(Q)$ and of measure $\eqsim \delta^{\frac{n}{m}}\abs{Q}$, we obtain 
\[ 
  \brB{ \frac{1}{N} \sum_{k \in \Lambda_\delta} W_{r,Q_k}(f)^\nu }^{\frac{1}{\nu}} \lesssim \brb{ {M}(F_j)(x) }^{\frac{1}{\nu}},
\]
where ${M}$ denotes the Hardy--Littlewood maximal operator.

Let us now evaluate the $t^{p,q}_\beta$-sequence norm. By partitioning the sum over all $Q \in \mc{D}$ into scales $j \in \Z$, and using the fact that the cubes at a fixed scale $j$ partition $\R^n$, we obtain
\begin{align*}
  \nrmb{ \{ W_{r,Q}(A_{u,\delta} f) \}_{Q \in \mc{D}} }_{t^{p,q}_\beta}
  &= \nrmB{ \brB{ \sum_{j \in \Z} \sum_{\ell(Q)=2^{-j}} 2^{-jm+jm\beta q} \cdot W_{r,Q}(A_{u,\delta} f)^q \ind_Q }^{1/q} }_{L^p_{\vphantom{\beta}}(\R^n)} \\
  &\lesssim \delta^{\gamma_\nu} \nrmB{ \brB{ \sum_{j \in \Z} 2^{-jm+jm\beta q}\cdot {M}(F_j)^{q/\nu} }^{1/q} }_{L^p_{\vphantom{\beta}}(\R^n)}\\
  &= \delta^{\gamma_\nu} \nrmB{ \brB{ \sum_{j \in \Z} 2^{-jm+jm\beta q}\cdot {M}(F_j)^{q/\nu} }^{\nu/q} }_{L^{p/\nu}_{\vphantom{\beta}}(\R^n)}^{1/\nu}.
\end{align*}
Now, using that  $q/\nu>1$ and $p/\nu>1$, the vector-valued Fefferman--Stein 
maximal inequality \cite{FS71}, which trivially also holds when $q=\infty$, yields 
\begin{align*}
  \nrmb{ \{ W_{r,Q}(A_{u,\delta} f) \}_{Q \in \mc{D}} }_{t^{p,q}_\beta}
  &\lesssim \delta^{\gamma_\nu } \nrmB{ \brB{ \sum_{j \in \Z} \brb{ 2^{-jm\nu/q+jm\beta \nu}\cdot {M}(F_j) }^{q/\nu} }^{\nu/q} }_{L^{p/\nu}_{\vphantom{\beta}}(\R^n)}^{1/\nu} \\
  &\lesssim \delta^{\gamma_\nu } \nrmB{ \brB{ \sum_{j \in \Z} \brb{ 2^{-jm\nu/q+jm\beta \nu} \cdot F_j }^{q/\nu} }^{\nu/q} }_{L^{p/\nu}_{\vphantom{\beta}}(\R^n)}^{1/\nu} \\
  &= \delta^{\gamma_\nu} \nrmB{ \brB{ \sum_{j \in \Z} \sum_{\ell(Q)=2^{-j}}  2^{-jm+jm\beta q}\cdot W_{r,Q}(f)^q \ind_Q }^{1/q} }_{L^p_{\vphantom{\beta}}(\R^n)} \\
  &= \delta^{\gamma +\varepsilon} \nrmb{ \{ W_{r,Q}(f) \}_{Q \in \mc{D}} }_{t^{p,q}_\beta}.
\end{align*}
Applying Proposition \ref{prop:dyadic} yields the claimed estimate.
\end{proof}

\begin{remark}\label{remark:noepsilonloss}
   The $\varepsilon$-loss in Lemma~\ref{lem:slice} can be avoided in two cases. First, if $\min\{p,q\}>u$, one may take $\nu=u$ in the proof, which directly gives the estimate with $\varepsilon=0$. Indeed, the only place where we used $\nu<\min\{p, q, u\}$ is in the application of the vector-valued Fefferman--Stein inequality, which only used $\nu <\min\cbrace{p,q}$.  
   
   Second, if $q\leq p,u$, one may also take $\varepsilon=0$. Indeed, taking $\nu=q$ in \eqref{eq:whitofav}, raising to the $q$-th power, and dualizing the resulting $L^{p/q}$-norm, one can reindex the shifted dyadic cubes and reduce the estimate to the scalar Hardy--Littlewood maximal operator. 
   
   We do not need these refinements below and leave the details to the interested reader.
\end{remark}

Using the boundedness of the auxiliary averaging operator $A_{u,\delta}$, we are now ready to prove the boundedness of $H^{\whit}_{\vec u,\delta}$ and $H^{\ca}_{\vec u,\delta}$ on $T^{p,q,r}_{\beta}$ for $p<\infty$.

\begin{prop}
\label{prop:modelops-whitneymixed}
Let \(p \in (0,\infty)\), $q,r \in (0,\infty]$, \(\beta\in\R\), \(\delta\in[1,\infty)\) and $\varepsilon>0$. Let
\(\vec u\in(0,\infty)^2\) satisfy \(u_t,u_x\le r\) and set
\begin{align*}
  \gamma
  :=
  \tfrac{n}{m}
  {\bracb{\min\cbrace{p,q,u_x},u_x}.}
\end{align*}
For all $f \in T^{p,q,r}_{\beta}$ we have
  \begin{align*}
    \nrm{H^{\whit}_{\vec u,\delta}f}_{T^{p,q,r}_{\beta}}
    \lesssim  \delta^{\gamma+\varepsilon}\,\nrm{f}_{T^{p,q,r}_{\beta}}.
  \end{align*}
 If $
    \beta>\gamma-[u_t,q]$, 
    we also have
  \begin{align*}
    \nrm{H^{\ca}_{\vec u,\delta}f}_{T^{p,q,r}_{\beta}}
    \lesssim
    \delta^{\gamma+\varepsilon}\,
    \nrm{f}_{T^{p,q,r}_{\beta}}.
  \end{align*}
 The implicit constants are independent of $\delta$.
\end{prop}
\begin{proof}
The claim for $H^{\whit}_{\vec u,\delta}$ follows almost immediately from Lemma \ref{lem:slice}. Indeed, since $u_t\le r$, Jensen's inequality in the $s$-average gives for $(t,x) \in \R^{1+n}_+$
\begin{align*}
      H^{\whit}_{\vec u,\delta}f(t,x) &=
\brB{\fint_{c^m t}^t\brB{\fint_{B(x,(\delta t)^{1/m})} \abs{f(s,y)}^{u_x} \dd y}^{\frac{u_t}{u_x}} \dd s}^{\frac{1}{u_t}}\\  
  &\lesssim
  \brB{
    \fint_{c^m t}^t
    A_{u_x,c^{-m}\delta}f(s,x)^r\dd s
  }^{1/r}.
\end{align*}
Therefore, for $(\tau,z) \in \R^{1+n}_+$
\begin{align}\label{eq:AudomHw}
\begin{aligned}
      W_r(H^{\whit}_{\vec u,\delta}f)(\tau,z)
  &\lesssim
  \brB{\fint_{2^{-m} \tau}^\tau
    \fint_{B(z,\tau^{1/m})}
    \fint_{c^m t}^t
    A_{u_x,c^{-m}\delta}f(s,x)^r\dd s\dd x\dd t
  }^{1/r} \\
  &\lesssim
  \brB{
    \fint_{(c/2)^{m} \tau}^\tau
    \fint_{B(z,\tau^{1/m})}
    A_{u_x,c^{-m}\delta}f(s,x)^r\dd x\dd s
  }^{1/r}.
\end{aligned}
\end{align}
On the right-hand side, we are integrating over a $t$-enlarged Whitney region. Therefore, taking $L^p_{\vphantom{\beta}}(\R^n;L^q_\beta(\R_+))$-norms and applying  
 Lemma \ref{lem:slice}, we conclude 
\[
  \|H^{\whit}_{\vec u,\delta}f\|_{T^{p,q,r}_\beta}
  \lesssim
  \|A_{u_x,c^{-m}\delta}f\|_{T^{p,q,r}_\beta}
  \lesssim
  \delta^{\gamma+\varepsilon} \|f\|_{T^{p,q,r}_\beta}.
\]

To prove the claim for $H^{\ca}_{\vec u,\delta}$, fix $(t,x) \in \R^{1+n}_+$. Define $I_k(t):= (c^{m(k+1)}t ,c^{mk}t )$ for $k\geq 0$ and decompose
\[
  (0,t )=\bigcup_{k=0}^\infty I_k(t)
\]
modulo null sets.
Define $\nu:=\min\cbrace{p,q,u_t}$. Then, since for $s \in I_k(t)$ we have $t \leq  c^{-m(k+1)}s \leq c^{-m} t$, we can estimate
\begin{align}\label{eq:pointdomHw}
\begin{aligned}
  H^{\ca}_{\vec u,\delta}f(t,x)
  &=
  \brB{
    \sum_{k=0}^\infty
    \frac{|I_k(t)|}{t }    \fint_{I_k(t)}
    \brB{      \fint_{B(x,(\delta t )^{1/m})}|f(s,y)|^{u_x}\dd y
    }^{u_t/u_x}    \dd s
  }^{1/u_t} \\
  &\lesssim
  \brB{\sum_{k=0}^\infty
  c^{\nu mk/u_t}  \brB{
    \fint_{I_k(t)}
    \Bigl(
      \fint_{B(x,(\delta t )^{1/m})}|f(s,y)|^{u_x}\dd y
    \Bigr)^{u_t/u_x}
    \dd s
  }^{\nu/u_t}}^{1/\nu}.\\
 &\lesssim
  \brB{\sum_{k=0}^\infty
  c^{\nu mk/u_t}  \brB{
    \fint_{I_k(t)}
    A_{u_x, c^{-m(k+1)}\delta}f(s,x)^{u_t}
    \dd s
  }^{\nu/u_t}}^{1/\nu}.
\end{aligned}
\end{align}
Taking the $T^{p,q,r}_\beta$-norm and using \eqref{eq:convex}
we obtain
\begin{equation}\label{eq:estHcafirst}
  \|H^{\ca}_{\vec u,\delta}f\|_{T^{p,q,r}_\beta} \lesssim \brB{ \sum_{k=0}^\infty c^{\nu mk/u_t} \cdot \nrmB{(t,x)\mapsto \brB{
    \fint_{I_k(t)}
    A_{u_x, c^{-m(k+1)}\delta}f(s,x)^{u_t}
    \dd s
  }^{1/u_t} }_{T^{p,q,r}_\beta}^\nu }^{1/\nu}.
\end{equation}
We will analyse the norm on the right-hand side for fixed $k\geq 0$. For $(\tau,z) \in \R^{1+n}_+$ the Whitney average of the $T^{p,q,r}_\beta$-norm can, by Jensen's inequality, be estimated as follows
\begin{align*}
  \brB{ \fint_{2^{-m}\tau}^\tau &\fint_{B(z,\tau^{1/m})} \brB{ \fint_{I_k(t)} A_{u_x, c^{-m(k+1)}\delta}f(s,x)^{u_t} \dd s }^{\frac{r}{u_t}} \dd x \dd t }^{\frac{1}{r}}\\&\lesssim \brB{ \fint_{c^{m(k+1)}2^{-m}\tau}^{c^{mk}\tau} \sup_{x \in B(z, \tau^{1/m})} A_{u_x, c^{-m(k+1)}\delta}f(s,x)^r  \dd s }^{\frac{1}{r}}\\
  &\leq \brB{ \fint_{c^{m(k+1)}2^{-m}\tau}^{c^{mk}\tau} \sup_{x \in B(z, {2c^{-(k+1)}s^{1/m}})} A_{u_x, c^{-m(k+1)}\delta}f(s,x)^r  \dd s }^{\frac{1}{r}}\\
  &\lesssim \brB{ \fint_{c^{m(k+1)}2^{-m}\tau}^{c^{mk}\tau} A_{u_x, c^{-m(k+2)}\delta}f(s,z)^r \dd s }^{\frac{1}{r}},
\end{align*}
where we used \eqref{eq:martingale} in the final step.
Taking the $L^p_{\vphantom{\beta}}(\R^n; L^q_\beta(\R_+))$-norm and first performing a change of variables $\sigma = c^{mk}\tau$ and then reintroducing the spatial average using \eqref{eq:martingaleback}, we obtain
\begin{align}\label{eq:dilatedest}
\begin{aligned}
    \nrmB{&(t,x)\mapsto \brB{
    \fint_{I_k(t)}
    A_{u_x, c^{-m(k+1)}\delta}f(s,x)^{u_t}
    \dd s
  }^{1/u_t} }_{T^{p,q,r}_\beta} \\&\lesssim  c^{mk(\beta - \frac{1}{q})} \nrmB{(\sigma,z)\mapsto \brB{ \fint_{c^m2^{-m}\sigma}^\sigma A_{u_x, c^{-m(k+2)}\delta}f(s,z)^r \dd s }^{1/r} }_{L^p_{\vphantom{\beta}}(\R^n; L^q_\beta(\R_+))}\\
  &\lesssim c^{mk(\beta - \frac{1}{q})} \nrmB{(\sigma,z)\mapsto \brB{ \fint_{c^m 2^{-m}\sigma}^\sigma \fint_{B(z, \sigma^{1/m})} A_{u_x, c^{-m(k+3)}\delta}f(s,w)^r \dd w \dd s }^{1/r} }_{L^p_{\vphantom{\beta}}(\R^n; L^q_\beta(\R_+))} \\
  &\lesssim c^{mk(\beta - \frac{1}{q})} \|A_{u_x, c^{-m(k+3)}\delta}f\|_{T^{p,q,r}_\beta},
\end{aligned}
\end{align}
where in the last step we used Proposition~\ref{prop:equivnorms}.
Combining with \eqref{eq:estHcafirst} and Lemma \ref{lem:slice}, we conclude
\begin{align*}
  \|H^{\ca}_{\vec u,\delta}f\|_{T^{p,q,r}_\beta}
  &\lesssim  \brB{\sum_{k=0}^\infty  c^{\nu mk(\beta+[u_t,q])}  \|A_{u_x, c^{-m(k+3)}\delta}f\|_{T^{p,q,r}_\beta}^\nu}^{1/\nu}\\
  &\lesssim
  \delta^{\gamma+\varepsilon}
  \brB{\sum_{k=0}^\infty
  c^{\nu mk(\beta+[u_t,q]-\gamma-\varepsilon)}}^{1/\nu}
  \|f\|_{T^{p,q,r}_\beta}.
\end{align*}
Since $0<c<1$, the series converges whenever
\[
  \beta+[u_t,q]-\gamma>\varepsilon.
\]
Since we may shrink $\varepsilon>0$ if necessary, this finishes the proof when $r<\infty$. If $r=\infty$, the usual modifications apply.
\end{proof}

The endpoint estimate $p=\infty$ for the model operators 
is as follows.

\begin{prop}\label{prop:cabddp>infty}
Let $q,r\in (0,\infty]$, \(\beta\in\R\), \(\delta\in[1,\infty)\). Let
\(\vec u\in(0,\infty)^2\) satisfy \(u_t,u_x\le r\) and set
\begin{align*}
  \gamma
  :=
  \tfrac{n}{m} {\bracb{\min\cbrace{q,u_x},u_x}}.
\end{align*}
For all $f \in T^{\infty,q,r}_{\beta}$ we have
  \begin{align*}
    \nrm{H^{\whit}_{\vec u,\delta}f}_{T^{\infty,q,r}_{\beta}}
    \lesssim  \delta^{\gamma}(1+\ln \delta)^{\frac1q}\,\nrm{f}_{T^{\infty,q,r}_{\beta}},
  \end{align*}
 and if $
    \beta>\gamma-[u_t,q]$, we also have
  \begin{align*}
    \nrm{H^{\ca}_{\vec u,\delta}f}_{T^{\infty,q,r}_{\beta}}
    \lesssim
    \delta^{\gamma} (1+\ln \delta)^{\frac1q}\,
    \nrm{f}_{T^{\infty,q,r}_{\beta}},
  \end{align*}
  where the implicit constants are independent of $\delta$.
\end{prop}

\begin{proof}
We only consider the case $q<\infty$, the case $q= \infty$ follows by the usual modifications. We first estimate the averaging operator $A_{u,\delta}$ from 
\eqref{eq:avop} for $u\le r$, using the dyadic characterization of $T^{\infty,q,r}_\beta$ from 
Proposition~\ref{prop:dyadic}.
Let $J$ be the integer in \eqref{eq:J}
and recall that
$2^{nJ} \eqsim N  =  \abs{\Lambda_\delta} \eqsim \delta^{n/m}$. Fix $Q^0\in\mc{D}$ and
consider the $j$-th generation of children of $Q^0$, denoted by $\mc{D}_j(Q^0)$.
It follows as in
\eqref{eq:whitofav}, replacing the embedding $\ell^q \hookrightarrow \ell^u$ by Jensen's inequality in case $q \geq u$, that
\begin{align*}
    \fint_{Q^0}&\sum_{Q\in\mc{D}_j(Q^0)} 
    W_{r,Q}(A_{u,\delta}f)^q \ind_Q(x)\dd x
    =2^{-jn}\sum_{Q\in\mc{D}_j(Q^0)} 
     W_{r,Q}(A_{u,\delta}f)^q\\
    &\lesssim 2^{-jn}\sum_{Q\in\mc{D}_j(Q^0)} 
    \brB{\frac 1N\sum_{k\in \Lambda_\delta} W_{r,Q_k}(f)^{u} }^{q/u} \\
     &\lesssim  2^{-jn}\sum_{Q\in\mc{D}_j(Q^0)} \delta^{\gamma_u q}
    \cdot {\frac 1N\sum_{k\in \Lambda_\delta} W_{r,Q_k}(f)^q } \\
    &=2^{-jn}\cdot \frac{\delta^{\gamma_u q}}{N}\sum_{Q\in\mc{D}_j(Q^0)} 
    \sum_{k\in \Lambda_\delta} \frac{1}{|Q_k|} \int_{Q_k} W_{r,Q_k}(f)^q  \ind_{Q_k}(x) \dd x 
    \\
    &= \frac{\delta^{\gamma_u q}}{N} \cdot \frac1{|Q^0|}\int_{R_j} G_j(x) \sum_{Q\in\mc{D}_j(Q^0)} 
     \sum_{k\in \Lambda_\delta}\ind_{Q_k}(x)\dd x,
\end{align*}
where
$\gamma_u:=\tfrac{n}{m}\brac{\min\cbrace{q,u},u}$,  $R_j:= \bigcup_{Q\in \mc{D}_j(Q^0),\, k\in \Lambda_\delta} Q_k$  and 
$$
 G_j(x)\colon = \sum_{Q\in\mc{D}: \ell(Q)= 2^{-j}\ell(Q^0)} 
    W_{r,Q}(f)^q \ind_Q(x), \qquad x\in \R^n.
$$
Note that, by the disjointness of the dyadic cubes at scale $j$, we have for all $x\in \R^n$,
\[
\sum_{Q\in\mc{D}_j(Q^0)} 
     \sum_{k\in \Lambda_\delta}\ind_{Q_k}(x) \le \min \cbraceb{\abs{\mc{D}_j(Q^0)}, |\Lambda_\delta|}\lesssim \min ( 2^{nj}, 2^{nJ}) \eqsim {N\cdot 2^{-n\max(J-j, 0)}}.
\]
Moreover, by the definition of $\Lambda_\delta$  we have 
\begin{itemize}
    \item If $0\le j<J$, then $R_j$ is contained in the union of $3^n$ dyadic cubes of side length $2^{J-j} \ell(Q^0)$ and
has measure $\eqsim 2^{n(J-j)} \abs{Q^0}$.
\item If $j \geq J$, then $R_j$ is contained in $3Q^0$, which is the union of $3^n$ dyadic cubes of side length $\ell(Q^0)$, and $R_j$ contains $Q^0$, hence has measure $\eqsim |Q^0|$.
\end{itemize}
Thus we have obtained 
\[
 \fint_{Q^0}\sum_{Q\in\mc{D}_j(Q^0)} 
    W_{r,Q}(A_{u,\delta}f)^q \ind_Q(x)\dd x \lesssim \delta^{\gamma_u q} \fint_{R_j} G_j(x) \dd x.
\]
Therefore
\begin{align*}
      \fint_{Q^0}&\sum_{Q\in \mc{D}(Q^0)}
 \ell(Q)^{m-m\beta q}\cdot  W_{r,Q}(A_{u,\delta}f)^q \ind_Q(x) \dd x \\
& =\sum_{j=0}^\infty (2^{-j}\ell(Q^0))^{m-m\beta q}
\fint_{Q^0}\sum_{Q\in\mc{D}_j(Q^0)} 
    W_{r,Q}(A_{u,\delta}f)^q \ind_Q(x)\dd x\\
    &\lesssim \sum_{j=0}^\infty \delta^{\gamma_u q} (2^{-j}\ell(Q^0))^{m-m\beta q}\fint_{R_j} G_j(x) \dd x 
    \\
    & \lesssim \delta^{\gamma_u q}
\sum_{j=0}^{J-1} \fint_{R_j} (2^{-j}\ell(Q^0))^{m-m\beta q} G_j(x) \dd x\\
 &\qquad\qquad+\delta^{\gamma_u q}
\fint_{3Q^0}\sum_{j=J}^\infty (2^{-j}\ell(Q^0))^{m-m\beta q} G_j(x) \dd x
\\
&\leq \delta^{\gamma_u q}
\sum_{j=0}^{J-1} \fint_{R_j} \sum_{Q\in \mc{D}, Q\subseteq R_j}
 \ell(Q)^{m-m\beta q}\cdot  W_{r,Q}(f)^q \ind_Q(x) \dd x\\
 &\qquad\qquad +\delta^{\gamma_u q}
\fint_{3Q^0}\sum_{Q\in \mc{D}, Q\subseteq 3Q^0}
 \ell(Q)^{m-m\beta q}\cdot  W_{r,Q}(f)^q \ind_Q(x)  \dd x\\
&\lesssim \delta^{\gamma_u q}\cdot J\cdot  3 ^n\cdot \nrm{f}_{T^{\infty,q,r}_\beta}^q+
\delta^{\gamma_u q}\cdot 3^n\cdot \nrm{f}_{T^{\infty,q,r}_\beta}^q\\
&\lesssim \delta^{\gamma_u q} (1+\ln \delta)\nrm{f}_{T^{\infty,q,r}_\beta}^q,
\end{align*}
where we have used the dyadic Carleson bound on $f$ from Proposition~\ref{prop:dyadic} on each of the $3^n$ dyadic cubes that $R_j$ is the union of, for each $0\le j<J$, and
 on each of the dyadic cubes that make up $3Q^0$ for the second term.
Taking supremum over all $Q^0$ proves
$$
    \nrm{A_{u,\delta}f}_{T^{\infty,q,r}_{\beta}}
    \lesssim
    \delta^{\gamma_u} (1+\ln \delta)^{\frac1q}\,
    \nrm{f}_{T^{\infty,q,r}_{\beta}}.
$$

With this estimate we can argue as in Proposition~\ref{prop:modelops-whitneymixed} for $H^{\whit}_{\vec{u}, \delta}$ and $H^{\ca}_{\vec{u}, \delta}$. 
Using again \eqref{eq:AudomHw} and replacing $u$ by $u_x$ so that  $\gamma_{u_x}=\gamma$, we immediately conclude that
$$
    \nrm{H^{\whit}_{\vec{u}, \delta}f}_{T^{\infty,q,r}_{\beta}}
    \lesssim
    \nrm{A_{u_x,c^{-m}\delta}f}_{T^{\infty,q,r}_{\beta}}
    \lesssim
    \delta^{\gamma} (1+\ln \delta)^{\frac1q}\,
    \nrm{f}_{T^{\infty,q,r}_{\beta}}.
$$
For $H^{\ca}_{\vec{u}, \delta}f$, we only need to note the analogue of \eqref{eq:dilatedest} for $p=\infty$.
For fixed $(\tau,z) \in \R^{1+n}_+$, we calculate 
\begin{align*}
    \int_0^\tau \fint_{B(z,\tau^{1/m})}
    & t^{-\beta q}
    \brB{ \fint_{c^{m(k+1)}2^{-m}t}^{c^{mk}t}
    A_{u_x, c^{-m(k+2)}\delta} f(s,x)^r \dd s}^{q/r} \dd x \dd t \\
    &= c^{mk(\beta q-1)}\int_0^{c^{mk} \tau} \fint_{B(z,\tau^{1/m})}
     t^{-\beta q}
    \brB{ \fint_{(c/2)^{m}t}^{t}
    A_{u_x, c^{-m(k+2)}\delta} f(s,x)^r \dd s}^{q/r}  \dd x\dd t \\
    &\lesssim c^{mk(\beta q-1)}\int_0^{\tau} \fint_{B(z,\tau^{1/m})}
    t^{-\beta q}
    \brB{ 
    {W_r^{(c/2)^m,1,1}} \brb{A_{u_x, c^{-m(k+3)}\delta} f}(t,x)}^q \dd x\dd t, 
\end{align*}
where we used a change of $t$-variable in the second line, and we used $s\geq (c/2)^m t$,   \eqref{eq:martingaleback}  and $c^{-1}\geq 2$ in the third line. The rest of the estimate of $\nrm{H^{\ca}_{\vec{u}, \delta}f}_{T^{\infty,q,r}_{\beta}}$ follows the case $p<\infty$ in Proposition \ref{prop:modelops-whitneymixed}, and we conclude that, using $\nu = \min\cbrace{q,u_t}$
\begin{align*}
  \|H^{\ca}_{\vec u,\delta}f\|_{T^{\infty,q,r}_\beta}
  &\lesssim
  \delta^{\gamma}
  \brB{\sum_{k=0}^\infty
  c^{\nu mk(\beta+[u_t,q]-\gamma)}
  (1+\ln\delta +mk\ln\tfrac 1c)^{\nu/q}
  }^{1/\nu}
  \|f\|_{T^{\infty,q,r}_\beta}\\
    &\le
  \delta^{\gamma}(1+\ln\delta)^{\frac 1q}
  \brB{\sum_{k=0}^\infty
  c^{\nu mk(\beta+[u_t,q]-\gamma)}
  (1+mk\ln\tfrac 1c)^{\nu/q}
  }^{1/\nu}
  \|f\|_{T^{\infty,q,r}_\beta}\\
      &\lesssim
  \delta^{\gamma}(1+\ln\delta)^{\frac 1q}
  \|f\|_{T^{\infty,q,r}_\beta},
\end{align*}
where we used that $\ln\delta\ge 0$ and
$\beta+[u_t,q]-\gamma> 0$.
\end{proof}

\begin{remark} \label{rem:HwODE}
Let $p,q,r \in (0,\infty]$, \(\beta\in\R\), \(\delta\in[1,\infty)\). Let
\(\vec u\in(0,\infty)^2\) satisfy \(u_t,u_x\le r\) and assume that
$$\gamma:=
  \tfrac{n}{m}
  \bracb{\min\cbrace{p,q,u_x},u_x}<\beta+[u_t,q].
$$
Let $E_t,F_t\subseteq\R_+$ with $\inf F_t>\sup E_t$.
    Then, for every $\varepsilon>0$, the operator $H^{\ca}_{\vec u,\delta}$ satisfies the 
    off-diagonal estimate
$$
\|\ind_{F_t\times\R^n} H^{\ca}_{\vec u,\delta} (\ind_{E_t\times \R^n}f)\|_{T^{p,q,r}_\beta}
  \lesssim \delta^{\gamma+\varepsilon} 
  \brB{\frac{\sup E_t}{\inf F_t}}^\eta
  \|f\|_{T^{p,q,r}_\beta}
$$
for some $\eta>0$, and with the $\log$-improvement of $\delta^\epsilon$ when $p=\infty$ as in Proposition~\ref{prop:cabddp>infty}.
This follows upon inspection of the pointwise estimate
\eqref{eq:pointdomHw}. Indeed, since $t>\inf F_t$ and $s<\sup E_t$, only the terms with $c^{mk}\le \sup E_t/\inf F_t$ will contribute. We conclude by summing the geometric series only over these $k$.
\end{remark}

\subsection{A priori estimates on singular operators}\label{sec:apriori}
{Combining the domination principles for singular operators in Theorems~\ref{thm:maindominationform} and~\ref{thm:quasimaindominationform}} with the boundedness of the model operators on tent spaces proved in Section~\ref{sec:modelops}, we can deduce our main result on the boundedness of singular operators on tent spaces. Recall our definition of a singular operator from Definition \ref{def:ODE}.

\begin{theorem}\label{thm:mainestimate}
Let $r \in (1,\infty)$, $\vec{u} \in [1,\infty)^2$ and $\vec{v} \in (1,\infty]^2$ be such that
$$\max\cbrace{u_t, u_x} \leq r\leq \min\cbrace{v_t, v_x}.$$ 
Let $\kappa, M\in\R$ and suppose that  $T$ is a singular operator of type $(\kappa,m,r,\vec u,\vec v, M)$.
Let $p,q\in (0,\infty]$ and $\beta\in\R$, and assume the decay condition
\begin{align}
    \tag{d}
    \label{e:main-decay-condition}
    M &> -[u_t,v_t] +\tfrac{n}{m}  {\bracb{\min\cbrace{p,q,u_x},u_x}} + \tfrac{n}{m}{\bracb{v_x,\max\cbrace{p,q,v_x}} } + \kappa \\
\intertext{
and the lower and upper bounds
}
    \tag{l}
    \label{e:main-beta-lower-bound}
    \beta &> -[u_t,q]+ \tfrac{n}{m}{\bracb{\min\cbrace{p,q,u_x},u_x}}, \\
    \tag{u}
    \label{e:main-beta-upper-bound}
    \beta&< \phantom{-}[q,v_t] -\tfrac{n}{m}{\bracb{v_x,\max\cbrace{p,q,v_x}} }-\kappa.
\end{align}
Then we have for all $f \in L^\infty_c(\R^{1+n}_+)$,
\begin{equation*}
  \nrm{Tf}_{T^{p,q,r}_{\beta+\kappa}} \lesssim \nrm{f}_{T^{p,q,r}_{\beta}}.
\end{equation*}
If $T$ is causal,  the upper bound \eqref{e:main-beta-upper-bound} is omitted. If $T$ is anti-causal,  the lower bound \eqref{e:main-beta-lower-bound} is omitted.
\end{theorem}

To illustrate the bounds of $\beta$, we give a graphic representation in Figure~\ref{fig:beta-p-diagram}. 
 
\begin{figure}[ht]
\label{fig:main}
    \centering
    \includegraphics[width=0.5\linewidth]{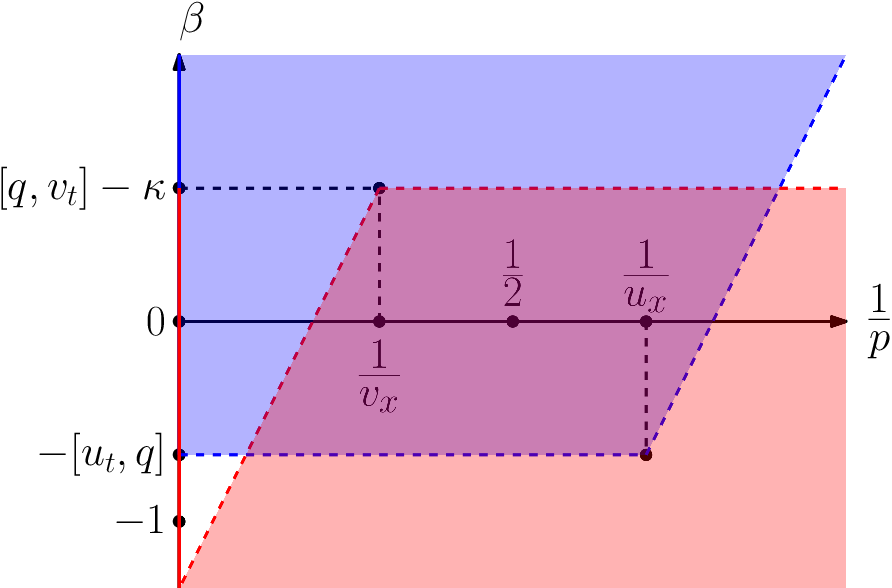}
    \caption{Illustration for the ranges of $\beta$ in \eqref{e:main-beta-lower-bound} and \eqref{e:main-beta-upper-bound} when $q=r$. The blue shaded trapezoid denotes the range of $\beta$ in \eqref{e:main-beta-lower-bound}, and the red shaded one denotes that in \eqref{e:main-beta-upper-bound}.}
    \label{fig:beta-p-diagram}
\end{figure}

\begin{remark}We make a few remarks on the assumptions in Theorem \ref{thm:mainestimate}.
\begin{enumerate}[(i)]
 \item For $u_x\le p,q\le v_x,$ the conditions \eqref{e:main-decay-condition}, \eqref{e:main-beta-lower-bound}, \eqref{e:main-beta-upper-bound} for $\beta$ and $M$  are 
 \[M>-[u_t,v_t]+\kappa,\quad \beta>-[u_t,q], \quad \beta<[q,v_t]-\kappa,\]
 and thus only depend on $u_t,v_t, q$ and $\kappa$. In particular when $p=q=r$,  this yields a sufficient condition for $L^r_\beta(\R^{1+n}_+)$ boundedness when $\max\cbrace{u_t, u_x} \leq r\leq \min\cbrace{v_t, v_x}.$
\item In the general case, the theorem also provides information about the admissible range of \(p\) and \(q\). Let \(\beta_-\) and \(\beta_+\) denote the lower and upper bounds for \(\beta\), respectively. The decay condition in Theorem~\ref{thm:mainestimate} can then be rewritten as
\[
  M>\beta_- - \beta_+.
\]
Thus, whenever \(\beta_-<\beta_+\), i.e., whenever the admissible range of $\beta$ is non-empty, condition \eqref{e:main-decay-condition} reduces to the negative lower bound
\[
  M>-(\beta_+-\beta_-).
\]
In particular, condition \eqref{e:main-decay-condition} is then automatically satisfied for every $M\geq 0$. Hence, increasing the decay parameter \(M\) beyond \(M=0\) does not enlarge the admissible range of \(p\) and \(q\) in this case.
In contrast, for causal or anti-causal operators, only one of the bounds on \(\beta\) is imposed. The admissible range is then \((\beta_-,\infty)\) or \((-\infty,\beta_+)\), respectively, and is therefore always non-empty. Thus, increasing the decay parameter \(M\) can substantially enlarge the range of \(p\) and \(q\) for which boundedness holds.
\item The proof of Theorem~\ref{thm:mainestimate} can be adapted to $r=1$ and $r=\infty$, but we foresee applications only for $r \in (1,\infty)$.
\end{enumerate}
\end{remark}

\begin{remark} The conclusion of the theorem is essentially optimal   as can be seen with ``elementary'' examples. Namely,  the various conditions on $M$ and $\beta$ are best possible for the class of singular operators we considered.  Of course, this does not exclude a stronger conclusion for a specific operator.  The upper bound and lower bound on $\beta$ are already optimal for the $L^r_\beta(\R^{1+n}_+) \to L^r_{\beta+\kappa}(\R^{1+n}_+)$ boundedness of the causal $\kappa$-Hardy operator given 
$$H^+_{\kappa}f(t,x)=\int_0^t (t-s)^{\kappa-1} f(s,x)\dd s, \qquad (t,x) \in \R^{1+n}_+, $$ and its  anti-causal adjoint $$H^-_{\kappa}f(t,x)=\int_t^\infty  (s-t)^{\kappa-1} f(s,x)\dd s,  \qquad (t,x) \in \R^{1+n}_+,$$ which are of type $(\kappa, m, r, (1,r),(\infty,r), M)$ for any $m>0, M\in \R$ and $1\le r\le \infty$ when $\kappa>0$. Optimality of the lower bound on $M$  can be checked with spatially lacunary operators of the type \[Tf(t,x)= \sum_{j=1}^N c_j\ind_{Q_j}(x) J_\kappa \bigg(s\mapsto \int_{Q_0} \phi(y)f(s,y) \dd y\bigg)(t),\] where $Q_j$ are unit cubes in $\R^n$ whose distance to $Q_0$ is $L^j$ for some large $L$,  $J_\kappa$ is the truncated fractional operator given by $$J_\kappa g(t)= \ind_{[1,2]}(t)\int_1^2 (t-s)^{\kappa - 1}g(s) \dd s, \qquad t \in \R_+,$$ and   the positive function $\phi$ and  the positive numbers $c_j$ satisfy  $L^{\rho}$-integrability and $\ell^\sigma$-summability conditions, respecively,  with appropriate $\rho,\sigma$.  The lacunarity implies  that the singular operator bounds do not depend on $N$ and one can compute explicitly the decay $M$.  On the other hand, the appropriate tent space norm of  $Tf$ grows with $N$ for  $f= \ind_{[1,2]\times Q_0}$.  The dependence of $M$ on $\max\{p,q\}$ and $\min\{p,q\}$ is also sharp. We leave the details to the interested readers.      
\end{remark}

\begin{proof}[Proof of Theorem \ref{thm:mainestimate}]
As  $r > 1$, we only need to distinguish whether $\min\cbrace{p,q}\ge 1$ or not.
Consider first the case $\min\cbrace{p,q}\ge 1$. 
By Proposition~\ref{prop:tentspacedensityduality}, it suffices to show for $f,g \in L^\infty_c(\R^{1+n}_+)$ 
that
$$
\int_{{\mathbb R}^{1+n}_+}|Tf||g|  \lesssim \nrm{f}_{\vphantom{T^{p',q',r'}_{-\beta-\kappa}}T^{p,q,r}_\beta} \cdot \nrm{g}_{T^{p',q',r'}_{-\beta-\kappa}}.
$$
To this end, we use the form domination from
Theorem~\ref{thm:maindominationform} and estimate each
of the four terms separately.

For the first term, we note that with $A^{\whit}_{r,\kappa}$ as in Lemma \ref{lemma:Whitneymap}, and  using disjointness of the Whitney cubes $\Qw{Q}$, $\Qw{Q} \subseteq \Qwtilde{Q}$ and Proposition~\ref{prop:tentspacedensityduality},
\begin{align*}
\sum_{Q \in \mc{D}} \ell(Q)^{m\kappa} \cdot \ip{f}_{r,\Qwtilde{Q}}\cdot \ip{g}_{\vphantom{\Qwtilde{Q}}r',\Qw{Q}} \cdot\abs{\Qw{Q}} &\lesssim  \int_{\R^{1+n}_+} A^{\whit}_{r,\kappa}f(t,x) \cdot A^{\whit}_{r',0}g(t,x)\dd t \dd x \\
&\le \nrm{A^{\whit}_{r,\kappa}f}_{\vphantom{T^{p',q',r'}_{-\beta-\kappa}}T^{p,q,r}_{\beta+\kappa}}\cdot \nrm{A^{\whit}_{r',0}g}_{\vphantom{T^{p',q',r'}_{-\beta-\kappa}}T^{p',q',r'}_{-\beta-\kappa}}
\lesssim  \nrm{f}_{\vphantom{T^{p',q',r'}_{-\beta-\kappa}}T^{p,q,r}_\beta} \cdot\nrm{g}_{\vphantom{T^{p',q',r'}_{-\beta-\kappa}}T^{p',q',r'}_{-\beta-\kappa}}.
\end{align*}
Next, for the other terms, we can use Propositions~\ref{prop:modelops-whitneymixed} and~\ref{prop:cabddp>infty} thanks to the conditions on $r$ and $\beta$, noting that  
$(1+\ln\delta)^{1/q}\lesssim \delta^\varepsilon$. 
For the term $(\mathrm{w}$-$\mathrm{w})$, we see that 
\begin{align*}
  \sum_{j=0}^\infty 2^{-jm(M+[u_t,v_t]-\kappa)} &\cdot 
\nrm{H^{\whit}_{(r,u_x), 2^{jm}}f}_{\vphantom{T^{p',q',r'}_{-\beta-\kappa}}T^{p,q,r}_\beta}\cdot \nrm{H^{\whit}_{(r', v_x'), 2^{jm}}g}_{\vphantom{T^{p',q',r'}_{-\beta-\kappa}}T^{p',q',r'}_{-\beta-\kappa}}
    \\&\lesssim \sum_{j=0}^\infty 2^{-jm\alpha} \cdot\nrm{f}_{\vphantom{T^{p',q',r'}_{-\beta-\kappa}}T^{p,q,r}_\beta} \cdot  \nrm{g}_{\vphantom{T^{p',q',r'}_{-\beta-\kappa}}T^{p',q',r'}_{-\beta-\kappa}},
\end{align*}
where 
\begin{equation}
    \label{eq:alpha}
        \alpha = M +[u_t, v_t]-\kappa -\tfrac{n}{m} {\bracb{\min\cbrace{p,q,u_x},u_x}} -\tfrac{n}{m}{\bracb{\min\cbrace{p',q',v_x'},v_x'}}-\varepsilon,
\end{equation}
for any $\varepsilon>0$.
By the decay condition \eqref{e:main-decay-condition}, we can choose $\varepsilon>0$ small enough 
so that $\alpha>0$ and the sum converges. 

The term $(\mathrm{ca}$-$\mathrm{w})$ is estimated similarly to term $(\mathrm{w}$-$\mathrm{w})$, noting that we need $\beta>\gamma-[u_t,q]= \tfrac n m \bracb{\min\cbrace{p,q,u_x},u_x}-[u_t,q]$
for the estimate of $H^{\ca}_{\vec{u},2^{jm}}f$, which follows from \eqref{e:main-beta-lower-bound}. This yields  
\begin{align*}
   \sum_{j=0}^\infty 2^{-jm(M+[u_t,v_t]-\kappa)} &\cdot 
\nrm{H^{\ca}_{\vec{u}, 2^{jm}}f}_{\vphantom{T^{p',q',r'}_{-\beta-\kappa}}T^{p,q,r}_\beta}\cdot \nrm{H^{\whit}_{\vec{v}', 2^{jm}}g}_{\vphantom{T^{p',q',r'}_{-\beta-\kappa}}T^{p',q',r'}_{-\beta-\kappa}}\\
&\lesssim\sum_{j=0}^\infty 2^{-jm\alpha} \cdot\nrm{f}_{\vphantom{T^{p',q',r'}_{-\beta-\kappa}}T^{p,q,r}_\beta} \cdot  \nrm{g}_{T^{p',q',r'}_{-\beta-\kappa}},
\end{align*}
with $\alpha$ as before, so the sum again converges.
Recall that if $T$ is anti-causal, then the term {$(\mathrm{ca}$-$\mathrm{w})$} is not present in the domination from
Theorem~\ref{thm:maindominationform}. Finally, unless $T$ is causal, the term $(\mathrm{w}$-$\mathrm{ca})$ is estimated identically to term $(\mathrm{ca}$-$\mathrm{w})$,
now using \eqref{e:main-beta-upper-bound} for the estimate
of $H^{\ca}_{\vec{v}', 2^{jm}}g$.

\medskip

Consider now the case $\nu:= {\min\cbrace{p,q}}\in(0,1)$.
{Using Proposition~\ref{prop:tentspacedensityduality}\ref{it:quasiduality}}, it suffices to prove that for $f,g \in L^\infty_c(\R^{1+n}_+)$
$$
\int_{{\mathbb R}^{1+n}_+}|Tf|^\nu|g|  \lesssim \nrm{f}_{\vphantom{T^{p',q',r'}_{-\beta-\kappa}}T^{p,q,r}_\beta}^\nu \cdot \nrm{g}_{T^{(p/\nu)',(q/\nu)',(r/\nu)'}_{-\nu(\beta+\kappa)}}^{}.
$$
For this we apply Theorem~\ref{thm:quasimaindominationform},
and estimate the four terms similarly to the above arguments.
{For the first term, as before together with  \eqref{eq:rescalenorm} and Proposition~\ref{prop:tentspacedensityduality},  
\begin{align*}
\sum_{Q \in \mc{D}} \ell(Q)^{\nu m\kappa} \cdot \ip{f}^\nu_{r,\Qwtilde{Q}}\cdot \ip{g}_{\vphantom{\Qwtilde{Q}}({r}/\nu)',\Qw{Q}} \cdot\abs{\Qw{Q}} & \lesssim \int_{\R^{1+n}_+} \brb{A^{\whit}_{r,\kappa}f(t,x)}^{\nu} \cdot A^{\whit}_{({r}/\nu)',0}g(t,x)\dd t \dd x \\
&\le \nrm{A^{\whit}_{r,\kappa}f}_{\vphantom{T^{p',q',r'}_{-\beta-\kappa}}T^{p,q,r}_{\beta+\kappa}}^\nu \cdot \nrm{A^{\whit}_{({r}/\nu)',0}g}_{T^{(p/\nu)'
   ,(q/\nu)',(r/\nu)'}_{-\nu(\beta+\kappa)}}^{}\\&
\lesssim  \nrm{f}_{\vphantom{T^{p',q',r'}_{-\beta-\kappa}}T^{p,q,r}_\beta}^\nu \cdot\nrm{g}_{T^{(p/\nu)',(q/\nu)',(r/\nu)'}_{-\nu(\beta+\kappa)}}^{}.
\end{align*}}
The last three terms are each estimated by
\begin{align*}
\sum_{j=0}^\infty  2^{-jm\nu(M+[u_t,v_t]-\kappa)} \cdot
&\brb{2^{jn({\brac{\min\cbrace{p,q,u_x},u_x}}+\varepsilon)} \nrm{f}_{T^{p,q,r}_\beta} }^\nu\\ &\cdot  2^{jn( {\brac{\min\cbrace{(p/\nu)',(q/\nu)',(v_x/\nu)'},(v_x/\nu)'}}+\varepsilon)} \nrm{g}_{T^{(p/\nu)',(q/\nu)',(r/\nu)'}_{-\nu(\beta+\kappa)}},    
\end{align*}
where assumption \eqref{e:main-decay-condition} shows that the sum converges.
As in the case $\min\cbrace{p,q}\geq 1$, the hypothesis on $\beta$ required for the estimate on $f$ in the term $(\mathrm{ca}$-$\mathrm{w})$
follows from \eqref{e:main-beta-lower-bound}. 
For the estimate on $g$ in term $(\mathrm{w}$-$\mathrm{ca})$, the hypothesis 
$$-\nu(\beta+\kappa)>\tfrac nm{ \bracb{\min\cbraceb{(\tfrac{p}{\nu})',(\tfrac{q}{\nu})',(\tfrac{v_x}{\nu})'},(\tfrac{v_x}{\nu})'}}-\bracb{(\tfrac{v_t}{\nu})', (\tfrac q\nu)'}$$ is needed,
which follows from \eqref{e:main-beta-upper-bound}. 
This completes the proof.
\end{proof}

\subsection{Extension} 
\label{sec:extension}

In Subsection \ref{sec:apriori}, we have obtained a priori estimates for singular operators. We will now extend these estimates to the full  tent spaces. 
The nature of the off-diagonal estimates allows for a non-conventional extension procedure. An earlier occurrence of this extension procedure was made in \cite[Theorem~3.1]{AP25} when $p=\infty$.

\begin{theorem} \label{thm:extension} Let $T$ be a singular operator satisfying the assumptions of Theorem~\ref{thm:mainestimate}and the pointwise almost everywhere Lipschitz estimate
\begin{equation}\label{eq:pointwiseLipschitz}
    |Tf-Tg|\le C \, |T(f-g)|, \qquad f,g \in L^\infty_c(\R^{1+n}_+)
\end{equation}
for some $C<\infty$.
Then $T$ has a unique bounded extension $\widetilde T:T^{p,q,r}_\beta\to T^{p,q,r}_{\beta+\kappa}$ satisfying for all compact $\mbs{K}\subseteq\R^{1+n}_+$, all sequences $(\mbs{K}_j)_j$  of compact sets, with
$\mbs{K}_j\subseteq \text{Int}\, \mbs{K}_{j+1}$,
that exhaust $\R^{1+n}_+$, and all $f\in T^{p,q,r}_\beta$, that
$$
 \ind_{\mbs{K}} \widetilde Tf:= \lim_{j\to\infty} \ind_{\mbs{K}} T(f\ind_{{\mbs{K}}_j}), 
$$
where the limit is in the  norm topology of $T^{p,q,r}_{\beta+\kappa}$.

Moreover, if $T$ is linear, then $\widetilde T$ coincides with the extension by norm continuity if $\max\{p,q\}<\infty$ and by  weak-star continuity if $\min\{p,q\}>1$. 
    \end{theorem}
    
\begin{remark}
  For linear operators, \eqref{eq:pointwiseLipschitz} always holds with $C=1$. For subadditive operators such that $Tf\ge 0$ and $T(f)=T(-f)$, this also holds with $C=1$.
Indeed, these additional assumptions imply that
$Tf\le Tg+T(f-g)$ and $Tg\le Tf+T(g-f)=Tf+T(f-g)$, hence  \eqref{eq:pointwiseLipschitz}. This covers many subadditive operators such as maximal operators and, more generally,
those coming from a vector-valued linear operator, that is $Tf=\|\mathcal{T}f\|_B$, where $\mathcal{T}$ is linear and $\|\cdot\|_B$ is a norm.  
\end{remark}

\begin{remark}
{This extension in Theorem \ref{thm:extension}  clearly preserves (anti-)causality.}
\end{remark}

The proof of Theorem \ref{thm:extension} depends on the following lemmas. 

\begin{lemma}\label{lem:extension} Let $T$ be as in Theorem \ref{thm:extension}. Then $T$ extends to a unique bounded operator $T^{p,q,r}_{\beta,c}\to T^{p,q,r}_{\beta+\kappa}$, where $T^{p,q,r}_{\beta,c}$ denotes the space of $T^{p,q,r}_\beta$ functions with compact support.
    \end{lemma}

    \begin{proof}
     The pointwise Lipschitz inequality implies $\nrm{Tf-Tg}\lesssim \nrm{T(f-g)}$ for any of the tent space norms above. Hence, $T$   initially defined on $L^\infty_c(\R^{1+n}_+)$,    can be defined uniquely on the closure of $L^\infty_c(\R^{1+n}_+)$ in $T^{p,q,r}_\beta$  by continuity  with
$\nrm{Tf}_{T^{p,q,r}_{\beta+\kappa}}\lesssim
  \nrm{f}_{T^{p,q,r}_{\beta}}.$
  If $\max\{p,q\}<\infty$, this closure is all of $T^{p,q,r}_\beta$ by Proposition~\ref{prop:tentspacedensityduality}\ref{it:density}.
  Otherwise, one checks that this closure contains $T^{p,q,r}_{\beta,c}$ (which consists of all compactly supported functions in $L^r(\R^{1+n}_+)$).   
    \end{proof}

\begin{lemma} \label{lem:annulitobox} Fix $(t_0,x_0) \in \R^{1+n}_+$ and define
\begin{align*}
   \mbs{A}_k&:= (2^{-(k+1)m}t_0, 2^{km}t_0)\times B(x_0,(2^{km}t_0)^{1/m}), &&k \in \N,\\
    \mbs{W}&:=(2^{-m}t_0, t_0)\times B(x_0,t_0^{1/m}).
\end{align*}
 Under the assumptions of Theorem~\ref{thm:extension},  
   there exists an $\varepsilon>0$ such that for all $f\in T^{p,q,r}_\beta$ with compact support and  $k\in\N$ we have
   $$
\nrm{\ind_{\mbs{W}}T(f\ind_{\R^{1+n}_+\setminus\mbs{A}_k})}_{T^{p,q,r}_{\beta+\kappa}} \lesssim 2^{-k\varepsilon}\nrm{f\ind_{\R^{1+n}_+\setminus \mbs{A}_k}}_{T^{p,q,r}_{\beta}},
   $$
   where $T$ is the extension from Lemma~\ref{lem:extension}. 
\end{lemma}

Before turning to the proof of Lemma \ref{lem:annulitobox}, let us deduce Theorem~\ref{thm:extension}.

\begin{proof}[Proof of Theorem~\ref{thm:extension}] 
Fix  an $f\in T^{p,q,r}_\beta$,  a compact subset $\mbs{K}\subseteq\R^{1+n}_+$ and an increasing sequence of compact sets $(\mbs{K}_j)_j$   
that exhaust $\R^{1+n}_+$. By a finite covering, we may assume that  $\mbs{K} \subseteq \mbs{W}$, and for any $k$ there is an integer $j_k$ such that $\mbs{A}_k\subseteq \mbs{K}_j$ for $j>j_k$, where $\mbs{W}$ and $\mbs{A}_k$ are as in Lemma \ref{lem:annulitobox}. Then $(\ind_{\mbs{K}} T(f\ind_{\mbs{K}_j}))$ is a Cauchy sequence in $T^{p,q,r}_{\beta+\kappa}$ by Lemma~\ref{lem:annulitobox}. Hence, one can set 
$$
  T_{\mbs{K}} f:= \lim_{j\to\infty} \ind_{\mbs{K}} T(f\ind_{\mbs{K}_j} ).
$$
Observe that for fixed $\mbs{K}$,  the limit does not  depend on the choice of the exhausting sequence. 
Moreover, 
$\nrm{T_{\mbs{K}}f}_{T^{p,q,r}_{\beta+\kappa}}\lesssim \nrm{f}_{T^{p,q,r}_{\beta}}$ 
with implicit constant independent of $f$ and ${\mbs{K}}$. 
Clearly, $T_{\mbs{K}}f$ and $T_{\tilde {\mbs{K}}}f$ agree on the intersection of two compact sets ${\mbs{K}}, \tilde {\mbs{K}}$. Thus we may define $\widetilde Tf$ by
$\widetilde Tf(t,x):= T_{\mbs{K}}f(t,x)$ if $(t,x)\in {\mbs{K}}$. By monotone convergence  we have 
\begin{equation*} 
    \nrm{\widetilde Tf}_{T^{p,q,r}_{\beta+\kappa}}\eqsim
  \sup_{\mbs{K}} \,\nrm{ T_{\mbs{K}} f}_{T^{p,q,r}_{\beta+\kappa}},    
\end{equation*}
so this yields a well-defined and bounded extension 
$ \widetilde T:T^{p,q,r}_\beta\to T^{p,q,r}_{\beta+\kappa}$ of $T$.

Uniqueness is immediate. Moreover, it is easy to see that when  $T$ is  linear, this extension  coincides with the extension by 
strong density if $\max\{p,q\}<\infty$ and weak-star density 
if $\min\{p,q\}>1$ by Proposition~\ref{prop:tentspacedensityduality}\ref{it:density}. 
\end{proof}

It remains to prove Lemma~\ref{lem:annulitobox}, which is the main technical difficulty in this subsection. We will once again rely on the decomposition we used to prove Theorem~\ref{thm:maindominationform}.

\begin{proof}[Proof of Lemma~\ref{lem:annulitobox}] By assumption, we know that $f$ has compact support and $f\in L^r(\R^{1+n}_+)$. Fix $k \in \N$ and define $f_k:=f\ind_{\R^{1+n}_+\setminus \mbs{A}_k}$. 
We will first consider the case $\min\cbrace{p,q}\ge 1$. By Proposition~\ref{prop:tentspacedensityduality}\ref{it:duality}, it suffices to control $
\int_{\mbs{W}}\abs{Tf_k}\abs{g} $ for bounded $g$ with support in $\mbs{W}$, where $T$ is the extension from Lemma~\ref{lem:extension}. We note that this extension satisfies Definition~\ref{def:ODE}, extended  to functions in $L^r(\R^{1+n}_+)$ with compact support by density.  One can thus perform the estimates on this integral leading to Theorem~\ref{thm:maindominationform}, i.e.,
    \begin{align*}
\int_{\mbs{W}}&\abs{Tf_k}\abs{g} 
\lesssim 
\sum_{Q \in \mc{D}} \ell(Q)^{m\kappa} \cdot \ip{f_k}_{r,\Qwtilde{Q}}\cdot \ip{g}_{\vphantom{\Qwtilde{Q}}r',\Qw{Q}} \cdot\abs{\Qw{Q}}
&& \\
&+\sum_{j=0}^{\infty} 2^{-jm\br{M+[u_t,v_t]-\kappa}} \cdot \int_{\R^{1+n}_+}t^\kappa\cdot   H^{\whit}_{(r,u_x),2^{jm}}f_k(t,x)\cdot H^{\whit}_{(r',v_x'), 2^{jm}}g(t,x)\dd x\dd t  &&
\text{$(\mathrm{w}$-$\mathrm{w})$}\\
&+\sum_{j=0}^{\infty} 2^{-jm\br{M+[u_t,v_t]-\kappa}} \cdot \int_{\R^{1+n}_+}t^\kappa\cdot   H^{\ca}_{\vec{u},2^{jm}}f_k(t,x)\cdot H^{\whit}_{\vec{v}', 2^{jm}}g(t,x)\dd x\dd t &&\text{$(\mathrm{ca}$-$\mathrm{w})$}\\
&+\sum_{j=0}^{\infty} 2^{-jm\br{M+[u_t,v_t]-\kappa}} \cdot
\int_{\R^{1+n}_+}t^\kappa\cdot   H^{\whit}_{\vec{u},2^{jm}}f_k(t,x)\cdot H^{\ca}_{\vec{v}', 2^{jm}}g(t,x)\dd x\dd t. &&\text{$(\mathrm{w}$-$\mathrm{ca})$}.
\end{align*}
 The idea is to leverage the distance between the supports of $f_k$ and $g$.

Clearly, for $k$ larger than some number depending only on dimension, $\ip{f_k}_{r,\Qwtilde{Q}}\cdot \ip{g}_{\vphantom{\Qwtilde{Q}}r',\Qw{Q}}=0$ for all $Q\in \mc{D}$, so the first term vanishes and it remains to control the three other terms. To simplify the exposition, we assume that the constant $c$ in the definition of the operator  $H^{\whit}$   is $\tfrac12$. The general case can be obtained using a finite covering in $\R_+$. 

For the (w-w) term, we observe that from the support condition on $g$, we only have $H^\whit g(t,x) \neq 0$ if $t\in (2^{-m}t_0,{2^m} t_0)$ and $|x-x_0|\le 2^jt^{1/m}+{t_0}^{1/m}$. 
Furthermore, by the support of $f_k$ we only have $H^\whit f_k(t,x) \neq 0$ if $|x-x_0|\ge 2^k t_0^{1/m}- 2^j t^{1/m}$. Hence the {integrand} vanishes unless {
$2^k t_0^{1/m} \le  2^{j+1} t^{1/m}+t_0^{1/m} < (2^{j+2}+1) t_0^{1/m}$ and thus $j+2\ge k$.} Therefore, the series in the (w-w)-term starts at {$j=k-2$} and we get the decay $2^{-k m\alpha}$ from the argument in 
Theorem~\ref{thm:maindominationform} using \eqref{e:main-decay-condition} with  $\alpha$  given by \eqref{eq:alpha}, i.e., we can use $\varepsilon := m\alpha$ for the (w-w)-term. 

For the (ca-w) term, as before  we only have $H^\whit g(t,x) \neq 0$ if $t\in (2^{-m}t_0,{2^m} t_0)$ and $|x-x_0|\le 2^jt^{1/m}+{t_0}^{1/m}$.
 If {$j\le k-3$}, the only contributions of $f_k$ come from its values near the boundary. 
We apply the off-diagonal estimates for $H^{\ca}_{\vec{u},2^{jm}}$ 
 noted in Remark~\ref{rem:HwODE} with $E_t=(0,2^{-(k+1)m}t_0)$
 and $F_t= (2^{-m}t_0,\infty)$, which shows that the series for  {$j\leq k-3$} in the (ca-w) term is bounded by 
\[   {\sum_{j=0}^{k-3} 2^{-jm\alpha}\cdot 2^{-km\eta} }\cdot\nrm{f_k}_{\vphantom{T^{p',q',r'}_{-\beta-\kappa}}T^{p,q,r}_\beta} \cdot  \nrm{g}_{\vphantom{T^{p',q',r'}_{-\beta-\kappa}}T^{p',q',r'}_{-\beta-\kappa}} \lesssim  2^{-km\eta} \cdot\nrm{f_k}_{\vphantom{T^{p',q',r'}_{-\beta-\kappa}}T^{p,q,r}_\beta} \cdot  \nrm{g}_{\vphantom{T^{p',q',r'}_{-\beta-\kappa}}T^{p',q',r'}_{-\beta-\kappa}},\]
for some $\eta>0$.
 The series for {$j\geq k-2$} is controlled by $2^{-k m\alpha} \cdot\nrm{f_k}_{\vphantom{T^{p',q',r'}_{-\beta-\kappa}}T^{p,q,r}_\beta} \cdot  \nrm{g}_{\vphantom{T^{p',q',r'}_{-\beta-\kappa}}T^{p',q',r'}_{-\beta-\kappa}}$ as for the (w-w)-term. Hence, for the  (ca-w)-term we can use $\varepsilon := m\min\cbrace{\alpha,\eta}$.

Finally, the (w-ca) term is handled differently. Fix $j$ and split the integral in the (w-ca) series according to $t\in (2^{(-1+\ell)m}t_0,2^{\ell m}t_0)$ for $\ell\ge 0$. Considering the $x$-supports of $f_k$ and $g$ shows that the integrand vanishes unless {$j+\ell+1\ge k$}.
We apply the off-diagonal estimates for $H^{\ca}_{\vec{v}',2^{jm}}$ 
 noted in Remark~\ref{rem:HwODE} with $E_t=(0,t_0)$
 and $F_t= (2^{(\ell-1)m}t_0,\infty)$.
 By these bounds for the $H^\ca$ and $H^\whit$ operators, the (w-ca) term is  
\[\lesssim \sum_{j=0}^{\infty}\sum_{\ell=0}^\infty  {\ind_{j+\ell+1\ge k}} \cdot \, 2^{-jm\alpha}\cdot  2^{-\ell m\eta} \cdot\nrm{f_k}_{\vphantom{T^{p',q',r'}_{-\beta-\kappa}}T^{p,q,r}_\beta} \cdot  \nrm{g}_{\vphantom{T^{p',q',r'}_{-\beta-\kappa}}T^{p',q',r'}_{-\beta-\kappa}},\]
for some $\eta>0$, which is controlled by 
\[   2^{-k \varepsilon } \cdot\nrm{f_k}_{\vphantom{T^{p',q',r'}_{-\beta-\kappa}}T^{p,q,r}_\beta} \cdot  \nrm{g}_{\vphantom{T^{p',q',r'}_{-\beta-\kappa}}T^{p',q',r'}_{-\beta-\kappa}}\]
for  $\varepsilon <m \min\{\alpha,\eta\}$.
This finishes the proof in the case $\min\{p,q\}\geq 1$.
The proof for $\min\{p,q\}<1$ is similar, using instead Theorem~\ref{thm:quasimaindominationform} applied to the extension.
\end{proof}

\begin{remark} It is possible to do a purely dyadic version of Lemma~\ref{lem:annulitobox} based on the decomposition in Proposition~\ref{prop:decomposition}
or the decomposition~\eqref{eq:pdecomp} by taking $g$ supported in a Whitney cube $\Qw{P}$ and $f_k$ supported in the complement of the union of all Whitney regions $\Qw{Q}$ with $2^{-km}\ell(P)\le \ell(Q)\le 2^{km}\ell(P)$ with $d(Q,P)\le 2^k\ell(P)$. This indicates that 
we did not lose too much information in controlling these decompositions using the model operators. 
\end{remark}

\begin{remark}
 The space $L^\infty_c(\R^{1+n}_+)$ has the advantage of being contained in all tent spaces.  Yet, it could happen that a singular operator is defined on a larger space to start with and other density procedures are available, leading to the same extension. For example,  the space of functions of bounded support (not necessarily away from the boundary) in $L^r(\R^{1+n}_+; t^{-\beta r}\dd t\dd x)\cap T^{p,\infty,r}_\beta$ is dense in $T^{p,\infty,r}_\beta$ for $p\in (0,\infty)$, $r\in (1,\infty)$ and  $\beta\in \R$. This was proved in \cite[Lemma 5.2]{HR23} when $\beta=0$ and the proof also holds for $0<p\le 1$. If the singular operator $T$ is initially defined on that dense space, one can check that our decompositions and estimates go through also for 
functions of bounded support, showing that, under the assumptions of Theorem~\ref{thm:mainestimate}, 
$T$ extends uniquely from this dense subspace to $T^{p,\infty,r}_\beta$.
The key point is that bounded support allows us to start the decomposition in 
Proposition~\ref{prop:decomposition} from a top Carleson box $\Qca{Q}_0$.
It does not seem that  we have density of some space of functions with bounded support when $p=\infty$ and our extension is the only one we know. 
\end{remark}

\subsection{Improvement on earlier results.}
\label{sec:comparison}
Let us explain how Theorems~\ref{thm:mainestimate} and \ref{thm:extension} fully recover and improve the main $T^{p,2,2}_{\beta} \to T^{p,2,2}_{\beta+\kappa}$ boundedness results for singular integral operators from \cite[Section 3.2]{AH25a} when $u_t=1$ and $v_t=\infty$. Following Proposition \ref{prop:compareAH} and Remark \ref{rem:KernelODE}\ref{it:kernelODEAH}, we will use the number   
\begin{equation*}
 M_{\text{AH}} :=  M+ 1-\kappa+\tfrac{n}{m}[u_x,v_x],
\end{equation*}
where $u_x,v_x \in [1,\infty]$ will be chosen, to express the off-diagonal parameter used there.  

To make the comparison easier, we  specify Theorem~\ref{thm:mainestimate} when $q=r=2$ in each of the four situations:  causal or anti-causal, $p\ge2$ or $p\le 2$, {where we understand that $T$ is extended from compactly supported functions to  tent spaces
as in Section~\ref{sec:extension}.}

\begin{cor}\label{cor:q=r=2} Let $T$ be a  singular operator of type $(\kappa, m, 2, (1,u_x), (\infty,v_x), M)$ with $\kappa, M\in \R$ and $u_x,v_x\in [1,\infty]$. Let $p\in (0,\infty]$ and $\beta\in \R$. Then     $T$ extends to a bounded operator from $T^{p,2,2}_{\beta}$ to $T^{p,2,2}_{\beta+\kappa}$ if $T$ satisfies the pointwise Lipschitz property and one of the following holds
\begin{enumerate}[(i)]
    \item \label{it:causalpge2} $T$ is causal, $u_x=2, v_x\ge 2$, $p\ge 2$ and   
    \begin{align*}
    \beta&>-\tfrac 12 
    \\
    M_{\textup{AH}}& > \tfrac{n}{m} \cdot  [2,\max \{p,v_x\}].
    \end{align*}
    \item \label{it:anticausalpge2} $T$ is anti-causal, $u_x=2, v_x\ge 2$, $p\ge 2$ and  \begin{align*}
    \beta&<\tfrac 12 - \tfrac{n}{m} \cdot [v_x,\max\cbrace{p,v_x}]-\kappa    \\M_{\textup{AH}} &> \tfrac{n}{m} \cdot  [2,\max \{p,v_x\}].
    \end{align*}
    \item \label{it:causalple2} $T$ is causal, $u_x\le 2, v_x= 2$, $p\le 2$ and  \begin{align*}
    \beta&>-\tfrac 12 +\tfrac{n}{m} \cdot  [ \min\cbrace{p,u_x},u_x]\\ M_{\textup{AH}} &> \tfrac{n}{m} \cdot  [\min \{p,u_x\}, 2].\end{align*}
     \item \label{it:anticausalple2} $T$ is anti-causal, $u_x\le 2, v_x= 2$, $p\le 2$ and
     \begin{align*}
     \beta&<\tfrac 12 -\kappa \\M_{\textup{AH}}& > \tfrac{n}{m} \cdot  [\min \{p,u_x\}, 2].
     \end{align*}
\end{enumerate}
\end{cor}

\begin{proof} These are simple verifications from Theorem~\ref{thm:mainestimate}. The lower bound on $M_{\text{AH}}$ follows from \eqref{e:main-decay-condition}, the lower bound on $\beta$ if $T$ is causal from \eqref{e:main-beta-lower-bound}, and the upper bound on $\beta$ if $T$ is anti-causal from \eqref{e:main-beta-upper-bound}. 
    \end{proof}

    \begin{remark}\label{rem:L2bdd}
    Note that lower bounds for $M_{\text{AH}}$ are $p$-dependent.
    An interesting  particular case  is $p=u_x=v_x=2$ which arises in each item: the condition  $M_{\text{AH}}>0$ yields   $L^2_{\beta}(\R^{1+n}_+)$ to $L^2_{\beta+\kappa}(\R^{1+n}_+)$ boundedness for \textbf{all} $\beta$ in the given range, whereas we  are only assuming  localized $L^2$-boundedness on Whitney regions.  Thus, if one knows only such boundedness for one value of $\beta$, as is often the case, then boundedness for any $\beta$ as above follows. 
    For maximal regularity operators this is a well-known phenomenon, see, e.g., \cite{CK18}.
    \end{remark}

    Before  making the comparison, we introduce the two exponents $p_{M_{\text{AH}}}$ and $p_{s}(\beta)$ defined implicitly by 
\[ M_{\text{AH}}\coloneqq  \tfrac{n}{m} \cdot [p_{M_{\text{AH}}}, 2] \qquad \text{and} \qquad \beta\coloneqq -\tfrac 12 + \tfrac{n}{m} \cdot [p_{s}(\beta),s].  \]
A calculation shows that the hypotheses on $\beta$ and $M_{\text{AH}}$ in item  \ref{it:causalple2} are equivalent to
\[ \beta>-\tfrac 12 \qquad \text{and} \qquad  \max\{p_{M_{\text{AH}}}, p_{u_x}(\beta)\}  < \min\cbrace{p,u_x}.\]
Similarly, the conditions on $\beta$ and $M_{\text{AH}}$ in item  \ref{it:anticausalpge2} are equivalent to 
\[ \beta+\kappa <\tfrac 12 \qquad \text{and} \qquad  \max\{p_{M_{\text{AH}}}, p_{v_x'}(-\beta-\kappa)\}  < \min\cbrace{p',v_x'}.\]
This shows a duality between the two statements. 

In \cite{AH25a}, for the same conclusion as above, $T$ is additionally assumed to be linear, to satisfy global  $L^2_\gamma(\R^{1+n}_+)$ to $L^2_{\gamma+\kappa}(\R^{1+n}_+)$ boundedness for $\gamma=0$ and $\gamma =\beta$ and to have an $L^2(\R^n)$-bounded kernel representation (see Definition~\ref{def:kernelODE}). Given the equivalences above, one can check that in the four separate statements there, the conditions on $\beta$ are exactly the same while {the lower bound on $M_{\text{AH}}$ is more restrictive and is, in particular, at least  $\frac{n}{2m}$}.

\begin{remark} 
  \cite[Corollary~2 and Corollary 3]{AH25a} also establish $T^{\infty,2, 2}_{\beta,[p,1]} \to T^{\infty,2,2}_{\beta+\kappa,[p,1]}$ boundedness for the adjoint $T^*$  in the range  $0<p<1$ and with $\beta$ for which one has $T^{p,2,2}_{-\beta-\kappa}$ to $T^{p,2,2}_{-\beta}$ boundedness for the singular integral operator $T$. This is obtained using duality since $T^{\infty,2, 2}_{\beta,[p,1]}$ (which we do not define here) is identified with the dual space of $T^{p,2, 2}_{-\beta}$ in the duality of $L^2(\R^{1+n}_+, \dd t\dd x)$. We can also  recover this result differently with the following argument.  Indeed, $T^*$  is easily seen to be a singular operator to which the form domination applies.   Applying the required model operator boundedness, see Lemma~\ref{lemma:Whitneymap} and Proposition \ref{prop:cabddp>infty}, and Theorem~\ref{thm:mainestimate} proves  $T^{\infty,\infty,2}_{\beta+ \frac nm [p,1]}$ to $T^{\infty,\infty,2}_{\beta+ \frac nm [p,1]+\kappa}$ boundedness of $T^*$ under appropriate conditions on $\beta+ \frac nm [p,1]$ and $M_{\text{AH}}$. As \cite[Proposition 3.27]{Haa26}, re-expressed in our context, shows that $T^{\infty,2,2}_{\beta,[p,1]} = T^{\infty,\infty,2}_{\beta+ \frac nm [p,1]}$ isomorphically, we conclude using these isomorphisms. As before, one can check that we get the same conditions on $\beta, p$ with a smaller lower bound on the decay parameter. 
   \end{remark}

\section{Applications to elliptic PDEs}
\label{sec:ellipticPDES}

In this section, we shall first improve results on Calder\'on--Zygmund operators and answer questions concerning maximal regularity for elliptic PDEs posed in \cite{HR23}.
We shall also prove new tent space estimates for Riesz transforms in the half-space.

\subsection{Calder\'on--Zygmund operators}

Consider  the classical Calder\'on--Zygmund operators 
$$
  Sf(t,x)=\text{p.v.}\int_{\R^{1+n}_+} k(t,x;s,y) f(s,y) \dd s\dd y, \qquad
  (t,x)\in \R^{1+n}_+,
$$
where we assume that the Calder\'on--Zygmund kernel has the estimate 
\begin{equation}  \label{eq:kest}
      |k(t,x;s,y)|\lesssim \frac 1{(|t-s|+|x-y|)^{1+n}},\qquad
  (t,x),(s,y)\in \R^{1+n}_+.
\end{equation}
If $S:L^r(\R^{1+n}_+)\to L^r(\R^{1+n}_+)$ is bounded for some $r \in (1,\infty)$, 
then $S$ is an SIO
of type $(0,1,r,1,\infty,0)$ in the sense of 
Definition~\ref{def:kernelODE}.

We will also consider the causal (upward mapping) truncation $S^+$ of $S$, with kernel
\begin{equation}  \label{eq:kplus}
      k^+(t,x;s,y):= k(t,x;s,y) \ind_{t>s},
\end{equation}
assuming that $S^+\colon L^r(\R^{1+n}_+)\to L^r(\R^{1+n}_+)$ is bounded.
Note that this does not follow from the $L^r$-boundedness of $S$. Instead, this typically 
requires a double cancellation of $k$, which is present for Beurling-type
SIOs, but not for Riesz-type transforms.
Similarly, we also consider the anti-causal (downward mapping) truncation $S^-$
of $S$, with kernel 
\begin{equation}  \label{eq:kminus}
      k^-(t,x;s,y):= k(t,x;s,y) \ind_{t<s}.
\end{equation}

Proposition~\ref{prop:compareAH} and
Theorem~\ref{thm:mainestimate} immediately give the following theorem, where we understand that $S$, $S^+$ and $S^-$ are extended from compactly supported functions to the tent spaces
as in Section~\ref{sec:extension}.

\begin{theorem}   \label{thm:cz}
Let $r\in(1,\infty)$, $p,q\in(0,\infty]$ and $\beta\in\R$, and assume that
$$
  \gamma:= 1-n\cdot \bracb{\min\cbrace{p,q,1},1}>0,
$$
or equivalently, $p,q>\frac{n}{n+1}$.
\begin{itemize}
    \item 
    If $S:L^r(\R^{1+n}_+)\to L^r(\R^{1+n}_+)$ is bounded and has a CZ kernel $k$
with pointwise upper bounds \eqref{eq:kest},
then $S: T^{p,q,r}_\beta\to T^{p,q,r}_\beta$ is bounded
if $q^{-1}-\gamma<\beta <q^{-1}$.
\item 
If $S^+:L^r(\R^{1+n}_+)\to L^r(\R^{1+n}_+)$ is bounded and has a causal 
CZ kernel $k^+$ as in \eqref{eq:kplus} with pointwise upper bounds \eqref{eq:kest}
on $k$,
then $S^+: T^{p,q,r}_\beta\to T^{p,q,r}_\beta$ is bounded
if $q^{-1}-\gamma<\beta$.
\item 
If $S^-:L^r(\R^{1+n}_+)\to L^r(\R^{1+n}_+)$ is bounded and has an anti-causal 
CZ kernel $k^-$ as in \eqref{eq:kminus} with pointwise upper bounds \eqref{eq:kest}
on $k$,
then $S^-: T^{p,q,r}_\beta\to T^{p,q,r}_\beta$ is bounded
if $\beta<q^{-1}$.
\end{itemize}
\end{theorem}

\begin{remark}~
\begin{itemize}
    \item Note that in the Banach range $1\le p,q\le \infty$, we have $\gamma=1$ in Theorem~\ref{thm:cz}.
In particular this recovers \cite[Theorem 5.1]{HR23}.
Indeed, for $S^+$ this uses $1<p<\infty$, $q=\infty$ and $\beta=0$, for which the above hypothesis holds. 
For $S^-$ this uses $1<p<\infty$, $q=1$ and $\beta=0$, for which the above hypothesis also holds using an equivalent  norm on $T^{p,1,r}_0$ based on a Carleson functional \cite[Theorem 3]{CMS85}.
Our Theorem~\ref{thm:cz} also answers the open problem posed in the second bullet in 
\cite[Section 6.1]{HR23}, in that it gives boundedness results for Calder\'on--Zygmund operators on general tent spaces.
\item  It is interesting to note that our strategy to obtain tent space estimates is completely different from that in \cite{HR23}. We exclusively rely on the pointwise upper bound \eqref{eq:kest}, yet we assume $L^r(\R^{1+n}_+)$-boundedness, which usually uses some regularity on $k$. In \cite{HR23}, H\"older regularity of $k$ is used  to show  weak $L^1(\R^{1+n}_+)$-boundedness of the SIO and the Lerner maximally truncated SIO,
\cite[Lemma 3.1 and Proposition 3.2]{HR23}, which in turn is used in the proof of
$L^r(\R^{1+n}_+)$ boundedness of $S, S^+$ and $S^-$ for $1<r<\infty$. The size estimate \eqref{eq:kest} was also needed in \cite{HR23} to show the weak $L^1(\R^{1+n}_+)$-estimates and to bound the 
technical $x$-separation term $II_1$ in the proof of the sparse domination (Theorem 4.1). 
\end{itemize}
\end{remark}

\subsection{ODE for \texorpdfstring{$DB$}{DB}}
\label{ssec:ODE-DB}
We next aim to show how Theorem~\ref{thm:mainestimate} implies 
tent space estimates for singular integral operators associated with a divergence form elliptic operator
\begin{equation} 
    \label{e:L}
  L=-\Div_x A(x) \nabla_x
\end{equation}
in $\R^n$,
with general bounded, measurable and accretive coefficients $A$.
Here the singular operators that appear typically have a non-causal structure.
To start with, we require off-diagonal estimates (using one possible interpretation of the subtitle) of resolvent operators.
For the second order operator $L$, we recall from
\cite{A07} and \cite[Chapters 6 and 12]{Auscher-Egert2023-book}
that the maximal open interval
$(p_-(L),p_+(L))\subseteq (1,\infty)$
where  we have
\begin{equation}
    \label{eq:p(L)}
    \sup_{t>0} \, \nrm{(I+t^2 L)^{-1}}_{L^p\to L^p} <\infty,\qquad p \in (p_-(L),p_+(L)),
\end{equation}
  is such that $\min\cbrace{[p_-(L),2], [2,p_+(L)]}> 1/n$ if $n\ge 3$, while $p_-(L)=1$ and $p_+(L)=\infty$ if $n=1,2$. 
We also recall that the maximal open interval $(q_-(L),q_+(L))\subseteq (1,\infty)$
where  we have
\begin{equation}
    \label{eq:q(L)}
  \sup_{t>0} \, \nrm{t\nabla(I+t^2 L)^{-1}}_{L^p\to L^p} <\infty,\qquad p \in (q_-(L),q_+(L)),
\end{equation}
is such that $q_-(L)=p_-(L)$  and $q_+(L)>2$.
Furthermore 
$[q_+(L),p_+(L)]\ge 1/n$ if $q_+(L)<n$, and otherwise $p_+(L)=\infty$.
Throughout this section, we work with $L^p$-estimates on $\R^n$ and we drop the subscript $x$ in the notation for simplicity.
Off-diagonal estimates for operators $\psi(L)$ in the holomorphic functional calculus of $L$ are well
understood, see \cite{Auscher-Egert2023-book}.

More generally, we shall make use of associated first order operators $DB$ on
$\R^n$, where 
$$
D= \begin{bmatrix}
   0 & -\Div \\
   \nabla & 0
\end{bmatrix}
$$ 
and $B$ is a bounded, measurable and accretive $(1+n)\times (1+n)$ matrix multiplier on $\R^n$, related to the coefficients $A$.
The relation between $DB$ and $L$ is roughly speaking the same as that between the Laplace operator and the Cauchy--Riemann operator, and will be explained in more detail below.
Note that in the special case of 
$B= \begin{bmatrix}
   1 & 0 \\
   0 & A
\end{bmatrix}$,
the operator
$(DB)^2$ is diagonal, with the $(1,1)$-element equal to $L$. However, in general we do not assume the coefficients $B$ to be block diagonal.
Operators $DB$ are often referred to as perturbed Dirac operators, although strictly speaking $-D^2$ only equals the Laplace operator when acting on tuples 
$\begin{bmatrix}
    u \\  v
\end{bmatrix}$,
where $v$ is a gradient vector field.

By \cite{AKMc06}, the operator $DB$ is bisectorial in $L^2(\R^n;\C^{1+n})$ with square function estimates and bounded bisectorial $H^\infty$-functional calculus. 
Extensions to $L^p(\R^n;\C^{1+n})$ can be found in 
\cite[Theorem 3.6]{AS16}, where $L^p$-bisectoriality of $DB$ and the bisectorial holomorphic functional calculus $H^\infty$ estimate
$$
  \|\varphi(DB)u\|_{L^p}\lesssim \|\varphi\|_{H^\infty} \|u\|_{L^p}
$$
is shown to hold for $p$ in an open interval
$(p_-(DB),p_+(DB))\subseteq (1,\infty)$ 
containing $p=2$ and depending on the coefficients in $B$.
Here $\varphi$ is a bounded and holomorphic symbol on an open bisector containing the spectrum of $DB$. 
In particular, for $p\in (p_-(DB), p_+(DB))$ we have
$$
  \sup_{t\in \R} \, \nrm{(I+it DB)^{-1}}_{L^p\to L^p} <\infty.
$$
In the specific block diagonal case
$B= \begin{bmatrix}
   1 & 0 \\
   0 & A
\end{bmatrix}$, we have $(p_-(DB), p_+(DB))= (q_+(L^*)', q_+(L))$. See \cite[Proposition 15.1]{Auscher-Egert2023-book} and \cite{ABE25}, where 
$p_+(DB)$ is identified with $q_+(L)$ and the Meyers exponent,
see \cite[Definition~7.2]{ABE25},
respectively, while $p_-(DB)$  is the H\"older conjugate of $p_+(DB^*)$.

Coming back to the general case, off-diagonal estimates for operators $\psi(DB)$ in the holomorphic functional calculus of $DB$, with general coefficients $B$, were proved in \cite[Chapter 15]{AS16}.
The complication, as compared to $L$, is that
the operator $DB$ is not injective when $n\ge 2$, but has an infinite dimensional null space.
For all $p\in (p_-(DB), p_+(DB))$, we have a decomposition into  (the closure of) the range and the null space
\begin{equation}   \label{eq:hodge}
      L^p(\R^n; \C^{1+n})= \overline{R(DB)}\oplus N(DB),
\end{equation}
where the closure and direct sum are in the $L^p$-topology. One can work within their intersections with the corresponding spaces defined on $L^2$ by density in $L^p$-topology. In other words, the calculi for different $p\in (p_-(DB), p_+(DB))$ are consistent.
Below we improve on the off-diagonal estimates from
 \cite{AS16}, starting with resolvent type operators.

\begin{lemma}  \label{lem:DBresODEs}
    Let $p_-(DB)<p_1<p_2<p_+(DB)$ and define the bisector
    $$
S_\mu:=\cbraceb{t\in \C: |\Im (t)| \le \tan(\mu) |\Re (t)| }
    $$
with $0<\mu<\frac{\pi}{2}$ depending on $B$.
\begin{enumerate}[(i)]
    \item \label{it:DBODE1}
    If $[p_1,p_2]\le \frac1n$, then the estimate 
\begin{equation*}   
\nrm{(I+it DB)^{-1} f}_{L^{p_2}}
\lesssim |t|^{-n[p_1,p_2]}\nrm{f}_{L^{p_1}}    
\end{equation*}
holds for all $0\ne t\in S_\mu$ and all $f\in \overline{R(D)}$.
    \item \label{it:DBODE2}
    If $[p_1,p_2]< \frac1n$, there exists
$c>0$ such that 
the off-diagonal estimates  
\begin{equation*}  
\nrm{\ind_F tDB(I+t^2 (DB)^2)^{-1} (f\ind_E) }_{L^{p_2}}
\lesssim |t|^{-n[p_1,p_2]}
 e^{-c \frac{d(E,F)}{|t|}}
\nrm{f\ind_E}_{L^{p_1}}    
\end{equation*}
hold for all $0\ne t\in S_\mu$, all $E,F\subseteq \R^n$, and all $f\in L^{p_1}(\R^n; \C^{1+n})$.
\item \label{item:ODEsN}
If $[p_1,p_2]< \frac{N}{n}$ for some $N\in\N$, then
there exists
$c>0$ and $0\leq \theta<1$ such that
the off-diagonal estimates
\begin{equation*} 
\nrm{\ind_F (I+it DB)^{-1} (f\ind_E) }_{L^{p_2}}
\lesssim |t|^{-n[p_1,p_2]}
\brB{\frac {|t|}{d(E,F)}}^{(2N-1)\theta}  e^{-c \frac{d(E,F)}{|t|}}
\nrm{f\ind_E}_{L^{p_1}}    
\end{equation*}
hold for all $0\ne t\in S_\mu$, all $E,F\subseteq \R^n$
with $d(E,F)>0$, and all $f\in L^{p_1}(\R^n; \C^{1+n})$.
\end{enumerate}
\end{lemma}

\begin{proof}
For \ref{it:DBODE1} we use the Sobolev embedding $ W^{1,p_1} \subseteq L^{p_2}$ and the ellipticity estimate
  $\nrm{\nabla u}_{L^p}\lesssim \nrm{Du}_{L^p}$ for $u\in \overline{R(D)}$, which follows from Riesz transform bounds for $p \in (1,\infty)$.
  Details are in \cite[Section 15.3, Claim 18]{AS16}.

For \ref{it:DBODE2} we first claim the estimate
$$
\nrm{tDB(I+t^2 (DB)^2)^{-1} f }_{L^{p_2}}
\lesssim |t|^{-n[p_1,p_2]}
\nrm{f}_{L^{p_1}}.  
$$
For $f\in \overline{R(DB)} = \overline{R(D)}$,
 this follows from \ref{it:DBODE1} upon writing
$$
  tDB(I+t^2 (DB)^2)^{-1}=
  \tfrac 1{2i}\brb{(I-itDB)^{-1}-(I+itDB)^{-1}}.
$$
For $f\in N(DB)$, we have $tDB(I+t^2 (DB)^2)^{-1} f =0$. Thus, \eqref{eq:hodge} proves the claim for general $f\in L^{p_1}(\R^n; \C^{1+n})$.

The off-diagonal estimate in \ref{it:DBODE2} for the operator
$Q:=\ind_F tDB(I+t^2 (DB)^2)^{-1}\ind_E$ 
follows by interpolation.
Indeed, taking $0<\theta<1$
    sufficiently close to $1$ and defining $\tilde p_1,\tilde p_2$ by the relations  $[p_1,2]=\theta[\tilde p_1,2]$
and $[2,p_2]=\theta[2,\tilde p_2]$, we have 
$p_-(DB)<\tilde p_1<\tilde p_2<p_+(DB)$ and 
    $[\tilde p_1,\tilde p_2]\le \frac1n$.
By the above claim,
$$
  \nrm{Q}_{L^{\tilde p_1}(E)\to 
L^{\tilde p_2}(F)}\lesssim |t|^{-n[\tilde p_1,\tilde p_2]}.
$$
Exponential $L^2$ off-diagonal estimates 
$$
  \nrm{Q}_{L^2(E)\to 
L^2(F)}\lesssim e^{-c d(E,F)/|t|}
$$
for $DB$ are proved as in \cite[Lemma 5.3]{CMcM13}.
Interpolation now gives \ref{it:DBODE2} upon replacing $c(1-\theta)$ by $c$ and noting that
$\theta[\tilde p_1, \tilde p_2]= [p_1,p_2]$.

For \ref{item:ODEsN}, start by assuming $[p_1,p_2]\le \frac1n$.
The argument for \ref{it:DBODE2}  does not directly apply to the resolvents 
$(I+it DB)^{-1}$, since the estimate
\ref{it:DBODE1} only holds for $f\in \overline{R(D)}$.
The non-trivial observation is that it is nevertheless possible to obtain an 
$L^{p_1}(E)\to L^{p_2}(F)$ estimate for all
$f\in L^{p_1}(\R^n; \C^{1+n})$ if $d(E,F)>0$.
To see this, we consider the commutator with a Lipschitz function $\eta:\R^n\to \R$, such that $\eta=1$ on $F$, $\eta=0$ on $E$, and $\nrm{\nabla\eta}_\infty\lesssim 1/d(E,F)$.
Note the commutator identity
\begin{align} \label{eq:firstcomm} \begin{aligned}
     \ind_F(I+itDB)^{-1}\ind_E&= \ind_F[\eta, (I+itDB)^{-1}]\ind_E \\
 &= -\ind_F(I+itDB)^{-1} it[\eta, D]B(I+itDB)^{-1}\ind_E 
\end{aligned}\end{align}
and that $[\eta,D]$ is a multiplier by a matrix
pointwise bounded by $|\nabla\eta|$.
We can further choose $\eta$ so that 
$G:= \supp{\nabla\eta}$ satisfies
$d(F,G)\eqsim d(G,E)\eqsim d(E,F)$.

Given $f\in L^{p_1}(\R^n; \C^{1+n})$, we write $\ind_E f=f_1+f_0$ using \eqref{eq:hodge}, with $f_1\in\overline{R(DB)}$
and $f_0\in N(DB)$.
For $f_1$ we estimate
$$
\nrmB{\brB{\frac 1{I+itDB} t[\eta, D]B} \brB{ \frac 1{I+itDB} }f_1}_{L^{p_2}}\lesssim
\frac {|t|}{d(E,F)} |t|^{-n[p_1,p_2]} \nrm{f_1}_{L^{p_1}},
$$
using $L^{p_2}$-boundedness of the operator in the left parenthesis, and \ref{it:DBODE1} for the right one.
For $f_0$ we estimate
\begin{align*}
\nrmB{\brB{\frac 1{I+itDB}} \brB{ t[\eta, D]B 
\frac 1{I+itDB} }f_0}_{L^{p_2}}
&\lesssim |t|^{-n[p_1,p_2]} \nrm{ t[\eta, D]B f_0}_{L^{p_1}}\\
&\lesssim
|t|^{-n[p_1,p_2]}\frac {|t|}{d(E,F)}  \nrm{f_0}_{L^{p_1}}, 
\end{align*}
using that $(I+itDB)^{-1}f_0= f_0\in N(DB)$ and
that $[\eta,D]B f_0= -D(\eta Bf_0)\in \overline{R(D)}$, i.e.,  \ref{it:DBODE1} applies 
to the operator in the left parenthesis.
Using \eqref{eq:hodge}, we have shown that 
\begin{equation}  \label{eq:resOD}
\nrm{\ind_F (I+it DB)^{-1} (f\ind_E) }_{L^{p_2}}
\lesssim |t|^{-n[p_1,p_2]}
\frac {|t|}{d(E,F)}
\nrm{f\ind_E}_{L^{p_1}}.     
\end{equation}

In general, if $[p_1,p_2]<\frac{ N}{n}$, we can iterate 
the argument with the commutator \eqref{eq:firstcomm}.
To this end, consider a chain of exponents
$$
  p_1= q_0<q_1<q_2<\ldots < q_N= p_2,
$$
where $n[q_{j-1},q_j]\le 1$ for $j=1,2,\ldots, N$.
Here \eqref{eq:resOD} applies with $p_1,p_2$ replaced
by $q_{j-1}, q_j$, and iterated use of 
\eqref{eq:firstcomm} proves that 
$$
\nrm{\ind_F (I+it DB)^{-1} (f\ind_E) }_{L^{p_2}}
\lesssim |t|^{-n[p_1,p_2]}
\brB{\frac {|t|}{d(E,F)}}^{2N-1}
\nrm{f\ind_E}_{L^{p_1}}.  
$$
Interpolation with the exponential $L^2$ off-diagonal estimates as in \ref{it:DBODE2} now proves 
the off-diagonal estimates \ref{item:ODEsN}, which completes the proof.
\end{proof}

\begin{remark}
To appreciate the off-diagonal estimates in Lemma \ref{lem:DBresODEs}\ref{item:ODEsN},
it is instructive to consider the case $B=I$.
Using the identity
$$
  (I-t^2 \delta d)(I-t^2 d\delta)= I-t^2 \Delta, 
$$
where $d$ and $\delta$ are the exterior and interior derivatives on vector fields, the formula
$$
\frac 1{I+itD}= 
\begin{bmatrix}
    1 & it\Div \\ -it\nabla & I-t^2\delta d
\end{bmatrix}
\frac 1{I-t^2 \Delta}
$$
follows. 
This reveals that the $(2,2)$-component of the 
matrix kernel of $(I+itD)^{-1}$ contains terms with
second derivatives of Bessel potentials.
This yields a Calder\'on--Zygmund kernel near the diagonal, but a smooth exponentially decaying kernel
away from the diagonal.
In particular we see that the estimate \ref{it:DBODE2} is not possible for the resolvents.
\end{remark}

\subsection{Elliptic maximal regularity operators} \label{sec:ellipticmaxreg}

In \cite{AAM10, AA11}, singular integral type representation formulas for solutions to
\begin{equation}  \label{eq:divformeq}
  \Div_{t,x} (A(t,x) \nabla_{t,x} u)=0
\end{equation}
were proved
and used to prove solvability
estimates for boundary value problems.
Here $t$ has the meaning of the coordinate transversal to the boundary $\R^n$ rather than time.
A first order PDE equivalent to \eqref{eq:divformeq} 
is the Cauchy--Riemann type system
\begin{equation}  \label{eq:CRsystem}
  \partial_t f + DB f=0
\end{equation}
for the conormal gradient $f=\begin{bmatrix}
    \partial_{\nu_A} u \\ \nabla_x u
\end{bmatrix}$
of $u$.
The coefficients $B$ are obtained from $A$ via a pointwise 
non-linear transformation of bounded accretive coefficients,
and $t$-dependency is handled perturbatively by writing
the equation as $$\partial_t f + DB_0 f=DB_0\mE f,$$ with 
$\mE= B_0^{-1}(B_0-B)$ and $B_0(x)= B(0,x)$.
Here the generator $DB_0$ is bisectorial in
$L^2(\R^n;\C^{1+n})$, and therefore Duhamel 
integration involves both a causal and an anti-causal term.
More precisely, letting  $$\Lambda:= |DB_0|= DB_0(E^+-E^-),$$ where 
$E^\pm=\ind_{\pm \Re(\lambda)>0}(DB_0)$ 
are the spectral projections of $DB_0$ for the sectors
$\pm \Re(\lambda)>0$,  the integrated, equivalent form of \eqref{eq:CRsystem} is 
$$
  f= \mathcal{C}^+ f_0+ (S^+ + S^-)\mE f,
$$
where now $f_0$ denotes the boundary trace 
$\lim_{t\to 0} f$.
The Cauchy operator $\mathcal{C}^+$ is defined by  $$\mathcal{C}^+ f_0(t):= e^{-t\Lambda}E^+f_0, \qquad f(t):= x\mapsto f(t,x),$$ (when $n=1$ and $A=I$, this is exactly the Cauchy operator from complex analysis). The causal elliptic maximal regularity operator is given by
$$
  (S^+ f)(t):= \int_0^t \Lambda e^{-(t-s)\Lambda}E^+ f(s) \dd s,
$$
and the anti-causal maximal regularity operator by $$
  (S^- f)(t):= \int_t^\infty \Lambda e^{-(s-t)\Lambda}E^- f(s) \dd s.
$$
It was shown for suitable small Carleson norm of the perturbation $\mE$, and with appropriate estimates of $S^\pm$, 
that this leads to a perturbed Cauchy-type reproducing formula
$$
  f= (I-(S^++S^-)\mE)^{-1}\mathcal{C}^+ f_0
$$
for $f$ in terms of its boundary trace $f_0$.
In \cite[Theorem 6.5]{AA11}, the boundedness of
$
S^+\colon T^{2,2}_{\beta}\to T^{2,2}_{\beta}
$
and dually $S^-\colon T^{2,2}_{-\beta}\to T^{2,2}_{-\beta}$,
was proved for $\beta>-1/2$, 
together with a weaker endpoint estimate for $\beta=-1/2$
in \cite[Theorem 6.8]{AA11}.

In the special case of a block form multiplier
$B_0=\begin{bmatrix}
   1 & 0 \\
   0 & A_0
\end{bmatrix}$,
$\Lambda=|DB_0|$ is block diagonal with $(1,1)$-
component equal to $\sqrt L$.
Here $A_0=A_0(x)$ is a bounded, measurable and accretive $n\times n$ matrix multiplier on $\R^n$, and $L= -\Div_x (A_0 \nabla_x)$.
The range of $L^p$ boundedness for the Poisson semigroup 
$e^{-t\sqrt{L}}$ is the larger interval
$(p_-(L),p_+(L))\subseteq (1,\infty)$, the same as the interval of $L^p$-boundedness for the resolvent   
and  the $H^\infty$-functional calculus for $L$, see \cite[Chapters~10, 12 and 13]{Auscher-Egert2023-book}.
We define the causal and anti-causal maximal regularity operators $T^\pm$ by 
 $$
  (T^+f)(t) := \int_0^t \sqrt L e^{-(t-s)\sqrt L} f(s) \dd s, \qquad (T^-f)(t) := \int_t^\infty \sqrt L e^{-(s-t)\sqrt L} f(s) \dd s,
$$
for the Poisson semigroup associated to $\sqrt{L}$, a priori defined for $f\in L^\infty_c(\R^{1+n}_+)$.

From Theorem~\ref{thm:mainestimate} we obtain the following general tent space maximal regularity estimates for these 
Poisson semigroups, where again the extension from compactly supported functions is as in Section~\ref{sec:extension}.
In particular this solves the open problems formulated in the first two bullets in
\cite[Section 6.1]{HR23} for $S^\pm$.
For $T^\pm$ it improves 
the results obtained when $q=r=2$ in \cite{AKMP12} for $T^\pm$, and extends them to a range of $q,r$.

\begin{theorem}   \label{thm:beurling}
Consider a perturbed Dirac operator $DB$, where we write $B$ for the matrix multiplier on $\R^n$.
  Let $p_\pm:= p_\pm(DB)$,  $r \in (p_-,p_+)$, and $\beta\in \R$. Suppose that $p,q\in(0,\infty]$ satisfy
$$
n\bracb{\min\cbrace{p,q,p_-},p_-} +
    n\bracb{p_+,\max\cbrace{p,q,p_+}} < 1.
$$
  Then $S^+$ is bounded on  $T^{p,q,r}_{\beta}$ $($of $\C^{1+n}$-valued functions$)$ 
  if 
\begin{gather*}
    \beta>\tfrac{1}{q}-1+n\bracb{\min\cbrace{p,q,p_-},p_-}.
\end{gather*}
$S^-$ is bounded on  $T^{p,q,r}_{\beta}$ $($of $\C^{1+n}$-valued functions$)$ if
\begin{gather*}
 \beta< \tfrac 1q-n\bracb{p_+,\max\cbrace{p,q,p_+}} 
\end{gather*}

The maximal regularity operators $T^\pm$ for the Poisson semigroup are bounded on  $T^{p,q,r}_{\beta}$ 
under the same conditions as for $S^\pm$, but with 
$p_\pm:=p_\pm(L)$.
\end{theorem}

\begin{remark}
  The $T^{p,2,2}_{1/2}$ boundedness of $T^+$ was shown in
  \cite[Proposition~1.7]{AKMP12}. Recall that
  $\min\cbrace{[p_-(L),2], [2,p_+(L)]}> \frac1n$ if $n\ge 3$, while $p_-(L)=1$ and $p_+(L)=\infty$ if $n=1,2$.
  Setting $q=r=2$ and
  $\beta=1/2$,  Theorem~\ref{thm:beurling} 
  recovers this result 
  in dimension $1\le n\le 4$, while 
  for $n\ge 5$, we obtain here 
  $$
    \frac{2n}{n+4}-\varepsilon<p<\frac{2n}{n-4}+\varepsilon',
  $$
  with $\varepsilon,\varepsilon'>0$ depending on $L$, which is a strict improvement. When $A$ has real coefficients, the range for $p$ is $(\tfrac{n}{n+1},\infty]$,  since $p_-(L)=1$ and $p_+(L)=\infty$ in all dimensions.
\end{remark}

\begin{proof}
    Consider first $T^+$ and let $p_-(L)<p_1<r<p_2<p_+(L)$. By
    Proposition~\ref{prop:compareAH} and Theorem~\ref{thm:mainestimate}, it suffices to show that $T^+$ is a singular integral operator of type $(0,1,r,p_1,p_2,0)$ in the sense of Definition~\ref{def:kernelODE},
    since $p_1>p_-(L)$ and $p_2<p_+(L)$ are arbitrary.
    Using the symbol $\psi(\lambda)= \lambda e^{-\lambda}1_{\Re(\lambda)>0}$, the operator-valued kernel of $T^+$ is
    $$
      K(t,s)=\ind_{s<t}\ (t-s)^{-1} \psi((t-s)\sqrt{L}).
    $$ 
    The boundedness of $T^+$ on $L^r(\R^{1+n}_+)$ follows from  
    $L^r$-maximal regularity of $\sqrt{L}$ on $L^r(\R^{n})$, which is a consequence of the 
     $H^\infty$-functional calculus on the same space, see \cite{KW04}.
    Off-diagonal estimates 
    \begin{equation}\label{eq:maxregdecay}
      \nrm{\ind_F \psi(\tau\sqrt{L})(f \ind_E )}_{L^{p_2}}\lesssim 
      \tau^{-n[p_1,p_2]}(1+d(E,F)/\tau)^{-1-n[p_1,p_2]}
      \nrm{f\ind_E }_{L^{p_1}}        
    \end{equation}
    follow from
    \cite[Proposition 12.7]{Auscher-Egert2023-book}, where again, we drop the $x$-subscript in the notation.
    We note that this is stated only for $p_2=2$,
    but inspecting the proof reveals that it goes through for any $p_2<p_+(L)$.
    This proves the result for $T^+$, and the proof for $T^-$ is similar.
    
    Consider next $S^+$, where we now let
     $p_-(DB)<p_1<r<p_2<p_+(DB)$.
    Once the off-diagonal estimates \eqref{eq:maxregdecay}, with $\sqrt{L}$ replaced
    by $DB$, for the operator-valued kernel
    $$
      K(t,s)=\ind_{s<t}\ (t-s)^{-1} \psi((t-s)DB)
    $$ 
    have been proved, the proof for $S^\pm$ is analogous to the proof for $T^\pm$. 
    These off-diagonal estimates however are more 
    complicated to see due to the non-injectivity
    of $DB$ and, with this optimal order of decay, are not in the literature.
    To this end, let $\alpha>0$ and consider the auxiliary symbol
    $\psi_\alpha(\lambda)= \lambda^\alpha e^{-\lambda}1_{\Re(\lambda)>0}$.
    By Dunford--Riesz calculus
    \begin{equation}   \label{eq:dunford}
      \psi_\alpha(\tau DB)= \frac 1{2\pi i}
      \int_{\gamma} \frac{(\tau \lambda)^\alpha e^{-\tau\lambda}}{I-\lambda^{-1}DB} \frac{d\lambda}\lambda, 
    \end{equation}
    where $\tau>0$ and $\gamma:=\partial S_{\nu+}$ with
    $S_{\nu+}:= \{\lambda \in \C:|\arg\lambda\, |<\nu\}$ for $\pi/2-\mu<\nu<\pi/2$ with $\mu$ as in Lemma~\ref{lem:DBresODEs}. Fractional powers are obtained using the principal branch of the logarithm.  We will prove \eqref{eq:maxregdecay} in three steps.

    \emph{Step 1:}
    We first prove $L^{p_1}-L^{p_2}$ boundedness of 
    $\psi_\alpha(\tau DB)$, where by \eqref{eq:hodge}
    we  may assume that $f\in \overline{R(D)}$, since $\psi_\alpha(\tau DB) f=0$ if $f\in N(DB)$.
    Let $N>n[p_1,p_2]$ and pick a chain of
    exponents 
$$
  p_1= q_0<q_1<q_2<\ldots < q_N= p_2,
$$
where $[q_{j-1},q_j]< \frac1n$ for $j=1,2,\ldots, N$.
Estimating \eqref{eq:dunford} using
Lemma~\ref{lem:DBresODEs}\ref{it:DBODE1} gives
$$
    \nrm{\psi_{\alpha/N}(\tau DB)f}_{L^{q_j}}\lesssim
      \int_{\gamma} (\tau |\lambda|)^{\alpha/N} e^{-{c_\nu}\tau|\lambda|} |\lambda|^{n[q_{j-1},q_j]}  \frac{|\ddn \lambda|}{|\lambda|} \nrm{f}_{L^{q_{j-1}}} 
      \lesssim \tau^{-n[q_{j-1},q_j]}\nrm{f}_{L^{q_{j-1}}}, 
$$
for $j=1,\ldots, N$, with $c_\nu>0$.
Composing these operators in the functional calculus 
of $DB$ gives
\begin{align*}
   \nrm{\psi_{\alpha}(\tau DB)f}_{L^{p_2}}
   &\eqsim \nrm{\brb{\psi_{\alpha/N}(\tfrac \tau N DB)}^N f}_{L^{p_2}} \\
&\lesssim \tau^{-n[q_{N-1},q_N]-n[q_{N-2},q_{N-1}]-\ldots -n[q_0,q_1]}
   \nrm{f}_{L^{p_1}}= \tau^{-n[p_1,p_2]}\nrm{f}_{L^{p_1}}.
\end{align*}

\emph{Step 2:}
We next prove off-diagonal estimates of 
$\psi_{\alpha}(\tau DB)$ for $E,F \subseteq \R^n$ with $d(E,F)\gtrsim \tau$ and large enough $\alpha$.
Estimating \eqref{eq:dunford} using instead
Lemma~\ref{lem:DBresODEs}\ref{item:ODEsN} gives
\begin{align*}
      \nrm{\ind_F\psi_{\alpha}&(\tau DB)(f \ind_E )}_{L^{p_2}}\\
      &\lesssim
      \int_{\gamma} (\tau |\lambda|)^{\alpha} e^{-{c_\nu}\tau|\lambda|} |\lambda|^{n[p_1,p_2]}
      \brb{|\lambda| d(E,F)}^{-(2N-1)\theta}
      e^{- c d(E,F) |\lambda|}
      \frac{|\ddn \lambda|}{|\lambda|} \nrm{f\ind_E }_{L^{p_1}}\\
  &\lesssim
  \frac{\tau^\alpha d(E,F)^{-(2N-1)\theta}}
  {({c_\nu}\tau+cd(E,F))^{\alpha+n[p_1,p_2]-(2N-1)\theta}}
  \nrm{f\ind_E }_{L^{p_1}}  \eqsim 
    \frac{\tau^\alpha}
  {d(E,F)^{\alpha+n[p_1,p_2]}}
  \nrm{f\ind_E }_{L^{p_1}},
\end{align*}
provided that $\alpha>(2N-1)\theta-n[p_1,p_2]$.

\emph{Step 3:} Proof of \eqref{eq:maxregdecay} for $DB$.
Combining the estimates from Steps 1 and 2, we have proved that
$$
      \nrm{\ind_F\psi_{\alpha}(\tau DB)(f \ind_E )}_{L^{p_2}}\lesssim
    \tau^{-n[p_1,p_2]}
    \brb{1+d(E,F)/\tau}^{-\alpha-n[p_1,p_2]}
  \nrm{f\ind_E }_{L^{p_1}},
$$
for any $E,F\subseteq \R^n$, $\tau>0$, 
and 
$\alpha>(2N-1)\theta-n[p_1,p_2]$.
By adapting the proof of \cite[Proposition 12.7]{Auscher-Egert2023-book}, we now extrapolate this 
result down to $\alpha=1$.
Using the Calder\'on reproducing formula
$$
  f= c\int_0^\infty (s\Lambda)^{\alpha-1} e^{-s\Lambda} f\frac {\ddn s}s,
$$
valid for $f\in \overline{R(DB)}\supset R(E^+)$ and $\alpha >1$,
writing $\Lambda= |DB|$ and $E^+= \ind^+(DB)$,
we get
$$
  \ind_F\tau \Lambda e^{-\tau\Lambda} E^+ (f\ind_E) 
  = c\int_0^\infty \frac {\tau s^{\alpha-1}}{(s+\tau)^\alpha} \ind_F \brb{(s+\tau)\Lambda}^\alpha 
  e^{-(s+\tau)\Lambda}E^+(f \ind_E)  \frac{\ddn s}s.
$$
This gives the estimate
\begin{align*}
  \nrm{\ind_F\tau \Lambda e^{-\tau\Lambda} &E^+ (f \ind_E )}_{L^{p_2}} \\
  &\lesssim \int_0^\infty \frac {\tau s^{\alpha-1}}{(s+\tau)^\alpha} 
  (s+\tau)^{-n[p_1,p_2]} 
  \brb{1+\tfrac{d(E,F)}{s+\tau}}^{-\alpha-n[p_1,p_2]}
  \nrm{f \ind_E }_{L^{p_1}} \frac{\ddn s}s\\
  &\lesssim 
  \tau^{-n[p_1,p_2]}
  \int_0^\infty \frac {\sigma^{\alpha-1}}{(\sigma+1)^{\alpha+n[p_1,p_2]}} 
  \brB{1+\tfrac{d(E,F)}{\tau(\sigma+1)}}^{-\alpha-n[p_1,p_2]}\frac{\ddn \sigma}\sigma
  \nrm{f \ind_E }_{L^{p_1}} \\
   &\lesssim 
   \tau^{-n[p_1,p_2]}\brb{1+\tfrac{d(E,F)}\tau}^{-1-n[p_1,p_2]}
  \nrm{f \ind_E }_{L^{p_1}},
\end{align*}
where the last integral estimate with the parameter 
$d(E,F)/\tau$ is seen as in 
\cite[Proposition 12.7]{Auscher-Egert2023-book}
with $\gamma=n[p_1,p_2]$ and $\theta \alpha$
replaced by $\alpha+n[p_1,p_2]$ and choosing $\alpha$ large enough.
This proves the off-diagonal estimates \eqref{eq:maxregdecay} for $DB$, which completes the
proof.
\end{proof}

\subsection{Generalized Riesz transforms on the half-space}
\label{sec:GRT}

We switch our perspective from dimension $n$ to dimension $n+1$ and consider operators $DB$ on $\R^{1+n}_+$. The pointwise multiplier $B= B(t,x)$ is now bounded, measurable and accretive on $\R^{1+n}_+$. To obtain self-adjoint Dirac operators $D$ in 
$L^2(\R^{1+n}_+; \C^{2+n})$, we impose either normal or tangential boundary conditions on the gradient and divergence operators respectively.
Consider these operators in 
$L^r$ {for $1<r<\infty$}.
The domain $\dom_r(\nabla_{t,x})$  consists of those $f\in L^r(\R^{1+n}_+)$ for which the distributional gradient $\nabla_{t,x}f$ is in $L^r(\R^{1+n}_+;\C^{1+n})$. 
The domain $\dom_r(\Div_{t,x})$  consists of those $f\in L^r(\R^{1+n}_+;\C^{1+n})$ for which the distributional divergence $\Div_{t,x}f$ is in $L^r(\R^{1+n}_+)$.
The operators $\underline{\nabla}_{t,x}$ and
$\underline{\Div}_{t,x}$ with natural boundary 
conditions are related to 
$\Div_{t,x}$ and $\nabla_{t,x}$ in $L^{r'}$
by duality:
$$
  \underline{\nabla}_{t,x}= -(\Div_{t,x})^*\quad
  \text{and}\quad 
  \underline{\Div}_{t,x}= -(\nabla_{t,x})^*.
$$
Concretely, the domain  
$\dom_r(\underline{\nabla}_{t,x})$ is $W^{1,r}_0(\R^{1+n}_+)$, and the boundary condition on fields $f$ in the domain  
$\dom_r(\underline{\Div}_{t,x})$ is that $f$ is tangential on $\R^n$ in the natural weak sense.

Define the Dirac operator with tangential boundary conditions
$$
D_\ta= \begin{bmatrix}
   0 & -\underline{\Div}_{t,x} \\
   \nabla_{t,x} & 0
\end{bmatrix},
$$ 
and with normal boundary conditions
$$
D_\no= \begin{bmatrix}
   0 & -\Div_{t,x} \\
   \underline{\nabla}_{t,x} & 0
\end{bmatrix}.
$$ 

A convenient way to handle these boundary conditions is the following reduction to operators on the full space $\R^{1+n}$, which
is an adaptation of the technique from \cite{AT03}.
Define the Dirac operator
$$
D= \begin{bmatrix}
   0 & -\Div_{t,x} \\
   \nabla_{t,x} & 0
\end{bmatrix}
$$ 
on $L^r(\R^{1+n};\C^{2+n})$.
Given the coefficients $B$ on $\R^{1+n}_+$ above, we define coefficients
$$
\widetilde B(t,x)
=\begin{cases}
    B(t,x), & t>0, \\
    RB(-t,x)R, & t<0, 
\end{cases}
$$
on $\R^{1+n}$, where 
$$
  R\begin{bmatrix} u \\ v
\end{bmatrix}\coloneqq \begin{bmatrix}
   u \\ R'v
\end{bmatrix} :=
\begin{bmatrix}
   u \\ v- 2\scl{v}{e_0}e_0
\end{bmatrix}
$$
where $R'$ is the reflection
across $\{0\}\times \C^n$ along the vertical vector $e_0$ in $\C^{1+n}=\C\times \C^n$.

\begin{lemma}   \label{lem:Jtrick}
  The operator $D\widetilde B$ on $\R^{1+n}$
  is similar to the direct sum of the operators
  $D_\ta B$ and $D_\no B$ on $\R^{1+n}_+$.
  More precisely, the isomorphism
  $J: L^r(\R^{1+n}; \C^{2+n})\to L^r(\R^{1+n}_+;\C^{2+n})\oplus L^r(\R^{1+n}_+;\C^{2+n})$ defined by
  $$
    (Jf)(t,x)= \frac 12
    \begin{bmatrix} f(t,x)+Rf(-t,x) \\ 
    f(t,x)-Rf(-t,x)
\end{bmatrix},\qquad t>0,
  $$
  intertwines the operators, that is, 
  $$
    \begin{bmatrix} D_\ta B & 0 \\ 0 & D_\no B
\end{bmatrix} J=
J (D\widetilde B)
  $$
  and $J\dom(D\widetilde B)= 
  \dom(D_\ta B)\oplus \dom(D_\no B)$.
\end{lemma}

\begin{proof}
  One checks that 
  $\begin{bmatrix} B & 0 \\ 0 & B\end{bmatrix} J=J \widetilde B$.
  Therefore it suffices to show that
    $\begin{bmatrix} D_\ta & 0 \\ 0 & D_\no
\end{bmatrix} J=J D
  $.
  To this end, let $f= \begin{bmatrix} u \\ v
\end{bmatrix}$ and
$Jf= \left( \begin{bmatrix} u_+ \\ v_+
\end{bmatrix}, \begin{bmatrix} u_- \\ v_-
\end{bmatrix} \right)$.
Concretely, this means for $t>0$ that
\begin{align*}
    u_\pm(t,x)&= \tfrac 12(u(t,x)\pm u(-t,x)), \\
    v_\pm(t,x)&= \tfrac 12(v(t,x)\pm R'v(-t,x)). 
\end{align*}
The chain rule implies that 
$\nabla(u(-t,x))= R'(\nabla u)(-t,x)$
and $\Div(R'v(-t,x))= (\Div v)(-t,x)$.

For smooth $f$, we see that $u_-(0,x)=0= \scl{v_+(0,x)}{e_0}$.
Conversely, if this holds, then $u$ and $\scl{v}{e_0}$ are continuous across $\R^n$.
In general, one checks that $u\in W^{1,r}(\R^{1+n})$ if and only if $u_-\in W^{1,r}_0(\R^{1+n}_+)$ and
$u_+\in W^{1,r}(\R^{1+n}_+)$, and also
$v\in \dom_r(\Div_{t,x})$ on $\R^{1+n}$ if and only if
$v_-\in \dom_r(\Div_{t,x})$ and
$v_+\in\dom_r(\underline{\Div}_{t,x})$ 
on $\R^{1+n}_+$.
This shows that 
  $$
    \begin{bmatrix} D_\ta & 0 \\ 0 & D_\no
\end{bmatrix} J=
JD,$$
with matching domains.
\end{proof}

Consider now the operators 
$\sgn(D_\ta B)$ and $\sgn(D_\no B)$, using the symbol
$\sgn(\lambda):= \frac{\Re(\lambda)}{|\Re(\lambda)|}$.
These generalize the  
Riesz transforms 
\begin{equation}
\label{eq:RTNeu}
   R_{N,A}\coloneqq \nabla_{t,x}(-\underline{\Div}_{t,x}(A\nabla_{t,x}))^{-1/2}
\end{equation}
and 
\begin{equation}
\label{eq:RTDir}R_{D,A}\coloneqq\underline{\nabla}_{t,x}(-\Div_{t,x}(A\underline{\nabla}_{t,x}))^{-1/2}
\end{equation}
associated with the divergence form equation
\eqref{eq:divformeq}, with Neumann and Dirichlet boundary conditions respectively.
Indeed, for the choice of coefficients
$B= \begin{bmatrix}
    1 & 0 \\ 0 & A
\end{bmatrix}$,
we have
$$
  \sgn(D_\ta B)=
  \begin{bmatrix}
    0 & (-\underline{\Div}_{t,x}(A\nabla_{t,x}))^{-1/2} (-\underline{\Div}_{t,x}A)\\
    \nabla_{t,x}(-\underline{\Div}_{t,x}(A\nabla_{t,x}))^{-1/2} & 0
  \end{bmatrix}
$$
and 
$$
  \sgn(D_\no B)=
  \begin{bmatrix}
    0 & (-\Div_{t,x}(A\underline{\nabla}_{t,x}))^{-1/2} (-\Div_{t,x}A)\\
    \underline{\nabla}_{t,x}(-\Div_{t,x}(A\underline{\nabla}_{t,x}))^{-1/2} & 0
  \end{bmatrix},
$$
containing the Riesz transforms along with their adjoints, upon swapping $A$ and $A^*$.
As in \cite{AT03}, we define the auxiliary
 maximal accretive operator $\tilde L=-\Div_{t,x}(\tilde A{\nabla}_{t,x})$ on $\R^{1+n}$, where $\tilde A(t,x)=A(t,x)$ if $t>0$ and $\tilde A(t,x)=R'A(-t,x)R'$ if $t<0$.

From Theorem~\ref{thm:mainestimate} we obtain the following new tent space estimates for these (generalized) Riesz transforms, where again the extension from compactly supported functions is as in Section~\ref{sec:extension}.
We note that the proof uses off-diagonal estimates with $u_t=u_x$ and $v_t=v_x$, rather than 
$u_t=1$ and $v_t=\infty$.

\begin{theorem}   \label{thm:riesz}
  Let $B(t,x)$ be a bounded, measurable and accretive multiplier on $\R^{1+n}_+$.
  Let $p_\pm= p_\pm(D\widetilde B)$, $r \in (p_-,p_+)$ and  $\beta\in\R$. 
  Let $p,q\in(0,\infty]$ satisfy 
  $$
n\bracb{\min\cbrace{p,q,p_-},p_-} +
    n\bracb{p_+,\max\cbrace{p,q,p_+}} < \brac{p_-,p_+}.
  $$
  Then the generalized Riesz transforms 
  $\mathrm{\sgn}(D_\ta B)$ and $\mathrm{\sgn}(D_\no B)$ of $\C^{2+n}$-valued functions on 
  $\R^{1+n}_+$ are bounded on $T^{p,q,r}_\beta$
   if 
\begin{gather*}
    -\brac{p_-,q}+n\bracb{\min\cbrace{p,q,p_-},p_-}<\beta<\brac{q,p_+}-n\bracb{p_+,\max\cbrace{p,q,p_+}}.
\end{gather*}
The same result holds for the Riesz transforms $R_{N,A}$ and $R_{D,A}$, mapping $\C$-valued functions to $\C^{1+n}$-valued
functions, in \eqref{eq:RTNeu} and \eqref{eq:RTDir} with $p_\pm(D\widetilde B)$ replaced by 
$q_\pm(\tilde L)$.
\end{theorem}

\begin{proof}
    Let $p_-(D\widetilde B)<p_1<r<p_2<p_+(D\widetilde B)$ and note that the operator $D\widetilde B$ on $\R^{1+n}$ from 
  Lemma~\ref{lem:Jtrick} has a bounded $H^\infty$ functional calculus in $L^r(\R^{1+n};\C^{2+n})$.
    Consider first the auxiliary generalized Riesz transform $S=\sgn(D\widetilde B)$, which is bounded on $L^r(\R^{1+n};\C^{2+n})$.
    To prove off-diagonal estimates for $S$, by functional calculus we have
    $$
      \sgn(D\widetilde B)= c_N\int_0^\infty
      \psi(\tau D\widetilde B) \ \frac{\ddn \tau}\tau,
    $$
    where $\psi(\lambda)= \lambda^N/(1+\lambda^2)^N$ for odd $N\in \N$  and
    $0<c_N<\infty$.
    From Lemma~\ref{lem:DBresODEs} in $\R^{1+n}$ and 
    \cite[Lemma 4.6]{Auscher-Egert2023-book}, we have
    $$
\|\ind_{\mbs{\widetilde F}} \psi(\tau D\widetilde B) (f \ind_{\mbs{\widetilde E}} )\|_{L^{p_2}_{t,x}}\le 
C \tau^{-(n+1)[p_1,p_2]} e^{- c d(\mbs{\widetilde E} , \mbs{\widetilde F})/\tau}
\|\ind_{\mbs{\widetilde E}}f \|_{L^{p_1}_{t,x}}
$$
for subsets $\mbs{\widetilde E},\mbs{\widetilde F}\subset \R^{1+n}$
and some $c>0$, and $N>(n+1)[p_1,p_2]$
which we fix.
We get
\begin{align*}
\|\ind_{\mbs{\widetilde F}} \sgn(D\widetilde B) ( f\ind_{\mbs{\widetilde E}} )\|_{L^{p_2}_{t,x}}
&\lesssim \int_0^\infty  \tau^{-(n+1)[p_1,p_2]} 
e^{- c d(\mbs{\widetilde E} , \mbs{\widetilde F})/\tau}
\|f \ind_{\mbs{\widetilde E}}\|_{L^{p_1}_{t,x}}
\frac{\dd \tau}\tau \\
&\lesssim d(\mbs{\widetilde E},\mbs{\widetilde F})^{-(n+1)[p_1,p_2]}
\|f \ind_{\mbs{\widetilde E}}\|_{L^{p_1}_{t,x}}.
\end{align*}
By Lemma~\ref{lem:Jtrick}, the same off-diagonal
estimates for subsets $\mbs{E},\mbs{F}$ of $\R^{1+n}_+$ hold for 
$\sgn(D_\ta B)$ and $\sgn(D_\no B)$,
since
  $$
    \begin{bmatrix} \sgn(D_\ta B) & 0 \\ 0 & \sgn(D_\no B)
\end{bmatrix} J=
J \,\sgn(D\widetilde B),
  $$
and noting that  we have 
$d(\mbs{\widetilde E}, \mbs{\widetilde F})= d(\mbs{E},\mbs{F})$,
where $\mbs{\widetilde E}= \mbs{E}\cup R'\mbs{E}$ and
$\mbs{\widetilde F}= \mbs{F}\cup R'\mbs{F}$,
acting pointwise with the reflection $R'$ on the sets. 

We have therefore shown \eqref{eq:odeprototype} with $M=0$, $m=1$, $\kappa=0$,  $(u_t,u_x)=(p_1,p_1)$ and $(v_t,v_x)=(p_2,p_2)$.  Proposition~\ref{prop:SIO-simple} 
applies to show  that $\sgn(D_\ta B)$ and $\sgn(D_\no B)$ are singular operators of type $(0,1,r,(p_1,p_1),(p_2,p_2),0)$ in the sense of Definition~\ref{def:ODE}.
Their tent space bounds now follow from Theorem~\ref{thm:mainestimate},
if $p_1,p_2$ are such that
\begin{align*}
    n\bracb{\min\cbrace{p,q,p_1},p_1} +
    n\bracb{p_2,\max\cbrace{p,q,p_2}} < \brac{p_1,p_2}
\end{align*}
and
\begin{align*}
-\brac{p_1,q}+n\bracb{\min\cbrace{p,q,p_1},p_1}<\beta<\brac{q,p_2}-n\bracb{p_2,\max\cbrace{p,q,p_2}}.
\end{align*}
Since $p_1>p_-(D\widetilde B)$ and $p_2<p_+(D\widetilde B)$ are arbitrary,
the result for $\mathrm{\sgn}(D_\ta B)$ and $\mathrm{\sgn}(D_\no B)$ follows.

The result for the Riesz transforms 
$R_{N,A}$ and $R_{D,A}$
is proved analogously, replacing $p_\pm(D\widetilde B)$ by 
$q_\pm(\tilde L)$, the endpoints of the maximal interval on which the Riesz transform $\nabla \tilde L^{-1/2}$ is $L^p$-bounded. See \cite[Theorem~7.3]{Auscher-Egert2023-book}.
\end{proof}

\section{Applications to parabolic PDEs}
\label{sec:parabolic}

In this section, we obtain tent space boundedness of the Duhamel and Lions operators related to solving the non-autonomous parabolic Cauchy problems 
\begin{equation}
    \label{e:NaPC}
    \tag{NaPC}
    \begin{cases}
    \partial_t u - \Div_x(A(t,x)\nabla_x u) = f + \Div_x F, \quad (t,x) \in (0,\infty) \times \R^n \\
    u(0) = u_0.
\end{cases}
\end{equation}

We first ignore the initial data and work on $\R^{1+n}$. Assume that the coefficient matrix $A \in L^\infty(\R^{1+n};\mat_n(\C))$ is bounded, measurable, and accretive. We do \emph{not} impose \emph{any} assumptions on the regularity or symmetry of $A$. Our methods also apply to systems of equations, with modifications reflecting the chosen notion of accretivity. We restrict to scalar equations for simplicity. We introduce the relevant function (distribution) spaces, derive the invertibility of the parabolic operator $\partial_t  - \Div_x (A(t,x)\nabla_x)$ via Sneiberg's lemma, and deduce some bounds on mixed-norm Lebesgue spaces $\Ltx{q}{p}$, as well as some local self-improvement properties of weak solutions.

The Duhamel and Lions operators arise when restricting to the half-space with $t>0$. Using the above estimates (with $A$ equal to the identity matrix $\bfI$ when $t<0$) allows us to show these operators fall into our new framework of singular operators. We finish with applications to the Cauchy problem.

\subsection{Function spaces}

Let $p,q\in (1,\infty)$. We want to work with homogeneous distributional spaces corresponding to their inhomogeneous versions ${\cV}^{q,p}\coloneq L^{\vphantom{1,p}q}_t {H}^{1,p}_{\vphantom{t}x} \cap {H}^{1/2,q}_tL^{\vphantom{1/2}p}_{\vphantom{t}x}$, which are adapted to the parabolic setting. The situation for inhomogeneous spaces was described in \cite[Section 6.1]{AEN20} and is well-known from e.g. \cite{Liz70}. But for the homogeneous counterparts, the naive expression $L^{\vphantom{1,p}q}_t \dot{H}^{1,p}_{\vphantom{t}x} \cap \dot{H}^{1/2,q}_t L^{\vphantom{1/2}p}_{\vphantom{t}x}$ cannot simply be taken as a definition, since they are typically defined in the space of tempered distributions modulo polynomials, while we want a space of distributions for the study of the PDE. As we have not seen the needed spaces explicitly introduced in the literature, we provide them here. Inspired by the approach in \cite{AE23} when $q=p=2$, we define our spaces via Littlewood-Paley theory, using the Fourier transform acting on tempered distributions $\cS'(\R^{1+n})$ with $(\tau,\xi)$ as the variables dual to $(t,x)$. With this definition at hand, it will not be difficult to transpose results from earlier literature on homogeneous anisotropic, mixed-norm Triebel--Lizorkin spaces, see e.g. \cite{CGN19} and \cite{Li21}. We require a particular \(\mathcal S'\)-realization of these known homogeneous anisotropic mixed-norm spaces, and have not found this realization stated in precisely the form needed here. Therefore, we will provide limited proofs.  

We first introduce a parabolic Littlewood--Paley decomposition of the identity. Define the parabolic quasi-norm 
\[ |(\tau, \xi)|_{\mathrm{par}}:=\left( \abs{\tau}^{2} + \abs{\xi}^4 \right)^{1/4} \eqsim |\tau|^{1/2} + |\xi|, \qquad (\tau,\xi) \in \R^{1+n}, \]
which is smooth away from the origin. Equip $\R^{1+n}$ with the induced parabolic quasi-metric. 

Let $\psi \in \cS(\R^{1+n})$ be an arbitrary real-valued function with Fourier transform supported in $\{(\tau,\xi)\in \R^{1+n} : 1/2 < |(\tau,\xi)|_{\mathrm{par}} < 4\}$ and $\hat \psi(\tau,\xi)>0$ when $1 \le |(\tau,\xi)|_{\mathrm{par}} \le 2$. For any $j \in \Z$, define $\psi_{j}(t,x) := (2^j)^{n+2} \psi(2^{2j} t,2^j x)$, and hence, $\widehat{{\psi}_j}(\tau,\xi) = \hat{\psi}(2^{-2j}\tau, 2^{-j}\xi)$. Using the Mikhlin multiplier theorem in anisotropic vector-valued mixed-norm Lebesgue spaces (see \cite[Theorem 7.1]{L21} with $X=\C$ and $Y=\ell^2(\Z)$ and vice versa), we have the characterisation 
\[\|f\|_{\Ltx{q}{p}} \eqsim \nrmb{\|(\psi_j\ast f)_{j\in \Z}\|_{\ell^2(\Z)} }_{\Ltx{q}{p}} . \]
Moreover, if $f\in \Ltx{q}{p}$, then the series $\sum_{j\in \Z} \psi_j \ast f$ converges to $f$ in $\cS'(\R^{1+n})$ when assuming the function $\psi$ is normalised such that 
\begin{equation*}
    \sum_{j = -\infty}^\infty \widehat{{\psi}_j}(\tau,\xi)=1, \qquad (\tau,\xi)\in \R^{1+n} \setminus \{(0,0)\}.
\end{equation*}

For $f\in \cS'(\R^{1+n})$, define the semi-norm
\[ \|f\|_{\dot\cV^{q,p}} := \nrmb{\|(2^{j}\psi_j\ast f)_{j\in \Z}\|_{\ell^2(\Z)} }_{\Ltx{q}{p}}. \]
For any $f \in \cS'(\R^{1+n})$, the series $\sum_{j \ge 0} \psi_j \ast f$ converges in $\cS'(\R^{1+n})$. For the series with $j < 0$, if $\|f\|_{\dot\cV^{q,p}} < \infty$, then one can show that
\begin{itemize}
    \item If $2/q+n/p>1$, then $\sum_{j< 0} \psi_j \ast f$ converges in $\cS'(\R^{1+n})$. 
    \item If $2/q+n/p \le 1$, then $\sum_{j< 0} \psi_j \ast f$ converges in $\cS'_0(\R^{1+n})$, the dual of the space $\cS_0(\R^{1+n})$ of Schwartz functions with Fourier transform vanishing at the origin.
\end{itemize}
Define the semi-normed space
\begin{equation*}
    \dot\cV^{q,p} \coloneq \left\{ f \in \cS'(\R^{1+n}) : \|f\|_{\dot\cV^{q,p}} < \infty \text{ and } \sum_{j \in \Z} \psi_j \ast f = f \right\},
\end{equation*}
equipped with the semi-norm $\|\cdot\|_{\dot\cV^{q,p}}$, where the equality for the series in the definition is understood from the above convergence. When $2/q+n/p>1$, it becomes a norm. The space $\dot\cV^{q,p}$ is complete. The Schwartz space $\cS(\R^{1+n})$ is contained in $\dot\cV^{q,p}$, and its subspace $\cZ(\R^{1+n})$ of functions with compactly supported Fourier transforms in $\R^{1+n} \setminus \{(0,0)\}$ is dense in $\dot\cV^{q,p}$. Any other choice of $\psi$ (not necessarily normalised) gives an equivalent norm. The dual space $(\dot\cV^{q,p})'$ identifies with a subspace of $\cS'(\R^{1+n})$ characterised by 
\[ \left\| \| (2^{-j}\psi_j\ast f)_{j \in \Z} \|_{\ell^2(\Z)} \right\|_{\Ltx{q'}{p'}}<\infty, \]
and $\cZ(\R^{1+n})$ is dense in $(\dot\cV^{q,p})'$. The proofs follow the classical arguments for Lebesgue spaces, with minor modifications for the parabolic quasi-norm. 

For completeness, we give the definition of the inhomogeneous space $\cV^{q,p}$. Let $\varphi$ be a real-valued Schwartz function with Fourier transform supported in $\{(\tau,\xi)\in \R^{1+n}:|(\tau,\xi)|_{\mathrm{par}} < 4\}$ and $\hat \varphi(\tau,\xi)>0$ when $|(\tau,\xi)|_{\mathrm{par}} \le 2$. Define the inhomogeneous $\cV^{q,p}$ norm by 
\[ \|f\|_{\cV^{q,p}} := \|\varphi\ast f \|_{\Ltx{q}{p}} + \nrmb{\|(2^{j}\psi_j\ast f)_{j\in \N}\|_{\ell^2(\N)} }_{\Ltx{q}{p}}. \]
If $\varphi$ and $\psi$ are normalized by
\begin{equation*}
    \widehat{\varphi}(\tau,\xi) + \sum_{j = 0}^\infty \widehat{{\psi}_j}(\tau,\xi)=1, \qquad (\tau,\xi)\in \R^{1+n},
\end{equation*}
then we have $f=\varphi\ast f+ \sum_{j\ge 0}\psi_j\ast f$ for all $f\in \cS'(\R^{1+n})$. Define the inhomogeneous space
\begin{equation*}
    \cV^{q,p} \coloneq \left\{ f\in \cS'(\R^{1+n}) :  \|f\|_{\cV^{q,p}}<\infty \right\}.
\end{equation*}
Note that $\cS(\R^{1+n})$ is dense in $\cV^{q,p}$.

The following proposition summarizes two basic properties of $\dot\cV^{q,p}$ and $\cV^{q,p}$. 
For $f\in \cS(\R^{1+n})$, define $D_t^{1/2}f \in \cS'(\R^{1+n})$ by the Fourier multiplication $\widehat{D_t^{1/2} f}(\tau,\xi) := |\tau|^{1/2} \hatf(\tau,\xi)$. The next result shows it will extend to $f \in \dot\cV^{q,p}$ by density.

\begin{prop}[Basic properties of $\dot\cV^{q,p}$ and $\cV^{q,p}$]
    \label{prop:Vrho,p}
    Let $q,p \in (1,\infty)$.
    \begin{enumerate}[label=\normalfont(\roman*)]
        \item \label{item:Vrho,p-eq-norms}
        {\normalfont (Equivalent norms)}
        It holds that
        \begin{align*}
            \|f\|_{\dot\cV^{q,p}} &\eqsim \|\nabla_x f\|_{\Ltx{q}{p}} + \|D_t^{1/2}f\|_{\Ltx{q}{p}}, && f \in \dot\cV^{q,p}, \\
            \|f\|_{\cV^{q,p}} &\eqsim \|f\|_{\Ltx{q}{p}} + \|\nabla_x f\|_{\Ltx{q}{p}} + \|D_t^{1/2}f\|_{\Ltx{q}{p}}, && f \in \cV^{q,p}.
        \end{align*}
        
        \item \label{item:Vrho,p-Sobolev-embed}
        {\normalfont (Sobolev embedding)}
        The space $\dot\cV^{q,p}$ embeds into $\Ltx{\ovq}{\ovp}$ if
        \begin{equation}
            \label{eq:parabolicembedding}
            2[q,\ovq]+n[p,\ovp]=1, \quad (\ovq,\ovp)\in [q,\infty) \times [p,\infty). 
        \end{equation}
        The space $\cV^{q,p}$ embeds into $\Ltx{\ovq}{\ovp}$ if
        \begin{equation}
            \label{eq:parabolicembeddinginhomogeneous}
            2[q,\ovq]+n[p,\ovp]\le 1,  \quad (\ovq,\ovp)\in [q,\infty)\times [p,\infty). 
        \end{equation}
        Moreover, assume $q>2$. Then the space $\cV^{q,p}$ embeds into $\Ltx{\infty}{\ovp}$ if 
        \begin{equation}\label{eq:parabolicembeddinginfty}
            \begin{cases}
            \ovp\in  [p,\infty]   &\text{if}\ \  2[q,\infty]+n[p,\infty] < 1 \\
            \ovp\in  [p,\infty)   &\text{if}\ \ 2[q,\infty]+n[p,\infty] = 1 \\
            \ovp\in  [p,\frac{n}{2[q,\infty]+n[p,\infty]-1})   &\text{if}\ \  2[q,\infty]+n[p,\infty] > 1.
        \end{cases}
        \end{equation}
    \end{enumerate}
    All the implicit constants depend only on $n, q, p, \ovq, \ovp$.
\end{prop}

To illustrate the ranges of $(1/q,1/p)$ and $(1/\ovq,1/\ovp)$ in Proposition \ref{prop:Vrho,p}\ref{item:Vrho,p-Sobolev-embed}, we give graphic representations in Figure \ref{fig:rho-p-diagram-parabolic}
in a specific case $\frac{2}{q}+\frac{n}{p}=C>1$ and $q>2$. Given a point $(1/q,1/p)$, the orange shaded trapezoid, including its boundary but not the part on the $1/p$-axis, is the range of $(1/\ovq,1/\ovp)$ in \eqref{eq:parabolicembeddinginhomogeneous}, its blue boundary part the range of $(1/\ovq,1/\ovp)$ in \eqref{eq:parabolicembedding}, and its boundary on the $1/p$-axis the range of $1/\ovp$ in \eqref{eq:parabolicembeddinginfty}. The intersection of the blue and black boundary lines is the point $(0, \frac{2[q,\infty]+n[p,\infty]-1}{n})$.
\begin{figure}[ht]
    \centering
    \includegraphics[width=0.4\linewidth]{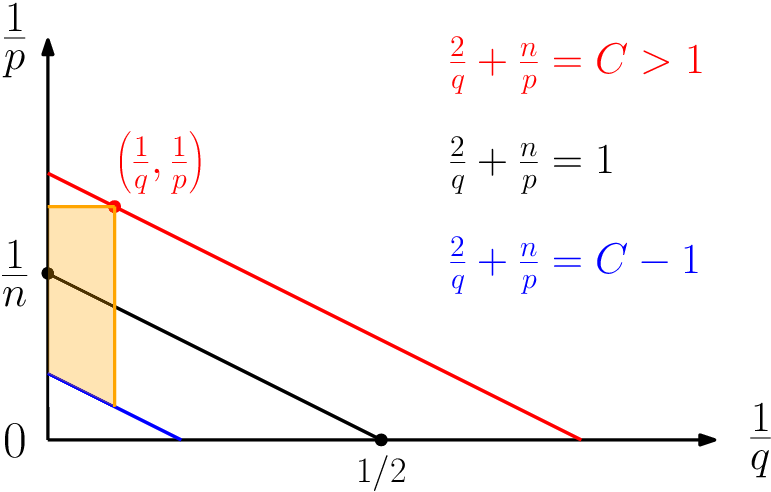}
    \caption{Illustration for the ranges of $(1/q,1/p)$, $(1/\ovq,1/\ovp)$, and $1/\ovp$}
    \label{fig:rho-p-diagram-parabolic}
\end{figure}

\begin{proof}[Proof of Proposition \ref{prop:Vrho,p}]
    We only present the proof in the homogeneous case. The inhomogeneous counterparts follow similarly. The reader can also refer to \cite[Theorem 1.1]{Li21} and \cite[Theorem 7]{Johnsen-Sickel2007-mixed-Sobolev-embedding} for the proof. By density, it suffices to prove these two statements for $f \in \cZ(\R^{1+n})$. Define $g \in \cZ(\R^{1+n})$ by the Fourier multiplication $\hatg(\tau,\xi) := |(\tau,\xi)|_{\mathrm{par}} \hatf(\tau,\xi)$. Define $\tilde \psi \in \cZ(\R^{1+n})$ by $\hat{\tilde\psi}(\tau,\xi) := |(\tau,\xi)|_{\mathrm{par}}^{-1} \hat\psi(\tau,\xi)$. Then we infer that
    \[ \|g\|_{\Ltx{q}{p}} \eqsim \nrmb{ \|(\tilde\psi_j\ast g)_{j\in \Z}\|_{\ell^2(\Z)} }_{\Ltx{q}{p}} = \nrmb{ \|(2^j\psi_j\ast f)_{j\in \Z} \|_{\ell^2(\Z)} }_{\Ltx{q}{p}} = \|f\|_{\dot\cV^{q,p}}. \]

    Let us prove \ref{item:Vrho,p-eq-norms}. Note that
    \[ (|\tau|^{1/2} + |\xi|) \hatf (\tau,\xi) = \frac{|\tau|^{1/2}+|\xi|}{|(\tau,\xi)|_{\mathrm{par}}} |(\tau,\xi)|_{\mathrm{par}} \hatf(\tau,\xi) = \frac{|\tau|^{1/2}+|\xi|}{|(\tau,\xi)|_{\mathrm{par}}} \hatg(\tau,\xi). \]
    The symbol $m(\tau,\xi) := \frac{|\tau|^{1/2}+|\xi|}{|(\tau,\xi)|_{\mathrm{par}}}$ satisfies Mikhlin's condition $| |\tau|^{\alpha} |\xi|^{|\beta|} \partial^{\alpha}_\tau \partial_\xi^\beta m(\tau,\xi) | \le C(\alpha,\beta) < \infty$ for all $\alpha \in \N$ and multi-indices $\beta \in \N^n$. Hence, the Mikhlin (or Marcinkiewicz) multiplier theorem in mixed-norm Lebesgue spaces (see \cite{Liz70}) shows that
    \[ \|D_t^{1/2} f\|_{\Ltx{q}{p}} + \|\nabla_x f\|_{\Ltx{q}{p}} \lesssim \|g\|_{\Ltx{q}{p}} \eqsim \|f\|_{\dot\cV^{q,p}}. \]
    Applying the same reasoning to $1/m$ gives us the reverse inequality. This proves \ref{item:Vrho,p-eq-norms}.

    Next, we prove \ref{item:Vrho,p-Sobolev-embed}. Observe that if $2[q,\infty]+n[p,\infty] \le 1$, then there do not exist such $\ovq$ and $\ovp$ satisfying \eqref{eq:parabolicembedding}. So assume $2[q,\infty]+n[p,\infty] > 1$. In this case, $\dot\cV^{q,p}$ embeds into $\cS'(\R^{1+n})$. Pick $\theta=2[q,\ovq]$, so that $\theta \in [0,2/q)$ and $1-\theta= n[p,\ovp]$ upon \eqref{eq:parabolicembedding}. Write
    \[ \hat f(\tau, \xi) = \frac{|\tau|^{\theta/2}|\xi|^{1-\theta}}{| (\tau,\xi) |_{\mathrm{par}}} |\tau|^{-\theta/2}|\xi|^{-1+\theta} \hat g(\tau, \xi). \]
    The symbol $\frac{|\tau|^{\theta/2}|\xi|^{1-\theta}}{| (\tau,\xi) |_{\mathrm{par}}}$ also satisfies Mikhlin's condition. Thus, using the Mikhlin multiplier theorem and boundedness of the Riesz potentials, we get
    \[ \|f\|_{\Ltx{\ovq}{\ovp}} \lesssim \|(|\tau|^{-\theta/2} |\xi|^{-1+\theta} \hat g)^\vee \|_{\Ltx{\ovq}{\ovp}} \lesssim \|( |\xi|^{-1+\theta} \hat g)^\vee \|_{\Ltx{q}{\ovp}} \lesssim \|g\|_{ \Ltx{q}{p} } \eqsim \|f\|_{\dot\cV^{q,p}}. \]
    The second inequality can be proved for example using representation in space variables of the Riesz potential in the $t$-variable and Minkowski integral inequality. This proves \ref{item:Vrho,p-Sobolev-embed} and completes the proof.
\end{proof}

From now on, we identify our spaces with subspaces of mixed-norm spaces via the embeddings in Proposition \ref{prop:Vrho,p}.

\subsection{The Sneiberg argument}

We now consider the forward and backward parabolic differential operators $\pm \partial_t - \Div_x(A(t,x)\nabla_x)$ on $\R^{1+n}$. These operators have been shown in \cite{AE23} (see also the streamlined presentation in \cite{Auscher-Baadi2025-fundamental-sol}) to be invertible from $\dot\cV^{2,2}$ onto $\big(\dot\cV^{2,2}\big)'$ and to satisfy the bounds
\[(\pm \partial_t - \Div_x(A(t,x)\nabla_x))^{-1}\colon L^2_t\dot{H}^{-1}_{\vphantom{t}x}+L^1_tL^2_{\vphantom{t}x} \to L^2_t\dot{H}^{1}_{\vphantom{t}x}\cap C_0 L^2_x, \]
where $C_0$ is the space of continuous functions on $\R$ with zero limits at $\pm\infty$, which excludes constants. Moreover, we have the duality relations
\begin{equation}
    \label{e:para-for-back-dual}
    (\pm \partial_t - \Div_x(A(t,x)\nabla_x))^\ast = \mp \partial_t - \Div_x(A^\ast(t,x)\nabla_x).
\end{equation}
The same holds for $(\pm \partial_t - \Div_x(A(t,x)\nabla_x)+1)^{-1}$ with the homogeneous spaces replaced by the inhomogeneous ones. 
Our next proposition extrapolates this result via Sneiberg's lemma.

\begin{prop}[Sneiberg extrapolation]
    \label{prop:sneiberg}
    There exists $\varepsilon>0$ such that if $q,p \in (1,\infty)$ satisfy $|[2,q]|<\varepsilon$ and $|[2,p]|<\varepsilon$, we have bounded extensions
    \begin{equation}
        \label{eq:homSneiberg}
        (\pm \partial_t - \Div_x(A(t,x)\nabla_x))^{-1} \colon (\dot\cV^{q',p'})' \to \dot\cV^{q,p}
    \end{equation}
    and 
    \begin{equation}
        \label{eq:inhomSneiberg} 
        (\pm \partial_t - \Div_x(A(t,x)\nabla_x)+1)^{-1} \colon (\cV^{q',p'})' \to \cV^{q,p}.
    \end{equation} 
\end{prop}

\begin{definition}[Critical numbers $v_\pm(A)$]
\label{def:v(A)}
    Let $(v_-(A),v_+(A)) \subset (1,\infty)$ be the largest open interval for which \eqref{eq:homSneiberg} and \eqref{eq:inhomSneiberg} hold for all $q,p\in (v_-(A),v_+(A))$.
\end{definition}
 
Proposition \ref{prop:sneiberg} shows that this interval is not empty and contains $2$ as an interior point. By duality in \eqref{e:para-for-back-dual}, one has
\[ v_-(A^\ast) = v_+(A)', \quad v_+(A^\ast) = v_-(A)'. \]

\begin{proof}
    The inhomogeneous counterpart \eqref{eq:inhomSneiberg} has been established in \cite[Section 6.1]{AEN20} via Sneiberg's extrapolation lemma. See also the appendix of \cite{ABES19a} for a proof of the quantitative version of Sneiberg's lemma justifying the use of two parameters in the extrapolation. The same argument applies to the homogeneous result \eqref{eq:homSneiberg}, where we first pass in the target to the quotient spaces $(\dot\cV^{q,p}+\C)/\C$  which form  a scale of complex interpolation Banach spaces and then conclude by lifting. 
\end{proof}

\begin{remark} 
    In the case where $A$ is the identity matrix $\bfI$, we have $v_-(\bfI)=1$ and $v_+(\bfI)=\infty$. Indeed, it amounts to proving the $\Ltx{q}{p}$ boundedness of the Fourier multipliers $\frac{(|\tau|^2+|\xi|^4)^{1/2}}{\pm i\tau + |\xi|^2} $ and $\frac{(|\tau|^2+|\xi|^4+1)^{1/2}}{\pm i\tau + |\xi|^2+1}$, which again follows from the Mikhlin multiplier theorem. It would be interesting to investigate the precise values of these critical exponents when $A$ has some regularity, e.g., when $x \mapsto A(t,x)$ belongs to $VMO$, to connect with   \cite{Krylov2007-para-ell-VMO,Krylov2007-parabolic-VMO,Dong-Kim2011-system-BMO}. We leave that to further developments.
\end{remark}

\begin{cor}[Operator bounds]
    \label{cor:para-op-mixed-Leb}
    Let $q,p \in (v_-(A),v_+(A))$. Let $\ovq \in [q,\infty)$ and $\ovp \in [p,\infty)$ with $2[q,\ovq]+n[p,\ovp] = 1$. Let $\ulq \in (1,q]$ and $\ulp \in (1,p]$ with $2[\ulq,q]+n[\ulp,p]=1$. Then
    \begin{align*}
        \nabla_x (\pm \partial_t - \Div_x(A(t,x)\nabla_x))^{-1}\Div_x &\colon \Ltx{q}{p} \to \Ltx{q}{p} \\
        (\pm \partial_t - \Div_x(A(t,x)\nabla_x))^{-1}\Div_x &\colon \Ltx{q}{p} \to \Ltx{\ovq}{\ovp} \\
        \nabla_x (\pm \partial_t - \Div_x(A(t,x)\nabla_x))^{-1} &\colon \Ltx{\ulq}{\ulp} \to \Ltx{q}{p} \\
        (\pm \partial_t - \Div_x(A(t,x)\nabla_x))^{-1} &\colon \Ltx{\ulq}{\ulp} \to \Ltx{\ovq}{\ovp}.
    \end{align*}
    For the operators $(\pm \partial_t - \Div_x(A(t,x)\nabla_x)+1)^{-1}$, the same statements hold with the modifications $2[q,\ovq]+n[p,\ovp] \le 1$ and $2[\ulq,q]+n[\ulp,p] \le 1$.
\end{cor}

\begin{proof} 
    To prove the first boundedness, we first infer from Proposition \ref{prop:Vrho,p}\ref{item:Vrho,p-eq-norms} that $L^{\vphantom{-1,p}q}_t \dot{H}^{-1,p}_{\vphantom{t}x}$ embeds into $(\dot\cV^{q',p'})'$ and $\dot\cV^{q,p}$ embeds into $L^{\vphantom{-1,p}q}_t \dot{H}^{1,p}_{\vphantom{t}x}$. Then it directly follows from \eqref{eq:homSneiberg}.
    Then the three others follow on applying the embeddings and duality, noting that the same argument applies to the adjoints $(\mp \partial_t - \Div_x(A^*(t,x)\nabla_x))^{-1}$ in parallel. The proof for the inhomogeneous case is similar. 
\end{proof}

\subsection{Local estimates}
Let us introduce some notation on exponents. For any $p \in (0,\infty]$, we define the Sobolev exponents $p_*,p^*,p^\sharp \in (0,\infty]$ by
\[  p_* := \frac{(n+2)p}{n+2+p}, \quad 
    p^* :=\begin{cases}
        \frac{(n+2)p}{n+2-p} & \text{ if } p<n+2, \\
        \infty & \text{ if } n+2 \le p \le \infty,
    \end{cases} \quad 
    p^\sharp := \begin{cases}
        \frac{np}{n+2-p} & \text{ if } p<n+2, \\
        \infty & \text{ if } n+2 \le p \le \infty.
    \end{cases}
\]
Note that $(n+2)[p_*,p]=1$ for all $p \in (0,\infty]$ and $(n+2)[p,p^*]=1$ when $p \le n+2$. One has $(p^\sharp)'= (p')^\sharp$ if $\tfrac{n+2}{n+1}\le p \le n+2$, and moreover, $p^\sharp < p_* < 1$ if $p<\tfrac{n+2}{n+1}$, $1 \le p_* \le p^\sharp \le 2$ if $\tfrac{n+2}{n+1}\le p\le 2$, $p\le p^\sharp<p^*$ if $2 \le p< n+2$, and $p^\sharp=p^*=\infty$ if $p \ge n+2$.

Let us recall two notions of solutions. Let $\Omega$ be an open subset of $\R^n$ and $-\infty\le a < b \le \infty$. Let $f$ and $F$ be distributions in $\ms{D}'((a,b) \times \Omega)$. A function $u$ is called a \emph{weak solution} to the equation $\partial_t u - \Div_x(A\nabla_x u) = f + \Div_x F$ if $u, \nabla_x u \in L^2_{\loc}((a,b)\times \Omega)$ and for any $\phi \in \Cc((a,b) \times \Omega)$,
\[ -\int_{(a,b) \times \Omega} u \cdot \partial_t \phi \dd x \dd t + \int_{(a,b) \times \Omega} A \nabla_x u \cdot \nabla_x \phi  \dd x\dd t= (f,\phi) - (F,\nabla_x \phi), \]
where $(\cdot,\cdot)$ denotes the dual pairing between distributions and test functions on $(a,b) \times \Omega$. A \emph{very weak solution} is defined similarly but with only  $u, \nabla_x u\in L^1_{\loc}((a,b)\times \Omega)$. The same notions also apply to the backward equation $-\partial_t u - \Div_x(A\nabla_x u) = f + \Div_x F$.

We state bounds on weak solutions and their gradients. The higher integrability of the gradient was first established in \cite{Giaquinta-Struwe1982}.  In addition, we also want to obtain a kind of reverse H\"older inequality of solutions with a prescribed range of exponents. To this end, we revisit the strategy of \cite[Theorem 8.1]{Auscher-Bortz-Egert-Saari2019-parabolicL2} incorporating our embeddings, in particular to obtain $p$-Caccioppoli inequalities and reverse H\"older inequalities that will simplify our application to boundedness on tent spaces. 

\begin{prop}[$p$-Caccioppoli inequalities and reverse H\"older inequalities]
    \label{prop:rhi} 
    Let $I \times B$ with $|I| \eqsim r(B)^2$ be a parabolic cylinder and $u \in L^2(\gamma^2 I; H^1(\gamma B))$ be a weak solution to $\partial_t u - \Div_x(A(t,x)\nabla_x u)=0$ on $\gamma^2 I \times \gamma B$ for some $\gamma>1$. Let $1<\delta<\gamma$.
    For $0 < p < v_+(A)$, we have
    \begin{equation}
        \label{eq:p-Caccioppoli1}
        \int_{I}\int_{B} |\nabla u(s,y)|^p \dd y \dd s \lesssim r(B)^{-p} \int_{\delta^2 I}\int_{\delta B} |u(s,y)|^p \dd y \dd s.
    \end{equation} 
    For $0< \ovp<\ v_+(A)^\sharp$ if $v_+(A)\le n+2$ and $0< \ovp \le \infty$ otherwise, we have
    \begin{equation}
        \label{eq:RHIu}
        \sup_{s \in I} \bigg(\fint_{B} |u(s,y)|^{\ovp} \dd y \bigg)^{1/\ovp}  \lesssim \bigg(\fint_{\delta^2 I} \fint_{\delta B} |u(s,y)|^2 \dd y \dd s\bigg)^{1/2},
    \end{equation}
    and 
    \begin{equation}
        \label{eq:RHIuq}
        \bigg(\fint_{I}\fint_{B} | u(s,y)|^{\ovp} \dd y\dd s\bigg)^{1/\ovp}  \lesssim  \fint_{\delta^2 I}\fint_{\delta B} |u(s,y)| \dd y \dd s
    \end{equation}
    with the canonical modification if $\ovp=\infty$. The implicit constants do not depend on $I, B$. The Euclidean balls on $\R^n$ can be replaced by cubes. The same statement holds for weak solutions of the backward equation $-\partial_t u - \Div_x(A(t,x)\nabla_x u)=0$.
\end{prop}

\begin{proof} 
Let us first prove the intermediate estimate
\begin{equation}
    \label{e:para-RH-intermediate}
    \| u\chi\|_{\cV^{p,p}} \lesssim \|u\tilde \chi\|_{L^2_{t,x}}
\end{equation}
for any $2\le p<v_+(A)$, assuming $I$  interval of length 1, $B$ ball of radius $r(B)\eqsim 1$, and $\chi,\tilde \chi$ two smooth functions with $\ind_{I \times B} \le \chi \le \tilde \chi \le \ind_{\gamma^2 I \times \gamma B}$, the implicit constant being independent of $I$ and $B$. To this end, let $\chi_i$, $i=1, 2, \ldots, 2k+2$, for some $k$ to be determined, be smooth functions supported in $\gamma_i^2 I \times \gamma_i B$ with $1<\gamma_i<\gamma$ and $\chi_{i+1}=1$ on this support. Write $v_i \coloneq \chi_i u$. Since $v_i$ is compactly supported, we may use H\"older's inequality freely. Then     
\begin{equation}
    \label{e:para-local-vi}
    \partial_t v_i - \Div_x(A(t,x)\nabla_x v_i) +v_i=  f_i+\Div_x F_{i}
\end{equation}
on $\R^{1+n}$, in the weak sense, with $f_i= (\partial_t\chi_i +\chi_i) v_{i+1} -A \nabla_x v_{i+1}\cdot \nabla_x \chi_i$ and $F_i=-A (v_{i+1}\nabla_x \chi_i)$ using the support assumptions. As $f_i \in L^{2_\ast}_{t,x}$, $F_i \in L^2_{t,x}$, and $v_i \in L^{2}_t H^1_{\vphantom{t}x}$, we invoke \cite[Proposition 3.1]{Auscher-Bortz-Egert-Saari2019-parabolicL2} to get $D_t^{1/2}v_i\in L^2_{t,x}$ and hence $v_i \in \cV^{2,2}$. Furthermore, we infer from existence and uniqueness of weak solutions to \eqref{e:para-local-vi} in $\cV^{2,2}$ (see \cite{AE23}) that
\[ v_i = (\partial_t - \Div_x(A(t,x)\nabla_x)+1)^{-1} (f_i+\Div_x F_{i}). \]
For $2 \le p < v_+(A)$, given $p_1 \in [2,p]$ with $(2+n)[p_1,p] = (2+n)[p',p_1'] \le 1$, the Sobolev embedding (cf. Proposition \ref{prop:Vrho,p}\ref{item:Vrho,p-Sobolev-embed}) yields that $\cV^{p_1,p_1} \hookrightarrow L^p_{t,x}$, $\cV^{p',p'} \hookrightarrow L^{p_1'}_{t,x}$, and hence $L^{p_1}_{t,x} \hookrightarrow (\cV^{p',p'})'$. We also get from Proposition \ref{prop:Vrho,p}\ref{item:Vrho,p-eq-norms} that $L^{\vphantom{-1,p}p}_t \dot{H}^{-1,p}_{\vphantom{t}x} \hookrightarrow (\cV^{p',p'})'$. So we have
\[ \|v_i\|_{\cV^{p,p}} \lesssim \|f_i\|_{L^{p_1}_{t,x}} + \|F_i\|_{L^p_{t,x}} \lesssim  \|\nabla_x v_{i+1}\|_{L^{p_1}_{t,x}} + \|v_{i+1}\|_{L^p_{t,x}}. \]
In the last inequality, we also used H\"older's inequality as $p_1 \le p$. Then as $p_1 \in [2,v_+(A))$, we use the Sobolev embedding $\cV^{p_1,p_1} \hookrightarrow L^p_{t,x}$ and the above inequality with $p=p_1$ to get
\[ \|v_{i+1}\|_{L^p_{t,x}} \lesssim \|v_{i+1}\|_{\cV^{p_1,p_1}} \lesssim  \|\nabla_x v_{i+2}\|_{L^{p_1}_{t,x}}+\|v_{i+2}\|_{L^{p_1}_{t,x}}. \]
Since $v_{i+1} = v_{i+2} \chi_{i+1}$, gathering these estimates gives us 
\[ \|v_i\|_{\cV^{p,p}} \lesssim  \|\nabla_x v_{i+2}\|_{L^{p_1}_{t,x}} + \|v_{i+2}\|_{L^{p_1}_{t,x}} \]
that we can bootstrap by iteration. In a finite number of iterations, one reaches $p_k=2$. So we obtain
\[ \|v_1\|_{\cV^{p,p}} \lesssim  \|\nabla_x v_{2k+1}\|_{L^2_{t,x}}+ \|v_{2k+1}\|_{L^2_{t,x}}  \lesssim \|v_{2k+2}\|_{L^2_{t,x}}. \]
In the last inequality, we use the classical Caccioppoli inequality. This proves the intermediate estimate \eqref{e:para-RH-intermediate}.

Now, we prove \eqref{eq:RHIu}. By parabolic rescaling, we may assume $|I|=1$. First consider the case $2 < v_+(A) \le n+2$. Let $0<\ovp<v_+(A)^\sharp$.  By H\"older's inequality, it suffices to prove \eqref{eq:RHIu} for $\ovp \ge 2$. Pick $p \in [2,v_+(A))$ such that $p \le \ovp < p^\sharp$. Then the estimate \eqref{eq:RHIu} follows from \eqref{e:para-RH-intermediate} and Sobolev's embedding as $u\chi \in \cV^{p,p}$ and $\cV^{p,p} \hookrightarrow \Ltx{\infty}{\ovp}$. The proof for the case $v_+(A)>n+2$ follows similarly, where the only modification is to assume first $\ovp\in (n+2,\infty]$ and  to take $p \in (n+2,v_+(A))$ with $p\le \ovp$  so that $\cV^{p,p} \hookrightarrow \Ltx{\infty}{\ovp}$. This proves \eqref{eq:RHIu}. 

To derive \eqref{eq:RHIuq}, the inequality \eqref{eq:RHIu} gives us an $L^{\ovp}-L^2$ reverse H\"older inequality for $u$. By rescaling we obtain it for all parabolic cylinders contained in $\gamma^2 I \times \gamma B$, and we may use well-known self-improvements to  $L^{\ovp}-L^r$ reverse H\"older inequalities for any $r<\ovp$, see \cite[Theorem~B.1]{Bernicot-Coulhon-Frey2016} for a simple proof.

To finish, let us prove \eqref{eq:p-Caccioppoli1}. For $p=2$, it is the classical Caccioppoli inequality obtained by using Lions' embedding theorem, see e.g. \cite[Proposition 3.6]{AMP19}. For $2<p<v_+(A)$, it follows from the intermediate estimate \eqref{e:para-RH-intermediate} via H\"older's inequality and parabolic rescaling. When $p<2$, we first apply H\"older's inequality on the left-hand side of \eqref{eq:p-Caccioppoli1} to raise $p$ to $2$ and then use 2-Caccioppoli inequality and   $L^2-L^p$ reverse H\"older inequalities for $u$.   This concludes the proof.
\end{proof}

\begin{remark}
    \label{rem:inhomo-Caccioppoli-RH}
    The above argument also extends to the inhomogeneous settings with source terms $f$ and $\Div_x F$, which will be used in Remark \ref{rem:tent-bdd-beta-sharp}. More precisely, let $f,F \in L^r_{\loc}(\gamma^2 I \times \gamma B)$ ($F$ being $\C^n$-valued) and $u \in L^2(\gamma^2 I; H^1(\gamma B))$ be a weak solution to $\pm \partial_t u - \Div_x(A\nabla_x u) = f + \Div_x F$. Then when $2 \le r < v_+(A)$, one has the $r$-inhomogeneous Caccioppoli inequality as
    \[ \left( \fint_I \fint_B |\nabla u|^r \right)^{1/r} \lesssim \frac{1}{r(B)} \left( \fint_{\delta^2 I} \fint_{\delta B} |u|^r \right)^{1/r} + r(B) \left( \fint_{\delta^2 I} \fint_{\delta B} |f|^r \right)^{1/r} + \left( \fint_{\delta^2 I} \fint_{\delta B} |F|^r \right)^{1/r}, \]
    which is also valid for $2_* \le r < v_+(A)$ when $F=0$. When $2 \le r < v_+(A)^\sharp$, one also has the $r$-inhomogeneous reverse H\"older inequality as
    \[ \sup_{s \in I} \left( \fint_B |u(s,y)|^r dy \right)^{1/r} \lesssim \fint_{\delta^2 I} \fint_{\delta B} |u| + r(B)^2 \left( \fint_{\delta^2 I} \fint_{\delta B} |f|^{r_\ast} \right)^{1/r_\ast} + r(B) \left( \fint_{\delta^2 I} \fint_{\delta B} |F|^r \right)^{1/r}. \] 
\end{remark}

\subsection{The fundamental solution}

The existence of fundamental solution of $\partial_t - \Div_x(A(t,x)\nabla_x)$ on the upper half-space, where $A$ does not necessarily have real coefficients, has been established in \cite{AMP19}. A universal full space construction was then developed in \cite{AE23} and  \cite{Auscher-Baadi2025-fundamental-sol} streamlined this construction to which we refer for the verifications of the following facts. See also \cite[Section 3.5]{Hou25} for this specific situation. The fundamental solution is a family of contractive operators on $L^2(\R^n)$, denoted by $\{\Gamma_A(t,s)\}_{-\infty<s \le t<\infty}$, with $\Gamma_A(t,t) = \id$. A particularly important property for us here is that for $f \in L^2_t\dot{H}^{-1}_{\vphantom{t}x}+\Ltx{1}{2}$, one has the representation
\begin{equation}
    \label{eq:representatioforward}
    (\partial_t - \Div_x(A(t,x)\nabla_x))^{-1}(f)(t)= \int_{-\infty}^t \Gamma_A(t,s)f(s)\, \dd s.
\end{equation}
The equality makes sense in $L^2(\R^n)$ for all $t \in \R$, where the integral converges weakly in $L^2(\R^n)$ for $f \in L^2_t\dot{H}^{-1}_{\vphantom{t}x}$ and strongly for $f \in \Ltx{1}{2}$. We also recall that the fundamental solution of $\partial_t - \Div_x(A(t,x)\nabla_x)+1$ is $e^{-(t-s)}\Gamma_A(t,s)$ and that the adjoint of $\Gamma_A(t,s)$ is the fundamental solution of $-\partial_t - \Div_x(A^*(t,x)\nabla_x)$, namely, 
\begin{equation}
    \label{eq:representatiobackwardward}
    (-\partial_t - \Div_x(A^*(t,x)\nabla_x))^{-1}(f)(s)= \int^{\infty}_s (\Gamma_A(t,s))^*f(t)\, \dd t.
\end{equation}

In particular, the representation allows one to see causality: if we assume that $f$ is supported on the upper half-space, then   $u:=(\partial_t - \Div_x(A(t,x)\nabla_x))^{-1}f$, vanishes in $L^2(\R^n)$ if $t\le 0$. Thus, one can consider the restriction of $(\partial_t - \Div_x(A(t,x)\nabla_x))^{-1}$ to functions supported on $\R^{1+n}_+$, and its fundamental solution is the same restricted to the range $0<s\le t<\infty$. One can also consider the restriction $\ind_{\R^{1+n}_+}(-\partial_t - \Div_x(A(t,x)\nabla_x))^{-1}\ind_{\R^{1+n}_+}$ for the backward operator  and observe that this is an anti-causal operator. We shall use this in the upcoming subsection.

Moreover, the fundamental solution identifies with the Green function. Specifically, for $\psi\in L^2(\R^n)$ and $s<t$, $\Gamma_A(t,s)\psi$ agrees with the evaluation at time $t$ of the unique  weak solution $u$ to the equation $\partial_t u - \Div_x(A(t,x)\nabla_x u)=0$ on $(s,\infty)\times \R^n$ with $u(s,\cdot)=\psi$ and $\nabla_x u \in L^2((s,\infty) \times \R^n)$. This identification is a consequence of the boundedness of the operator $(\partial_t - \Div_x(A^*(t,x)\nabla_x))^{-1}$ from $\Ltx{1}{2}$ to $C^{}_0 L^2_{\vphantom{0}x}$ observed above (see \cite[Theorem 6.25]{Auscher-Baadi2025-fundamental-sol}). The fundamental solution also satisfies the Chapman-Kolmogorov relation $\Gamma_A(t,s)=\Gamma_A(t,r)\Gamma_A(r,s)$ as bounded operators on $L^2(\R^n)$ when $s<r<t$. The same applies to the backward case.

The existence of $L^p$ bounds for $p\ne 2$ of $\Gamma_A(t,s)$ for arbitrary non-autonomous complex coefficients is first due to \cite{Z20}. It is also based on Sneiberg's extrapolation, and its proof furnishes also some off-diagonal bounds. There are many ways to recover them, see e.g.~\cite[Proposition 3.9]{Hou25}. Here we make clear the maximal range of exponents. 

\begin{definition}[Critical numbers $p_\pm(A)$]
\label{def:p(A)}
    Let $(p_-(A), p_+(A)) \subset (1,\infty)$ be the maximal open interval for which the family $\{\Gamma_A(t,s)\}_{-\infty<s \le t<\infty}$ is uniformly bounded in $L^p(\R^n)$.
\end{definition}
Note that $p_\pm(A^*)$ are the H\"older conjugates of $ p_\mp(A)$. In the autonomous case with $L=-\Div_x (A(x) \nabla_x)$, they are identical to $p_\pm(L)$ introduced in Section \ref{ssec:ODE-DB}, see Section \ref{ssec:parabolic-auto}.

\begin{lemma}[$t$-pointwise $L^{p_1}$-$L^{p_2}$ off-diagonal estimates]
\label{lem:ODEGammaA}
    Let $\max\{v_-(A)^\sharp, 1\} < p_1 \le p_2 < v_+(A)^\sharp$. There exists a constant $\eta>0$ so that for $-\infty< s<t<\infty$, Borel sets $E,F \subset \R^n$, and $f \in L^2 \cap L^{p_1}(\R^n)$,
    \begin{equation}
        \label{e:GammaA-t-ptwise-ODE}
        \left\| \ind_F \Gamma_A(t,s) (f\ind_E) \right\|_{L^{p_2}_x} \lesssim |t-s|^{-\frac{n}{2}[p_1,p_2]} \exp\left( -\eta \frac{d(E,F)^2}{t-s} \right) \|f\ind_E\|_{L^{p_1}_x}.
    \end{equation} 
    Consequently, $\Gamma_A(t,s)$ is uniformly bounded on $L^p(\R^n)$ for $\max\{v_-(A)^\sharp, 1\} < p < v_+(A)^\sharp$, so
    \[ p_-(A)\le \max\{v_-(A)^\sharp, 1\}, \quad p_+(A)\ge v_+(A)^\sharp. \]
    In particular, if $(\tfrac{n+2}{n+1}, n+2)\subset (v_-(A), v_+(A))$, then $(p_-(A),p_+(A))=(1,\infty)$. In fact, the estimate \eqref{e:GammaA-t-ptwise-ODE} holds for all $p_-(A) < p_1 \le p_2 < p_+(A)$.
\end{lemma}

\begin{proof}  
    The estimate \eqref{e:GammaA-t-ptwise-ODE} is known for $p_1=p_2=2$, again without further assumption on $A$, see e.g. \cite[Proposition 3.19]{AMP19}. Then \cite{Z20} shows that if \eqref{eq:RHIu} holds for some $q \in (2,\infty)$, then one has
    \[ \|\ind_{B(y,\sqrt{t-s})} \Gamma_A(t,s) f\|_{L^q_x} \lesssim |t-s|^{-\frac{n}{2}[2,q]}\|f\|_{L^2_x}, \quad y \in \R^n. \]
    In particular, we have verified \eqref{eq:RHIu} for $2<q<v_+(A)^\sharp$ by Proposition \ref{prop:rhi}, so by interpolation \eqref{e:GammaA-t-ptwise-ODE} with $p_1=p_2=2$ and covering arguments with balls imply \eqref{e:GammaA-t-ptwise-ODE} for $p_1=2<p_2<q$ and general $F$, and next for  $2<p_1\le p_2<q$. As $q$ is arbitrary, we obtain the desired range with $p_1\ge 2$. By duality from the same result for $\Gamma_A(t,s)^*$ associated to the backward adjoint equation, we obtain the estimate \eqref{e:GammaA-t-ptwise-ODE} for $\max\{v_-(A)^\sharp,1\} < p_1 \le p_2\le 2$. Finally, the Chapman-Kolmogorov relation and the composition of exponential off-diagonal estimates allow one to get the full range. The consequences for $p_\pm(A)$ follow from the definition and properties of the  ${}^\sharp$ rule. 
\end{proof}

Let us state a corollary that was not obtainable just using the above embeddings in Proposition \ref{prop:Vrho,p}.

\begin{cor} 
    For $p_-(A)<p<p_+(A)$, $(\pm\partial_t - \Div_x(A(t,x)\nabla_x))^{-1}\colon \Ltx{1}{p} \to C^{}_0L^p_{\vphantom{0}x}$ are bounded.
\end{cor}

\begin{proof} 
    Clearly, using the representation \eqref{eq:representatioforward} and the fundamental solution bound for $p$, we obtain the bound $\Ltx{1}{p} \to \Ltx{\infty}{p}$ by density. To see the bound into $C^{}_0L^p_{\vphantom{0}x}$, choose $q \in (p_-(A),p_+(A))$ so that $p\in [2,q)$ or $p\in (q,2]$. Note that the images of functions in $\Ltx{1}{q} \cap \Ltx{1}{2}$ are contained in $C^{}_0L^2_{\vphantom{0}x}\cap \Ltx{\infty}{q}$, and hence in $C^{}_0L^p_{\vphantom{0}x}$ by interpolation. Since $C^{}_0L^p_{\vphantom{0}x}$ is closed in $\Ltx{\infty}{p}$, the conclusion follows by density of $\Ltx{1}{q} \cap \Ltx{1}{2}$ into $\Ltx{1}{p}$. The argument for the backward operator is the same. 
\end{proof}

\begin{remark}[Sharpness of $p_\pm(A)$]
    A family of examples by Mooney \cite{Moo21} can be used to see that the inequality $p_+(A)>2$ is sharp (at least when $n\ge 2$) among all such operators when $A$ is bounded, elliptic, complex and non-autonomous, in stark contrast with the autonomous case (see Section \ref{ssec:parabolic-auto}). We thank Moritz Egert for explaining to us  this example, designed to disprove local boundedness of weak solutions to forward complex equations (the construction can be adapted to backward equations as well). Inspection of this construction also furnishes the conclusion that for any $p>2$ there is an equation such that $\Gamma_A(0,-1)$ can be unbounded on $L^p(\R^n)$. As $p_-(A)=p_+(A^*)'$, the conclusion $p_-(A)<2$ is also sharp. 
    When $A$ is real though it is known that the fundamental solutions are contractive on $L^p(\R^n)$ spaces for all $p\in [1,\infty]$, hence $p_+(A)=\infty$ and $p_-(A)=1$, see \cite{Ar68}.
\end{remark}

\subsection{Boundedness in tent spaces}
\label{ssec:lions-tent}
We define the following operators on $\R^{1+n}_+$: 
\begin{equation*}
    \begin{split}
        \cL^A_1 &:= \ind_{\R^{1+n}_+} (\partial_t - \Div_x(A(t,x)\nabla_x))^{-1} \ind_{\R^{1+n}_+}, \\
        \cB^A_1 &:= \ind_{\R^{1+n}_+} (-\partial_t - \Div_x(A(t,x)\nabla_x))^{-1} \ind_{\R^{1+n}_+}, \\
        \cR^A_{1/2} &:=\cL^A_1 \Div_x,
    \end{split}
    \qquad \qquad 
    \begin{split}
        \cL^A_{1/2} &:= \nabla_x \cL^A_1, \\
        \cB^A_{1/2} &:= \nabla_x \cB^A_1, \\
        \cR^A_0 &:= \nabla_x \cR^A_{1/2} = \nabla_x \cL^A_1 \Div_x.
    \end{split}
\end{equation*}

We call the first four the \emph{Duhamel operators}. The first two are related to solving Cauchy problems \eqref{e:NaPC} with $F=0$ and for $\cL^A_1$, there is a representation of Duhamel type \eqref{eq:representatioforward}. The next two correspond to backward Cauchy problems with zero final data and $F=0$ and $\cB^A_1$ enjoys the dual representation formula \eqref{eq:representatiobackwardward}. The last two are called the \emph{Lions operators} because they are related to solving Cauchy problems \eqref{e:NaPC} with $f=0$, first solved by Lions in the $L^2$ setting. Being associated to the PDE, these operators map $L^\infty_{\rmc}(\R^{1+n}_+)$ into $L^2_{\loc}(\R^{1+n}_+)$. Observe the duality relations
\begin{equation}
    \label{e:parabolic-op-duality}
    \cB^{A^*}_1 = (\cL^A_1)^*, \quad \cB^{A^*}_{1/2} = -(\cR^A_{1/2})^*
\end{equation}
in the sense that for all $f,g\in L^\infty_c(\R^{1+n}_+)$,
\[\langle \cB^{A^*}_1 f, g\rangle= \langle  f,  \cL^A_1 g\rangle, \quad \langle \cB^{A^*}_{1/2} f, g\rangle=- \langle  f,  \cR^A_{1/2} g\rangle, \]
where $\langle f,g \rangle = \int_{\R^{1+n}_+} f(t,x) \ovg(t,x) \dd t\dd x$. Here, we introduce the backward operators because we use duality arguments. Note that our methods also yield results for the backward Lions operators $\cB^A_1 \Div_x$ and $\cB^A_{1/2}\Div_x$. We leave the precise statements to interested readers.

We use our framework to prove boundedness of the Duhamel operators and the Lions operators in a range of weighted tent spaces. Our first theorem establishes boundedness of the forward Duhamel operators $\cL^A_1$ and $\cL^A_{1/2}$ in tent spaces.

\begin{theorem}[Boundedness of the forward Duhamel operators in tent spaces]
    \label{thm:Duhamel-tent}
    Let $p,q \in (0,\infty]$ and 
    \[ \beta>-[1,q]+\tfrac{n}{2}[\min\{p,q,p_-(A)\},p_-(A)]. \]
    \begin{enumerate}[label=\normalfont(\roman*)]
        \item \label{item:LA1-bd-tent}
        If $p_-(A)<r<p_+(A)$, then $\cL^A_1$ extends to a bounded operator from $T^{p,q,r}_\beta$ to $T^{p,q,r}_{\beta+1}$.

        \item \label{item:LA1/2-bd-tent}
        If $\max\{p_-(A),v_-(A)_\ast\}<r<v_+(A)$, then $\cL^A_{1/2}$ extends to a bounded operator from $T^{p,q,r}_\beta$ to $T^{p,q,r}_{\beta+1/2}$.
    \end{enumerate}
\end{theorem}

The next theorem is devoted to the backward Duhamel operators $\cB^A_1$ and $\cB^A_{1/2}$.
\begin{theorem}[Boundedness of the backward Duhamel operators in tent spaces]
    \label{thm:back-Duhamel-tent}
    Let $p,q \in (0,\infty]$.
    \begin{enumerate}[label=\normalfont(\roman*)]
        \item \label{item:BA1-bd-tent}
        If $p_-(A)<r<p_+(A)$ and
        \[ \beta < [q,\infty] - \tfrac{n}{2}{\bracb{p_+(A),\max\cbrace{p,q,p_+(A)}} } - 1, \]
        then $\cB^A_1$ extends to a bounded operator from $T^{p,q,r}_\beta$ to $T^{p,q,r}_{\beta+1}$.

        \item \label{item:BA1/2-bd-tent}
        If $\max\{p_-(A),v_-(A)_*\} < r < v_+(A)$ and
        \[ \beta < [q,v_+(A)] -\tfrac{n}{2}{\bracb{v_+(A),\max\cbrace{p,q,v_+(A)}} } - \tfrac{1}{2}, \]
        then $\cB^A_{1/2}$ extends to a bounded operator from $T^{p,q,r}_\beta$ to $T^{p,q,r}_{\beta+1/2}$.
    \end{enumerate}
\end{theorem}

Finally, we present the results for boundedness of the Lions operators $\cR^A_{1/2}$ and $\cR^A_{0}$  in tent spaces.

\begin{theorem}[Boundedness of the Lions operators in tent spaces]
    \label{thm:Lions-tent}  
    Let $p,q \in (0,\infty]$ and 
    \begin{equation}
        \label{e:Lions-beta-bound}
        \beta>-[v_-(A),q]+\tfrac{n}{2}\left[ \min\{p,q,v_-(A)\},v_-(A) \right].
    \end{equation}
    \begin{enumerate}[label=\normalfont(\roman*)]
        \item \label{item:RA1/2-bd-tent}
        If $v_-(A)<r<\min\{p_+(A),v_+(A)^\ast\}$, then $\cR^A_{1/2}$ extends to a bounded operator from $T^{p,q,r}_\beta$ to $T^{p,q,r}_{\beta+1/2}$.

        \item \label{item:RA0-bd-tent}
        If $v_-(A) < r < v_+(A)$, then $\cR^A_{0}$ extends to a bounded operator from $T^{p,q,r}_\beta$ to $T^{p,q,r}_{\beta}$.
    \end{enumerate}
\end{theorem}

\begin{remark}
    In the special case $p=q=r$, since $T^{p,q,r}_\beta$ identifies with $L^r_\beta(\R^{1+n}_+)$, our theorems imply boundedness of the Duhamel and Lions operators from $L^r_\beta(\R^{1+n}_+)$ to $L^r_{\beta+\kappa}(\R^{1+n}_+)$. Aside from the case of $\cL^A_1$, we do not know a direct proof of such weighted Lebesgue bounds.
\end{remark}

\begin{remark}
    \label{rem:tent-bdd-beta-sharp}
    These are the bounds one can obtain using the method of this article from the lemmas below. However, some other techniques such as local regularity of weak solutions, atomic decomposition of tent spaces, duality and interpolation, may provide further bounds. We just describe one here without proof as this would go beyond the scope of this article.
    
    For instance, the bounds of $\beta$ in Theorems \ref{thm:back-Duhamel-tent}\ref{item:BA1/2-bd-tent} for $\cB^A_{1/2}$ and \ref{thm:Lions-tent} for $\cR^A_{1/2}$ and $\cR^A_0$ may be improved. When $q=r=2$, it has been shown in \cite[Theorem 5.1]{Hou25} that $\cR^A_{1/2}$ (resp. $\cR^A_0$) is also bounded from $T^{p,2,2}_\beta$ to $T^{p,2,2}_{\beta+1/2}$ (resp. $T^{p,2,2}_\beta$) for $\beta>\frac{n}{2}[\min\{p,p_-(A)\},p_-(A)]$. As $p_-(A) \le \max\{v_-(A)^\sharp,1\}$, this range is not fully covered by \eqref{e:Lions-beta-bound}. The argument there involves an inhomogeneous version of Caccioppoli's inequality (for the adjoint operator $\cB^{A^*}_{1/2}$) when $p>1$ and atomic decomposition of the tent spaces $T^{p,2,2}_\beta$ when $p \le 1$. 
    
    Inspired by this result, we conjecture that for any $p,q \in (0,\infty]$, $v_-(A)<r<v_+(A)$, and 
    \begin{equation}
        \label{e:Lions-beta-conj}
        \beta>-[2,q]+\tfrac{n}{2} \left[ \min\{p,q,p_-(A)\},p_-(A) \right],
    \end{equation}
    the Lions operator $\cR^A_{1/2}$ (resp. $\cR^A_0$) is bounded from $T^{p,q,r}_{\beta}$ to $T^{p,q,r}_{\beta+1/2}$ (resp. $T^{p,q,r}_{\beta}$). If the conjecture holds, then by interpolation, the lower bound of $\beta$ in \eqref{e:Lions-beta-bound} can be improved as
    \begin{equation}
        \label{e:Lions-beta-conj-full}
        \begin{cases}
        -[v_-(A),q]+\frac{n}{2} \left[ \min\{q,v_-(A)\},v_-(A) \right]  & \text{ if } p \ge v_-(A) \\
        -[2,q]+\frac{n}{2} \left[ \min\{p,q\},p_-(A) \right]  & \text{ if } p \le p_-(A) \\
        -[(v_-(A),2)_\theta,q] + \frac{n}{2}  \left[ \left(\min\{q,v_-(A)\},\min\{q,p_-(A)\} \right)_\theta, (v_-(A) ,p_-(A))_\theta \right] & \text{ if } p=(v_-(A),p_-(A))_\theta
        \end{cases},
    \end{equation}
    where $\theta \in (0,1)$ and $(p,q)_\theta$ denotes the real number such that $\frac{1}{(p,q)_\theta} = \frac{1-\theta}{p} + \frac{\theta}{q}$. As we shall see in Section \ref{ssec:parabolic-auto}, in the autonomous case, since $v_-(A)=q_+(L^\ast)'$, this new lower bound agrees with that in \cite{AH25b} for $q=r=2$.

    When $p,q \in (1,\infty]$, one can verify the conjecture \eqref{e:Lions-beta-conj} by adapting the argument in \cite[Lemma 5.5]{Hou25} with the inhomogeneous Caccioppoli inequality and reverse H\"older inequality established in Remark \ref{rem:inhomo-Caccioppoli-RH}. But in the other cases, for the moment, we do not know the atomic decomposition for Whitney-averaged tent spaces $T^{p,q,r}_\beta$ when $p \le 1$ except when $q=r \ge p$.

    To illustrate our conjecture on the bounds of $\beta$, we give graphic representations in Figure \ref{fig:Lions-beta-conj-full}. 
    \begin{figure}[ht]
        \centering
        \includegraphics[width=0.5\linewidth]{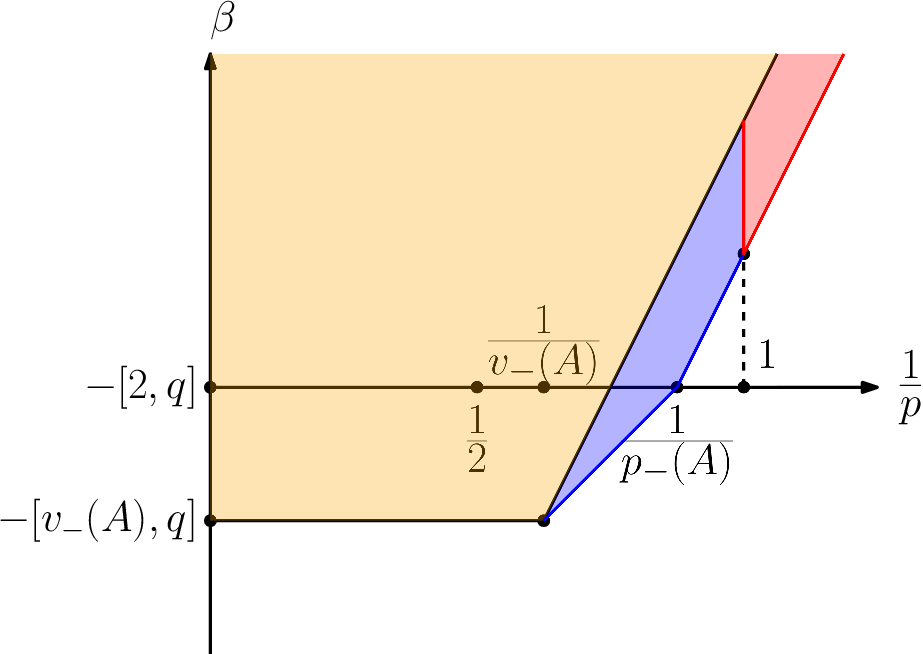}
        \caption{Illustration for the bounds of $\beta$ in \eqref{e:Lions-beta-conj-full}. The orange, blue, and red trapezoids represent, respectively, the range of $\beta$ in \eqref{e:Lions-beta-bound}, the range for which boundedness of the Lions operator is established in this remark, and the range for which such boundedness remains unproved.}
        \label{fig:Lions-beta-conj-full}
    \end{figure}
\end{remark}

Theorems~\ref{thm:Duhamel-tent},~\ref{thm:back-Duhamel-tent}, and~\ref{thm:Lions-tent} are direct consequences of Theorem~\ref{thm:mainestimate} via the following lemmas, which establish the singular operator properties.

\begin{lemma}[Forward Duhamel operators are singular operators]
    \label{lemma:Duhamel-SIO}
    ~
    \begin{enumerate}[label=\normalfont(\arabic*)]
        \item \label{item:LA1-SIO}
        $\cL^A_1$ is a singular operator of type $(1,2,r,\boldu,\boldv,M)$ when $\max\{\min\{p_-(A),v_-(A)_*\}, 1\}<r<\max\{p_+(A),v_+(A)^*\}$, $\boldu=(1,p_1)$, $\boldv = (\infty,p_2)$ with $p_-(A) < p_1 \le p_2 < p_+(A)$, and $M \ge 0$. 
        
        \item \label{item:LA1/2-SIO}
        $\cL^A_{1/2}$ is a singular operator of type $(1/2,2,r,\boldu,\boldv,M)$ when $\max\{v_-(A)_*, 1\}<r<v_+(A)$, $\boldu = (1,p_1)$, $\boldv = (p_2,p_2)$ with $p_-(A) <  p_1 \le p_2 < v_+(A)$, and $M \ge 0$.
    \end{enumerate}
\end{lemma}

\begin{lemma}[Backward Duhamel operators are singular operators]
    \label{lemma:back-Duhamel-SIO}
    ~
    \begin{enumerate}[label=\normalfont(\arabic*)]
        \item \label{item:BA1-SIO}
        $\cB^A_1$ is a singular operator of type $(1,2,r,\boldu,\boldv,M)$ when $\max\{\min\{p_-(A),v_-(A)_*\}, 1\}<r<\max\{p_+(A),v_+(A)^*\}$, $\boldu = (1,p_1)$, $\boldv = (\infty,p_2)$ with $p_-(A) < p_1 \le p_2 < p_+(A)$, and $M \ge 0$.

        \item \label{item:BA1/2-SIO}
        $\cB^A_{1/2}$ is a singular operator of type $(1/2,2,r,\boldu,\boldv,M)$ when $\max\{v_-(A)_*, 1\}<r<v_+(A)$, $\boldu = (1,p_1)$, $\boldv = (p_2,p_2)$ with $p_-(A) < p_1  \le p_2 < v_+(A)$, and $M\ge 0$.
    \end{enumerate}
\end{lemma}

\begin{lemma}[Lions operators are singular operators]
    \label{lemma:Lions-SIO}
    ~
    \begin{enumerate}[label=\normalfont(\arabic*)]
        \item \label{item:RA1/2-SIO}
        $\cR^A_{1/2}$ is a singular operator of type $(1/2,2,r,\boldu,\boldv,M)$ when $v_-(A)<r<v_+(A)^*$, $\boldu = (p_1,p_1)$ and $\boldv = (\infty,p_2)$ with $v_-(A) < p_1 \le  p_2 < p_+(A)$, and $M\ge 0$.
        
        \item \label{item:RA0-SIO}
        $\cR^A_0$ is a singular operator of type $(0,2,r,\boldu,\boldv,M)$ when $v_-(A)<r<v_+(A)$, $\boldu = (p_1,p_1)$ and $\boldv = (p_2,p_2)$ with $v_-(A) < p_1  \le p_2 < v_+(A)$, and $M\ge 0$.
    \end{enumerate}
\end{lemma}

We make a comment on the above statements. Note that the backward Duhamel operators $\cB^A_1$ and $\cB^A_{1/2}$ in Theorem~\ref{thm:back-Duhamel-tent} are anti-causal operators. So the upper bound on $\beta$ there relies on the parameter $\boldv$ and $\kappa$. This is why the upper bounds on $\beta$ differ for $\cB^A_1$ and $\cB^A_{1/2}$, in contrast with the lower bound on $\beta$ for the causal operators $\cL^A_1$ and $\cL^A_{1/2}$, which depends only on $\boldu$ and hence is the same for both.

\begin{remark}
\label{rem:conjecture}
    Recall that the conditions on $r$ in these theorems come from the fact that the statement of Theorem~\ref{thm:mainestimate} imposes $p_1 \le r \le p_2$. For example, we conjecture that the lower bound for $r$ in Theorem~\ref{thm:Duhamel-tent}\ref{item:LA1/2-bd-tent} should be $p_-(A)$ but we have not succeeded to prove it. It would amount to improve the lower bound on $r$ in Lemma~\ref{lemma:Duhamel-SIO}\ref{item:LA1/2-SIO} to $\max\{\min\{p_-(A),v_-(A)_*\},1\}$ exactly as for $\cL^A_1$. In the autonomous case, we give an affirmative answer to this conjecture by using the properties of the analytic semigroup generated by $-\Div(A(x)\nabla)$, see Theorem \ref{thm:parabolic-autonomous}.
\end{remark}

We next prove the lemmas in this order. 

\begin{proof}[Proof of Lemma \ref{lemma:Duhamel-SIO}]
Let us start with $\cL^A_1$ in \ref{item:LA1-SIO}. It suffices to show that $\cL^A_1$ is a singular integral operator of type $(1,2,r,p_1,p_2,M)$ for any $M \ge 0$ in the sense of Definition \ref{def:kernelODE}, thanks to Proposition \ref{prop:compareAH} and the requirements on $r, p_1, p_2$. The kernel representation \eqref{eq:representation} has been verified in \eqref{eq:representatioforward}.

Let us first verify condition \ref{it:Lr}. First, assume $p_-(A)<r<p_+(A)$. The condition \ref{it:Lr} with $\gamma=0$ is equivalent to boundedness of $\cL^A_1$ from $L^r(\R^{1+n}_+)$ to $L^r_1(\R^{1+n}_+)$. Let $t>0$ and $f \in L^r(\R^{1+n}_+)$. Since $\Gamma_A(t,s)$ are uniformly bounded in $L^r(\R^n)$ for $p_-(A)<r<p_+(A)$, we get
\[ \|t^{-1} \cL^A_1(f)(t)\|_{\Lx{r}} \lesssim t^{-1} \int_0^t \|f(s)\|_{\Lx{r}} \dd s = H(g)(t), \]
where $H$ is the Hardy operator and $g(s)=\|f(s)\|_{\Lx{r}}$. We conclude from the boundedness of $H$ on $L^r(\R_+)$.

Next, assume $p_+(A)\le r<v_+(A)^*$. This case may only occur when $v_+(A)<n+2$ (i.e., $v_+(A)^*<\infty$), which we assume. Thanks to Remark \ref{rem:pointwiseimmliesnonpointwise}, we may instead  check condition \ref{it:ODEDef1} in Definition \ref{def:ODE} directly. To see this, Corollary \ref{cor:para-op-mixed-Leb} asserts that $\cL^A_1$ is bounded from $L^{q_*}_{t,x}$ to $L^{q^*}_{t,x}$ for any $q \in (v_-(A),v_+(A))$. Such $q$ can be chosen arbitrarily close to $v_+(A)$ so that $1<q_*\le r\le q^*<\infty$, since the interval $(\max\{v_-(A)_*, 1\},v_+(A)^*)$ contains ``double Sobolev intervals" with endpoints $q_*, q^*$. Indeed, one has
\begin{align*}
    (n+2)[\max\{v_-(A)_*, 1\},v_+(A)^*] &
    = (n+2)[\max\{v_-(A)_*, 1\},v_+(A)] +1 \\&
    >(n+2)[v_+(A)_*,v_+(A)] + 1 = 2.
\end{align*}
Therefore, applying Proposition \ref{prop:boundedimplySOwithM=0} with $p_t=p_x=q_*$, $q_t=q_x=q^*$, and $u_t=u_x=v_t=v_x=r$ yields that $\cL^A_1$ satisfies condition \ref{it:ODEDef1}.

In the case where $\max\{v_-(A)_*, 1\}<r \le p_-(A)$, which may occur only when $v_-(A)_*>1$, we have
\begin{align*}
    (n+2)[\max\{v_-(A)_*, 1\},v_+(A)^*] &=1+  (n+2)[v_-(A),v_+(A)^*]  \\&>1+ (n+2)[v_-(A), v_-(A)^*]  = 2,
\end{align*}
and one can choose  $q \in (v_-(A),v_+(A))$ arbitrarily close to $v_-(A)$ such that $1<q_*\le r \le q^*<\infty$. We conclude with the same reasoning and this ends the verifications for condition \ref{it:Lr}.

Next, thanks to Remark \ref{rem:KernelODE}\ref{it:kernelODEAH}, we get the desired estimates in conditions \ref{item:kernelODE-t-sep} and \ref{item:kernelODE-x-sep} for $\kappa=1$, $m=2$, $u_x=p_1$, $v_x=p_2$, and arbitrary $M \ge 0$ from \eqref{e:GammaA-t-ptwise-ODE}. This is independent of the choice of $r$. This proves \ref{item:LA1-SIO}.

\medskip

We continue the proof and consider $\cL^A_{1/2}$ in \ref{item:LA1/2-SIO}. 
We verify the four configurations in Definition \ref{def:ODE} for the sets $\boldE= E_t\times E_x$, $\boldF=F_t\times F_x$ with
\[ E_t = (a_0,b_0), \quad E_x=Q(x_0,r_0), \quad F_t = (a_1,b_1), \quad F_x=Q(x_1,r_1), \]
where $Q(x,r)$ denotes a cube centred at $x$ with sidelength $2r$. Set the geometric variables $d_t := d(E_t,F_t)$, $d_x := d(E_x,F_x)$, $D_t := \diam(E_t\cup F_t)$, and recall that $\ell_x = \diam(E_x\cup F_x)$ and $\ell_t = \sup(E_t\cup F_t)$. Fix $f \in L^\infty_c(\R^{1+n}_+)$.

\medskip

\emph{Configuration \ref{it:ODEDef1}: local boundedness on Whitney-type regions}.
It suffices to show that there exists $q \in (v_-(A), v_+(A))$ so that $1<q_*\le r\le q$. Indeed, thanks to Proposition \ref{prop:boundedimplySOwithM=0}, the desired estimate follows from boundedness of $\cL^A_{1/2}$ from $L^{q_*}_{t,x}$ to $L^{q}_{t,x}$ by Corollary \ref{cor:para-op-mixed-Leb}. To this end, we notice that the interval $(\max\{v_-(A)_*, 1\},v_+(A))$ contains ``Sobolev intervals'' with endpoints $q_*,q$ since
$$(n+2)[\max\{v_-(A)_*, 1\},v_+(A)]>(n+2)[2_*,2]=  1.$$ 
As $\max\{v_-(A)_*, 1\}<r<v_+(A)$, one easily finds the required $q$.

\medskip

\emph{Configuration \ref{it:ODEDef2}: $t$-separation}.
Assume $c\ell_x \le \ell_t^{1/2} \le \ell_x$ and $c^2 \ell_t \le d_t$. Note that if $b_1<a_0$, causality implies $\cL^A_1(f \ind_{\boldE})=0$ on $\boldF$. We assume now   $a_0\le b_1$. Then as $d_t>0$, we have $a_1-b_0>0$, and hence $a_1\ge a_1-b_0=d_t$ and also $\ell_t=b_1$. As $\boldF$ is not necessarily a parabolic cylinder, we first take parabolic cylinder enlargements $\boldF \subset \boldF_1 \subset \boldtilF$ so that $\boldtilF$ and $\boldE$ are also $t$-separated. Note that $E_x \cup F_x \subset Q(x_1,\ell_x)$, so we take
\[  \boldF_1 := (a_1-\gamma^2 d_t,b_1) \times Q(x_1,\ell_x+\gamma d_t^{1/2}), \quad \boldtilF := (a_1-4\gamma^2 d_t,b_1+4\gamma^2 d_t) \times Q(x_1,\ell_x + 2\gamma d_t^{1/2}), \]
where $\gamma\in (0,1/4)$ is a constant to be determined. The inclusions are immediate. Observe that $\boldtilF$ and $\boldE$ satisfy the $t$-separation configuration with the parameter $c/(4\sqrt{n})$ replacing $c$, if $\gamma$ is small enough. Indeed, the new geometric variables satisfy $\tilde\ell_t=\ell_t+4\gamma^2 d_t$, $\tilde d_t=d_t-4\gamma^2 d_t$ and $2\ell_x +4\gamma d_t^{1/2} \le \tilde \ell_x\le \sqrt{n}(2 \ell_x + 4\gamma d_t^{1/2})$ from comparison of Euclidean norm with the sup norm. Since $u=\cL^A_1(f \ind_{\boldE})$ is a weak solution to $\partial_t u - \Div(A\nabla u) = 0$ on $\boldtilF$, as $p_2 < v_+(A)$, we apply $p_2$-Caccioppoli inequality \eqref{eq:p-Caccioppoli1} to $u$ from $\boldF_1$ with $R\eqsim d_t^{1/2}$  to get 
\begin{align*}
    \left\| \ind_{\boldF} \cL^A_{1/2}(f \ind_{\boldE}) \right\|_{L^{p_2}_{t,x}}
    &\le \left\| \ind_{\boldF_1} \cL^A_{1/2}(f \ind_{\boldE}) \right\|_{L^{p_2}_{t,x}} \\&\lesssim d_t^{-\frac{1}{2}} \left\| \ind_{\boldtilF} \cL^A_{\vphantom{1/2}1}(f \ind_{\boldE}) \right\|_{L^{p_2}_{t,x}} 
    \lesssim d_t^{-[1,p_2]-\frac{n}{2}[p_1,p_2]+\frac{1}{2}} \left\| f \ind_{\boldE} \right\|_{ \Ltx{1}{p_1} }.
\end{align*}
In the last inequality, we used the facts that $\tilde d_t \eqsim d_t$ and $\cL^A_1$ is a singular operator of type $(1,2,2,(1,p_1),(p_2,p_2),M)$, which follows from Lemma \ref{lemma:Duhamel-SIO}\ref{item:LA1-SIO} and contraction of parameters in Proposition \ref{prop:ODE-contraction}, as $p_-(A)<p_1 \le p_2<p_+(A)$.

\medskip

\emph{Configuration \ref{it:ODEDef3}: $t$-$x$-separation}.
Assume $\ell_t^{1/2} \le \ell_x \le c^{-1} d_x$ and $c^2 \ell_t \le d_t$. Take the parabolic cylinder enlargements $\boldF \subset \boldF_1 \subset \boldtilF$ with 
\[ \boldF_1 := (a_1-\gamma^2 d_t, b_1) \times Q(x_1,r_1+\gamma d_t^{1/2}), \quad \boldtilF := (a_1-4\gamma^2 d_t, b_1+4\gamma^2 d_t) \times Q(x_1,r_1+2\gamma d_t^{1/2}), \]
where $\gamma \in (0,c/(2\sqrt n))$ is a constant to be chosen small enough so  that $\boldtilF$ and $\boldE$ satisfy the $t$-$x$-separation configuration with the parameter $c/2$ replacing $c$. Indeed, the new geometric variables satisfy $\tilde\ell_t=\ell_t +4\gamma^2 d_t $, $\tilde d_t=d_t-4\gamma^2 d_t$, $\ell_x +2\gamma d_t^{1/2} \le \tilde \ell_x\le \ell_x+2\sqrt n\gamma d_t^{1/2}$ and $d_x-2\sqrt n \gamma d_t^{1/2}\le \tilde d_x \le d_x-2 \gamma d_t^{1/2}$ following from $d_x\ge cd_t^{1/2}$ and comparison of Euclidean norm with the sup norm. The same reasoning as above by using $p_2$-Caccioppoli inequality \eqref{eq:p-Caccioppoli1} on $\boldF_1$ with $R\eqsim d_t^{1/2}$ gives us
\begin{align*}
    \left\| \ind_{\boldF} \cL^A_{1/2} (f \ind_{\boldE}) \right\|_{L^{p_2}_{t,x}} 
    &\lesssim d_t^{ -\frac{1}{2} } \left\| \ind_{\boldtilF} \cL^A_1(f \ind_{\boldE}) \right\|_{L^{p_2}_{t,x}} 
    \\&\lesssim D_t^{ -\frac{1}{2} } d_x^{-2[1,p_2]-n[p_1,p_2]+2} \left( \frac{D_t}{d_x^2} \right)^{M+\frac{1}{2}} \left\| f \ind_{\boldE} \right\|_{ \Ltx{1}{p_1} } 
    \\
    &\lesssim d_x^{-2[1,p_2]-n[p_1,p_2]+1} \left( \frac{D_t}{d_x^2} \right)^{M} \|f \ind_{\boldE}\|_{ \Ltx{1}{p_1} }.
\end{align*}
In the second inequality, we used $d_t \eqsim \ell_t \eqsim D_t$, and the fact that $\cL^A_1$ is a singular operator of type $(1,2,2,(1,p_1),(p_2,p_2),M+\frac{1}{2})$.

\medskip

\emph{Configuration \ref{it:ODEDef4}: $x$-separation.}
Assume $\ell_t^{1/2} \le \ell_x \le c^{-1} d_x$. As $D_t$ may be much smaller than $r_1^2$, we proceed by covering $\boldF$ by a tiling of small parabolic cylinders of spatial size comparable to $D_t^{1/2}$ in order to apply \eqref{eq:p-Caccioppoli1}. First consider the case $r \le p_2$. H\"older's inequality yields that
\[ \left\| \ind_{\boldF} \cL^A_{1/2} (f \ind_{\boldE}) \right\|_{ \Ltx{r}{p_2} } \le |F_t|^{[r,p_2]} \left\| \ind_{\boldF} \cL^A_{1/2} (f \ind_{\boldE}) \right\|_{ L^{p_2}_{t,x} } \le D_t^{[r,p_2]} \left\| \ind_{\boldF} \cL^A_{1/2} (f \ind_{\boldE}) \right\|_{ L^{p_2}_{t,x} }. \]
Cover $\boldF$ by disjoint (up to a null set) parabolic cylinders $\boldC_{i,j} = I_i \times Q_j$ with $\sup(I_i) \le b_1$ for all $i$, $|I_i| = \gamma^2 D_t$, and $\ell(Q_j) = \gamma D_t^{1/2}$, where $\gamma \in (0,c/(2\sqrt n))$ is a constant to be chosen. Observe that the enlarged cylinders $\boldtilC_{i,j} = 4I_i \times 2Q_j$ have bounded overlap and their intersections with $\R^{1+n}_+$ are contained in
\[ \boldtilF = \left( \max\{0,a_1-4\gamma^2 D_t\}, b_1+4\gamma^2 D_t \right) \times Q(x_1,r_1+2\gamma D_t^{1/2}). \]
Observe also that if $\gamma$ is small enough, then $\boldtilF$ and $\boldE$ satisfy the $x$-separation configuration with the parameter $c/2$ replacing $c$. Indeed, the geometric variables satisfy $ \tilde \ell_t \le \ell_t + 4\gamma^2 D_t$, $\ell_x +2\gamma D_t^{1/2} \le \tilde \ell_x \le \ell_x+2\sqrt n\gamma D_t^{1/2}$ and $d_x-2\sqrt n \gamma D_t^{1/2}\le \tilde d_x \le d_x$. 
Moreover, we have $D_t \le \tilde D_t \le D_t+4\gamma^2 D_t$. Now we apply Minkowski inequality, the $p_2$-Caccioppoli inequalities \eqref{eq:p-Caccioppoli1} and the $L^{p_2}-L^1$ reverse H\"older inequalities \eqref{eq:RHIuq} on these small parabolic cylinders $\boldC_{i,j}$ to get
\begin{align*}
    \left\| \ind_{\boldF} \cL^A_{1/2} (f \ind_{\boldE}) \right\|_{ L^{p_2}_{t,x} } 
    &\lesssim \sum_{i,j} D_t^{(1+\frac{n}{2})\frac{1}{p_2}} \left( \fint_{\boldC_{i,j}} |\nabla u|^{p_2} \right)^{1/p_2} \lesssim D_t^{(1+\frac{n}{2})\frac{1}{p_2}-\frac12} \sum_{i,j} \fint_{\boldtilC_{i,j}} |u| \\
    &\lesssim D_t^{(1+\frac{n}{2})[p_2,1]-\frac12} \|\ind_{\boldtilF} \cL^A_1(f\ind_{\boldE}) \|_{L^1_{t,x}} \\
    &\lesssim D_t^{[p_2,r]+\frac{n}{2}[p_2,1]-\frac12} d_x^{n[1,p_1]} \|\ind_{\boldtilF} \cL^A_1(f\ind_{\boldE}) \|_{ \Ltx{r}{p_1} } \\
    &\lesssim D_t^{[p_2,r]+\frac{n}{2}[p_2,1]-\frac12} d_x^{n[1,p_1]} d_x^2 \left( \frac{D_t}{d_x^2} \right)^{M+[1,p_2]+\frac{n}{2}[1,p_2]+\frac{1}{2}} \left\|f \ind_{\boldE} \right\|_{ \Ltx{r}{p_1} } \\
    &\lesssim D_t^{[p_2,r]} d_x^{-n[p_1,p_2]+1} \left( \frac{D_t}{d_x^2} \right)^{M+[1,p_2]} \left\| f \ind_{\boldE} \right\|_{ \Ltx{r}{p_1} }.
\end{align*}
In the third-to-last and second-to-last inequalities, we used $\tilde D_t \eqsim D_t$, and
the fact that $\cL^A_1$ is a singular operator of type $(1, 2, r, (r,p_1), (r,p_1), M+[1,p_2]+\frac{n}{2}[1,p_2]+\frac{1}{2})$. Gathering these inequalities gives us the required estimate
\[ \left\| \ind_{\boldF} \cL^A_{1/2} (f \ind_{\boldE}) \right\|_{ \Ltx{r}{p_2} } \lesssim d_x^{-n[p_1,p_2]+1} \left( \frac{D_t}{d_x^2} \right)^{M+[1,p_2]} \left\| f \ind_{\boldE} \right\|_{ \Ltx{r}{p_1} }. \]

The case for $r>p_2$ follows similarly by using the $r$-Caccioppoli inequalities \eqref{eq:p-Caccioppoli1} and the $L^{r}-L^1$ reverse H\"older inequalities \eqref{eq:RHIuq} instead. Details are left to the reader. This completes the proof of \ref{item:LA1/2-SIO}.
\end{proof}

\begin{proof}[Proof of Lemma \ref{lemma:back-Duhamel-SIO}]
    Statement~\ref{item:BA1-SIO} follows from  Proposition~\ref{prop:duality} and Lemma~\ref{lemma:Duhamel-SIO}\ref{item:LA1-SIO}, thanks to the duality that $\cB^A_1 = (\cL^{A^\ast}_1)^\ast$ in \eqref{e:parabolic-op-duality} and $p_\pm(A) = p_\mp(A^\ast)'$. Also, the proof of Statement~\ref{item:BA1/2-SIO} is an adaptation of that of Lemma~\ref{lemma:Duhamel-SIO}\ref{item:LA1/2-SIO}, by using Corollary~\ref{cor:para-op-mixed-Leb}, Caccioppoli inequality~\eqref{eq:p-Caccioppoli1} and reverse H\"older inequality~\eqref{eq:RHIuq} for backward solutions on appropriately extended $\boldF$ regions. Details are left to the reader. 
\end{proof}

\begin{proof}[Proof of Lemma \ref{lemma:Lions-SIO}]
Statement \ref{item:RA1/2-SIO} for $\cR^A_{1/2}$ follows from the singular operator properties of $\cB^A_{1/2}$ in Lemma~\ref{lemma:back-Duhamel-SIO}\ref{item:BA1/2-SIO} by duality (see Proposition \ref{prop:duality}) and $\cR^A_{1/2} = -(\cB^{A^\ast}_{1/2})^\ast$
from \eqref{e:parabolic-op-duality}.

Next, we consider $\cR^A_0$ in \ref{item:RA0-SIO}. We also verify Definition~\ref{def:ODE} for the sets $\boldE = (a_0,b_0) \times B_0$ and $\boldF = (a_1,b_1) \times B_1$, with $B_0$ and $B_1$ cubes in $\R^n$ as before. Condition~\ref{it:ODEDef1} on local Whitney-type regions directly follows from the boundedness of $\cR^A_0$ in $L^r(\R^{1+n}_+)$ in Corollary \ref{cor:para-op-mixed-Leb}. The verifications of conditions \ref{it:ODEDef2}, \ref{it:ODEDef3}, and \ref{it:ODEDef4} are verbatim adaptations of those for $\cL^A_{1/2}$ in Lemma~\ref{lemma:Duhamel-SIO}\ref{item:LA1/2-SIO}, by using Caccioppoli and reverse H\"older inequalities together with the singular operator properties of $\cR^A_{1/2}$ just established. Details are left to the reader.
\end{proof}

\subsection{Applications to the Cauchy problem \texorpdfstring{\eqref{e:NaPC}}{(NaPC)}}
We apply the boundedness of these operators in tent spaces to the non-autonomous parabolic Cauchy problem \eqref{e:NaPC} with initial data in the homogeneous Triebel-Lizorkin spaces $\DotF^s_{p,q}$. Here, we also define homogeneous Triebel-Lizorkin spaces as (semi-)normed spaces in tempered distributions $\cS'(\R^n)$. The construction is similar to that of homogeneous Hardy-Sobolev spaces $\DotH^{s,p} \eqsim \DotF^s_{p,2}$  in \cite[Section 2.1]{AH25b} to which we refer. As a further preparation, we need the heat characterisation of the spaces $\DotF^s_{p,q}$.
\begin{prop}[Heat characterisation of $\DotF^s_{p,q}$]
    \label{prop:heat-char-Fspq} 
    Let $0<p,q,r \le \infty$. Let $s<0$. Then a tempered distribution $u_0 \in \cS'(\R^n)$ belongs to $\DotF^s_{p,q}$ if and only if its heat extension $(t,x) \mapsto (e^{t\Delta} u_0)(x)$ belongs to $T^{p,q,r}_{s/2+1/q}$. The equivalence of norms holds as
    \[ \|u_0\|_{\DotF^s_{p,q}} \eqsim \|e^{t\Delta} u_0\|_{T^{p,q,r}_{s/2+1/q}}. \]
\end{prop}

\begin{proof}
    It is a direct consequence of \cite[Proposition~2.6]{ABH26} (for $p<\infty$ and $p=q=\infty$) and \cite[Proposition~2.7]{Haa26} (for $p=\infty$ and $q<\infty$), after a change of variable as we deal with parabolic homogeneity $m=2$ and Lebesgue measure $\dd x \dd t$, whereas these references use  elliptic homogeneity $m=1$ and  invariant measure $\dd x \dd t/t$.
\end{proof}

\begin{cor}[Heat extension is an isomorphism]
    Let $-1<s<1$ and $\frac{n}{n+s+1} \le p \le \infty$. Let $0 < q,r \le \infty$. The heat extension $u_0 \mapsto (e^{t\Delta} u_0)(x)$ is an \emph{isomorphism} from $\DotF^s_{p,q}$ to the space of distributional heat solutions $u$ with $\nabla_x u \in T^{p,q,r}_{s/2-1/2+1/q}$. In particular, the equivalence holds that
    \[ \|u_0\|_{\DotF^{s}_{p,q}} \eqsim \|\nabla_x e^{t\Delta} u_0\|_{T^{p,q,r}_{s/2-1/2+1/q}}. \]
\end{cor}

\begin{proof}
    The equivalence follows from Proposition \ref{prop:heat-char-Fspq} as $\nabla_x$ commutes with the heat semi-group and $\|u_0\|_{\DotF^{s}_{p,q}} \eqsim \|\nabla_x u_0\|_{\DotF^{s-1}_{p,q}}$. We also infer that when $s<1$, the heat extension is injective from $\DotF^s_{p,q}$ to the space of distributional heat solutions $u$ with $\nabla_x u \in T^{p,q,r}_{s/2-1/2+1/q}$. 
    
    To prove the surjectivity, it suffices to show that any distributional heat solution $u$ with $\nabla_x u \in T^{p,q,r}_{s/2-1/2+1/q}$ can be represented as $u(t)=e^{t\Delta} u_0$ for some $u_0 \in \scrS'(\R^n)$. The proof is a verbatim adaptation of that in \cite[Proposition 3.5]{AH25b} for $q=r=2$, by using \cite[Theorem 1]{Auscher-Hou2024-Repheat}. Details are left to the reader.
\end{proof}

This will allow us to extend the fundamental solution $\Gamma_A(\cdot,0)$ to the Triebel-Lizorkin spaces $\DotF^s_{p,q}$ by its representation as a perturbation of the one from the Laplacian. Together with the Duhamel operators and the Lions operators, this extension gives a very weak solution to \eqref{e:NaPC}. For convenience, we parametrize our results by the parameter
\[ \beta = \tfrac{s}{2}+[q,2], \quad \text{i.e.,} \quad s=2(\beta-\tfrac{1}{q})+1. \]
The following theorem extends \cite[Theorem~1.7]{Hou25} to $q,r\ne 2$.

\begin{theorem}[Existence and uniqueness of solutions to \eqref{e:NaPC}]
    \label{thm:NAPC}
    Let $0<p,q \le \infty$ but $\max\{p,q\} \ne \infty$ when $\min\{p,q\} \le 1$. Let $v_-(A) < r < v_+(A)$. Let
    \[ \beta>-[2,q]+\frac{n}{2}\left[ \min\{p,q,v_-(A)\},v_-(A) \right]. \]
    \begin{enumerate}[label=\normalfont(\alph*)]
        \item \label{item:NAPC-ext-propagator}
        Let $\beta<1/q$. Then the fundamental solution map $u_0 \mapsto \Gamma_A(\cdot,0)u_0$ extends to a bounded operator from $\DotF^{2(\beta-1/q)+1}_{p,q}$ to $L^1_{\loc}(\R_+;W^{1,1}_{\loc}(\R^n))$ with
        \begin{equation}
            \label{e:NAPC-ext-propagator-est}
            \|\nabla_x \Gamma_A(t,0) u_0\|_{T^{p,q,r}_{\beta}} \lesssim \|u_0\|_{\DotF^{2(\beta-1/q)+1}_{p,q}}.
        \end{equation}
        
        \item \label{item:NAPC-sol}
        Let $F \in T^{p,q,r}_\beta$ and $f \in T^{p,q,r}_{\beta-1/2}$.
        \begin{enumerate}[label=\normalfont(\arabic*)]
            \item If $\beta<1/q$, then for any $u_0 \in \DotF^{2(\beta-1/q)+1}_{p,q}$, there exists a global very weak solution $u$ to \eqref{e:NaPC} on $(0,\infty) \times \R^n$, given by the explicit formula
            \[ u= \Gamma_A(\cdot, 0)u_0+ \cL^A_1(f) +\cR^A_{1/2}(F), \]
            such that $\nabla_x u \in T^{p,q,r}_\beta$ with the estimate
            \[ \|\nabla_x u\|_{T^{p,q,r}_\beta} \lesssim \|u_0\|_{\DotF^{2(\beta-1/q)+1}_{p,q}} + \|f\|_{T^{p,q,r}_{\beta-1/2}} + \|F\|_{T^{p,q,r}_\beta}. \]
            
            \item If $\beta \ge 1/q$, then the same statements also hold when $u_0$ is a constant.
        \end{enumerate}
        In both cases, we have $u \in C([0,\infty);\cS'(\R^n))$ with the initial condition $u(0)=u_0$.       
        
        \item \label{item:NAPC-unique}
        When $r\ge 2$, they are weak solutions and they are unique in the solution class $\nabla_x u \in T^{p,q,r}_\beta$.
    \end{enumerate}   
\end{theorem}

\begin{proof} 
First consider \ref{item:NAPC-ext-propagator}. We follow the method in \cite{AH25b,Hou25}. It has been shown there that for $u_0 \in L^2(\R^n)$, one has
\begin{equation*}
    \Gamma_A(t,0)u_0 = \Gamma_{\bfI}(t,0)u_0+ \cR^A_{1/2}((A-\bfI)\nabla_x \Gamma_{\bfI}(\cdot,0)u_0)(t), \quad t>0,
\end{equation*}
where $\Gamma_{\bfI}(t,0)$ agrees with the heat semigroup $e^{t\Delta}$. Then using Theorem \ref{thm:Lions-tent}\ref{item:RA0-bd-tent} for boundedness of $\cR^A_0 = \nabla_x \cR^A_{1/2}$ and then Proposition \ref{prop:heat-char-Fspq}, we get that for $u_0 \in \DotF^{2(\beta-1/q)+1}_{p,q} \cap L^2(\R^n)$,
\[ \|\nabla_x \Gamma_A(t,0)u_0\|_{T^{p,q,r}_{\beta}} \lesssim \|\nabla_x \Gamma_\bfI(t,0)(u_0)\|_{T^{p,q,r}_{\beta}} \lesssim \|u_0\|_{\DotF^{2(\beta-1/q)+1}_{p,q}}. \]
In particular, as $T^{p,q,r}_\beta$ embeds into $L^r_{\loc}(\R^{1+n}_+)$, we infer that $u \in L^1_{\loc}(\R_+;W^{1,1}_{\loc}(\R^n))$. Thus, the extension and the estimate \eqref{e:NAPC-ext-propagator-est} follow by a density argument (or weak*-density if either $p=\infty$ or $q=\infty$). This proves \ref{item:NAPC-ext-propagator}.

Next, for \ref{item:NAPC-sol}, that $u$ is a very weak solution is obtained by taking weak limits in the equation from good data (e.g. $u_0 \in \DotF^s_{p,q} \cap L^2(\R^n)$) and sources (e.g. $f,F \in L^\infty_c(\R^{1+n}_+)$), which are dense in the prescribed spaces, respectively (cf. Proposition \ref{prop:tentspacedensityduality}\ref{it:density}). This density argument is ensured by boundedness of the related operators, see Theorems \ref{thm:Duhamel-tent} and \ref{thm:Lions-tent} and \ref{thm:NAPC}\ref{item:NAPC-ext-propagator}. The continuity in $\scrS'(\R^n)$ follows from a verbatim adaptation of the reasoning in \cite{Hou25} for $q=r=2$, using the properties of tent spaces $T^{p,q,r}_\beta$ established in the recent works \cite{ABH26,Haa26}. Details are left to the reader.

Finally, consider \ref{item:NAPC-unique}. The condition $r\ge 2$ implies that $u$ and $\nabla_x u$ belong to $L^2_{\loc}(\R^{1+n}_+)$, so $u$ is a weak solution. For the uniqueness, we follow the reasoning in \cite[Theorem 7.1]{Hou25} by using an interior representation of $u$ established in \cite{AMP19}.
\end{proof}

\begin{remark}
    First, the result we obtain is partially extendable. Indeed, if we look at $f$ and $F$ in a separate fashion, there is a larger range of $\beta$ for which the existence of very weak solutions still holds, but we cannot verify uniqueness for the moment. Statements \ref{item:NAPC-ext-propagator} and \ref{item:NAPC-sol} still hold for $u_0 \in \DotF^{2(\beta-1/q)+1}_{p,q}$, $f \in T^{p_1,q,r}_{\beta_1}$, and $F \in T^{p,q,r}_\beta$, provided that
    \[ \begin{cases}
    0<p \le \infty, & \beta>-[v_-(A),q]+\frac{n}{2}\left[ \min\{p,q,v_-(A)\},v_-(A) \right], \\
    0<p_1 \le p, & \beta_1 > -[1,q]+\frac{n}{2}\left[ \min\{p_1,q,p_-(A)\},p_-(A) \right], \\
    \end{cases} \]
    with
    \[ 2\beta-\frac{n}{p}=2\beta_1+1-\frac{n}{p_1}. \] 
    When $q=r=2$, this formulation improves the existence part of \cite[Theorem 1.7]{Hou25}. For example, when $p=\infty$, we get existence of weak solutions for $-[v_-(A),2]<\beta<0$, which was not obtained there. 

    Moreover, if the bounds conjectured in Remark \ref{rem:tent-bdd-beta-sharp} hold, then one can readily prove the existence of very weak solutions in the corresponding range.
    
    However, we do not know how to prove uniqueness if $\beta \le -[2,q]+\frac{n}{2}\left[ \min\{p,q,v_-(A)\},v_-(A) \right]$ in our framework. When $q=r=2$, \cite[Theorem 1.7]{Hou25} includes a region below $\beta=0$ for which uniqueness holds, see \cite[Figure 1]{Hou25}, as a consequence of a Sneiberg argument implemented directly for the Cauchy problem and only in weighted Lebesgue spaces $L^2_\beta(\R^{1+n}_+)$ with $\beta$ near 0. Setting this Sneiberg argument in our context would require  showing that the operator $u \mapsto (\partial_t u-\Delta u ,\tr(u))$ is an isomorphism, where the solution class consists of functions $u$ such that $\partial_t u, \Delta u$ lie in the divergence of the weighted tent space $T^{p,q,r}_\beta$ and $\tr$ is a trace taking values in the Triebel-Lizorkin space. This is out of the scope of this article.
\end{remark}

\begin{remark}
    The uniqueness also holds in the solution class $u \in T^{p,q,r}_{\beta+1/2}$ under the same conditions as in Theorem~\ref{thm:NAPC}\ref{item:NAPC-unique}.
\end{remark}

\begin{remark}
    There are possible statements for the backward Lions operators $\cB^A_1 \Div$ and $\nabla \cB^A_1 \Div = \cB^A_{1/2} \Div$ and solutions to the backward equations with null final data at $t=\infty$. We leave precise formulations to interested readers.
\end{remark}

\subsection{Autonomous case}
\label{ssec:parabolic-auto}
Recall $L=- \Div_x(A(x)\nabla_x)$ was defined in \eqref{e:L} together with the critical numbers $p_\pm(L)$ and $q_\pm(L)$ in \eqref{eq:p(L)} and \eqref{eq:q(L)}. The numbers $p_\pm(A)$ and $v_\pm(A)$ when $A$ also depend on $t$ are in Definitions~\ref{def:p(A)} and~\ref{def:v(A)}, respectively. Recall that  the inequalities $p_-(L)=q_-(L)<\frac{2n}{n+2}$ and $p_+(L)>\frac{2n}{n-2}$ are sharp when $n\ge 3$, and $q_+(L)>2$ are sharp in all dimensions, see \cite[Remark 6.8(ii)]{Auscher-Egert2023-book}.

\begin{theorem}[Critical numbers in the autonomous case]
    Assume $A(t,x)=A(x)$ is independent of $t$. 
    \begin{enumerate}
        \item $(p_-(A), p_+(A))=(p_-(L), p_+(L))$.
        \item $(v_-(A),v_+(A)) = (q_+(L^*)',q_+(L))$.
    \end{enumerate}
\end{theorem}

\begin{proof}
    For the first point,  when $A$ does not depend on $t$, $\Gamma_A(t,s)= e^{-(t-s)L}$, the semi-group  generated by $-L$. Thanks to the  definition of the  exponents  $p_\pm(L)$ and $p_\pm(A)$, we must have $p_\pm(L)=p_\pm(A)$.

    For the second point, in contrast, the definitions of $q_+(L)$ and $v_+(A)$ do not look alike. The result follows from the next proposition and the definition of $v_\pm(A)$.
\end{proof}

\begin{prop}
    Let $p\in (1,\infty)$.
    \begin{enumerate}
        \item If $p\in (q_+(L^*)', q_+(L))$ and $1<q<\infty$, then $(\pm \partial_t - \Div_x(A(x)\nabla_x))^{-1} \colon (\dot\cV^{q',p'})' \to \dot\cV^{q,p}$. The same holds for the inhomogeneous operators and spaces. 
        \item If there exists  $q\in(1,\infty)$ such that $(\partial_t - \Div_x(A(x)\nabla_x))^{-1} \colon (\dot\cV^{q',p'})' \to \dot\cV^{q,p}$, then $p\in (q_+(L^*)', q_+(L))$.
    \end{enumerate}
    \end{prop}
    
\begin{proof}
Proof of (i). We only do the homogeneous case as the inhomogeneous case is similar. As usual we only look at the forward operator  $( \partial_t - \Div_x(A(x)\nabla_x))^{-1}$ which we write as $(\partial_t+L)^{-1}$. 
By norm equivalences  and density we need to show for
\(f\in \cS(\R^{1+n})\)  and
\(F\in \cS(\R^{1+n};\C^n)\): 
\[
\begin{aligned}
\|D_t^{1/2}(\partial_t+L)^{-1}(D_t^{1/2}f+\Div_x F)\|_{\Ltx{q}{p}}  +\|\nabla_x&(\partial_t+L)^{-1}(D_t^{1/2}f+\Div_x F)\|_{\Ltx{q}{p}}  \\
&\lesssim
\|f\|_{\Ltx{q}{p}}
+\|F\|_{\Ltx{q}{p}}.
\end{aligned}
\]

Using a Fourier transform in \(t\), we equivalently need to show that the following four operator-valued Fourier multipliers are bounded
on \(\Ltx{q}{p}\):
\begin{align*}
    |\tau|(i\tau+L)^{-1},\\
    |\tau|^{1/2}\nabla_x(i\tau+L)^{-1}&=\brb{\nabla_x L^{-1/2}} \circ \brb{
    |\tau|^{1/2}L^{1/2}(i\tau+L)^{-1}},\\
    |\tau|^{1/2}(i\tau+L)^{-1}\Div_x  &= \brb{|\tau|^{1/2}L^{1/2}(i\tau+L)^{-1}} \circ \brb{L^{-1/2}\Div},\\
    \nabla_x(i\tau+L)^{-1}\Div&=\brb{\nabla_x L^{-1/2} }\circ\brb{L(i\tau+L)^{-1}}\circ \brb{
    L^{-1/2}\Div_x}.
\end{align*}
As we assume $p \in (q_+(L^*)',q_+(L))$, then we can estimate all Riesz transforms $\nabla_x L^{-1/2}$ and $L^{-1/2}\Div_x= -(\nabla_x {L^*}^{-1/2})^*$ away because they are bounded on $L^p$ spaces and independent of $\tau$. That leaves us with Fourier multiplier operators with symbols
\[ m_\alpha(\tau):= {|\tau|(|\tau|^{-1} L)^\alpha} {(i\tau+L)}^{-1}=-i\sgn(\tau){(|\tau|^{-1} L)^\alpha} {(1+(i\tau)^{-1} L)}^{-1}  , \qquad \tau \in \R^*, \]
for $\alpha = 0,\frac12,1$. We can study $\tau>0$ and $\tau<0$ separately by the boundedness of the Fourier multipliers $\ind_{(0,\infty)}$, $\ind_{(-\infty,0)}$ on  $L^q(\R;X)$, where $X =L^p(\R^n)$ is a UMD space (see \cite[Proposition 5.3.10]{HNVW16}). 
Thus we can only look at $\tau>0$. Then both $m_\alpha$ and $\tau \mapsto \tau m'_\alpha(\tau)$ are  functions of the form $\varphi(\tau^{-1}L)$ for some $\varphi \in H^\infty$. Note that $L$ has a bounded $H^\infty$-calculus on $L^p(\R^n)$ as $p\in (p_-(L), p_+(L))$, see \cite[Section 5]{A07}, so by \cite[Theorem 10.3.4(3)]{HNVW17} we know that for any fixed $\varphi \in H^\infty$ on an appropriate sector,  we have that
$$
\cbrace{\varphi(tL), t>0}
$$
is $\mc{R}$-bounded on $L^p(\R^n)$.
Thus the required boundedness follows from the vector-valued Mikhlin multiplier theorem in the UMD space $L^p(\R^n)$ \cite{We01b}, see also \cite[Theorem 5.3.18]{HNVW16}.

Proof of (ii). By the assumed boundedness and the norm equivalences in Proposition~\ref{prop:Vrho,p},
the operator
\(
\nabla_x(\partial_t+L)^{-1}\Div_x
\)
is bounded on $L^q(\R;L^p(\R^n;\C^n))$. Fix $\Psi,\widetilde\Psi\in\cS(\R^n;\C^n)$ and define for $h \in L^q(\R)$
\[
T_{\Psi,\widetilde\Psi}h(t)
=
\ipb{
\nabla_x(\partial_t+L)^{-1}\Div_x(h \otimes \Psi)(t),
\widetilde\Psi
} .
\]
Then $T_{\Psi,\widetilde\Psi}$ is bounded on $L^q(\R)$.
After Fourier transform in $t$, this scalar operator has multiplier
\[
m_{\Psi,\widetilde\Psi}(\tau)
=
-\left\langle
(i\tau+L)^{-1}\Div_x\Psi,
\Div_x\widetilde\Psi
\right\rangle.
\]
 Since any Fourier multiplier on $L^q(\R)$ is bounded pointwise {a.e.}~by the operator norm, we obtain for almost every $\tau$  
\[
|m_{\Psi,\widetilde\Psi}(\tau)| \le C \|\Psi\|_{L^{\vphantom{p_1'}p}_x}\|\widetilde \Psi\|_{L^{{\vphantom{p_1'}}p'}_x}.
\]
Let $\psi=\Div_x \Psi$, $\tilde \psi= \Div_x \widetilde \Psi $. Then they both belong to $\cS(\R^n)$ and $-m_{\Psi,\widetilde\Psi}$ is the Fourier transform of 
$f(t):=\ind_{(0,\infty)}(t)\langle  e^{-tL}  \psi, \tilde \psi\rangle$.
Note that $f$ is bounded, using that both $\psi,\tilde \psi$ are in $L^2(\R^n)$. 
Next,  using the off-diagonal estimates we have for any $p_-(L)<p_1\le p_2<p_+(L)$ and all $t> 0$ 
\[ |f(t)|\lesssim t^{-\tfrac{n}{2}[p_1,p_2]}\|\psi\|_{L^{\vphantom{p_2'}p_1}_x}\|\tilde \psi\|_{L^{p_2'}_x}.\]
As $[p_-(L),p_+(L)]> \tfrac{2}{n}$ when $n\ge 3$, one can choose $p_1,p_2$ with $\tfrac{n}{2}[p_1,p_2]>1$. The argument for $n=2$ requires a stronger property that the semigroups generated by $-L$ and $-L^*$ map the Hardy space $H^p(\R^n)$ to $L^2(\R^n)$ for some $p<1$ with bound $t^{1/2-1/p}$. This follows from the results in \cite[Chapters 6 \& 12]{Auscher-Egert2023-book}. Using $\langle e^{-tL}  \psi, \tilde \psi\rangle= \langle  e^{-(t/2)L}  \psi, e^{-(t/2)L^*}\tilde \psi\rangle$ and that $\psi, \tilde\psi\in H^p(\R^n)$ as the divergence of  Schwartz functions we obtain an integrable decay at $\infty$. For $n=1$, this mapping property is for all $1/2<p<1$ with bound $t^{1/4-1/2p}$, which again gives integrable decay at $\infty$. It follows that in any dimension, $f$ is integrable on $\R$ and thus $m_{\Psi,\widetilde\Psi}$ is continuous. In particular at $\tau=0$, we have obtained
\[|\langle   L^{-1}\Div_x \Psi, \Div_x \widetilde \Psi\rangle|\le C \|\Psi\|_{L^{\vphantom{p_1'}p}_x}\|\widetilde \Psi\|_{L^{{\vphantom{p_1'}}p'}_x}\]
for any $\Psi, \widetilde \Psi\in \cS(\R^{n}; \C^{n})$.
Thus, the operator $\nabla_xL^{-1}\Div_x$ extends to a bounded operator on $L^p(\R^{n}; \C^{n})$ and by \cite[Theorem~13.12]{Auscher-Egert2023-book}, this implies that $p\in (q_+(L^*)', q_+(L))$.
\end{proof}

We now specify boundedness of the operators in the autonomous case and verify the conjecture in Remark~\ref{rem:conjecture} in this case. We also consider the \emph{maximal regularity operator}
\[ \cL^A_0 f(t) := \int_0^t Le^{-(t-s)L}f(s) \dd s,  \quad f \in L^{\infty}_c(\R^{1+n}_+). \]

\begin{theorem}[Boundedness in the autonomous case]
    \label{thm:parabolic-autonomous}
    Assume $A=A(x)$. Let $p,q \in (0,\infty]$. 
    \begin{enumerate}[label=\normalfont(\arabic*)]
        \item \label{item:Duhamel-auto}
        Let $p_-(L)<r<p_+(L)$ and
        \[ \beta>-[1,q]+\tfrac{n}{2}[\min\{p,q,p_-(L)\},p_-(L)]. \]
        Then the Duhamel operator $\cL^A_1$ extends to a bounded operator from $T^{p,q,r}_\beta$ to $T^{p,q,r}_{\beta+1}$, and the maximal regularity operator $\cL^A_0$ extends to a bounded operator on $T^{p,q,r}_\beta$. 
        
        Moreover, if $p_-(L) < r < q_+(L)$, then $\cL^A_{1/2}$ extends to a bounded operator from $T^{p,q,r}_\beta$ to $T^{p,q,r}_{\beta+1/2}$.

        \item \label{item:Lions-auto}
        Let $q_+(L^*)'<r<p_+(L)$ and
        \[ \beta>-[1,q]+\tfrac{n}{2}\left[ \min\{p,q,q_+(L^*)'\},q_+(L^*)' \right]. \]
        Then the Lions operator $\cR^A_{1/2}$ extends to a bounded operator from $T^{p,q,r}_\beta$ to $T^{p,q,r}_{\beta+1/2}$.
        
        Moreover, if $q_+(L^*)' < r < q_+(L)$, then $\cR^A_{0}$ extends to a bounded operator on $T^{p,q,r}_\beta$.
    \end{enumerate}
\end{theorem}

\begin{proof} 
    First consider \ref{item:Duhamel-auto}. For $\cL^A_1$, as $p_\pm(A)=p_\pm(L)$, the statement is the same as in the non-autonomous case (cf. Theorem \ref{thm:Duhamel-tent}\ref{item:LA1-bd-tent}). 
    
    For $\cL^A_0$, the same reasoning as in Lemma~\ref{lemma:Duhamel-SIO}\ref{item:LA1-SIO} for $\cL^A_1$ also shows that $\cL^A_0$ is a singular operator of type $(0,2,r,\boldu,\boldv,M)$ when $p_-(L) < r < p_+(L)$, $\boldu = (1,p_1)$, $\boldv=(\infty,p_2)$ with $p_-(L) < p_1 \le p_2 < p_+(L)$, and $M \ge 0$. To verify condition \ref{it:Lr}, we only need to mention that $\cL^A_0$ is bounded in $L^r(\R^{1+n}_+)$ follows from the $\cR$-boundedness in $L^r(\R^n)$. The conditions \ref{item:kernelODE-t-sep} and \ref{item:kernelODE-x-sep} follow immediately from \cite[Proposition 3.15]{A07}.
    
    For $\cL^A_{1/2}$, when $p_-(L)<r<q_+(L)$, the family $(\sqrt{t-s}\,\nabla_x e^{-(t-s)L})_{0 \le s < t < \infty}$ is uniformly bounded in $L^r(\R^n)$ and enjoys $t$-pointwise $L^{p_1}-L^{p_2}$ off-diagonal estimates (cf. Lemma~\ref{lem:ODEGammaA}) for $p_-(L) < p_1 \le p_2 < q_+(L)$ (which is optimal). Hence, we can directly show that $\cL^A_{1/2}$ is bounded from $L^r(\R^{1+n}_+)$ to $L^r_{1/2}(\R^{1+n}_+)$ and adapt the argument in Lemma~\ref{lemma:Duhamel-SIO}\ref{item:LA1-SIO} to get that $\cL^A_{1/2}$ is a singular operator of type $(1/2,2,r,\boldu,\boldv,M)$ when $p_-(L) < r < q_+(L)$, $\boldu = (1,p_1)$, $\boldv=(\infty,p_2)$ with $p_-(L) < p_1 \le p_2 < q_+(L)$, and $M \ge 0$. This proves \ref{item:Duhamel-auto}.
    
    The statements in \ref{item:Lions-auto} for $\cR^A_{1/2}$ and $\cR^A_{0}$ follow similarly, since $v_-(A)=v_+(A^*)'=q_+(L^*)'$ and $v_+(A)=q_+(L)$. This also refines the singular operator properties of $\cR^A_{1/2}$ and $\cR^A_0$ asserted in Lemma \ref{lemma:Lions-SIO}. Indeed, for $\cR^A_{1/2}$, it yields that $\boldu=(1,p_1)$ and $\boldv=(\infty,p_2)$ with $q_+(L^\ast)'<p_1 \le p_2 < p_+(L)$. For $\cR^A_0$, it gives $\boldu=(1,p_1)$ and $\boldv=(\infty,p_2)$ with $q_+(L^\ast)'<p_1 \le p_2 < q_+(L)$.
    
    This completes the proof.
\end{proof}

\section{Boundedness of singular operators on \texorpdfstring{$Z$}{Z}-spaces}\label{sec:extensions}

We conclude this paper with further estimates of singular operators derived from the form domination
in Section~\ref{sec:domination}.
Namely, we consider the twin $Z$-spaces to the tent spaces.
For $p,q,r\in (0,\infty]$, $\beta\in \R$, 
we define $Z^{p,q,r}_{\beta}$ as the space of all measurable $f \colon \R_+^{1+n} \to \C$ such that
\begin{align*}
\nrm{f}_{Z^{p,q,r}_{\beta}}
  &:=
  \nrm{W_r f}_{L^q_\beta(\R_+\,;\, L^p_{\vphantom{\beta}}(\R^n))}<\infty.
\end{align*}

At $p=q$, $Z$-spaces and tent spaces agree by Fubini's theorem.
When $p=q=r$, this is the weighted space $L^r_\beta(\R_+^{1+n})= L^r_{\vphantom{\beta}}(\R_+^{1+n}; t_{\vphantom{\beta}}^{-\beta r}\dd x \dd t)$.  

We remark that $Z$-spaces are real interpolants of tent spaces; see \cite{Haa26} for a complete description. However, interpolation would not give as good results as  the direct estimates below. 
The reversal of the order in taking the norms compared to tent spaces will show that the spatial parts of the exponents depend on $p$ only, which makes the decay exponent $M$ independent of $q$.  

For a direct proof of the estimates, one starts again from the dominations in the Banach and quasi-Banach ranges in Section~\ref{sec:domination}, and one has to prove 
boundedness results for the model operators as in Section~\ref{sec:modelops}.  In contrast with tent spaces, they come with no logarithmic loss at $p=\infty$ and no loss in the exponent for the  aperture  $\delta$. We shall use the dyadic characterization 
\[
  \nrm{f}_{Z^{p,q,r}_\beta} \eqsim \nrmB{ \nrmB{  \sum_{Q \in \mc{D}, \ell(Q)=2^{-j}} \ell(Q)^{m/q-m\beta }\cdot  W_{r,Q}(f) \ind_Q } _{L^p_{\vphantom{\beta}}(\R^n)}}_{\ell^q_{\vphantom{\beta}}(\Z)}
\]
of the norm corresponding to Proposition~\ref{prop:dyadic}, which is valid for all exponents, see \cite{Haa26}.

\begin{lemma}\label{lemma:WhitneymapZ}
Let $p,q,r \in (0,\infty]$, $\beta,\kappa \in \R$ . Then
$$
  A_{r,\kappa}^{\whit} f:=  \sum_{Q \in \mc{D}} \ell(Q)^{m\kappa} \cdot \ip{f}_{r,\Qwtilde{Q}} \ind_{\Qw{Q}}, \qquad f\in L^0(\R^{1+n}_+),
$$
defines a bounded operator  from $Z^{p,q,r}_\beta$ to $Z^{p,q,r}_{\beta+\kappa}$.
\end{lemma}

\begin{proof} This is immediate for the dyadic characterization, using that $\Qwtilde{Q}$ is contained in a finite union of Whitney regions $\Qw{P}$, with a bounded number of dyadic cubes $P$ of size comparable to $Q$. 
 \end{proof}

\begin{lemma}\label{lem:sliceZ}
    Let $u \in (0,\infty)$, $p,q,r \in (0,\infty]$ with $u\leq r$, $\beta \in \R$  and $\delta \ge 1$. Set
\[
  \gamma
  :=\tfrac{n}{m}
  \bracb{\min\cbrace{p,u},u}.
\]
Then for $A_{u,\delta}$ given by \eqref{eq:avop} we have for all $f \in Z^{p,q,r}_\beta$ that
\[
  \nrm{A_{u,\delta}f}_{Z^{p,q,r}_\beta}
  \lesssim
  \delta^{\gamma}\,
  \|f\|_{Z^{p,q,r}_\beta}.
\]
The implicit constant is independent of $\delta$.
\end{lemma}

\begin{proof} We start from the upper bound 
\[ W_{r,Q}(A_{u,\delta} f)\lesssim\brB{ \frac{1}{N} \sum_{k \in \Lambda_\delta} W_{r,Q_k}(f)^u }^{\frac{1}{u}},
  \]
  where $N\eqsim \delta^{n/m}$
as in the proof of Lemma~\ref{lem:slice}. Then take the $L^p(\R^n)$-norm. By Jensen's inequality if $p\ge u$ and subadditivity of $s\mapsto s^{p/u}$ if $p<u$, we have
\[ \bigg( \sum_{ \ell(Q)=2^{-j}} |Q|\cdot  W_{r,Q}(A_{u,\delta} f)^p\bigg)^{1/p} \lesssim 
 N^{\max\{0, 1/p-1/u\}} \bigg( \sum_{ \ell(Q)=2^{-j}} |Q|\cdot  W_{r,Q}(f)^p\bigg)^{1/p},
 \]
for each $j\in \Z$. 
This gives the result when $p<\infty$ by taking the $\ell^q(\Z)$ norm. The modification when $p=\infty$ is straightforward. 
\end{proof}

With the above lemma, the following proposition for the model Hardy operators is a straightforward adaptation of the corresponding results established in Section~\ref{sec:modelops} on tent spaces. We leave details to the interested reader.

\begin{prop}
\label{prop:modelops-whitneymixedZ}
Let  $p, q,r \in (0,\infty]$, \(\beta\in\R\), \(\delta\in[1,\infty)\). Let
\(\vec u\in(0,\infty)^2\) satisfy \(u_t,u_x\le r\) and set
\begin{align*}
  \gamma
  :=
  \tfrac{n}{m}
  {\bracb{\min\cbrace{p,u_x},u_x}.}
\end{align*}
For all $f \in Z^{p,q,r}_{\beta}$ we have
  \begin{align*}
    \nrm{H^{\whit}_{\vec u,\delta}f}_{Z^{p,q,r}_{\beta}}
    \lesssim  \delta^{\gamma}\,\nrm{f}_{Z^{p,q,r}_{\beta}},
  \end{align*}
 If $
    \beta>\gamma-[u_t,q]$, then
    we also have
  \begin{align*}
    \nrm{H^{\ca}_{\vec u,\delta}f}_{Z^{p,q,r}_{\beta}}
    \lesssim
    \delta^{\gamma}\,
    \nrm{f}_{Z^{p,q,r}_{\beta}}.
  \end{align*}
 The implicit constants are independent of $\delta$.
\end{prop}

Finally, the main result for $Z$-spaces corresponding to the tent space estimates in Theorem~\ref{thm:mainestimate} and the extension result in Theorem~\ref{thm:extension} is the following.

\begin{theorem}\label{thm:mainestimateZ}
Let $r \in (1,\infty)$, $\vec{u} \in [1,\infty)^2$ and $\vec{v} \in (1,\infty]^2$ be such that
$$\max\cbrace{u_t, u_x} \leq r\leq \min\cbrace{v_t, v_x}.$$ 
Let $\kappa, M\in\R$ and suppose that  $T$ is a singular operator of type $(\kappa,m,r,\vec u,\vec v, M)$.
Let $p,q\in (0,\infty]$ and $\beta\in\R$, and assume the decay condition
\begin{align}
    \tag{dZ}
    \label{e:main-decay-conditionZ}
    M &> -[u_t,v_t] +\tfrac{n}{m}  {\bracb{\min\cbrace{p,u_x},u_x}} + \tfrac{n}{m}{\bracb{v_x,\max\cbrace{p,v_x}} } + \kappa \\
\intertext{
and the lower and upper bounds
}
    \tag{lZ}
    \label{e:main-beta-lower-boundZ}
    \beta &> -[u_t,q]+ \tfrac{n}{m}{\bracb{\min\cbrace{p,u_x},u_x}}, \\
    \tag{uZ}
    \label{e:main-beta-upper-boundZ}
    \beta&< \phantom{-}[q,v_t] -\tfrac{n}{m}{\bracb{v_x,\max\cbrace{p,v_x}} }-\kappa.
\end{align}
Then we have 
\begin{equation*}
  \nrm{Tf}_{Z^{p,q,r}_{\beta+\kappa}} \lesssim \nrm{f}_{Z^{p,q,r}_{\beta}},
\end{equation*}
for all $f \in L^\infty_c(\R^{1+n}_+)$.
If $T$ is causal,  the upper bound \eqref{e:main-beta-upper-boundZ} is omitted. If $T$ is anti-causal,  the lower bound \eqref{e:main-beta-lower-boundZ} is omitted.
Moreover, if $T$ satisfies the Lipschitz pointwise estimate \eqref{eq:pointwiseLipschitz}, $T$ has an extension to all $Z^{p,q,r}_{\beta}$ as in Theorem~\ref{thm:extension}. 
\end{theorem}

\begin{remark}
 If $u_{x}\le p\le v_{x}$, the result only requires
 \[ M > -[u_t,v_t] + \kappa, \quad \beta> -[u_t,q], \quad \beta< [q,v_t] -\kappa.
 \]
\end{remark}

\subsection*{AI disclosure}
During the preparation of this paper, ChatGPT was used for several minor tasks. More precisely, draft versions of Figures~\ref{fig:EF-picturesdef} and~\ref{fig:EF-pictures} were prepared using ChatGPT 5.2. Moreover, ChatGPT 5.6 Sol Pro assisted in the $q<p,u$ case in Remark \ref{remark:noepsilonloss}, developing the sharpness discussion following the main tent-space estimate in Subsection~\ref{sec:apriori} and working out the routine details of the \(Z\)-space extension in Section~\ref{sec:extensions}. Furthermore, we found a slight simplification,
based on the observation in Remark~\ref{rem:HwODE},
of the extension procedure in Subsection~\ref{sec:extension}. Finally, during the final stages of preparation, GPT-5.6 Sol, accessed through Codex, was used to identify typos, mathematical errors and missing literature.

\bibliographystyle{alpha}
\bibliography{bibliography}

\end{document}